\documentclass[11pt,a4paper,  twoside, titlepage]{report}
\usepackage{etex}
\usepackage[latin1]{inputenc}
\usepackage[T1]{fontenc}
\usepackage[ngerman, english]{babel} 
\usepackage[fleqn]{amsmath}
\usepackage{amssymb, amscd, amsthm}
\usepackage{mathrsfs}
\usepackage{placeins}
\usepackage{tikz}
\usepackage{cases}
\usepackage{multicol}
\usepackage{enumerate}

\usepackage{enumitem}
\usepackage{rotating}
\usetikzlibrary{matrix}

\usepackage{booktabs}
\usepackage{array}
\usepackage{xcolor}
\usepackage{longtable}
\usepackage{ltcaption}
\usepackage{caption}
\usepackage[dvips, all,  color]{xy}
\usepackage[a4paper, includehead, left=3cm, right=4cm, bottom=3.5cm]{geometry}
\usepackage{fancyhdr}
\usepackage[pdfpagelabels]{hyperref}
\usepackage{pdflscape}

\newcounter{enumglobal}

\newenvironment{myalign}{%
    \setlength{\mathindent}{0pt}%
    \setlength{\abovedisplayskip}{-\baselineskip}%
    \setlength{\abovedisplayshortskip}{\abovedisplayskip}%
    \align
  }%
  {\endalign}
\newenvironment{myalignst}{%
    \setlength{\mathindent}{0pt}%
    \setlength{\abovedisplayskip}{-\baselineskip}%
    \setlength{\abovedisplayshortskip}{\abovedisplayskip}%
    \csname align*\endcsname}
  {\csname endalign*\endcsname}

\newcommand{\fr}{\mathfrak}
\newcommand{\und}{\underline}
\newcommand{\ove}{\overline}
\def\re{\color{red}}
\def\blu{\color{cyan}}
\def\gre{\color{green}}
\def\ora{\color{orange}}
\def\deg{\operatorname{deg}}
\def\cups{\operatorname{cups}}
\def\caps{\operatorname{caps}}

\def\filt{\operatorname{filt}}

\def\res{\operatorname{res}}
\def\down{\vee}
\def\up{\wedge}
\def\id{\operatorname{id}}
\def\Id{\operatorname{Id}}
\def\C{{\mathbb C}}
\def\F{{\mathbb F}}
\def\D{{\mathcal D}}
\def\A{{\mathcal A}}

\def\K{{\mathcal K}}

\def\Z{{\mathbb Z}}

\def\O{{\mathcal O}}

\def\hom{{\operatorname{Hom}}}
\def\Ext{{\operatorname{Ext}}}
\def\End{{\operatorname{End}}}

\def\Mod{{\operatorname{Mod}}}
\def\mod{{\operatorname{mod}}}
\def\gmod{{\operatorname{gmod}}}
\def\res{{\operatorname{res}}}

\def\d{{\operatorname{d}}}
\def\nes{{\operatorname{nes}}}
\def\eps{{\varepsilon}}
\def\phi{{\varphi}}
\def\emptyset{{\varnothing}}
\def\la{{\lambda}}
\def\La{{\Lambda}}
\def\Ga{{\Gamma}}
\def\ga{{\gamma}}

\def\al{{\alpha}}
\def\be{{\beta}}
\def\alch{{\al \check{ \ }}}
\def\p{{\mathfrak p}}
\numberwithin{equation}{chapter}
\newtheorem{satz}{Satz}[chapter]
\newtheorem{Theorem}[satz]{Theorem}
\newtheorem{Cor}[satz]{Corollary}
\newtheorem{Lemma}[satz]{Lemma}
\newtheorem{Remark}[satz]{Remark}
\newtheorem{Example}[satz]{Example}
\newtheorem{Prop}[satz]{Proposition}
\newtheorem{Conjecture}[satz]{Conjecture}
\theoremstyle{definition}
\newtheorem{Def}[satz]{Definition}
\newtheorem{Notation}[satz]{Notation}

\renewcommand{\chaptermark}[1]{\markboth{#1}{}}
\renewcommand{\sectionmark}[1]{\markright{#1}}
\fancypagestyle{plain}{%
\fancyhead{} 
}			

\let\origdoublepage\cleardoublepage
\newcommand{\clearemptydoublepage}{%
  \clearpage
  {\pagestyle{empty}\origdoublepage}%
}

\begin{document}
	\pagestyle{empty}
	\begin{titlepage}%
\Large
	\begin{center}
		\textsc{Diplomarbeit}\\
		\bigskip
		\textit{$A_\infty$-Strukturen der $\Ext$-Algebra von Verma Moduln in der parabolischen Kategorie $\O$ \\ ($A_\infty$-structures on the algebra of extensions of Verma modules in the parabolic category $\O$)}
	\end{center}
	\vspace{\stretch{1}}
	\begin{center}
		Angefertigt am \\
		Mathematischen Institut
	\end{center}
	\vspace{\stretch{1}}
	\begin{center}
		Vorgelegt der \\
		Mathematisch-Naturwissenschaftlichen Fakult\"at der \\
		Rheinischen Friedrich-Wilhelms-Universit\"at Bonn
	\end{center}
	\vfill
	\begin{center}
		September 2010\\
		\bigskip
		von\\
		Angela Klamt\\
		\bigskip
		aus\\
		Leverkusen
	\end{center}
\end{titlepage}%

	\cleardoublepage
	
	\pagestyle{fancy}
	\pagenumbering{Roman}
	
	\tableofcontents
	
	\clearemptydoublepage

	\pagenumbering{arabic}
	\pagestyle{fancy}

%

\chapter*{Introduction}
\chaptermark{Introduction}
\addcontentsline{toc}{chapter}{Introduction}

\section*{English introduction}
\sectionmark{English introduction}
\addcontentsline{toc}{section}{English introduction}

In 1988 Shelton computed the dimensions of the $\Ext$-spaces $\Ext^k(M(\la), M(\mu))$ of Verma modules $M(\la)$ and $M(\mu)$ in the parabolic category $\O$ in the Hermitian symmetric cases \cite{Shel1988}. More recently Biagioli reformulated these recursion formulas combinatorially and gave a closed formula \cite{Biag2004}. A nice feature of (parabolic) Verma modules is that they form an exceptional sequence in the sense that there is a partially ordered set $(\La, \leq)$ of highest weights labelling these Verma modules such that for all $i\geq 0$ holds:
$$\hom(M(\la), M(\la))=\C \text{ and } \Ext^i(M(\la),M(\mu))=0 \text{ unless } \la \leq \mu.$$
The set $\La$ is infinite, but the category $\O^\p$ decomposes into indecomposable summands, so-called blocks, which each contain only finitely many parabolic Verma modules. In particular, restriction to the principal block yields finitely many Verma modules. Moreover, taking $M$ to be the direct sum of all Verma modules in this block leads to an algebra $\Ext(M,M)=\bigoplus e_\mu \Ext(M(\mu), M(\la)) e_\la$ with idempotents $e_\mu$ and $e_\la$. Therefore, we have much more structure if we regard it as an algebra than only looking at each vector space. The algebra structure can be obtained by viewing $\Ext(M,M)$ as the homology of the algebra $\hom(P_\bullet, P_\bullet)$ with $P_\bullet$ a projective resolution of $M$. In this situation multiplication reduces to the composition of maps between complexes. The construction of these projective resolutions and chain maps requires a deeper knowledge of the projective modules and morphisms between them. Note that already the question about non-vanishing $\hom$-spaces between parabolic Verma modules is non-trivial (cf. \cite{Boe85} or \cite[Theorem 9.10]{Hump08}).

 In $\cite{brun32008}$ Brundan and Stroppel developed a combinatorial description of the category $\O^\mathfrak{p}$ via a slight generalization of Khovanov's diagram algebra in the case of $\mathfrak{g}=\mathfrak{gl}_{n+m}$ and $\mathfrak{p}$ the parabolic subalgebra with Levi component $\mathfrak{gl}_n \oplus \mathfrak{gl}_m$ (cf. Theorem \ref{Th:equcat}). Using these combinatorial techniques along with classical Lie theoretical results, in the first part of this thesis we are able to compute projective resolutions and their morphisms. One crucial technical result, which is needed later, is a new proof of a version of the Delorme-Schmid theorem (cf. \cite{Delo77}, \cite{schm81}) in our situation.
 
 The main results of the first part are Theorems \ref{mainalg1} and \ref{mainalg2}:
 \begin{itemize}
 \item an explicit description of the $\Ext$-algebra in terms of a path algebra of a quiver with relations in the case for $n=1$ and $n=2$, respectively.
 \end{itemize}
The first algebra also occurs while analysing Floer Homology, as Khovanov and Seidel point out in \cite{khovanov2002quivers} (for more details cf. \cite{Asae2008}).
Since we finally want to obtain the algebra with all its structure as an algebra over the field $\C$, we carefully worked out the signs of the multiplication, too.

In the context of Fukaya categories these algebras come along with a natural $A_\infty$-structure. $A_\infty$-algebras are a generalization of algebras encoding more information about the object. An $A_\infty$-algebra, also known in topology as a strongly homotopic associative algebra, has higher multiplications satisfying so-called Stasheff relations (cf. \cite{Kell2001}). For example Keller points out that working with minimal models provides the possibility to recover the algebra of complexes filtered by a family of modules $M(i)$ from some $A_\infty$-structure on $\Ext(\bigoplus M(i), \bigoplus M(i))$. This $A_\infty$-structure is constructed as a minimal model, i.e. deduced from the algebra structure on $H^*(\hom(\bigoplus P(i)_{\bullet}, \bigoplus P(i)_{\bullet}))$. This kind of $A_\infty$-structure becomes interesting and is a natural structure on $\Ext(\bigoplus M(i), \bigoplus M(i))$.

Therefore, in the second part of the thesis we construct minimal models for our $\Ext$-algebras from above. Since we already viewed $\Ext(M,M)$ as the $\hom$-algebra of the projective resolutions, the previous results allow us to analyse the higher multiplications. For the construction of the minimal models we use a similar approach as it is worked out in \cite{Lu2009}. We combine formulas obtained by Merkulov \cite{Merk99} (for the case of superalgebras) and Kontsevich and Soibelman \cite{Kont2001}(for the $\F_2$-case). Again we have to keep track of the signs (which sometimes yields to tedious computations).

As a result, for $n=1$ we achieve the \textbf{first vanishing theorem} (Theorem \ref{Th:1stvanish}). In this theorem we prove the formality of the $\Ext$-algebra, i.e. we construct a minimal model with vanishing $m_k$ for $k \geq 3$. 

For $n=2$ in the \textbf{second vanishing theorem} (Theorem \ref{Th:2ndvanish}) we get an $A_\infty$-structure with non-vanishing $m_3$ but vanishing $m_k$ for $k \geq 4$. Therefore, we obtain an example of an $A_\infty$-algebra with non-trivial higher multiplications.

The main result of this thesis is presented in the \textbf{general vanishing theorem} (Theorem \ref{Th:genvanish}). It says that for arbitrary $n$ we get a minimal model with vanishing $m_k$ for $k \geq n^2+2$. The tools used for this proof are developed throughout the entire thesis.
\\
\\
\textbf{Structure of the thesis:} Part I of the thesis presents as the main result the algebra structure. In Chapter \ref{sec:Cat} the properties of the categories $\O$ and $\O^\p$ are stated. A review of the tools needed from homological algebra is given in Chapter \ref{ch:NotHom}. Chapter \ref{ch:Khov} starts with the basic theorem about the equivalence between $\O^\p$ and $K_m^n-\mod$. Later on the required definitions about $K_m^n$ are introduced and results about the grading, the endomorphisms of projective modules and the shape of projective resolutions are proved. In Chapter \ref{ch:Ext} we state Shelton's results and work them out for the Hermitian pair  ($\fr{sl}_{m+n}$, $(\fr{gl}_m \oplus \fr{gl}_n) \cap \fr{sl}_{m+n}$). Subsequent, in Chapter \ref{ch:spec}, we compute the algebra structure for $n=1$ and $n=2$. For ease of presentation most of the tedious computations are performed in the Appendix. At the end of the first part in Chapter \ref{ch:Koszul} the achieved results together with Koszul duality are used to give a proof of a graded version of Verma's Theorem in special cases.

The second part of the thesis deals with the $A_\infty$-structure. A short introduction and some basic ideas are given in Chapter \ref{ch:Ainf}. Finally in Chapter \ref{ch:expAinf} results from the first part and the previous chapter are combined to prove the main theorems mentioned above.
\\
\\
\textbf{Acknowledgements:} I would like to thank all those who supported me during the writing progress. My special gratitude belongs to my advisor Professor Catharina Stroppel for her encouraging support and helpful advice.
\newpage
\selectlanguage{ngerman}

\section*{German introduction}
\sectionmark{German introduction}
\addcontentsline{toc}{section}{German introduction}

1988 hat Shelton die Dimension der $\Ext$-Räume $\Ext(M(\la), M(\mu))$ von Vermamoduln $M(\la)$ und $M(\mu)$ in der parabolischen Kategorie $\O$ im hermitisch symmetrischen Fall berechnet (s. \cite{Shel1988}). Später hat Biagioli diese Rekursionsformeln umformuliert und in einer geschlossene Formel angegeben (s. \cite{Biag2004}). Eine besondere Eigenschaft von (parabolischen) Vermamoduln ist, dass diese eine exzeptionelle Folge bilden, das heißt, dass es eine partiell geordnete Menge $(\La, \leq)$ gibt, welche diese Vermamoduln indiziert und für alle $i\geq 0$ gilt:
$$\hom(M(\la), M(\la))=\C \text{ und } \Ext^i(M(\la),M(\mu))=0 \text{ außer für } \la \leq \mu.$$

Die Menge $\La$ ist nicht endlich, aber die Kategorie $\O^\p$ zerfällt in unzerlegbare Summanden, so genannte Blöcke, wovon jeder nur endlich viele Vermamoduln enthält. Insbesondere erhält man durch Einschränkung auf den prinzipalen Block eine Kategorie mit endlich vielen Vermamoduln. Wenn man $M$ als die direkte Summe aller Vermamoduln in diesem Block wählt, so erhält man eine Algebra $\Ext(M,M)=\bigoplus e_\mu \Ext(M(\mu), M(\la)) e_\la$ mit Idempotenten $e_\mu$ und $e_\la$. Wenn wir diesen Raum nun also als Algebra und nicht nur als Vektorraum betrachten, haben wir mehr Struktur, welche man erhalten kann, indem man die Homologie der Algebra $\hom(P_\bullet, P_\bullet)$ bestimmt, wobei $P_\bullet$ eine projektive Auflösung von $M$ ist. In dieser Situation wird die Multiplikation zur Verknüpfung von Abbildungen zwischen Komplexen. Die Konstruktion dieser Auflösung und der Kettenabbildungen erfordert detaillierte Kenntnis der projektiven Moduln. Schon die Frage, ob Abbildungen zwischen parabolischen Vermamoduln existieren, ist nicht trivial (s. \cite{Boe85} oder \cite[Theorem 9.10]{Hump08}).

In $\cite{brun32008}$ entwickeln Brundan und Stroppel mit Hilfe einer leichten Verallgemeinerung von Khovanovs Diagrammalgebra eine kombinatorische Beschreibung der Kategorie $\O^\mathfrak{p}$ für den Fall $\mathfrak{g}=\mathfrak{gl}_{n+m}$ und $\mathfrak{p}$ einer parabolischen Unteralgebra mit Levikomponente $\mathfrak{gl}_n \oplus \mathfrak{gl}_m$ (s. Theorem \ref{Th:equcat}). Unter Benutzung dieser kombinatorischen Techniken und klassischer Lie-theoretischer Resultate können wir im ersten Teil der Arbeit projektive Auflösungen und Morphismen zwischen diesen bestimmen. Ein wichtiges Resultat, welches wir später verwenden, ist ein Beweis einer Version des Delorme-Schmid Theorems (s. \cite{Delo77}, \cite{schm81}) in unserem Fall.

Die Hauptresultate des ersten Teils sind die Theoreme \ref{mainalg1} und \ref{mainalg2}:
 \begin{itemize}
 \item Eine explizite Beschreibung der $\Ext$-Algebra mit Hilfe von Wegealgebren eines Köchers mit Relationen für die Fälle $n=1$ bzw. $n=2$.
 \end{itemize}
 
Wie Khovanov und Seidel in \cite{khovanov2002quivers} erwähnen, taucht die erste Algebra auch bei der Analyse von Floer Homologie auf (für mehr Details s. \cite{Asae2008}). 
Da wir die Algebren mit ihrer gesamten Struktur als Algebren über $\C$ berechnen wollen, ist es notwendig, dass wir auch die Vorzeichen in den Multiplikationsregeln bestimmen.

Im Zusammenhang von Fukaya-Kategorien tragen diese Algebren natürliche $A_\infty$-Strukturen. $A_\infty$-Algebren sind eine Verallgemeinerung von Algebren, welche mehr Information über das Objekt beinhalten. Eine $A_\infty$-Algebra besitzt höhere Multiplikationen, welche die sogenannten Stasheff-Relationen erfüllen (s. \cite{Kell2001}). Keller zeigt in seiner Arbeit, dass es mithilfe der Theorie der minimalen Modelle möglich ist, aus einer $A_\infty$-Struktur auf $\Ext(\bigoplus M(i), \bigoplus M(i))$ die Kategorie der Komplexe, welche durch eine Familie von Moduln $M(i)$ filtriert sind, wiederzugewinnen.
Diese $A_\infty$-Struktur wird als minimales Modell konstruiert, d.h. sie kommt von einer $A_\infty$-Struktur auf $H^*(\hom(\bigoplus P(i)_{\bullet}, \bigoplus P(i)_{\bullet}))$. Eine solche $A_\infty$-Strukturen ist natürlich auf $\Ext(\bigoplus M(i), \bigoplus M(i))$.

Aufgrund dieser Betrachtungen konstruieren wir im zweiten Teil der Arbeit minimale Modelle für unsere $\Ext$-Algebra. Da wir diese schon zuvor als die Homologie der $\hom$-Algebra der projektiven Auflösungen betrachtet haben, können wir die Resultate aus dem ersten Teil nutzen um die höheren Multiplikationen zu analysieren. Zunächst konstruieren wir minimale Modelle mit einem ähnlichen Ansatz wie in \cite{Lu2009}. Dazu kombinieren wir Merkulovs Formeln \cite{Merk99} (im Fall für Superalgebren) und die von Kontsevich und Soibelman \cite{Kont2001} (im $\F_2$-Fall). Auch hier müssen wir auf die Vorzeichen achten (was manchmal zu sehr langwierigen Rechnungen führt).

Als Ergebnis zeigen wir für $n=1$ das "`\textbf{first vanishing theorem}"'(Theorem \ref{Th:1stvanish}). In diesem Satz beweisen wir die Formalität der $\Ext$-Algebra, d.h. wir konstruieren ein minimales Modell bei dem alle $m_k$ mit $k \geq 3$ verschwinden.

Für $n=2$ erhalten wir im "`\textbf{second vanishing theorem}"' (Theorem \ref{Th:2ndvanish}) eine $A_\infty$-Struktur mit nichtverschwindenen $m_3$, jedoch $m_k=0$ für $k \geq 4$. Daher haben wir ein Beispiel für eine $A_\infty$-Algebra mit nicht trivialen höheren Multiplikationen konstruiert.

Das Hauptresultat dieser Arbeit ist das "`\textbf{general vanishing theorem}"' (Theorem \ref{Th:genvanish}). Es besagt, dass es für beliebiges $n$ ein minimales Modell gibt, so dass $m_k$ für $k \geq n^2+2$ verschwindet. Die für den Beweis benötigten Hilfsmittel werden in der gesamten Arbeit entwickelt.
\\
\\
\textbf{Struktur der Arbeit:} 
Im ersten Teil der Arbeit werden die Hauptresultate über die Struktur der Algebra erarbeitet. In Kapitel \ref{sec:Cat} wird eine Einführung in die Kategorien  $\O$ and $\O^\p$ gegeben. Einen Überblick über die wichtigsten Hilfsmittel aus der homologischen Algebra findet man in Kapitel \ref{ch:NotHom}. Kapitel \ref{ch:Khov} beginnt mit dem wichtigen Satz über die Äquivalenz der Kategorien $\O^\p$ und $K_m^n-\mod$. Im Folgenden werden die benötigten Definitionen über $K_m^n$ eingeführt und Resultate über die Graduierung, Endomorphismen von projektiven Moduln und die Form der projektiven Auflösungen bewiesen. In Kapitel \ref{ch:Ext} stellen wir Sheltons Ergebnisse vor und arbeiten diese für das hermitische Paar ($\fr{sl}_{m+n}$, $(\fr{gl}_m \oplus \fr{gl}_n) \cap \fr{sl}_{m+n}$) aus. Anschließend, in Kapitel \ref{ch:spec}, berechnen wir die Algebrenstruktur für $n=1$ und $n=2$. Um die Lesbarkeit zu erhöhen, sind die meisten aufwändigen Rechnungen in den Anhang ausgelagert. Am Ende des ersten Teils werden in Kapitel \ref{ch:Koszul} die erzielten Resultate mit der Koszul-Dualität in Verbindung gesetzt und eine graduierte Version von Vermas Theorem in diesem speziellen Fall bewiesen.

Der zweite Teil der Arbeit beschäftigt sich mit der $A_\infty$-Struktur. Zunächst wird in Kapitel \ref{ch:Ainf} eine kurze Einleitung gegeben und grundlegende Ideen präsentiert. Abschließend verwenden wir in Kapitel \ref{ch:expAinf} Ergebnisse aus dem ersten Teil und dem vorhergehenden Kapitel um die oben erwähnten Hauptresultate zu beweisen.
\\
\\
\textbf{Danksagung:} An dieser Stelle möchte ich mich bei allen bedanken, die mich während meiner Diplomarbeit unterstützt haben. Mein besonderer Dank gilt meiner Professorin Catharina Stroppel für die umfangreiche und unermüdliche Betreuung.
	\clearemptydoublepage
\renewcommand{\sectionmark}[1]{\markright{\thesection\ #1}}	
\selectlanguage{english}
\part{Algebra structure}\label{Part:Alg}

\chapter{Categories $\O$ and $\O^\p$}\label{sec:Cat}

\section{Notation}
We just resume some basic notations about semisimple and reductive Lie algebras. For a more detailed explication see \cite{Hump08} for the semisimple case and \cite{Mood95} for the more general case of a reductive Lie algebra. Note that since a reductive Lie algebra is the direct sum of a semisimple Lie algebra and its center, there are only slight differences if one generalizes the theory from semisimple to reductive Lie algebras.

Let $\fr{g}$ be a reductive Lie algebra over $\C$ and $\fr{h} \subset \fr{b} \subset \fr{g}$ fixed \textit{Cartan} and \textit{Borel subalgebras}. Let $\p$ be a \textit{parabolic subalgebra}, i.e. $\fr{b} \subset \fr{p} \subset \fr{g}$. From now on fix $\fr{h}$ and $\fr{b}$. If $\fr{g}=\fr{gl}_n$ or $\fr{sl}_n$ we take the standard subalgebras, i.e. for $\fr{b}$ the upper triangular matrices and for $\fr{h}$ the diagonal matrices.

Denote by $\Phi \subset \fr{h}^*$ the \textit{root system} of $\fr{g}$ relative to $\fr{h}$. By choosing $\fr{b}$ as above we get a system of \textit{simple roots} $\Delta \subset \Phi$ and a \textit{positive system} $\Phi ^+ \subset \Phi$ respectively, such that we have the \textit{Cartan decomposition}
$$\fr{g}=\fr{n}^- \oplus \fr{h} \oplus \fr{n}$$
with $\fr{n}=\bigoplus_{\al \in \Phi^+} \fr{g}_\al$ and $\fr{h}$ acting on $\fr{g}_\alpha$ as $\al(\fr{h})$, such that $\fr{b}=\fr{n} \oplus \fr{h}$. 

We always work with a \textit{standard basis} of $\fr{n}$ and $\fr{n^-}$ consisting of root vectors $x_\al \in \fr{g}_\alpha$ and $y_\al \in \fr{g}_{-\alpha}$, $\al>0$ and vectors $h_\al=[x_\al, y_\al] \in \fr{h}$ for $\al \in \Delta$ so that all $\al(h_\al)=2$. Note that for a semisimple Lie algebra the $h_\al$ already give a basis of $\fr{h}$. In the reductive case there can be an additional direct summand from the center of the Lie algebra.

We define a \textit{dual root system} $\Phi \check{ \ }$ satisfying the condition
$$\langle \beta, \alch \rangle = \beta(h_\al) \text{ for all } \al \in \Phi.$$
Denote by $\rho=\frac{1}{2}\sum_{\al \in \Phi^+} \al$ the special weight satisfying $\langle \rho, \alch \rangle=1 \  \forall \  \al \in \Delta$ and by $\la_0$ the zero weight.

The symmetric group attached to $\Phi$ is its \textit{Weyl group} $W$, the group generated by all reflections $s_\al$ with $\al \in \Delta$. On $W$ we define a \textit{length function} $l(w)$, which gives us the smallest number of simple reflections needed to get $w$. If $w=s_1 \dots s_n$ with $s_i$ simple reflections and $l(w)=n$ we call such an expression \textit{reduced}. Define the \textit{Chevalley-Bruhat ordering} of $W$ as follows 
$$w' \leq w \Leftrightarrow w' \text{ occurs as a subexpression in a reduced expression for } w.$$
We get a natural action of $W$ on $\fr{h}^*$ with fixed point zero.

We also work with the \textit{dot-action} where the action is shifted by $-\rho$, i.e.
$$\text{ for }w \in W, \ \la \in \fr{h}^* \text{ define } w \cdot \la=w(\la+\rho)-\rho.$$
Later on we will need the \textit{integral weight lattice} defined as 
$$\La := \{ \la \in \fr{h}^*| \langle \la, \alch \rangle \in \Z \text{ for all } \al \in \Phi \}.$$
Its elements are called \textit{integral weights}. This space is stable under the action of $W$.\\
There is a natural partial ordering on $\La$ defined by $\mu \leq \la$ if and only if $\la-\mu$  is a $\Z^+$-linear combination of simple roots. Write $\Delta=\{ \al_1,\dots, \al_k\}$, then the group $\La$ is free abelian of rank $k$. The basis elements $\{\eps_1, \dots \eps_k\}$ satisfying $\langle \eps_i,  \al_j \check{ \ } \rangle = \delta_{ij}$ are called the \emph{fundamental weights}.

Also define the \textit{dominant weights} by
$$\La^+ := \{ \la \in \fr{h}^*| \langle \la, \alch \rangle \geq 0 \text{ for all } \al \in \Phi^+ \}.$$

The universal enveloping algebra $U(\fr{a})$ of a Lie algebra $\fr{a}$ is an essential tool in the construction of representations of the Lie algebra $\fr{a}$ which are the same as $U(\fr{a})$-modules. Note that every unital associative algebra $A$ is a Lie algebra with the Lie bracket given by the commutator $AB-BA$.
\begin{Def}
The \emph{universal enveloping algebra} $U(\fr{a})$ of a Lie algebra $\fr{a}$ is an associative unital algebra together with a Lie algebra morphism $\sigma: \fr{a} \to U(\fr{a})$ satisfying the following universal property:

For any unital associative $\C$-algebra $A$ and Lie algebra morphism $f:\fr{a} \to A$ there exists a unique morphism of unital associative algebras $\widetilde{f}:U(\fr{a}) \to A$ with $\widetilde{f} \circ \sigma = f$.
\begin{displaymath}
\xymatrix{\fr{a} \ar[r]^{\forall f}\ar[d]_{\sigma} & A\\ U(\fr{a}) \ar[ur]_{\exists \widetilde{f}}}
\end{displaymath}
\end{Def}
By standard arguments for universal properties the universal enveloping algebra is unique up to unique isomorphism. \\
One may also write down the algebra explicitly by taking the tensor algebra $T(\fr{a})$ and dividing out the ideal $I=\langle a \otimes b - b \otimes a -[a,b] | a,b \ \in \fr{a} \rangle$.\\
The adjoint action of $\fr{a}$ on itself induces an action of $\fr{a}$ on $U(\fr{a})$. For $x \in \fr{a}$ this action is given by $x\cdot u= xu - ux$ for $u \in U(\fr{a})$.

In the following we work with left $U(\fr{g})$-modules. Denote the category by $U(\fr{g})-\Mod$. 

For an arbitrary $U(\fr{g})$-module $M$ we define the \textit{weight space} relative to the action of the Cartan subalgebra $\fr{h}$. For $\la \in \fr{h}^*$ the weight space of weight $\la$ is defined as
$$M_\la:= \{v \in M | h \cdot v=\la(h)v \ \forall \ h \: \in \: \fr{h}\}.$$
\section{The category $\O$}
\subsection{Definition}
\begin{Def} The BGG-category $\O$ is defined to be the full subcategory of $U(\fr{g})-\Mod$  whose objects $M$ are satisfying the following conditions:
\begin{itemize}
		\item[$\O 1)$] $M$ is a finitely generated $U(\fr{g})$-module
		\item[$\O 2)$] $M$ is $\fr{h}$-semisimple, i.e., $M$ is a direct sum of its weight spaces $M=\bigoplus_{\la \in 		\fr{h}^*} M_\la$
		\item[$\O 3)$] $M$ is locally $\fr{n}$-finite: for each $v \in M$, the subspace $U(\fr{n}) \cdot v$ of $M$ is finite dimensional
\end{itemize}
\end{Def}
\begin{Lemma}[{\cite[Theorem 1.1 and Theorem 1.11]{Hump08}}]
$\O$ is an artinian, noetherian, abelian category with finite dimensional morphism spaces.
\end{Lemma}
\subsection{Special modules in $\O$} 
In the category $\O$ as in the other categories we will see later, there are three distinguished classes of modules whose isomorphism classes give a basis of the Grothendieck group of $\O$ and which can be labelled by weights $\la \in \fr{h}^*$. To avoid confusion and because we are later mostly working in $\O^\p$, we label the modules in $\O$ by an extra upper index, for instance $M^{\O}$. 

The Verma modules $M^\O(\la)$ are so-called highest weight modules of highest weight $\la$.
\begin{Def}
For $\la \in \fr{h}^*$ let $\C_\la$ be the $1$-dimensional $\fr{b}$-module with trivial $\fr{n}$-action (so for $h \in \fr{h} \ h \cdot x=\la(h)x \ \forall \ x \in \C_\la$). The \textit{Verma module} $M^\O(\la)$ is defined to be
$$M^\O(\la) := U(\fr{g}) \otimes_{U(\fr{b})} \C_\la$$
which has a natural structure of a left $U(\fr{g})$-module.
\end{Def}
\begin{Lemma}[{\cite[Theorem 1.2]{Hump08}, \cite[§2.3, Prop. 4]{Mood95}}]
The Verma module $M^\O(\la)$ has a unique maximal submodule and hence it has a simple head. 
\end{Lemma}
\begin{Def}
For $\la \in \fr{h}^*$ denote the simple head of $M^\O(\la)$ by $L^\O(\la)$.
\end{Def}
The third class of modules we are interested in are the projective modules. 
\begin{Theorem}[{\cite[Theorem 3.8]{Hump08}, \cite[§2.10, Prop. 17]{Mood95}}]
Category $\O$ has enough projectives.
\end{Theorem}

\begin{Def}
For $\la \in \fr{h}^*$ denote by  $P^\O(\la)$ the projective cover of $L^\O(\la)$.
\end{Def}
\begin{Lemma}[{\cite[Theorem 3.9]{Hump08}, \cite[§2.10, Prop. 10]{Mood95}}]
Every indecomposable projective module in $\O$ is isomorphic to some $P^\O(\la)$.
\end{Lemma}
\subsection{Blocks}\label{sec:block}
\begin{Theorem}[{\cite[Prop. 1.12]{Hump08}, \cite[§2.12 Prop.1, §6.7 Prop 4]{Mood95}}]
The above category splits into direct summands (so-called ``blocks'' )  $O_\mu$ such that
$$\O = \bigoplus_\mu O_\mu$$
where $\mu$ runs over a complete system of representatives for orbits under the dot action and $\O_\la$ the summand generated by the simple modules $L(\mu)$ with $\mu \in W \cdot \la$. 
\end{Theorem}
Note that for integral weights $\la \in \La$ the subcategory $O_\la$ is a block in the usual sense \cite[Prop. 1.13]{Hump08}.

One can show that $M(\mu)$ and $P(\mu)$ are lying in $\O_\mu$.

Note that since an element of the center of $\fr{g}$ operates by a scalar on $M(\mu)$ and this central character $\chi(z)$ is invariant under the dot-action of $W$ (cf \cite[Chapter 1.7]{Hump08}) the center acts by $\chi(z)$ on the whole block.

Note that $\Ext_\O^i(M,N)={0}$ for $M$, $N$ in different blocks.
$\O_0$ is called the \textit{principal block} since it is the block containing the trivial representation. 

\section{The category $\O^\p$}
\subsection{Notations}
We work with a standard parabolic subalgebra $\p$. For a subset $I \subset \Delta$ we get a root system $\Phi_I \subset \Phi$ and the standard parabolic subalgebra $\p=\p_I=\fr{l}_I \oplus \fr{u}_I$ which is the Lie algebra with the chosen root system and $\fr{l}=\fr{l}_I$ the corresponding Levi subalgebra ($\fr{l}=\fr{h} \oplus \bigoplus_{\al \in \Phi_I} \fr{g}_\al$). For $I=\emptyset$ we get $\p= \fr{b}$, for $I=\Delta$ we get $\p_I=\fr{g}$. Let $W_\fr{l}$ be the Weyl group generated by all $\al \in I$. Denote by $W^\fr{l}$ the set of minimal-length coset representatives for $W_\fr{l} \backslash W$, that is 
\begin{equation*}W^\fr{l}=\{ w \in W |\  \forall \  \al \in I : l(s_\al w)>l(w)\}.\end{equation*}
We also have $\p$-dominant weights
\begin{equation}\label{pdom}\La_I^+:=\{\la \in \fr{h}^* | \langle \la, \alch \rangle \in \Z^+ \text{ for all } \al \in I \}\end{equation}
Denote by $E(\la)$ the finite dimensional $\fr{l}_I$-module with highest weight $\la \in \La_I^+$.
\subsection{Definition}
\begin{Def}
The category $\O^\p$ is the full subcategory of $\O$ consisting of all modules which are locally $\p$-finite.
\end{Def}
\subsection{Special modules in $\O^\p$}
A simple module in $\O^\p$ must be simple in $\O$. Moreover we get
\begin{Prop}[{\cite[Prop 9.3. and Theorem 9.4.]{Hump08}}]
The simple module $L^\O(\la)$ lies in $\O^\p$ if and only if $\la \in \La_I^+$. Denote those modules just by $L(\la)$.
\end{Prop} 
Now we proceed with the parabolic Verma modules. 
\begin{Def}
For $\la \in \La_I^+$ we define the \textit{parabolic Verma module} 
$$M(\la):= U(\fr{g}) \otimes_{U(\fr{p}_I)} E(\la)$$
\end{Def}
$M(\la)$ has unique simple quotient $L(\la)$.
Sometimes it might be easier to work with the following identification
\begin{Prop}[{\cite[Corollary 1.3]{Stro2005}}]
For $\la \in \La_I^+$ there is an isomorphism $M(\la) \cong M^\O(\la)/M$, where M denotes the smallest submodule containing all composition factors not contained in $\O^\p$.
\end{Prop}
Now we are left to define the projective modules.
\begin{Def}
For $\la \in \La_I^+$ define $P(\la)$ to be the projective cover of $L(\la)$.
\end{Def}
Similar to the Verma modules the projective modules in $\O^\p$ are quotients of those in $\O$. 
\begin{Prop}[{\cite[Proposition 1.2]{Stro2005}\label{Pro:projOp}}]
Let $Q \in \O^\p$ with projective cover $P \in O$. Then the projective cover of Q in $O^\p$ is (up to isomorphism) the quotient $P/M$, where M is the smallest submodule of $P$ containing all composition factors of $P$ not contained in $\O^\p$.
\end{Prop} 
Using the Proposition one easily shows that $\O^\p$ has enough projectives.

Let $\O^\p_0$ be the principal block of $\O^\p$ such that all $M \in \O^\p$ only have composition factors $L(\mu)$ with $\mu \in W \cdot {\la_0} \cap \La_i^+$. 

Note that these weights are exactly those given by $w \cdot {\la_0}$ with $w \in W^\p$. Since we work with left cosets, for better readability we write $P(x \cdot \la)=:P(\la.x)$ and similar for simple and Verma modules.

\section{Blocks of $\O(\fr{sl}_n)$ versus blocks of $\O(\fr{gl}_n)$}
Later on we want to combine results obtained for $\O(\fr{gl}_n)$ with those obtained for $\O(\fr{sl}_n)$. 
To make sure that this is possible we want to stress that on the level of blocks the two categories are almost the same. Making this more precise is the purpose of this section.

Note that since $\fr{gl}_n$ only differs from $\fr{sl}_n$ by the span of one more central element (the identity matrix $E$), they have the same root space decomposition and only differ in the Cartan subalgebra.
If $\fr{h}$ is a Cartan for $\fr{sl}_n$, then $\fr{h}'=\fr{h} \oplus \C E$ is a Cartan for $\fr{gl}_n$. In our setup $\fr{h}$ is the Lie algebra of diagonal matrices with trace zero and $\fr{h}'$ is the space of all diagonal matrices. 

\begin{Def}
For a $U(\fr{gl}_n)$-module $M$ denote by $\res(M)$ the corresponding $U(\fr{sl}_n)$-module.

For a $U(\fr{sl}_n)$-module $M$ and $a \in \C$ denote by $F_a(M)$ the $U(\fr{gl}_n)$-module with 
$$E.m=a\cdot m \quad \forall \ m \ \in M.$$
\end{Def}

\begin{Lemma}
If $M$ is in $\O(\fr{gl}_n)$, then $\res(M)$ lies in $\O(\fr{sl}_n)$.

The functor $\res$ maps a Verma module $M(\la')$ with $\la' \in \fr{h}'$ to the Verma module $M(\la)$ with $\la=\la'|_{h}$. 
\end{Lemma}
\begin{proof}
We check the properties $\O1-\O3$. 
\begin{enumerate}
\item[$\O1$)] The module $\res(M)$ is still finitely generated, since a central element does not generate anything.
\item[$\O2$)] As the subalgebra $\fr{h'}$ is restricted to a smaller one, semisimplicity holds.
\item[$\O3$)] $U(\fr{n})$ does not change.
\end{enumerate}
The second statement follows then directly from the definitions and the facts that a highest weight vector is sent to a highest weight vector and that $U(\fr{n})$ does not change.
\end{proof}
\begin{Lemma}
If $M$ is in $\O(\fr{sl}_n)$, then $F_a(M)$ lies in $\O(\fr{gl}_n)$.

The functor $F_a$ maps the Verma module $M(\la)$ with $\la \in \fr{h}^*$ to the Verma module $M(\la')$ with $\la' \in \fr{h}'^*$ such that $\la'(h)=\la(h) \ \forall h \in \fr{h}$ and $\la'(E)=a$. 
\end{Lemma}
\begin{proof}
Again, we check the properties $\O1-\O3$ without any difficulties appearing.

By the definition of $F_a(M)$ the highest weight vector $v \in M(\la)$ is mapped to a highest weight vector in $F_a(M(\la))$ with $E.v=a \cdot v$.
\end{proof}

\begin{Lemma}\label{resf}
$\res(F_a(M))=M$ for all $a \in \C$.
\end{Lemma}
Now we can check that these functors map blocks to blocks. Therefore recall from Section \ref{sec:block} that each element $z$ in the center acts by a scalar $\chi(z)$ on the modules in a block. Especially for an element in $Z(U(\fr{g})) \cap \fr{h}$ this scalar has to be $\la(z)$ since $z.v = \la(z)v$ for a highest weight vector.

\begin{Lemma}\label{fres}
For a module $M \in \O(\fr{gl}_n)_{\la'}$ with $\la' \in \fr{h'}^*$ we have
$F_a(\res M)=M$ with $a=\la'(E)$ .
\end{Lemma}

\begin{Theorem}
There is an equivalence of categories
\begin{align} F_a: \O(\fr{sl}_n)_\la  \to \O(\fr{gl}_n)_{\la'}\end{align}
 with $\la' \in \fr{h}'^*$ such that $\la'(h)=\la(h) \ \forall h \in \fr{h}$ and $\la'(E)=a$. The inverse functor is given by
 \begin{align} \res: \O(\fr{gl}_n)_{\la'}  \to \O(\fr{sl}_n)_\la.\end{align}

In particular, the principal blocks of both categories are equivalent.
\end{Theorem}
\begin{proof}
The theorem follows immediately by Lemma \ref{resf} and Lemma \ref{fres}.
\end{proof}
\begin{Remark}
Since the parabolic category $\O$ is a Serre subcategory of the ordinary category $\O$ given by all modules with composition factors from a prescribed set of simple highest weight modules, the above result also holds in the parabolic setting. Particularly, we can derive the result for the principal block which we are going to use later on.
\end{Remark}
\begin{Cor}
The functor $F_0$ induces an equivalence of categories
\begin{equation*}
O^{\p'}_0(\fr{gl}_{m+n})\cong O^\p_0(\fr{sl}_{m+n})
\end{equation*}
where $\p'$ is the parabolic subalgebra with corresponding Levi component $gl_m \oplus gl_n$ and $\p=\p' \cap \fr{sl}_{m+n}$.
\end{Cor}
\chapter{Notations and homological algebra}\label{ch:NotHom}
For later use we recall some notations and definitions from homological algebra.
\section{Complexes}
\begin{Def} Let $A$ be a graded algebra. A \emph{chain complex} $C_\bullet$ of graded $A$-modules is a family of graded $A$-modules $\{C_n\}_{n\in \Z}$ together with degree zero $A$-module morphisms $d_n:C_n \to C_{n-1}$ such that $d_{n-1} \circ d_n=0$. The object $C_n$ is often called the \textit{$n$th component} of $C_\bullet$. A complex is called \emph{acyclic} or \emph{exact} if its homology is zero.
\end{Def}
\begin{Notation}\label{Not:sign}
We have to define different kinds of shifts.

For a complex $C_\bullet$ define $C[i]_\bullet$ by $C[i]_j=C_{j-i}$. The differential is $d[i]_j=(-1)^i d_j$.\\
For explicit calculations in Section \ref{ch:spec} we want to define a shift functor $[\quad]_\hom$ where $C[i{]_\hom} _\bullet$ is precisely as above but leaving the differential unchanged.\\
For $M$ a graded $A$-module define the internal shift $M\langle i \rangle$ by $M\langle i \rangle_j=M_{j-i}$.\\
We denote by $C_\bullet\langle i \rangle$ the (internally) shifted complex $C_\bullet$ which one obtains just by shifting each object. Hence the internal grading is shifted up, the differential maps stay to be of degree zero.
\end{Notation}
\section{Projective resolutions}
\begin{Def} Choose $A$ as above and let $M$ be a graded $A$-module. A \emph{projective resolution} $P_\bullet$ of $A$ is a chain complex consisting of projective modules with $P_i=0 \ \forall \ i<0$ such that the complex $P_{\geq 0}\to M \to 0$ is acyclic.
\end{Def}
\begin{Def} A projective resolution $P_\bullet$ is called \emph{linear} if the $n$th part $P_n$ is generated in degree $n$.
\end{Def}
\begin{Lemma}\label{Le:Coneconstr}
Given an exact sequence of graded $A$-modules 
$$\xymatrix{0 \ar[r] &L\langle 1 \rangle \ar[r]^{f} &M \ar[r]^{g}& N \ar[r] &0}$$
and finite linear projective resolutions $P_\bullet^L$, $P_\bullet^M$ of $L$ and $M$ respectively, one may construct a linear projective resolution $P_\bullet^N$ of $N$. The construction works as follows: Set $f_\bullet: P_\bullet^L \langle 1 \rangle \to P_\bullet^M$ the map induced by $f$ and choose $P_\bullet^N= C(f_\bullet)$ where $C(f_\bullet)$ denotes the cone of the chain map $f_\bullet$. The map $P_0^N \to N$ is given by the composition of the maps $P_0^M \to M \to N$.
\end{Lemma}
\begin{proof}
Writing down the induced map $f$, we get a commuting diagram
$$\xymatrix{ \cdots \ar[r] &P_k^L\langle 1 \rangle \ar[d]^{f_k} \ar[r]^{d^L_k}  &\cdots \ar[r] &P_0^L\langle 1 \rangle \ar[d]^{f_0} \ar[r]^{d^L_0} &L \langle 1 \rangle \ar[d]^{f} \ar[r] &0 \\ \cdots \ar[r] &P_k^M\ar[r]^{d^M_k} &\cdots \ar[r] &P_0^M \ar[r] ^{d^M_0}&M \ar[r] &0}$$
The cone of this complex is an acyclic complex

\begin{displaymath}\xymatrix@M=3pt@C=20pt{ \cdots \ar[r] &P_{k-1}^L \langle 1 \rangle \oplus P_k^M \ar[r]^{\hspace{0.7cm}\widetilde d^N_k}  &\cdots \ar[r] &P_0^L \langle 1 \rangle \oplus P_1^M \ar[r]^{\widetilde{d}^N_1} &L \langle 1 \rangle \oplus P_0^M  \ar[r]^{\hspace{0.7cm} \widetilde{d}^N_0} &M \ar[r] & 0 }\end{displaymath}
\normalsize
with $\widetilde{d}_k^N=\left( \begin{matrix} -d_{k-1}^L&0\\f_{k-1}&d_{k}^M \end{matrix}\right)$. There is an obvious map between complexes

$$\xymatrix@M=3pt@C=20pt{\cdots \ar[r] &0\ar[d] \ar[r]& \cdots \ar[r] &0 \ar[r]\ar[d] & L \langle 1 \rangle \ar[r]\ar[d] & L \langle 1 \rangle \ar[r]\ar[d] &0\\
\cdots \ar[r] &P_{k-1}^L \langle 1 \rangle \oplus P_k^M \ar[r]^{\hspace{0.7cm}\widetilde d^N_k}  &\cdots \ar[r] &P_0^L \langle 1 \rangle \oplus P_1^M \ar[r]^{\widetilde{d}^N_1} &L \langle 1 \rangle \oplus P_0^M  \ar[r]^{\hspace{0.7cm}\widetilde{d}^N_0} &M \ar[r] & 0 }$$
\normalsize
Taking the quotient by the upper complex, we get a new acyclic complex 
$$\xymatrix{ \cdots \ar[r] &P_{k-1}^L \langle 1 \rangle \oplus P_k^M \ar[r]^{\hspace{0.7cm} d^N_k}  &\cdots \ar[r] &P_0^L \langle 1 \rangle \oplus P_1^M \ar[r]^{\hspace{0.7cm}{d}^N_1} &P_0^M  \ar[r]^{{d}^N_0} &N \ar[r] & 0 }$$
with the differentials $d_k^N=\begin{cases} \widetilde{d}_k^N &\text{ for }k\geq 2, \\
																									(f_0,d_1^M) &\text{ for } k=1, \\
																								 g \circ d_0 &\text{ for } k=0 \end{cases}$\\
which are the differentials assumed in the Lemma.

Since the part belonging to $P_\bullet^L$ is shifted in both directions, it occurs as $P_\bullet^L \langle 1 \rangle  [1]$. The terms belonging to $P_\bullet^M$ stay unchanged and so this projective resolution is linear.											
\end{proof}

\section{The algebra structure on $\Ext$}
\subsection{$\Ext$-spaces}\label{sec:Exthom}
We shortly review different ways to define $\Ext(A,B)$ with $A, B \in Ob(\A)$ and $\A$ an abelian category. 

If $\A=R-\Mod$ is the category of modules over a ring $R$, the most usual way is defining $\Ext_R^k(\_,B)$ as the $k$th right derived functor of $\hom_R(\_,B)$. For this one chooses a projective resolution $P_\bullet$ of the module $A$ and computes $Ext_R^k(A,B)=H^n(\hom_R(P_\bullet,B))$.

>From now on assume that $A$ and $B$ have finite projective dimension.
For $Q_\bullet$ a finite projective resolution of $B$, we define a differential graded structure on $\hom(P_\bullet,Q_\bullet)$ with $\hom(P_\bullet,Q_\bullet)^r=\prod_p \hom(P_p,Q_{p+r})$ and the differential $d_p(f)=d {\circ}f-(-1)^p f {\circ}d$ (c.f. \cite[chapter III.6.13]{Gelf88}). Now we are able to compute $\Ext$ using the derived category (for the arguments see  \cite[Chapter III]{Gelf88}): 
\begin{align*}
\Ext^k(A,B)&=	\hom_{\D(\A)}(A[0],B[k]) \\
&=	\hom_{\D(\A)}(P[0],Q[k]) \\
&=	\hom_{\K(\A)}(P[0],Q[k])	&\text{ since $P_\bullet$ is projective and bounded}\\
&= 	\hom_{\K(\A)}(P_\bullet,Q_\bullet)[k]	\\
&=	H^0(\hom(P_\bullet,Q_\bullet)[k])\\
&= H^k(\hom(P_\bullet,Q_\bullet))
\end{align*}

Therefore, one can also compute the homomorphism spaces of the projective resolutions and afterwards take its cohomology. \\
Cycles in $\hom(P_\bullet,Q_\bullet)$ are chain maps (according to the degree commuting or anticommuting) and boundaries are homotopies (up to sign). If one regards them as chain maps between translated complexes (i.e. in $\hom_{\D^b(\A)}(P[0],Q[k])$) the sign convention from Notation \ref{Not:sign} leads to cycles being commuting chain maps and boundaries being usual homotopies.

For getting less confused with the signs in the complexes, in our computations we do not change signs while shifting (i.e. we use the $[ \ ]_\hom$-shift), but therefore we have to check that the maps are commutative or anticommutative, respectively.

Note that for $A=\bigoplus_\al A_\al$ and $B=\bigoplus_\beta B_\beta$ two finite direct sums one has 
$$\Ext^k(A,B)=\bigoplus_{\al, \beta} \Ext^k(A_\al,B_\beta).$$

\subsection{Multiplication}
>From now on we choose $A=B$ and compute $\Ext^k(A,A)=H^k(\hom(P_\bullet,P_\bullet))$. For the ease of presentation, multiplication in the algebra $\hom(P_\bullet, P_\bullet)$, which is given by composing of chain maps, is written from left to right, i.e. for $\al$, $\be \in \hom(P_\bullet, P_\bullet)$ we have $(\al \cdot \be)(x)=\be(\al(x))$. The multiplication in $\Ext(A,A)$ is the induced multiplication, therefore it is also given by composing the corresponding chain maps.
If $A=\bigoplus\limits_{\al \in I} A_\al$, $P_{\al \bullet}$ is a projective resolution of $A_\al$ and $P_\bullet=\bigoplus\limits_{\al \in I} P_{\al \bullet}$, we have $\Id_\al=[\id]\in \Ext^0(A_\al,A_\al)$.\\
The $\Id_\al$ form a system of mutual orthogonal idempotents, hence we can write 
\begin{equation*}
\Ext^k(A,A)=\bigoplus_{\al, \beta \in I} \Id_\al \Ext^k(A_\al, A_\beta) \Id_\beta.
\end{equation*}
Therefore, it is enough to compute $\Ext^k(A_\al, A_\beta)$ and to look at products of elements $x \in \Ext^k(A_\al, A_\beta)$ and $y \in \Ext^l(A_\beta, A_\ga)$ interpreting their product 
$$x \cdot y =\Id_\al x\Id_\beta \Id_\beta y \Id_\gamma \in \Ext^{k+l}(A, A).$$

\chapter{$\O^\p$ via Khovanov's diagram algebra}\label{ch:Khov}
Now we want to specify our situation. Let $\fr{g}=\fr{gl}_{m+n}(\C)$ and $\p$ the parabolic subalgebra associated to the Levi subalgebra $\fr{l}=\fr{gl}_{m}(\C) \oplus \fr{gl}_{n}(\C)$. The key ingredient of the whole work is the main theorem from \cite{brun32008}.
\begin{Theorem}[{\cite[Corollary 8.21.]{brun32008}}]\label{Th:equcat}
There is an equivalence of categories 
$$\mathbb{E}:\O(m,n,I) \to K(m,n,I)-\mod$$ 
such that $\mathbb{E}(L(\la))\cong L(\la)$, $\mathbb{E}(M(\la))\cong M(\la)$ and $\mathbb{E}(P(\la))\cong P(\la)$ for each $\la \in \La(m,n,I)$.
\end{Theorem}
Here $\O(m,n,I)$ is an infinite sum of certain blocks of $\O^\p$ and $K(m,n)$ is an infinity algebra given by the direct sum of finite dimensional algebras $K_\La$ (cf. \cite[Section 2]{brun32008}). Here $K_\La$ is the algebra defined diagrammatically in \cite{brun32008} with an explicit basis given by certain diagrams and a multiplication defined by an explicit ``surgery'' construction.
The basis is in fact a (graded) cellular basis in the sense of Graham and Lehrer \cite{Grah2004} in the graded version of Hu and Mathas \cite{Hu2010}.
The algebra is shown to be quasi-hereditary in \cite[Section 5]{Brun12008}. Hence we have  standard modules $M(\la)$, their projective covers $P(\la)$ and irreducible quotients $L(\la)$. This is meant by the notation used in the theorem. These terms will be explained in detail below.

Since we are not going to work in this general setting, we do not introduce the notation in detail. We only need certain finite dimensional summands corresponding to the principal block of $\O^\p$ which we will introduce below.\\
The theorem simplifies by restriction to the principal block of $\O^\p$ and we deduce the following corollary

\begin{Cor}\label{Cor:equcat}
There is an equivalence of categories of the principal block of $\O^\p$ to the category of finite dimensional left modules over the Khovanov diagram algebra, $K_m^n-\mod$, sending the simple module $L(\la) \in \O^\p$ to the simple module $L(\la) \in K_m^n-\mod$, the Verma modules $M(\la)$ to the cell modules $M(\la)$ and the indecomposable projectives to the corresponding indecomposable projectives.
\end{Cor}

\begin{Remark}\label{grading}
As we will see in the following sections, $K_m^n-\mod$ possesses a natural grading (cf. \cite[Theorem 5.3]{Brun12008}). Using the above equivalence from Corollary \ref{Cor:equcat} this gives rise to a graded version of the principal block of $\O^\p$, denoted by $\O_0^{\p \Z}$. Using the grading from $K(m,n)-\mod$ which will not be explained in this thesis, one obtains a graded category $\O^{\p \Z}$. By the unicity of Koszul gradings \cite[Section 2.5]{Beil1996} this is equivalent to the graded version of $\O^\p$ one can define geometrically \cite[Section 3.11]{Beil1996} (and \cite{Stro2005} for the principal block). 
\end{Remark}

\section{The algebra $K_m^n$ and its basic properties}
\subsection{Basic definitions} \label{basic}
For the construction of elements in $K_m^n$, we recall the notions of weights, cup/cap diagrams and finally circle diagrams in our situation (cf. {\cite[Section 2]{Brun12008}).

A {\it weight} $\la$ in the block $\La_m^n:=\La(m,n;m+n)$ belonging to $K_m^n=K_{\La_m^n}$ is an element obtained by permuting $n \ \up$'s and $m \  \down$'s placed at the $(m+n)$ places $i \in I=\{0,\dots,m+n-1\}$ on the number line. The {\it zero weight} $\la_0$ is the one having all $\up$'s on the left and all $\down$'s on the right (cf. figure \ref{fig:zeroweight}).
\begin{figure}
\caption{the zero weight for $n=2$ and $m=3$}
	\label{fig:zeroweight}
 \center
\begin{tikzpicture}
\begin{scope}
		\draw (-1, 0)--(2.6, 0);
		\draw (0,0) node[below=-3pt] {$\up$};
		\draw (0.4,0) node[below=-3pt] {$\up$};
		\draw (0.8, 0) node[above=-3pt] {$\down$};
		\draw (1.2,0) node[above=-3pt] {$\down$};
		\draw (1.6,0) node[above=-3pt] {$\down$};
		\end{scope}
		\begin{scope}[yshift=0cm]
		\scriptsize
		\draw (-1.4, 0) node {$\cdots$};
		\draw (-0.8, 0.1)--(-0.8, -0.1);
		\draw (-0.9, 0) node [above] {$-2$};
		\draw (-0.4, 0.1)--(-0.4, -0.1);
		\draw (-0.5,0) node[above] {$-1$};
		\draw (2,0) node[above] {$5$};
		\draw (2, 0.1)--(2, -0.1);	
		\draw (2.4,0) node[above] {$6$};
		\draw (2.4, 0.1)--(2.4, -0.1);		
		\draw (3,0) node {$\cdots$};
		\end{scope}
\end{tikzpicture}
\end{figure}
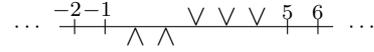

The connection with our previously defined weights is given by a weight dictionary similar to the one in \cite[Section 1]{brun32008}. For the ease of presentation we shift the weight $\rho$ used to define the zero weight:\\
Take $$\rho=\eps_{m+n-1}+2\eps_{m+n-2}+ \cdots +(m+n-1) \eps_{1} \in \fr{h}^*$$
and for $\la$ in $\La_I^+$ from equation \eqref{pdom} define
\begin{align*}
I_\down(\la)&:= \{(\la+\rho,\eps_1), \ldots, (\la+\rho,\eps_m)\}\\
I_\up(\la)&:= \{(\la+\rho,\eps_{m+1}), \ldots, (\la+\rho,\eps_{m+n})\}.
\end{align*}
Now label the $i$th vertex of the numberline by
\begin{equation*}
\begin{cases}\down &\text{if $i$ belongs to $I_\down(\la)$}\\
\up &\text{if $i$ belongs to $I_\up(\la)$} \end{cases}\end{equation*}
respectively. The obtained weight is the weight $\la \in \La_m^n$. 
The weights are partially ordered by the {\it Bruhat order}, i.e. an element becomes bigger by swapping $\down$'s to the right.

The symmetric group $S_{m+n}$ acts transitively on the set of weights by permuting the symbols. Since the zero weight has stabilizer $S_m \times S_n$, we get a bijection between $W^\fr{l}$ and the set of weights. Observe that the elements $s_{n}\cdot \ldots \cdot s_{i_n} \cdot s_{n-1}\cdot \ldots \cdot s_{i_{n-1}} \cdot \ldots \cdot s_{1}\cdot \ldots \cdot s_{i_1}$ with $n+m\geq i_n>i_{n-1}> \ldots >i_1\geq 0$ give a set of shortest elements (First permute the $n$th $\up$, then the $(n-1)$st and so on, cf. figure \ref{fig:permuting}).
The order of the elements in $\La_m^n$ corresponds precisely to the Bruhat order on the Weyl group via Corollary \ref{Cor:equcat} and the part after Prop \ref{Pro:projOp} and the explicit description given above. 
\begin{figure}
\caption{the Weyl group operating on the zero weight for $n=2$ and $m=3$}
	\label{fig:permuting}
 \center
\begin{tikzpicture}
\begin{scope}
		\draw (0.8, 1) node {$\la_0$};
		\draw (0,0) node[above=-1.55pt] {$\up$};
		\draw (0.4,0) node[above=-1.55pt] {$\up$};
		\draw (0.8, 0) node[above=-1.55pt] {$\down$};
		\draw (1.2,0) node[above=-1.55pt] {$\down$};
		\draw (1.6,0) node[above=-1.55pt] {$\down$};
		\draw[->] (1.9,0.3)--(2.8,0.3);
		\draw (2.3,0.3) node[above]{$s_2$};
\end{scope}
\begin{scope}[xshift=3cm]
		\draw (0.8, 1) node {$\la_0.s_2$};
		\draw (0,0) node[above=-1.55pt] {$\up$};
		\draw (0.4,0) node[above=-1.55pt] {$\down$};
		\draw (0.8, 0) node[above=-1.55pt] {$\up$};
		\draw (1.2,0) node[above=-1.55pt] {$\down$};
		\draw (1.6,0) node[above=-1.55pt] {$\down$};
		\draw[->] (1.9,0.3)--(2.8,0.3);
		\draw (2.3,0.3) node[above]{$s_3$};
\end{scope}
\begin{scope}[xshift=6cm]
		\draw (0.8, 1) node {$\la_0.s_2s_3$};
		\draw (0,0) node[above=-1.55pt] {$\up$};
		\draw (0.4,0) node[above=-1.55pt] {$\down$};
		\draw (0.8, 0) node[above=-1.55pt] {$\down$};
		\draw (1.2,0) node[above=-1.55pt] {$\up$};
		\draw (1.6,0) node[above=-1.55pt] {$\down$};
		\draw[->] (1.9,0.3)--(2.8,0.3);
		\draw (2.3,0.3) node[above]{$s_1$};		
\end{scope}
\begin{scope}[xshift=9cm]
		\draw (0.8, 1) node {$\la_0.s_2s_3s_1$};
		\draw (0,0) node[above=-1.55pt] {$\down$};
		\draw (0.4,0) node[above=-1.55pt] {$\up$};
		\draw (0.8, 0) node[above=-1.55pt] {$\down$};
		\draw (1.2,0) node[above=-1.55pt] {$\up$};
		\draw (1.6,0) node[above=-1.55pt] {$\down$};
\end{scope}
\end{tikzpicture}
\end{figure}
For $\la=\la_0.x$ with $x \in W^\fr{l}$ we write $l(\la)$ for $l(x)$.\\
For each index $i$ define \emph{the relative length}
\begin{equation*}
\begin{split}
l_i(\la,\mu)
:=&
\#\{j \in I\:|\:j \leq i
\text{ and vertex $j$ of $\la$ is labelled $\down$}\}
\\&-
\#\{j \in I\:|\:j \leq i
\text{ and vertex $j$ of $\mu$ is labelled $\down$}\}\\
=& \#\{j \in I\:|\:j \leq i
\text{ and vertex $j$ of $\mu$ is labelled $\up$}\}
\\
&-
\#\{j \in I\:|\:j \leq i
\text{ and vertex $j$ of $\la$ is labelled $\up$}\}.
\end{split}\end{equation*}
and note that by \cite[Section 5]{Brun12008} 
\begin{align}\label{diffl}
l(\la)-l(\mu)= \sum_{i \in I} \ell_i(\la,\mu).
\end{align}
A {\it cup diagram} is a diagram obtained by attaching rays and finitely many cups (lower semicircles) to the line of length $m+n$, so that rays join vertices down to infinity and do not intersect cups.\\
A {\it cap diagram} is the mirror image of a cup diagram, so caps (i.e. upper semicircles) instead of cups are used. The mirror image of a cup
 (resp. cap) diagram $c$ is denoted by $c^*$.\\
If $c$ is a cup diagram and $\la$ a weight in $\La_m^n$, we can glue $c$ and $\la$ and obtain a new diagram denoted $c\la$. $c\la$ is called {\it oriented cup diagram} if 
\begin{itemize}
	\item each cup is oriented, i.e. one of the vertices is labelled $\down$ and the other one $\up$
	\item there are not two rays in $c$ labelled $\down \up$ in this order from left to right.
\end{itemize}
An example is given in figure \ref{fig:anOrientedCupDiagram}.

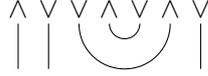
\begin{figure}
\caption{an oriented cup diagram}
	\label{fig:anOrientedCupDiagram}
 \center
\begin{tikzpicture}
\begin{scope}
		\draw (0,0) node[above=-1.55pt] {$\up$};
		\draw (0.4,0) node[above=-1.55pt] {$\down$};
		\draw (0.8, 0) node[above=-1.55pt] {$\down$};
		\draw (1.2,0) node[above=-1.55pt] {$\up$};
		\draw (1.6,0) node[above=-1.55pt] {$\down$};
		\draw (2,0) node[above=-1.55pt] {$\up$};
		\draw (2.4, 0) node[above=-1.55pt] {$\down$};
		\draw (0,0) -- (0,-0.6);
		\draw (0.4,0) -- (0.4,-0.6);
	  \draw (0.8,0) arc (180:360:0.6);
		\draw (1.2,0) arc (180:360:0.2);
		\draw (2.4,0) -- (2.4,-0.6);	
		\end{scope}
\end{tikzpicture}
\end{figure}
If $c$ is a cap diagram, it is called {\it oriented cap diagram} if $c^*\la$ is an oriented cup diagram.\\
A {\it circle diagram} is obtained by gluing a cup and a cap diagram. It consists of circles and lines.\\
An {\it oriented circle diagram} $a\la b$ is the diagram obtained by gluing the oriented cup diagram $a\la$ underneath the oriented cap diagram $\la b$. For an example look at figure \ref{fig:anOrientedCircleDiagram}. 

\begin{figure}
\caption{an oriented circle diagram}
	\label{fig:anOrientedCircleDiagram}
 \center
\begin{tikzpicture}
\begin{scope}
		\draw (0,0) node[above=-1.55pt] {$\up$};
		\draw (0.4,0) node[above=-1.55pt] {$\down$};
		\draw (0.8, 0) node[above=-1.55pt] {$\down$};
		\draw (1.2,0) node[above=-1.55pt] {$\up$};
		\draw (1.6,0) node[above=-1.55pt] {$\down$};
		\draw (2,0) node[above=-1.55pt] {$\up$};
		\draw (2.4, 0) node[above=-1.55pt] {$\down$};
		\draw (2.4,0.4) -- (2.4,0.8);
		\draw (0,0.4) arc (180:0:0.2);
		\draw (0.8,0.4) arc (180:0:0.2);
		\draw (1.6,0.4) arc (180:0:0.2);
		\draw (0,0) -- (0,-0.6);
		\draw (0.4,0) -- (0.4,-0.6);
	  \draw (0.8,0) arc (180:360:0.6);
		\draw (1.2,0) arc (180:360:0.2);
		\draw (2.4,0) -- (2.4,-0.6);	
		\end{scope}
\end{tikzpicture}
\end{figure}
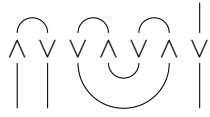
The {\it degree} of an oriented cup/cap diagram $a\la$ (or $\la b$) means the total number of oriented cups (caps) that it contains. So in $K_m^n$ one has $deg(a \la) \leq n$, since there are at most $n$ cups (caps).
The {\it degree} of an oriented circle diagram $a\la b$ is defined as the sum of the degree of $a\la$ and the degree of $\la b$.

The {\it cup diagram associated to a weight $\la$} is the unique cup diagram $\underline{\la}$ such that $\underline{\la} \la$ is an oriented cup diagram of degree 0. The construction works as follows: Look for two neighboured vertices labelled by $\down \up$ and connect them by a cup. Proceed this procedure ignoring vertices which are already joined to others. Finally draw rays to all vertices which are left. An example is given in figure \ref{fig:assocDiagram}.\\
\begin{figure}
\caption{The cup diagram associated to a weight $\la$}
	\label{fig:assocDiagram}
 \center
\begin{tikzpicture}
\begin{scope}
		\draw (-0.7,0) node[above=-1.55pt] {$\la$};
		\draw (-0.7,-0.6) node[above=-1.55pt] {$\und{\la}$};
		\draw (0,0) node[above=-1.55pt] {$\up$};
		\draw (0.4,0) node[above=-1.55pt] {$\down$};
		\draw (0.8, 0) node[above=-1.55pt] {$\up$};
		\draw (1.2,0) node[above=-1.55pt] {$\down$};
		\draw (1.6,0) node[above=-1.55pt] {$\down$};
		\draw (2,0) node[above=-1.55pt] {$\up$};
		\draw (2.4, 0) node[above=-1.55pt] {$\down$};
		\draw (0,0) -- (0,-0.4);
		\draw (1.2,0) -- (1.2,-0.4);
	  \draw (0.4,0) arc (180:360:0.2);
		\draw (1.6,0) arc (180:360:0.2);
		\draw (2.4,0) -- (2.4,-0.4);	
		\end{scope}
\end{tikzpicture}
\end{figure}
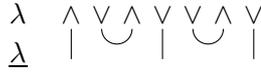
The {\it cap diagram associated to a weight $\la$} is defined as $\overline{\la}:=(\underline{\la})^*$.

In a cup (cap) diagram we number the cups (caps) by the order of their right ending points from left to right. 

For a cup (cap) diagram $a$ we denote by $\nes_a(i)$ for $1 \leq i \leq \# \{ \text{cups} \}$ the number of cups nested in the $i$th cup.

\subsection{The algebra}
The underlying vector space leading to $K_m^n$ has a basis
\begin{align*}
	\left\{(a\la b) \left| \text{for all oriented circle diagrams with } \la \in \La_m^n\right\} \right. .
\end{align*}
We take the grading given by the degree defined above. 

$e_\la$ is defined to be the diagram $\underline{\la}\la \overline{\la}$.
The product of two circle diagrams $a\la b$ and $c \mu d$ is zero except for $b=c^*$. The multiplication of $a \la b$ and $b^* \mu d$ works by the rules of the generalised surgery procedure defined below, with the first diagram drawn below the second and all rays stitched together. An example for the multiplication is given in figure \ref{fig:multi}.
\begin{Def}[{\cite[Section 3 and Theorem 6.1.]{Brun12008}}]
Given two circle diagrams $a \la b$ and $b^* \mu d$, the first one drawn below the second and all corresponding rays stitched together (i.e. one has a diagram with a symmetric middle part) the \textit{generalised surgery procedure} works by the following steps:
\begin{enumerate}
\item Choose a symmetric pair of a cup and a cap in the middle section of the diagram that can be connected without crossings.
\item Check if these cup and cap belong to one circle/line segment and if it applies note down $1$ (anti-clockwise oriented circle), $x$ (clockwise oriented circle) or $y$ (line segment), respectively.
\item If they belong to different circles/line segments note down their kind for each of them ($1$, $x$ or $y$).
\item Delete the orientations from the circles/line segments these cup and cap are belonging to. Cut open the cup and the cap and stitch the loose ends together to form a pair of vertical line segments.
\item Re-orient the obtained circle diagram using the following rules:
\begin{enumerate}
\item If one has cut one circle/line segment into two parts, one uses the rules:
\begin{align*}
&1 \mapsto 1 \otimes x + x \otimes 1, \quad x \mapsto x \otimes x, \quad y \mapsto x \otimes y
\end{align*}
where the first rule means that the diagram goes to a sum of two, with one of the two new circles oriented clockwise and the other anti-clockwise in each of these two diagrams, the rule $x \mapsto x \otimes x$ means that the clockwise oriented circle transforms into two clockwise oriented circles and the rule $y \mapsto x \otimes y$ indicates that the line segment becomes a clockwise oriented circle and a line segment.
\item If one has made one circle/line segment out of two, one has to use these rules:
\begin{align*}
&1 \otimes 1 \mapsto 1, \quad 1\otimes x \mapsto x, \quad x \otimes 1 \mapsto x, \quad x \otimes x \mapsto 0,\\
&1 \otimes y \mapsto y,\quad y \otimes 1 \mapsto y, \quad x \otimes y \mapsto 0, \quad y \otimes x \mapsto 0\\
&y \otimes y \mapsto \begin{cases} y \otimes y & \text{if both rays from one of the lines are oriented $\up$}\\ & \text {and both rays from the other line are oriented $\down$}\\ 0 & \text{otherwise}\end{cases} \end{align*}
For instance, the rule $x \otimes 1 \mapsto x$ here indicates that an anti-clockwise and a clockwise circle transform to one clockwise circle.
\end{enumerate}
\item Iterate the procedure on all summands until there are no cups and caps left in the middle part, then identify the two numberlines.
\end{enumerate}
\end{Def}

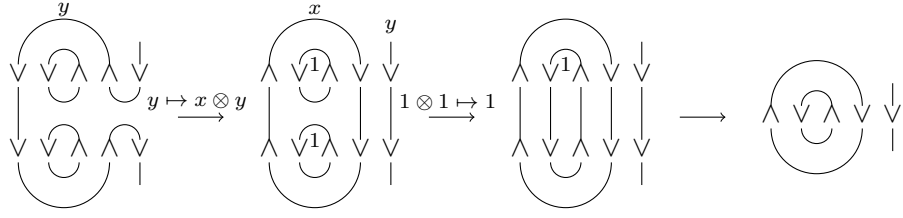
\begin{figure}
\caption{Multiplication of two elements}
	\label{fig:multi}
 \center
		\begin{tikzpicture}
		\begin{scope}

		\draw (0.4,0) node[above=-1.55pt] {$\down$};
		\draw (0.8, 0) node[above=-1.55pt] {$\down$};
		\draw (1.2,0) node[above=-1.55pt] {$\up$};
		\draw (1.6,0) node[above=-1.55pt] {$\up$};
		\draw (2,0) node[above=-1.55pt] {$\down$};
		\draw (0.8,0.3) arc (180:0:0.2);
		\draw (0.4,0.3) arc (180:0:0.6);
		\draw (2.0,0.3) -- (2.0,0.6);
		\draw (1.6,0) arc (180:360:0.2);
		\draw (0.8,0) arc (180:360:0.2);
		\draw (0.4,0) -- (0.4,-0.7);
		\end{scope}
		
				\begin{scope}[yshift=-1cm]
		\draw (0.4,0) node[above=-1.55pt] {$\down$};
		\draw (0.8, 0) node[above=-1.55pt] {$\down$};
		\draw (1.2,0) node[above=-1.55pt] {$\up$};
		\draw (1.6,0) node[above=-1.55pt] {$\up$};
		\draw (2,0) node[above=-1.55pt] {$\down$};
		\draw (0.8,0.3) arc (180:0:0.2);
		\draw (1.6,0.3) arc (180:0:0.2);
		\draw (0.4,0) arc (180:360:0.6);
		\draw (0.8,0) arc (180:360:0.2);
		\draw (2,0) -- (2,-0.3);
		\end{scope}
			\scriptsize
		\draw (1,1.03) node {$y$};		
		\draw (2.75, -0.2) node{$y \mapsto x \otimes y$};
		\normalsize
		\draw[->] (2.5,-0.4)--(3.1,-0.4);
		
		\begin{scope}[xshift=3.3cm]
				\draw (0.4,0) node[above=-1.55pt] {$\up$};
		\draw (0.8, 0) node[above=-1.55pt] {$\down$};
		\draw (1.2,0) node[above=-1.55pt] {$\up$};
		\draw (1.6,0) node[above=-1.55pt] {$\down$};
		\draw (2,0) node[above=-1.55pt] {$\down$};
		\draw (0.8,0.3) arc (180:0:0.2);
		\draw (0.4,0.3) arc (180:0:0.6);
		\draw (2.0,0.3) -- (2.0,0.6);
		\draw (0.8,0) arc (180:360:0.2);
		\draw (0.4,0) -- (0.4,-0.7);
		\draw (1.6,0) -- (1.6,-0.7);
		\draw (2,0) -- (2,-0.7);

				\begin{scope}[yshift=-1cm]
		\draw (0.4,0) node[above=-1.55pt] {$\up$};
		\draw (0.8, 0) node[above=-1.55pt] {$\down$};
		\draw (1.2,0) node[above=-1.55pt] {$\up$};
		\draw (1.6,0) node[above=-1.55pt] {$\down$};
		\draw (2,0) node[above=-1.55pt] {$\down$};
		\draw (0.8,0.3) arc (180:0:0.2);	
		\draw (0.4,0) arc (180:360:0.6);
		\draw (0.8,0) arc (180:360:0.2);
		\draw (2,0) -- (2,-0.3);
				\end{scope}
				
				\scriptsize
		\draw (1,1.03) node {$x$};	
		\draw (2,0.8) node {$y$};		
		\draw (1,0.3) node {$1$};
		\draw (1,-0.7) node {$1$};		
		\draw (2.75, -0.2) node{$1 \otimes 1 \mapsto 1$};
		\normalsize
		\draw[->] (2.5,-0.4)--(3.1,-0.4);
		
		\end{scope}

		\begin{scope}[xshift=6.6cm]
				\draw (0.4,0) node[above=-1.55pt] {$\up$};
		\draw (0.8, 0) node[above=-1.55pt] {$\down$};
		\draw (1.2,0) node[above=-1.55pt] {$\up$};
		\draw (1.6,0) node[above=-1.55pt] {$\down$};
		\draw (2,0) node[above=-1.55pt] {$\down$};
		\draw (0.8,0.3) arc (180:0:0.2);
		\draw (0.4,0.3) arc (180:0:0.6);
		\draw (2.0,0.3) -- (2.0,0.6);
		\draw (0.4,0) -- (0.4,-0.7);
		\draw (0.8,0) -- (0.8,-0.7);
		\draw (1.2,0) -- (1.2,-0.7);
		\draw (1.6,0) -- (1.6,-0.7);
		\draw (2,0) -- (2,-0.7);

				\begin{scope}[yshift=-1cm]
		\draw (0.4,0) node[above=-1.55pt] {$\up$};
		\draw (0.8, 0) node[above=-1.55pt] {$\down$};
		\draw (1.2,0) node[above=-1.55pt] {$\up$};
		\draw (1.6,0) node[above=-1.55pt] {$\down$};
		\draw (2,0) node[above=-1.55pt] {$\down$};
		\draw (0.4,0) arc (180:360:0.6);
		\draw (0.8,0) arc (180:360:0.2);
		\draw (2,0) -- (2,-0.3);
				\end{scope}
		\scriptsize	
		\draw (1,0.3) node {$1$};

		\normalsize
		\draw[->] (2.5,-0.4)--(3.1,-0.4);
				
		\end{scope}
		
		\begin{scope}[yshift=-0.55cm, xshift=9.9cm]
		\draw (0.4,0) node[above=-1.55pt] {$\up$};
		\draw (0.8, 0) node[above=-1.55pt] {$\down$};
		\draw (1.2,0) node[above=-1.55pt] {$\up$};
		\draw (1.6,0) node[above=-1.55pt] {$\down$};
		\draw (2,0) node[above=-1.55pt] {$\down$};
		\draw (0.8,0.3) arc (180:0:0.2);
		\draw (0.4,0.3) arc (180:0:0.6);
		\draw (2,0.3) -- (2,0.6);
		\draw (0.4,0) arc (180:360:0.6);
		\draw (0.8,0) arc (180:360:0.2);
		\draw (2,0) -- (2,-0.3);		
			\end{scope}
			\end{tikzpicture}
\end{figure}

The vectors $\left\{ e_\alpha | \alpha \in \La_m^n \right\} $ form a complete set of mutual orthogonal idempotents in $K_m^n$. We get
\begin{align*}
K_m^n= \bigoplus_{\alpha, \beta \in \La_m^n} e_\alpha K_m^n e_\beta
\end{align*}
and the summand $e_\alpha K_m^n e_\beta$ has a basis 
\begin{align*}
\left\{(\und{\alpha} \la \ove{\beta}) \left| \lambda \in \La_m^n \text{ the diagram is oriented} \right\} \right. .
\end{align*}
\subsection{Modules}
Following \cite[Section 5]{Brun12008}, we look at graded unital left $K_m^n$-modules $M$.
As mentioned already there are different types of important modules which we describe now in more detail:
\begin{itemize}
	\item The simple modules $L(\la)$ with $\la \in \La_m^n$. \\
	$x \in K_m^n$ operates on $L(\la)$ by 1 if $x=e_\la$ and by 0 otherwise. Note that these are onedimensional modules.\\
	By shifting the degree one gets all isomorphism classes of simple graded $K_m^n$-modules.
	\item The projective covers of the simple modules $L(\la)$ denoted by $P(\la):=K_m^n e_\la$, which have a basis
	\begin{align}
	\left\{ (\und{\alpha}\mu\ove{\la}) \left| \text{ for all } \alpha, \mu \in \La_m^n \text{ such that the diagram is oriented}\right\} \right. .
	\end{align}
	By shifting the degree one gets a full set of indecomposable projective modules.
	\item The cell or standard modules $M(\mu)$ with basis
		\begin{align*}
	\left\{
(c \mu | \:\big|\:
\text{for all oriented cup diagrams $c \mu$}\right\}
		\end{align*}
		and $(a\la b)(c\mu |)=(a\mu |)$ or 0 depending on the elements.
\end{itemize}
\begin{Remark}
Under the equivalence from Corollary \ref{Cor:equcat} these modules correspond to simples, projectives and Verma modules in the principal block of $O^\p$.
\end{Remark}
\subsection{q-decomposition numbers}
We have the following theorems about cell module filtrations of projectives and Jordan-Hoelder filtrations of cell modules, which say that $K_m^n$ is quasi-hereditary in the sense of Cline, Parshall and Scott \cite{Clin1988}.
\begin{Theorem}[{\cite[Theorem 5.1]{Brun12008}}]
\label{qh1}
For $\la \in \La_m^n$,
enumerate the elements of the set
$\{\mu \in \La_m^n\:|\:\und{\la}\mu \text{ is oriented}\}$
as $\mu_1,\mu_2,\dots,\mu_n = \la$
so that $\mu_i > \mu_j$ implies $i < j$.
Let $M(0) := \{0\}$ and for $i=1,\dots,n$ define
$M(i)$ to be the subspace of $P(\la)$
generated by $M(i-1)$ and the vectors
$$
\left\{
(c \mu_i \overline{\la} ) \:\big|\:
\text{for all oriented cup diagrams $c \mu_i$}\right\}.
$$
Then
$$
\{0\} = M(0) \subset M(1) \subset\cdots\subset M(n) = P(\la)
$$
is a filtration of $P(\la)$ as a $K_m^n$-module such that
$$
M(i) / M(i-1) \cong M(\mu_i) \langle \deg(\mu_i
\overline{\la})\rangle
$$
for each $i=1,\dots,n$.
\end{Theorem}

\begin{Theorem}[{\cite[Theorem 5.2]{Brun12008}}]\label{qh2}
For $\mu \in \La_m^n$, let $N(j)$ be the submodule of $M(\mu)$ spanned by all
graded pieces of degree $\geq j$. Then
$$
M(\mu) = N(0) \supseteq N(1) \supseteq N(2) \supseteq \cdots
$$
is a
filtration of $M(\mu)$ as a $K_m^n$-module such that
$$
N(j) / N(j+1) \cong \bigoplus_{\substack{\la  \subset \mu\text{\,with}\\
\deg(\underline{\la} \mu) = j}}
  L(\la) \langle j \rangle
$$
for each $j \geq 0$. Moreover, we have that $N(j) = 0$ for $j \gg 0$, i.e.
$M(\mu)$ is finite dimensional.
\end{Theorem}
By the BGG reprocity \cite[Theorem 9.8(f)]{Hump08} the two numbers
$d_{\la,\mu}^i(q):=[M(\mu):L(\la)\langle i\rangle]$ and $[P(\la):M(\mu)\langle i \rangle]$ are equal.

If we define the polynomial  
\begin{equation*} 
d_{\la,\mu}(q):=\sum d_{\la,\mu}^i(q) \cdot q^i =\left\{
\begin{array}{ll}
q^{\deg(\underline{\la} \mu)} &\text{if $\und{\la}\mu$ is oriented,}\\
0&\text{otherwise}
\end{array}\right.
\end{equation*} we get the \textit{q-decomposition matrix} $D$ which encodes the multiplicities in the above filtrations and is defined as 
\begin{equation}\label{decmat}
D_{\La_m^n}(q) = (d_{\la,\mu}(q))_{\la,\mu \in \La_m^n}.
\end{equation}
So in terms of the Grothendieck group, the theorems above tell us
\begin{align}
[M(\mu)] &= \sum_{\la \in \La_m^n} d_{\la,\mu}(q) [L(\la)],\label{MinL}\\
[P(\la)] &= \sum_{\mu \in \La_m^n} d_{\la,\mu}(q) [M(\mu)]\label{PinM}.
\end{align}
Note that there are no higher multiplicities by the theorems. This multiplicity freeness is a general phenomenon of symmetric hermitian pairs (see \cite[Theorem 1.1]{Boe90}), here reproved and illustrated nicely in terms of diagrams.

\begin{Lemma}
The radical filtration of the projective modules and the cell modules agrees with their grading filtration.

Moreover, in the case of the cell modules the filtration also coincides with the socle filtration, i.e. they are rigid.
\end{Lemma}
\begin{proof}
By Theorem \ref{qh1} and Theorem \ref{qh2} the projective and cell modules have simple heads. By \cite[Corollary 5.7.]{Brun22008} the algebra $K_m^n$ is Koszul. Using \cite[Prop 2.4.1.]{Beil1996} the two gradings agree.

By \cite[Corollary 6.7.]{Brun22008} the cell modules are rigid, i.e. the filtration also coincides with the socle filtration.
\end{proof}

Now we analyse the structure given by the filtration more detailly and work out some technical Lemmas.
\begin{Lemma}\label{Le:highdegrsimp}
The highest degree where a simple can occur in $P(\mu)$ is $2n$.
\end{Lemma}
\begin{proof}
Since $\deg(\underline{\la} \mu)\leq n$, one obtains that the highest degree of $M(\mu)$, where a simple could occur, is $n$. Putting the two formulas \eqref{MinL} and \eqref{PinM} together, we obtain the the stated result.
\end{proof}
We also look at the \textit{q-Cartan matrix}
\begin{equation}\label{cartanmat}
C_{\La_m^n}(q) = (c_{\la,\mu}(q))_{\la,\mu \in\La_m^n}
\end{equation}
where
\begin{equation*}
c_{\la,\mu}(q) := \sum_{j \in \Z} q^j \dim \hom_{K_{\La_m^n}}(P(\la),P(\mu))_j \in \Z((q)).
\end{equation*}
We obtain the following results
\begin{Lemma}\label{Le:geshifmaps}
In $K_m^n$ we have $c_{\la,\mu}(q)=0$ for $q>2n$, i.e.
$$\hom_{K_{\La_m^n}}(P(\la)\langle i \rangle ,P(\mu))_0=0$$ unless $0 \leq i \leq 2n$.
\end{Lemma}
\begin{proof}
We know that a nonzero morphism $f \in \hom_{K_{\La_m^n}}(P(\la)\langle i \rangle ,P(\mu))_0$ maps the head of $P(\la)\langle i \rangle$ to a simple sitting in degree $i$ in $P(\mu)$. As we proved in Lemma \ref{Le:highdegrsimp} this can only exist for $0 \leq i \leq 2n$.
\end{proof}
The following provides lower and upper bounds for the decomposition numbers.
\begin{Lemma}\label{Le:condmaps}
In $K_m^n$ we have $d_{\la,\mu}=0$ unless $$0 \leq l(\la)-l(\mu) \leq n+2\sum_i \nes_{\und{\la}}(i) \leq n^2.$$
In particular we get $c_{\la, \mu}=0$ unless $$l(\la)-l(\mu) \leq n+2\sum_i \nes_{\und{\la}}(i) \leq n^2.$$
\end{Lemma}
\begin{proof}
Assume $d_{\la,\mu}(q)\neq 0$ for some $q$. This means that $\und{\la}\mu$ is oriented. By \cite[Lemma 2.3.]{Brun12008} it follows that $\la \leq \mu$ in the Bruhat ordering, which leads to $l(\la) \geq l(\mu)$. \\
Now we want to find $\la$ and $\mu$ such that $l(\la)-l(\mu)$ is maximal and $\und{\la} \mu$ is oriented. Assume $\la$ is fixed. We have to look for the weight $\mu$ with the smallest length such that $\und{\la}\mu$ is oriented. This is obviously obtained if all $\up$'s and $\down$'s on the end of a cup in $\la$ are interchanged. Since a $\up$ on the $i$th cup has been moved $1+2 \nes_{\und{\la}}(i)$ positions to the right, the length is changed by $\sum_i(2 \nes_{\und{\la}}(i)+1)$.\\ Therefore, we obtain
$$0 \leq l(\la)-l(\mu) \leq n+2\sum_i \nes_{\und{\la}}(i).$$
Since $\sum_i \nes_{a}(i)$ for any cup diagram becomes the biggest, if all cups are nested (and then the first cup contains no other, the second one and so on), in that case we obtain 
$$2\sum_i \nes_{a}(i)=2\sum_{i=1}^n (i-1)=(n-1)n$$ and therefore the inequality is shown.

For $c_{\la, \mu}$ one has to look for a simple $L(\la)$ occurring in $P(\mu)$, especially occurring in some $M(\nu)$, i.e. $d_{\la,\nu}\neq 0$ and $d_{\mu,\nu}\neq 0$. This yields
\begin{align*}
l(\la)-l(\nu) \leq n+2\sum_i \nes_{\und{\la}}(i) 
\end{align*} and $0 \leq l(\mu)-l(\nu)$ and therefore
\begin{equation*}
l(\la)-l(\mu) \leq l(\la)-l(\nu) \leq n+2\sum_i \nes_{\und{\la}}(i).\qedhere
\end{equation*} 
\end{proof}
\begin{Remark}
Later on in Section \ref{ch:spec} we will deduce and use stronger inequalities which however require an explicit and detailed knowledge of the structure of Vermas and projectives and are obtained by combining the information one has about the shifts with those about the decomposition numbers.
\end{Remark}
\section{$\End(\bigoplus P(\la))$}\label{End}
Taking the above description for projective modules, we see that a minimal projective generator of $K_{\La_m^n}-\Mod$ is $\bigoplus P(\la)\cong K_{\La_m^n}$.
Any endomorphism is given by multiplication with an element of the algebra. Writing down the $\hom$-spaces (cf. \cite[equation (5.9)]{Brun12008}) we get:
$$
\hom_{K_{\La_m^n}}(P(\la), P(\mu)) = \hom_{K_{\La_m^n}} (K_{\La_m^n} e_\la, K_{\La_m^n} e_\mu)= e_\la
K_{\La_m^n} e_\mu
$$
and $e_\la K_{\La_m^n} e_\mu$ has basis
$$\left\{(\underline{\la}\nu\overline{\mu})\:\big|\:
\nu \in \La_m^n \text{ such that }\und{\la}\nu\ove{\mu} \text{ is an oriented circle diagram}\right\}.$$
We are now interested in the degree 1 component of $\hom_{K_{\La_m^n}}(P(\la), P(\mu))$, i.e. we look for elements $\nu$ s.t.
$\deg(\und{\la}\nu\ove{\mu})=1$:\\
Since then $1=\deg(\und{\la}\nu\ove{\mu})=\deg(\und{\la}\nu)+\deg(\nu\ove{\mu})$, one summand has to be $0$ and the other one has to be $1$.
\begin{enumerate}
	\item $\deg(\und{\la}\nu)=0$, i.e. $\la=\nu$, so we look for an oriented cap diagram $\la\ove{\mu}$ of degree $1$. It exists iff $\la>\mu$ and $\mu= \la.w$ with $w$ changing the $\up$ and $\down$ (in this ordering) at the end of a cup into a $\down$ and $\up$.
	
	\item $\deg(\nu \ove{\mu})=0$, i.e. $\mu=\nu$, so we look for an oriented cup diagram $\und{\la}\mu$ of degree 1. It exists iff $\mu>\la$ and $\la=  \mu.w$ with $w$ changing the $\down$ and $\up$ at the end of a cap.
\end{enumerate}
So we get $\dim(\hom_{K_{\La_m^n}}(P(\la), P(\mu))_1)\leq 1$. \\
Our first goal is to write down explicitly the morphisms corresponding to the multiplication by $\und{\la}\nu\ove{\mu}$  in terms of their action on basis vectors. In this way we determine all relations between compositions of degree $1$ morphisms. (Note that this determines the algebra completely, since it is Koszul \cite[Theorem 5.6.]{Brun22008}, in particular quadratic by \cite[Corollary 2.3.3.]{Beil1996}).

\section{Projective functors}

\subsection{Crossingless matchings and $K_{\La \Ga}^t$}
In \cite{Brun22008} the new category of geometric bimodules for $K_\La$ is defined. We will cite some general results and then only work in detail with two specific examples. For the general construction look in \cite[chapter 2 and 3]{Brun22008}.
\begin{Def}
For $\La=\La_m^n$ and $\Ga=\La_{m'}^{n'}$, an \emph{oriented $\La\Ga$-matching} is a diagram $\la t \mu$ with $\la \in \La$, $\mu \in \Ga$ such that 
\begin{itemize}
\item $t$ is obtained as cup diagram drawn above a cap diagram with connected rays
\item All cups, caps and rays in $t$ are oriented.
\end{itemize}

\end{Def}

\begin{Example}[The $\La_m^n\La_{m-1}^{n-1}$-matching $\la t_i \la'$]\label{Ex:ti}
There is an oriented $\La_m^n\La_{m-1}^{n-1}$-matching $\la t_i \la'$ with
$\la'$ obtained from $\la$ by deleting the $i$th and $i+1$st vertex and $t_i$ matching the $i$th and $i+1$st vertex by a cap. An example is given in figure \ref{fig:lat3la'}.
\begin{figure}
\caption{$\La_3^2\La_2^1$-matching $\la t_3 \la'$}
	\label{fig:lat3la'}
 \center
 \begin{tikzpicture}
{\begin{scope}
		\draw (0.6, 0) node[above=-1.55pt] {$\down$};
		\draw (1.2,0) node[above=-1.55pt] {$\up$};
		\draw (1.8,0) node[above=-1.55pt] {$\down$};

		\draw (0.6,0) -- (0.4,-0.7);
		\draw (1.2,0) -- (0.8,-0.7);
		\draw (1.8,0) -- (2,-0.7);

				\begin{scope}[yshift=-1cm]
						\draw (1.2,0.3) arc (180:0:0.2);
				\draw (0.4,0) node[above=-1.55pt] {$\down$};
		\draw (0.8, 0) node[above=-1.55pt] {$\up$};
		\draw (1.2,0) node[above=-1.55pt] {$\down$};
		\draw (1.6,0) node[above=-1.55pt] {$\up$};
		\draw (2,0) node[above=-1.55pt] {$\down$};
				\end{scope}
				\end{scope}}
\end{tikzpicture}
\end{figure}

\begin{figure}
\caption{$\La_2^1\La_3^2$-matching $\la' t_3^* \la$}
	\label{fig:lat'la}
 \center
 \begin{tikzpicture}
{\begin{scope}
				\draw (0.4,0) node[above=-1.55pt] {$\down$};
		\draw (0.8, 0) node[above=-1.55pt] {$\up$};
		\draw (1.2,0) node[above=-1.55pt] {$\down$};
		\draw (1.6,0) node[above=-1.55pt] {$\up$};
		\draw (2,0) node[above=-1.55pt] {$\down$};
		\draw (1.2,0) arc (180:360:0.2);
		\draw (0.4,0) -- (0.6,-0.7);
		\draw (0.8,0) -- (1.2,-0.7);
		\draw (2,0) -- (1.8,-0.7);

				\begin{scope}[yshift=-1cm]
		\draw (0.6, 0) node[above=-1.55pt] {$\down$};
		\draw (1.2,0) node[above=-1.55pt] {$\up$};
		\draw (1.8,0) node[above=-1.55pt] {$\down$};
				\end{scope}
				\end{scope}}
\end{tikzpicture}
\end{figure}				
Similarly one defines the $\La_{m-1}^{n-1}\La_m^n$-matching $\la' t_i^* \la$ (figure \ref{fig:lat'la}).
\end{Example}
\begin{Def}
An \emph{oriented $\La\Ga$-circle diagram} is a diagram obtained by gluing a cup diagram $a \in \La$ and a cap diagram $b \in \Ga$ below and on top of an oriented $\La\Ga$-matching such that everything is oriented.

The set of these diagrams is denoted by $K_{\La\Ga}^t$. Similarly to $K_\La$ one defines a bimodule structure via a surgery procedure from the top and bottom on this space (cf. \cite[chapter 3]{Brun22008}). Therefore, one also obtains a graded structure.

The \emph{upper reduction} of a diagram $t\la b$ means the oriented cap diagram obtained by removing the upper number line and all upper circles and lines, i.e. those that do not cross the bottom number line. An example is given in figure \ref{fig:upred}.

\end{Def}

\begin{figure}
\caption{The upper reduction of $t\la b$}
	\label{fig:upred}
 \center
		\begin{tikzpicture}
		\begin{scope}

		\draw (0.4,0) node[above=-1.55pt] {$\down$};
		\draw (0.8, 0) node[above=-1.55pt] {$\down$};
		\draw (1.2,0) node[above=-1.55pt] {$\up$};
		\draw (1.6,0) node[above=-1.55pt] {$\up$};
		\draw (2,0) node[above=-1.55pt] {$\down$};
		\draw (0.8,0.3) arc (180:0:0.2);
		\draw (0.4,0.3) arc (180:0:0.6);
		\draw (2.0,0.3) -- (2.0,0.6);
		\draw (0.8,0) arc (180:360:0.2);
		\draw (0.4,0) -- (0.8,-0.7);
		\draw (1.6,0) -- (1.2,-0.7);
		\draw (2,0) -- (1.6,-0.7);
		
		\end{scope}
		
				\begin{scope}[yshift=-1cm]
		\draw (0.8,0) node[above=-1.55pt] {$\down$};
		\draw (1.2, 0) node[above=-1.55pt] {$\up$};
		\draw (1.6,0) node[above=-1.55pt] {$\down$};
		\end{scope}
		\draw[->] (2.5,-0.4)--(3.1,-0.4);

		\begin{scope}[yshift=-0.55cm, xshift=3.3cm]
		\draw (0.4,0) node[above=-1.55pt] {$\down$};
		\draw (0.8, 0) node[above=-1.55pt] {$\up$};
		\draw (1.2,0) node[above=-1.55pt] {$\down$};
		\draw (0.4,0.3) arc (180:0:0.2);
		\draw (1.2,0.3) -- (1.2,0.6);
			\end{scope}
			\end{tikzpicture}
\end{figure}
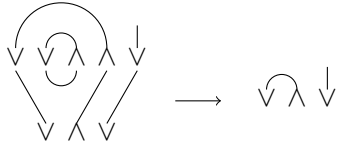

\subsection{Projective functors}
We want to study functors obtained by tensoring with geometric bimodules. 
In \cite[Chapter 4]{Brun22008} the functor
\begin{equation}\label{pfun}
G_{\La\Ga}^t := K_{\La\Ga}^t \langle -\caps(t)\rangle \otimes_{K_\Ga} ?
:\mod{K_\Ga} \rightarrow \mod{K_\La}
\end{equation} 
between graded module categories is defined.\\
\begin{Remark}\rm{
In the special case of $t=t_i$ defined above and $\La=\La_m^n$ and $\Ga=\La_{m-1}^{n-1}$ the functor $G_{\La\Ga}^t$ is a functor well-known in Lie Theory. It is obtained as the composition of Jantzen's translation functors $\psi_{\alpha_i}: \O(m,n,0) \to \O(m,n,i)$ (tensoring with a finite dimensional module and projecting to another block, cf. \cite[chapter 10]{Enri1987}, \cite[Chapter 2]{Jan79}) and the Enright-Shelton equivalence (cf. \cite[Prop. 11.2]{Enri1987}) $\O(m,n,i) \to \O(m-1,n-1,0)$. \\
The functor in the case $t=t_i^*$ defined above and $\La=\La_{m-1}^{n-1}$ and $\Ga=\La_m^n$ is the composition of the Enright-Shelton equivalence $\O(m-1,n-1,0) \to \O(m,n,i)$ with the translation functor $\phi_{\alpha_i}: \O(m,n,i) \to \O(m,n,0)$.\\
Having in mind this translation to classical Lie theory in these examples, the following theorems can be understood as a generalization of structure theorems for the special projective functors.}
\end{Remark}

First we give a theorem which justifies the name ``projective functor'', since it sends projective modules to projectives:

\begin{Theorem}[{\cite[Theorem 4.2]{Brun22008}}]
\label{pf}
Let $t$ be a proper $\La\Ga$-matching and $\gamma \in \Ga$.
\begin{itemize}
\item[\rm(i)]
We have that
$G^t_{\La\Ga} P(\ga)
\cong K_{\La\Ga}^t e_\ga \langle -\caps(t)\rangle$
as left $K_\La$-modules.
\item[\rm(ii)]
The module $G^t_{\La\Ga} P(\ga)$ is non-zero if and only if the rays of each
 upper line in $t \ga \overline{\ga}$
are oriented so that one is $\up$ and one is $\down$.
\item[\rm(iii)]
Assuming the condition from (ii) is satisfied,
define $\la \in \La$ by declaring that $\overline{\la}$ is
the upper reduction of $t \overline{\gamma}$,
and let $n$ be the number of upper circles
removed in the reduction process.
Then
$$
G_{\La \Ga}^t P(\ga) \cong
P(\la) \otimes R^{\otimes n} \langle
\cups(t)-\caps(t)\rangle.
$$
as graded left $K_\La$-modules (where the $K_\La$-action on
$P(\la) \otimes R^{\otimes n}$ comes from its action on the
first tensor factor and $R=\C[x]/(x^2)$).
\end{itemize}
\end{Theorem}

Brundan and Stroppel also prove that projective functors preserve the category of modules with graded cell module filtration. More precisely
\begin{Theorem}[{\cite[Theorem 4.5(i)]{Brun22008}}]
\label{vf}
Let $t$ be a proper $\La \Ga$-matching and $\ga \in \Ga$.
The $K_\La$-module
$G^t_{\La\Ga} M(\ga)$
has a filtration
$$
\{0\} = N(0) \subset N(1) \subset \cdots \subset N(n)
= G^t_{\La\Ga}M(\ga)
$$
such that
$N(i) / N(i-1) \cong M(\mu_i) \langle \deg(\mu_i t \ga)  - \caps(t)\rangle$
for each $i$.
Here $\mu_1,\dots,\mu_n$ denote the elements of the set
$\{\mu \in \La\:|\:\mu t \ga \text{ is oriented}\}$ ordered so that
$\mu_i > \mu_j$ implies $i < j$.
\end{Theorem}

\subsection{Construction of linear projective resolutions of cell modules}\label{sec:linprojres}
In this part we give an explicit way to construct projective resolutions of cell modules in $K_m^n-\mod$ using tools from the proof of \cite[Theorem 5.3]{Brun22008}. This construction works by an interesting simultaneous induction varying the underlying algebra and the highest weights. We assume that we know the resolutions in $K_{m-1}^{n-1}-\mod$. Note that for $K_{m}^0$ and $K_0^n$ we only have one indecomposable module, which is projective, simple and cell module at once. We want to compute the projective resolution of $M(\la)$, assuming that we also know the projective resolutions for $M(\mu$) with $\mu > \la$. \\
We start with a general definition.
\begin{Def}[{\cite[(5.6)]{Brun22008}}]
For an index $0\leq i<m+n-1$ define
\begin{equation}\label{laij}
\La_{i,i+1}^{\!\vee \wedge}
:= \left\{\nu\in \La_m^n\:\Bigg|\:
\begin{array}{l}
\text{the $i$th vertex of $\nu$ is labelled $\down$}\\
\text{the $(i+1)$st vertex of $\nu$ is labelled $\up$}
\end{array}
\right\}.
\end{equation}
\end{Def}
Now we are able to compute linear projective resolutions of cell modules.
\begin{Theorem}\label{Th:constrproj}
Each cell module $M(\la)$ in $K_m^n$ has a linear projective resolution which can be constructed inductively from knowing the resolutions in $K_{m-1}^{n-1}$ and those of $M(\mu)$ for $\mu > \la \in \La_m^n$.
\end{Theorem}
\begin{proof}
To start the induction, first note that for the dominant weight in $K_m^n$ we have $M(\la_0)=P(\la_0)$ and hence a projective resolution is given by $0 \to P(\la_0) \to M(\la_0) \to 0$. We also know that in $K_m^0-\mod$ and $K_0^n-\mod$ there only exists one cell module which is a simple projective cell module, so there we know the resolution for this module and hence for all cell modules in $K_m^0-\mod$ and $K_0^n-\mod$, respectively. So we may assume that for a fixed $\la \in K_m^n$ we already have resolutions for all cell modules in $K_{m-1}^{n-1}-\mod$ and for all $M(\mu)$ with $\mu > \la$.

Set $\La=\La_m^n$ and $\Ga=\La_{m-1}^{n-1}$. For $\la \in \La_{i,i+1}^{\!\vee \wedge}$ let $\la' \in \Ga$ be the weight obtained from $\la$ by deleting the $i$th and $(i+1)$st vertex.\\
 Let $t_{i}$ be the $\La \Ga$-matching given in Example \ref{Ex:ti} with a cap joining the $i$th and $(i+1)$st position and lines on the others. By the definition of $\ove{\la}$ it equals the upper reduction of $t_i \ove{\la'}$ and no circles are removed, hence 
\begin{equation}\label{prback}
G^{t_i}_{\La\Ga} P(\la')  \cong P(\la)\langle-1\rangle
\end{equation}
 by Theorem \ref{pf}.\\
Define $\la''$ to be the weight one gets by interchanging the $i$th and $(i+1)$st vertex in $\la$ (i.e. $\la''=\la.s_i$), then Theorem \ref{vf} provides a short exact sequence
\begin{equation}\label{step2}
\xymatrix{0 \ar[r] &M(\la'')\langle 1 \rangle
\ar[r]^{f}
&G^t_{\La\Ga} M(\la')\langle 1 \rangle
\ar[r] & M(\la)
\ar[r]& 0}.
\end{equation}
Now by our induction hypothesis, we already have constructed linear projective resolutions $P_\bullet(\la'')$ of $M(\la'')$ for $\la'' \in \La$ and $P_\bullet(\la')$ of $M(\la')$ with $\la' \in \Ga$. Denote the differential maps by $d_\bullet(\la'')$ and $d_\bullet(\la')$.\\
Applying the exact functor $G^{t_i}_{\La\Ga}$ to the second complex, we get a resolution of $G^{t_i}_{\La\Ga}
M(\la')$ by the complex $G^{t_i}_{\La\Ga}P_\bullet(\la')$ with differentials $G^{t_i}_{\La\Ga}d_\bullet(\la')$.
By (\ref{prback}) and the hypothesis that we already have a smaller linear projective resolution, we know that all $G^t_{\La\Ga} P_{k}(\la')\langle 1 \rangle$ are sums of projective modules with the head sitting in degree $k$.
 Now we are in the situation of Lemma \ref{Le:Coneconstr} and can apply the lemma to the complexes $P_\bullet(\la'') \to M(\la'')$ and $G^{t_i}_{\La\Ga}P_\bullet(\la') \langle 1 \rangle \to G^{t_i}_{\La\Ga}
M(\la')\langle 1 \rangle$.\\
By this a projective resolution of $M(\la)$ is given by the cone 
$$C(f_\bullet)=G^{t_i}_{\La\Ga}P_\bullet(\la') \langle 1 \rangle \oplus P_\bullet(\la'')\langle 1 \rangle [1].$$ Putting $d_0(\la'')=0$, for $n \geq 1$ we have the differentials
$${d}_n(\la)=\left( \begin{matrix} -d_{n-1}(\la'')&0\\f_{n-1}&G^t_{\La\Ga}(d_n(\la')) \end{matrix}\right)$$
and the map $d(\la): P_0(\la) \to M(\la)$ equals to $P_0(\la) \to G^{t_i}_{\Ga\La}M(\la') \to M(\la)$.
\end{proof}
\begin{Cor}\label{Cor:dif}
For two cell modules $M(\mu)$ and $M(\la)$ in $K_m^n-\mod$ and their projective resolutions $P_\bullet(\mu)$ and $P_\bullet(\la)$ constructed as in Theorem \ref{Th:constrproj}, we have
\begin{align*}
\hom^k(P_\bullet(\la), P_\bullet(\mu) \langle j \rangle)=0 \text{ unless } 0 \leq k-j \leq 2n.
\end{align*}
\end{Cor}
\begin{proof}
Since $P_\bullet(\mu)$ and $P_\bullet(\la)$ are linear projective resolutions, we know that the head of the $i$th component of $P_\bullet(\la)\langle j \rangle [k]$ sits in degree $i-k+j$. Therefore, an element of $\hom^k(P_\bullet(\la), P_\bullet(\mu) \langle j \rangle)$ is a morphism from a projective $P(\nu)$ in the $i$th component of $P_\bullet(\la)$ to a projective $P(\nu')\langle j-k \rangle$, i.e. a morphism in 
$$\hom_{K_\La}(P(\nu)\langle k-j \rangle ,P(\nu'))_0$$ which is only unequal zero for $0 \leq k-j \leq 2n$ by Lemma \ref{Le:geshifmaps}.
\end{proof}

\subsection{The combinatorial Kazhdan-Lusztig polynomials}
In this section we recall the definition of the combinatorial Kazhdan-Lusztig polynomials which describe the terms occurring in the projective resolution of cell modules.\\
We first give the recursive definition stated in \cite[Lemma 5.2.]{Brun22008}, going back originally to work of Lascoux and Schuetzenberger \cite{Lasc1981}.
\begin{Def}\label{recurse}
Let $p_{\la,\mu}(q) \in \Z[q]$ be the polynomials defined
by the following recursion formulas:
\begin{itemize}
\item[(i)]
If $\la = \mu$ then $p_{\la,\mu}(q) = 1$
and if $\la \not \leq \mu$ then $p_{\la,\mu}(q) = 0$.
\item[(ii)]
If $\la < \mu$, pick some index $i$
such that $\la \in \La_{i,i+1}^{\!\vee\wedge}$.
Then
\begin{equation*}
p_{\la,\mu}(q) =
\left\{
\begin{array}{ll}
p_{\la',\mu'}(q) + q p_{\la'',\mu}(q)&\text{if $\mu \in \La_{i,i+1}^{\!\vee\wedge}$,}\\
q p_{\la'',\mu}(q)&\text{otherwise,}
\end{array}\right.
\end{equation*}
where $\la'$ and $\mu'$ denote the weights in $\La_{m-1}^{n-1}$
obtained from $\la$
and $\mu$ by deleting vertices $i$ and $i+1$ (as above),
and $\la''$ is the weight obtained from $\la$ by interchanging the labels
on vertices $i$ and $i+1$.
\end{itemize}
\end{Def}
For easier calculation we also restate the direct construction of $p_{\la,\mu}$ given in \cite[chapter 5]{Brun22008}.
First of all, set $p_{\la,\mu}=0$ if $\la \not\le \mu$. 
A \emph{labelled cap diagram} C is a cap diagram whose unbounded chambers are labelled by zero and given two chambers separated by a cap, the label in the inside chamber is greater than or equal to the label in the outside chamber.

Denote by $D(\la,\mu)$ all labelled cap diagrams obtained by labelling the chambers of $\ove{\mu}$ in such a way that the label inside every inner cap (a cap containing no smaller one) is less than or equal to $l_i(\la,\mu)$ (defined in \eqref{diffl}), where i is the index of the vertex of the cap labelled by $\down$. Now define $\left|C\right|$ to be the sum of all labels in C. The polynomials are given by
\begin{equation*}
p_{\la,\mu}(q) := q^{l(\la)-l(\mu)}\sum_{C \in D(\la,\mu)} q^{-2|C|}.
\end{equation*}
In \cite[Lemma 5.2.]{Brun22008} it is shown that both definitions agree.
\begin{Example}
For $\la$ and $\mu$ the possibilities for the labelled cap diagrams $C \in D(\la, \mu)$ are presented in figure \ref{fig:kaz}. The possible values for $C$ therefore are $0$ and $1$. Since $(\la)-l(\mu)=4$, we get $p_{\la,\mu}(q)=q^{4}+q^{2}$.
		\begin{figure}
		 \center
		\caption{the labelled cap diagram}
				\label{fig:kaz}
		\begin{tikzpicture}[scale=1]
		\draw (-0.2,0) node[above=-1.55pt] {$\mu$};

		\draw (0.4,0) node[above=-1.55pt] {$\down$};
		\draw (0.8, 0) node[above=-1.55pt] {$\up$};
		\draw (1.2,0) node[above=-1.55pt] {$\down$};
		\draw (1.6,0) node[above=-1.55pt] {$\down$};
		\draw (2,0) node[above=-1.55pt] {$\up$};
		\draw (2.4, 0) node[above=-1.55pt] {$\down$}; 
		\draw (0.8,0.3) arc (0:180:0.2);
		\draw (1.2, 0.3) -- (1.2, 0.6);
		\draw (2,0.3) arc (0:180:0.2);
		\draw (2.4, 0.3) -- (2.4, 0.6);
		\begin{scope}[yshift=-0.4cm]
		\draw (-0.2,0) node[above=-1.55pt] {$l_i$};
		\draw (0.4,0) node[above=-1.55pt] {$0$};
		\draw (1.6,0) node[above=-1.55pt] {$1$};
		\end{scope}
		\begin{scope}[yshift=0.15cm]
		\scriptsize	
		\draw (0.6,0) node[above=-1.55pt] {$0$};
		\tiny
		\draw (1.8,0) node[above=-1.55pt] {$0/1$};	
		\end{scope}	
		\begin{scope}[yshift=-0.8 cm]
		\draw (-0.2,0) node[above=-1.55pt] {$\la$};
		\draw (0.4,0) node[above=-1.55pt] {$\down$};
		\draw (0.8,0) node[above=-1.55pt] {$\down$};
		\draw (1.2, 0) node[above=-1.55pt] {$\down$};
		\draw (1.6,0) node[above=-1.55pt] {$\down$};
		\draw (2.0,0) node[above=-1.55pt] {$\up$};
		\draw (2.4,0) node[above=-1.55pt] {$\up$};
		\end{scope}
		\end{tikzpicture}
\end{figure}
\end{Example}
Now we have the tools to write down the terms occurring in the projective resolution constructed in \ref{sec:linprojres}.
\begin{Theorem}[{\cite[Theorem 5.3]{Brun22008}}]\label{sumocc}
For $\la \in \La_m^n$, the resolution of $M(\la)$ constructed in \ref{sec:linprojres}
$$
\cdots
\stackrel{d_1}{\longrightarrow}
P_{1}(\la)
\stackrel{d_0}{\longrightarrow} P_{0}(\la)
\stackrel{\eps}{\longrightarrow} M(\la) \longrightarrow 0
$$
consists of $P_0(\la) := P(\la)$ and
$P_i(\la) := \bigoplus_{\mu \in \La_m^n}
p_{\la,\mu}^{(i)} P(\mu)\langle i \rangle$
for $i \geq 0$.
\end{Theorem}				
Using the above definition we can already make some explicit restrictions on terms occurring in the resolution.
\begin{Lemma}\label{Le:termprojres}
If a projective module $P(\nu)$ occurs as a direct summand in $P_i(\la)$ with $P_\bullet(\la)$ being the projective resolution constructed above, one has
\begin{align*}
l(\la)-i-(n^2-n-2\sum_i \nes_{\und{\nu}}(i)) \leq l(\nu) \leq l(\la) -i
\end{align*}
\end{Lemma}
\begin{proof}
To get $p_{\la,\nu}^{(i)}$ we have to compute the possible $C \in D(\la,\nu)$. Look on the cap belonging to the $j$th $\up$ occurring in $\ove{\nu}$, which is numbered by $k_j$ if the caps are counted by their endpoints, and this cap having starting point $i$. For the label of this cap we get
\begin{align*} 
l_i(\la, \nu)=&\{k|\:k \leq i
\text{ and vertex $k$ of $\nu$ is labelled $\up$}\}\\
 &-\{k|\:k \leq i \text{ and vertex $k$ of $\la$ is labelled $\up$}\} \\
&\leq \{k|\:k \leq i
\text{ and vertex $k$ of $\nu$ is labelled $\up$}\} \\
\end{align*}
So we have to count the number of $\up$'s to the left edge of the cap. There are $j-1-\nes_{\und{\nu}}(k_j)$ $\up$'s to the left, i.e. all which are left to the $j$th $\up$ (at the end of the cap) without those lying inside the cap.

For this we get 
\begin{align*}
0 \leq |C| &\leq \sum_{\substack{j \in \{ 1,\ldots n\}\\ \text{ a cap ending on the $j$th $\up$}}}{(j-1-\nes_{\und{\nu}}(k_j))}\\
&\leq\frac{n(n-1)}{2}-\sum_i \nes_{\und{\nu}}(i).
\end{align*}
If a term occurs in the resolution, one has $p_{\la,\nu}^{(i)} >0$, i.e. there is a $C$ such that
$i=l(\la)-l(\nu)-2|C|$. Taking the upper and lower bound for $C$ obtained before, one gets
$$l(\la)-i-(n^2-n-2\sum_i \nes_{\und{\nu}}(i)) \leq l(\nu) \leq l(\la) -i$$
and the claim of the Lemma follows.
\end{proof}										
\subsection{Possible maps between projective resolutions}
Using what we have obtained so far about our projective resolutions, we deduce the following result:
\begin{Lemma}\label{Le:mapprojres}
For $\la$, $\mu \in K_m^n$ we have
\begin{align}
\hom^k(P_\bullet(\la), P_\bullet(\mu))=0 \text{ unless } l(\la) \leq l(\mu)+n^2+k \label{inequ}
\end{align}
\end{Lemma}
\begin{proof}
A map between $P_\bullet(\la)$ and $P_\bullet(\mu)[k]$ is in each component a morphism between projectives. Including the shift we look for morphisms between projectives $P(\nu)$ occurring in $P_i(\la)$ and projectives $P(\nu')$ in $P_{i-k}(\mu)$. 
By Lemma \ref{Le:termprojres} we know
$$l(\la)-i-\left(n^2-n-2\sum_i \nes_{\und{\nu}}(i)\right) \leq l(\nu)$$
and
$$l(\nu') \leq l(\mu)-(i-k).$$
Therefore, we have:
\begin{equation}\label{1stnu}
l(\la)-i-\left( n^2-n-2\sum_i \nes_{\und{\nu}}(i)\right)-\left(l(\mu)-(i-k)\right)\leq l(\nu)-l(\nu')
\end{equation}
Since we have a morphism between these projectives we get from Lemma \ref{Le:condmaps}
\begin{equation}\label{2ndnu}
l(\nu)-l(\nu') \leq n+ 2\sum_i \nes_{\und{\nu}}(i).\end{equation}
Combining the two inequalities \eqref{1stnu} and \eqref{2ndnu}, we obtain
\begin{equation*}
l(\la)-i-\left( n^2-n-2\sum_i \nes_{\und{\nu}}(i)\right)-\left(l(\mu)-(i-k)\right) \leq n+2\sum_i \nes_{\und{\nu}}(i)\end{equation*}
which gives
\begin{equation*}l(\la) \leq l(\mu)+n^2+k. \qedhere \end{equation*}
\end{proof}															
\begin{Remark}

In Lemma \ref{Le:posschmaps} we will deduce this Lemma again in a special case by also incorporating the grading shifts. The results in Section \ref{sec:mapproj} imply amongst other things that the inequality in \eqref{inequ} is sharp.
\end{Remark}

\chapter{The $\Ext$-algebra of $\bigoplus M(\la)$}\label{ch:Ext}																			
In this section we introduce a tool which is not necessarily needed for the computations below but which simplifies them. Knowing the dimension of the $\Ext$-spaces will 
be quite convenient later. Namely when we construct explicitly extensions it will 
save us from showing that we have constructed enough of them.
														
In \cite{Shel1988}, Shelton computes the dimensions of these spaces in the hermitian symmetric cases, to which our case belongs to. In general, there is no explicit formula, not even a candidate.

\section[Shelton's dimension formulas]{The proof of the dimension formulas by Shelton and results about embeddings}
\subsection{Notation}
For the sake of better readability we are comparing Shelton's notation to ours. As we do, he takes $\p$ to be a parabolic subalgebra in a complex semisimple Lie algebra $\fr{g}$. $\fr{u}$ is the nilradical of $\p$ and $\fr{m}$ its complement, i.e. $\p=\fr{u}\oplus\fr{m}$. In our notation $\fr{m}$ is denoted by $\fr{l}$. $W$ is the Weyl group of $\fr{g}$ and $W_\fr{m}$ the Weyl group of $\fr{m}$. This is notated by $W_\fr{l}$ in our notation. The set of shortest elements $W^\fr{l}$ Shelton denotes by $W^{\fr{m}}$. $\omega_0$ is the longest element in $W$ and $\omega_{\fr{m}}$ the longest element in $W_{\fr{m}}$.\\
For a fixed antidominant integral weight $\la \in \fr{h}^*$ the parabolic Verma module of highest weight $\omega_{\fr{m}}y\la-\rho$ is denoted by $N_y$. If we take instead of $\la$ antidominant a dominant weight $\la_0=\omega_0 \la$ then $N_y=M(\omega_{\fr{m}}y \omega_0 \cdot \la_0)$ in our notation. The labelling sets agree by the following observation:

\begin{Lemma}
For arbitrary $y, x \in W$ the following holds:
\begin{enumerate}
	\item $\omega_{\fr{m}}y \omega_0 \in W^{\fr{m}} \Leftrightarrow y \in W^{\fr{m}}$
	\item $\omega_{\fr{m}}y \omega_0 <\omega_{\fr{m}}x \omega_0 \Leftrightarrow y > x$ in the Bruhat order
\end{enumerate}
\end{Lemma}
\begin{proof}
The basic ingredients of the proof are the knowledge, that in $W$ multiplication with $\omega_0$ changes the order (i.e. $x<y \Leftrightarrow x\omega_0>y\omega_0$), that for $x \in W^{\fr{m}}$ and $y_1 < y_2 \in W_{\fr{m}}$ we have $y_1 x<y_2 x$ and that for $x_1 < x_2 \in W^{\fr{m}}$ and $y \in W_{\fr{m}}$ we have $yx_1 < yx_2$. 
\begin{enumerate}
\item
First we show that for a root $s \in W_{\fr{m}}$ we have: $s$ simple $\Leftrightarrow s'=\omega_{\fr{m}}s\omega_{\fr{m}}$ is simple.\\
Since $W_\fr{m}$ is generated by a subset of simple reflections out of $W$, a reflection lying in $W_\fr{m}$ is simple in $W_\fr{m}$ if and only if it is simple in $W$. So we can use the ordinary length function on $W_\fr{m}$. Since $\omega_{\fr{m}}$ is the longest element in $W_\fr{m}$ it must become shorter by multiplication with another element $w$. We obtain the following equality for the length function:
$$l(w \cdot \omega_{\fr{m}})=l(\omega_{\fr{m}} \cdot w)=l(\omega_{\fr{m}})-l(w)$$
in $W_\fr{m}$ and by this also in $W$.\\
Therefore, we get
$$l(s')=l((\omega_{\fr{m}}s)\omega_{\fr{m}})= l(\omega_{\fr{m}})-(l(\omega_{\fr{m}}s))=l(\omega_{\fr{m}})-(l(\omega_{\fr{m}})-l(s))=l(s).$$
By this we have $l(s)=l(s')$ and $s$ is simple in $W$ if and only if $s'$ is simple in $W$.
But then they are also simple in $W_\fr{m}$.\\
Now we can return to the statement needed to show:
\begin{align*}
y \in W^{\fr{m}} &\Leftrightarrow &sy&>y \: &&\forall \text{ simple } s \in W_{\fr{m}} \\
& \Leftrightarrow &\omega_{\fr{m}}sy&<\omega_{\fr{m}}y \:&&\forall \text{ simple } s \in W_{\fr{m}}\\
& \Leftrightarrow &s'\omega_{\fr{m}}y&<\omega_{\fr{m}}y \: &&\forall \text{ simple } s' \in W_{\fr{m}}\left( \text{take } s'\omega_{\fr{m}}=\omega_{\fr{m}}s \right)\\
& \Leftrightarrow &s'\omega_{\fr{m}}y\omega_0&>\omega_{\fr{m}}y\omega_0 \: &&\forall \text{ simple } s' \in W_{\fr{m}}\\
& \Leftrightarrow &\omega_{\fr{m}}y \omega_0 &\in W^{\fr{m}}
\end{align*} 
\item 
For $x,y \in W^{\fr{m}}$ we have
\begin{align*}
x<y &\Leftrightarrow &\omega_{\fr{m}}x&<\omega_{\fr{m}}y\\
&\Leftrightarrow &\omega_{\fr{m}}x\omega_0&>\omega_{\fr{m}}y\omega_0
\end{align*} 
\end{enumerate}
\end{proof}

Taking $\fr{g}=\fr{sl}_{m+n}$ and $\fr{m}=(\fr{gl}_m \oplus \fr{gl}_n)\cap \fr{sl}_{m+n}$ we get $W=S_{m+n}$ and $W_{\fr{m}}=S_{m} \times S_{n}$. 

We want to compute the bijective map 
\begin{align*}
&&\pi : W^{\fr{m}} & \to W^{\fr{m}}\\
&&y & \mapsto \omega_{\fr{m}} y \omega_0
\end{align*}
\begin{Lemma}
An element $w \in W^{\fr{m}}$ has the form $w=s_{n}\cdot \ldots \cdot s_{i_n} \cdot s_{n-1}\cdot \ldots \cdot s_{i_{n-1}} \cdot \ldots \cdot s_{1}\cdot \ldots \cdot s_{i_1}$ and is send to the element $w'=s_{n}\cdot \ldots \cdot s_{m+n-1-i_1} \cdot s_{n-1}\cdot \ldots \cdot s_{m+n-1-i_{2}} \cdot \ldots \cdot s_{1}\cdot \ldots \cdot s_{m+n-1-i_n}$ by $\pi$.
\end{Lemma}
\begin{proof}
For the first statement just look at \cite[Appendix A]{Stro2005}.

The proof goes by induction on the length of $w$. We start with $w=e$, i.e. $i_j=j-1$ and see that $\pi(e)=\omega_{\fr{m}} \omega_0$ is the longest element in $W^{\fr{m}}$. For $w'$ from above we get $w'=s_{n}\cdot \ldots \cdot s_{m+n-1} \cdot s_{n-1}\cdot \ldots \cdot s_{m+n-2} \cdot \ldots \cdot s_{1}\cdot \ldots \cdot s_{m+n-n}$, so the hypothesis holds.

Assuming that the Lemma is true for all elements of length $l$, we have to show that it holds for an element of length $l+1$. Take $y=x s_k=s_{n}\cdot \ldots \cdot s_{i_n} \cdot \ldots \cdot s_{l}\cdot \ldots \cdot s_{i_{j}+1} \cdot \ldots \cdot s_{1}\cdot \ldots \cdot s_{i_1} \in W^{\fr{m}}$ with $l(x)=l$ and $l(y)=l+1$ and $i_j+1=k$. We want to show that $\pi(y)=s_{n}\cdot \ldots \cdot s_{m+n-1-i_1} \cdot \ldots \cdot s_{n-l+1}\cdot \ldots \cdot s_{m+n-1-{i_{j}+1}} \cdot \ldots \cdot s_{1}\cdot \ldots \cdot s_{m+n-1-i_n}$ which is the same as $\pi(x)s_{m+n-1-k+1}$.

We see that $\pi(y)=\omega_{\fr{m}}xs_k\omega_0=\omega_{\fr{m}}x\omega_0s'$. So we only have to show that $s_k \omega_0=\omega_0 s_{m+n-k}$. \\
For this, we write $\omega_0$ and $s_k$ in the matrix notation. We get
\begin{align*}
s_k \omega_0=
&\left(
\begin{matrix} 1 & 2& \ldots& k & k+1 & \ldots & m+n\\ 
							1& 2 &\ldots& k+1& k& \ldots & m 
\end{matrix}
\right)
\\
\circ &\left(
\begin{matrix} 1 & 2& \ldots& m+n-k & m+n-k+1 & \ldots & m+n\\ 
							m+n& m+n-1 &\ldots& k+1& k& \ldots & 1 
\end{matrix}
\right)\\
=&\left(
\begin{matrix} 1 & 2& \ldots& m+n-k & m+n-k+1 & \ldots & m+n\\ 
							m+n& m+n-1 &\ldots& k& k+1& \ldots & 1 
\end{matrix}
\right)\\
=&\left(
\begin{matrix} 1 & 2& \ldots& m+n-k & m+n-k+1 & \ldots & m+n\\ 
							m+n& m+n-1 &\ldots& k+1& k& \ldots & 1 
\end{matrix}
\right)\\
\circ &\left(
\begin{matrix} 1 & 2& \ldots& m+n-k &m+n-k+1 & \ldots & m+n\\ 
							1& 2 &\ldots& m+n-k+1& m+n-k& \ldots & m 
\end{matrix}
\right)
\\
=&\omega_0 s_{m+n-k}
\end{align*}
\end{proof}
This finished the explanation of the passage from the Shelton notation to ours.

\subsection{Dimension formulas}
\begin{Notation}
For $x,y \in W^{\fr{m}}$ write $E^k(x,y)=\dim \Ext^k(M(x.\la_0), M(y.\la_0))$.
\end{Notation}
Using our notation \cite[Theorem 1.3]{Shel1988} translates into the following statement:
\begin{Theorem}[Dimension of $\Ext$-groups]\label{Th:dim}
Take $\fr{g}$ and $\p$ as above. Let $x,y \in W^{\fr{l}}$ and let $s$  be a simple reflection with $x>xs$ and $xs \in W^{\fr{l}}$. We can compute  $E^k(x,y)$ by the formulas:
\begin{align}
1. \ &E^k(x,y)=&&0  &&\forall \ k \text{ unless } y<x \\
2. \ &E^k(x,x)=&&\begin{cases} 1 &\text{  for } k=0\\
																			 0 &\text{ otherwise} \end{cases} \\
\intertext{For $y<x$ we get in addition the following recursion formulas:}
3.\ & E^k(x,y) =&&E^k(xs,ys) &&\text{ if }y>ys \text{ and }ys \in W^{\fr{l}} \label{ca:numcas1}\\ 
4.\ & E^k(x,y) =&&E^{k-1}(xs,y) &&\text{ if }ys \notin W^{\fr{l}}\label{ca:numcas2} \\
5.\ & E^k(x,y) =&&E^{k-1}(xs,y) +E^k(xs,y) &&\text{ if }ys>y\text{ but }xs \not> ys \label{ca:numcas3}\\
6.\ &E^k(x,y) =&&E^{k-1}(xs,y) - E^{k+1}(xs,y)\notag \\ 
	&&&+E^k(xs,ys)
	&&\text{ if }x>xs>ys>y	\label{ca:numcas4}
	 \end{align}

\end{Theorem}

\section{Some $\Ext$-vanishing}
Later on we need some information about the $\Ext$-spaces in general and do not want to compute their dimensions in all cases. Therefore, we reprove the Delorme-Schmid Theorem (cf. \cite{Delo77}, \cite{schm81}) in our situation:
\begin{Lemma} \label{Le:Extrest}
For $\la, \mu \in \La_m^n$ we have
\begin{align*}
\Ext^k(M(\la),M(\mu))=0 \ \forall \ k >l(\la) - l(\mu).
\end{align*}
\end{Lemma}
\begin{proof}
Assume we have a chain map $P_\bullet(\la) \to P_\bullet(\mu)[k]$ with $k >l(\la) - l(\mu)$. We have to show that it is homotopic to zero.\\
On the $k$th component we would have a map $f_k:P_k(\la) \to P_0(\mu)=P(\mu)$. For $P(\nu)$ occurring as a direct summand in $P_k(\la)$ we know by Lemma \ref{Le:termprojres} that $$l(\nu) \leq l(\la) -k <l(\la)-(l(\la)-l(\mu))=l(\mu).$$ 
Hence by Lemma \ref{Le:condmaps} $L(\nu)$ does not occur in $M(\mu)$ and therefore the composition $P(\nu) \to P(\mu) \to M(\mu)$ is zero. We denote by $P^T_\bullet(\la)$ the truncated complex with 
\begin{align*}
P^T_i(\la) = \begin{cases} 0 & i<0\\ P_{i+k}(\la) &i \geq 0 \end{cases}
\end{align*}
which is a projective resolution of $\operatorname{im} d_k$. The truncation of the map $f_\bullet$ delivers a map $$\widetilde{f}_\bullet: P^T_\bullet(\la) \to P_\bullet(\mu)$$
Therefore, we are in the situation
\begin{displaymath}
\xymatrix{ 0 \ar[r] &\cdots \ar[r] \ar[d] &P^T_0(\la)\ar[d]^{\widetilde{f}_0} \ar[r] & \operatorname{im} d_k \ar[r] \ar[d]^{0} \ar[r] &0\\
 0 \ar[r] &\cdots \ar[r]  &P(\mu) \ar[r] & M(\mu) \ar[r]&0}
\end{displaymath}
and $\widetilde{f}$ is a lift of the zero map. Since the zero map between the complexes is also a lift of the zero map and two lifts of a map are equal up to homotopy (cf. \cite[Theorem III.1.3]{Gelf88}) the map $\widetilde{f}$ is nullhomotopic by a homotopy $H:P_\bullet^T(\la) \to P_\bullet(\mu)[-1]$. This extends to a homotopy $H:P_\bullet(\la) \to P_\bullet(\mu)[-1]$ by defining it to be zero on the other terms. Therefore, it is obvious that the map $f$ becomes nullhomotopic which proves the Lemma.
\end{proof}
\begin{Remark}
The result of Lemma \ref{Le:Extrest} could also be deduced from Shelton's formulas or from the explicit formulas Biagioli computes in \cite[Theorem 3.4.]{Biag2004}.
\end{Remark}

\chapter{Special cases}\label{ch:spec}
Now we want to compute the $\Ext$-algebra in two cases. The first one (Section \ref{sec:k1}), when $n=1$ and $m=N$, is easy to handle but it gives a good idea of what we have to do. This algebra also becomes important while studying Floer cohomology (cf. \cite{khovanov2002quivers}). In the second case (Section \ref{sec:Mod2}) we take $n=2$ and $m=N-1$. 
\section{Modules in $K_{N}^1-\mod$}\label{sec:k1}

We label our reflections in $W^\fr{l}$ by $s_1 \cdots s_j$ and denote the corresponding weight by $(j)=\la_0.s_1s_2\dots s_j$.
\subsection*{q-decomposition numbers in $K_{N}^1-\mod$}
The first step is to compute the q-decomposition numbers and write down projectives in terms of Vermas (Table \ref{proverm1}) and Vermas in terms of simple modules (Table \ref{vermsimp1}). Since there is at most one cup, the decomposition numbers are either $0, 1$ or $q$. Combining the two tables, we write down projectives in terms of simples (Table \ref{prosimp1}).

\begin{table}[ht]
\footnotesize
\caption{\label{proverm1}$\mu$ such that $d_{\la,\mu}=\_$}
	\centering		\begin{tabular}{|lc|l|l|}
		\toprule
		$\la=(j)$ &  &$1$&$q$ \\ \hline
		$j\neq 0$ &
		\begin{tikzpicture}[scale=1]
		\draw (0,0) node[above=-1.55pt] {$\cdots$};
		\draw (0.4,0) node[above=-1.55pt] {$\down$};
		\draw (0.8, 0) node[above=-1.55pt] {$\up$};
		\draw (1.2,0) node[above=-1.55pt] {$\cdots$};
		\draw (0.4,0) arc (180:360:0.2);
		\end{tikzpicture}
														
		 &$(j)$&$(j-1)$\\ \hline

		$j=0$ &
		\begin{tikzpicture}[scale=1]
		\draw (0,0) node[above=-1.55pt] {$\up$};
		\draw (0.4,0) node[above=-1.55pt] {$\cdots$};
		\draw (0.0, -0.5)--(0.0,0);
		\end{tikzpicture}
														
		 & $(0)$ & \\ \hline
		 
		\end{tabular}
\end{table}

\begin{table}
\footnotesize
\caption{\label{vermsimp1}$\la$ such that $d_{\la,\mu}=\_$}
	\centering		\begin{tabular}{|lc|l|l|}
		\toprule
		$\mu=(j)$ &  &$1$&$q$ \\ \hline
		$j <N$ &
		\begin{tikzpicture}
		\draw (0,0) node[above=-1.55pt] {$\cdots$};
		\draw (0.4,0) node[above=-1.55pt] {$\down$};
		\draw (0.8, 0) node[above=-1.55pt] {$\up$};
		\draw (1.2,0) node[above=-1.55pt] {$\down$};
		\draw (1.6,0) node[above=-1.55pt] {$\cdots$};		
				\end{tikzpicture}
														
		 &$(j)$&$(j+1)$ \\ \hline
		 
		$j=N$&
		\begin{tikzpicture}
		\draw (0,0) node[above=-1.55pt] {$\cdots$};
		\draw (0.4,0) node[above=-1.55pt] {$\down$};
		\draw (0.8, 0) node[above=-1.55pt] {$\up$};
				\end{tikzpicture}
														
		 &$(N)$& \\ \hline

		\end{tabular}
\end{table}

\begin{table}
\footnotesize
\caption{\label{prosimp1} Filtration of projective module $P(\la)$ by simple modules, same colour belonging to the same Verma module}
	\centering		\begin{tabular}{|l|c|}
		\toprule
		$\la=(j)$ & $P(j)$\\
		\hline
		$\begin{array}{l} j \neq 0 \\ j \neq N\end{array}$&
		$\begin{array}{l}
\hspace*{1cm}\blu L(j)\\
\blu L(j+1)   \re L(j-1) \\
\hspace*{1cm} \re L(j) 
\end{array}$\\ \hline

		$j=0$&
		$\begin{array}{l}
\blu L(0)\\
\blu L(1)   
\end{array}$\\ \hline

		$j=N$&
		$\begin{array}{l}
\blu L(N)\\
 \re L(N-1) \\
\re L(N) 
\end{array}$\\ \hline

		\end{tabular}
\end{table}

\subsection*{Morphisms between projective modules}\label{sec:mapproj1}
Recalling from Section \ref{End} the results about degree one morphisms, we compute their compositions. For this, take $\la$ and $\mu$ as in Section \ref{End}, so we have $\la>\mu$ or $\mu>\la$. In the first case the morphism $P(\la) \to P(\mu)$ is given by the vector $\und{\la}\la\ove{\mu}$, in the second case by $\und{\la}\mu\ove{\mu}$. Composing such maps just means to multiply the vectors (viewed as elements in the algebra).
From Table \ref{prosimp1} we see that for $\nu \neq \la$ there is no degree two morphism $P(\la) \to P(\nu)$ and so each composition $P(\la) \to P(\mu) \to P(\nu)$ is zero.\\
Be careful that one has to read the diagrams form bottom to top, so multiplying $a\la b$ with $b^*\mu c$, a is the lowest cap diagram, above one has $\la$, then $b$ then $b^*$ on top, etc. 
\begin{table}
\footnotesize
\caption{\label{tab:mapprojdownup1}$P(\la) \to P(\mu) \to P(\la)$ with $\la >\mu$}
	\centering		\begin{tabular}{|m{1.7cm}|m{1.5cm}|m{5.3cm}| m{2cm}|}
		\toprule
		 
				$\la$ & $\mu$ &$\und{\la}\la\ove{\mu} \cdot \und{\mu}\la\ove{\la}=\und{\la}\gamma\ove{\la}$&$\gamma$\\ \hline
		$\begin{array}{l} (j)\\j \geq 1\end{array}$ &$(j+1)$&
		\vspace{0.4cm}
		\begin{tikzpicture}
		\begin{scope}
		\draw (0,0) node[above=-1.55pt] {$\cdots$};
		\draw (0.4,0) node[above=-1.55pt] {$\down$};
		\draw (0.8, 0) node[above=-1.55pt] {$\up$};
		\draw (1.2,0) node[above=-1.55pt] {$\down$};
		\draw (1.6,0) node[above=-1.55pt] {$\cdots$};
  	\draw (0.4,0.3) arc (180:0:0.2);
		\draw (1.2,0.3) -- (1.2,0.6);
		\draw (0.8,0) arc (180:360:0.2);
		\draw (0.4,0) -- (0.4,-0.3);
		\end{scope}
		
				\begin{scope}[yshift=-1cm]
		\draw (0,0) node[above=-1.55pt] {$\cdots$};
		\draw (0.4,0) node[above=-1.55pt] {$\down$};
		\draw (0.8, 0) node[above=-1.55pt] {$\up$};
		\draw (1.2,0) node[above=-1.55pt] {$\down$};
		\draw (1.6,0) node[above=-1.55pt] {$\cdots$};
		\draw (0.8,0.3) arc (180:0:0.2);
		\draw (0.4,0.3) -- (0.4,0.6);
		\draw (0.4,0) arc (180:360:0.2);
		\draw (1.2,0) -- (1.2,-0.3);
		\end{scope}
		
		\draw[->] (2,-0.4)--(2.3,-0.4);
		
		\begin{scope}[yshift=-0.55cm, xshift=2.9cm]
		\draw (0,0) node[above=-1.55pt] {$\cdots$};
		\draw (0.4,0) node[above=-1.55pt] {$\up$};
		\draw (0.8, 0) node[above=-1.55pt] {$\down$};
		\draw (1.2,0) node[above=-1.55pt] {$\down$};
		\draw (1.6,0) node[above=-1.55pt] {$\cdots$};
		\draw (0.4,0.3) arc (180:0:0.2);
		\draw (1.2,0.3) -- (1.2,0.6);
		\draw (0.4,0) arc (180:360:0.2);
		\draw (1.2,0) -- (1.2,-0.3);
		\end{scope}
				\end{tikzpicture}
					 &$(j-1)$\\ \hline

		$(0)$ &$(1)$&
		\vspace{0.4cm}
		\begin{tikzpicture}
		\begin{scope}
		\draw (0,0) node[above=-1.55pt] {$\up$};
		\draw (0.4,0) node[above=-1.55pt] {$\down$};
		\draw (0.8, 0) node[above=-1.55pt] {$\cdots$};
		\draw (0,0.3) -- (0,0.6);
		\draw (0.4,0.3) -- (0.4,0.6);
		\draw (0,0) arc (180:360:0.2);
		\end{scope}
		
				\begin{scope}[yshift=-1cm]
		\draw (0,0) node[above=-1.55pt] {$\up$};
		\draw (0.4,0) node[above=-1.55pt] {$\down$};
		\draw (0.8, 0) node[above=-1.55pt] {$\cdots$};
		\draw (0,0.3) arc(180:0:0.2);
		\draw (0,0) -- (0,-0.3);
		\draw (0.4,0) -- (0.4,-0.3);
				\end{scope}
		
		\draw[->] (2,-0.4)--(2.3,-0.4);
		
		\begin{scope}[yshift=-0.55cm, xshift=2.9cm]
		\normalsize
		\draw (0,0) node[above=-1.55pt] {$0$};
		\end{scope}
				\end{tikzpicture}

		 &\hfill\\ \hline	 		 
		\end{tabular}
\end{table}

\begin{table}
\footnotesize
\caption{\label{tab:mapprojupdown1}$P(\la) \to P(\mu) \to P(\la)$ with $\la <\mu$}
	\centering		\begin{tabular}{|m{1.7cm}|m{1.5cm}|m{5.3cm}| m{2cm}|}
		\toprule
	 		$\la$ & $\mu$ &$\und{\la}\la\ove{\mu} \cdot \und{\mu}\la\ove{\la}=\und{\la}\gamma\ove{\la}$&$\gamma$\\ \hline
		$\begin{array}{l} (j)\\j > 1\end{array}$ &$(j-1)$&
		\vspace{0.4cm}
		\begin{tikzpicture}
		\begin{scope}
		\draw (0,0) node[above=-1.55pt] {$\cdots$};
		\draw (0.4,0) node[above=-1.55pt] {$\down$};
		\draw (0.8, 0) node[above=-1.55pt] {$\up$};
		\draw (1.2,0) node[above=-1.55pt] {$\down$};
		\draw (1.6,0) node[above=-1.55pt] {$\cdots$};
  	\draw (0.8,0.3) arc (180:0:0.2);
		\draw (0.4,0.3) -- (0.4,0.6);
		\draw (0.4,0) arc (180:360:0.2);
		\draw (1.2,0) -- (1.2,-0.3);
		\end{scope}
		
				\begin{scope}[yshift=-1cm]
		\draw (0,0) node[above=-1.55pt] {$\cdots$};
		\draw (0.4,0) node[above=-1.55pt] {$\down$};
		\draw (0.8, 0) node[above=-1.55pt] {$\up$};
		\draw (1.2,0) node[above=-1.55pt] {$\down$};
		\draw (1.6,0) node[above=-1.55pt] {$\cdots$};
		\draw (0.4,0.3) arc (180:0:0.2);
		\draw (1.2,0.3) -- (1.2,0.6);
		\draw (0.8,0) arc (180:360:0.2);
		\draw (0.4,0) -- (0.4,-0.3);
		\end{scope}
		
		\draw[->] (2,-0.4)--(2.3,-0.4);
		
		\begin{scope}[yshift=-0.55cm, xshift=2.9cm]
		\draw (0,0) node[above=-1.55pt] {$\cdots$};
		\draw (0.4,0) node[above=-1.55pt] {$\down$};
		\draw (0.8, 0) node[above=-1.55pt] {$\up$};
		\draw (1.2,0) node[above=-1.55pt] {$\down$};
		\draw (1.6,0) node[above=-1.55pt] {$\cdots$};
		\draw (0.8,0.3) arc (180:0:0.2);
		\draw (0.4,0.3) -- (0.4,0.6);
		\draw (0.8,0) arc (180:360:0.2);
		\draw (0.4,0) -- (0.4,-0.3);
		\end{scope}
				\end{tikzpicture}
					 &$(j-1)$\\ \hline

		$(1)$ &$(0)$&
		\vspace{0.4cm}
		\begin{tikzpicture}
		\begin{scope}
		\draw (0,0) node[above=-1.55pt] {$\up$};
		\draw (0.4,0) node[above=-1.55pt] {$\down$};
		\draw (0.8, 0) node[above=-1.55pt] {$\cdots$};
		\draw (0,0) -- (0,-0.3);
		\draw (0.4,0) -- (0.4,-0.3);
		\draw (0,0.3) arc (180:0:0.2);
		\end{scope}
		
				\begin{scope}[yshift=-1cm]
		\draw (0,0) node[above=-1.55pt] {$\up$};
		\draw (0.4,0) node[above=-1.55pt] {$\down$};
		\draw (0.8, 0) node[above=-1.55pt] {$\cdots$};
		\draw (0,0) arc(180:360:0.2);
		\draw (0,0.3) -- (0,0.6);
		\draw (0.4,0.3) -- (0.4,0.6);
				\end{scope}
		
		\draw[->] (2,-0.4)--(2.3,-0.4);
		
		\begin{scope}[yshift=-0.55cm, xshift=2.9cm]
		\draw (0,0) node[above=-1.55pt] {$\up$};
		\draw (0.4,0) node[above=-1.55pt] {$\down$};
		\draw (0.8, 0) node[above=-1.55pt] {$\cdots$};
		\draw (0,0) arc(180:360:0.2);
		\draw (0,0.3) arc (180:0:0.2);
				\end{scope}
				\end{tikzpicture}

		 &$(0)$\\ \hline
		\end{tabular}
\end{table}
\normalsize
From now on denote by $P(\la) \to P(\mu)$ the degree one morphism computed in Tables \ref{tab:lamulaupm} - \ref{tab:lamunuspecdnm} (if it is nonzero). Note that these morphisms are only unique up to scalar hence we have made a choice.
\subsection*{The quiver of $\End(P)$}\label{sec:quiver1}

As it might be helpful for later work, we want to summarise the results in the quiver of $\End(P)$. 

For an introduction into the topic of quivers see \cite{Ausl97}. In our situation it is enough to interpret an arrow in the quiver as a degree one morphism between the corresponding projectives. Composing two arrows corresponds to the composition of the maps in $\End(P)$. The relations belong to the relations in $\End(P)$. For ease of presentation we denote the starting point of the relation by a bullet. By \cite[Proposition 1.15]{Ausl97} the opposed quiver is the one belonging to the $\Ext$-algebra of the simples.

\begin{Theorem}
The algebra $\End(P)$ is given as the path algebra of the quiver
\begin{displaymath}
\xymatrix{(0)\ar@/^/[r] &(1) \ar@/^/[r] \ar@/^/[l] &(2) \ar@/^/[r] \ar@/^/[l] &\hspace{0.5cm}\cdots \hspace{0.5cm}\ar@/^/[r] \ar@/^/[l]&(N-1) \ar@/^/[r] \ar@/^/[l] &(N)  \ar@/^/[l]} \end{displaymath}
with relations
\begin{enumerate}
\item $ \xymatrix{\hfill \ar@/^/[r]  &\bullet  \ar@/^/[l] \ar @{} [r] |{=}&\bullet \ar@/^/[r]  &\hfill  \ar@/^/[l]}$
\item $\xymatrix{\bullet \ar@/^/[r] &\hfill  \ar@/^/[l] }=0$ for the map starting in $(0)$
\item $\xymatrix{\bullet \ar@/^/[r] &\hfill  \ar@/^/[r] & \hfill }=0$
\item $\xymatrix{\hfill &\hfill  \ar@/^/[l] & \ar@/^/[l] \bullet }=0$ 
\end{enumerate}
\end{Theorem}

\subsection*{Terms occurring in the resolutions}
Now we compute the combinatorial Kazhdan-Luztig polynomials and therefore determine the terms of the resolution.
We only have to look for $\mu \geq \la$. For $\mu=(s)$ and $\la=(j)$ we obtain:
$$\begin{tikzpicture}[scale=1]
		\draw (-0.6,0) node[above=-1.55pt] {$\mu$};
		\draw (0,0) node[above=-1.55pt] {$\cdots$};
		\draw (0.4,0) node[above=-1.55pt] {$\down$};
		\draw (0.8, 0) node[above=-1.55pt] {$\up$};
		\draw (1.2,0) node[above=-1.55pt] {$\cdots$};
		\draw (0.8,0.3) arc (0:180:0.2);
		\begin{scope}[yshift=-0.4cm]
		\draw (-0.6,0) node[above=-1.55pt] {$\ell_i$};
		\draw (0.4,0) node[above=-1.55pt] {$0$};
		\end{scope}
		\begin{scope}[yshift=0.15cm]
		\scriptsize	
		\draw (0.6,0) node[above=-1.55pt] {$0$};
		\end{scope}	
		\begin{scope}[yshift=-0.8 cm]
		\draw (-0.6,0) node[above=-1.55pt] {$\la$};
		\draw (0,0) node[above=-1.55pt] {$\cdots$};
		\draw (0.4,0) node[above=-1.55pt] {$\down$};
		\draw (0.8,0) node[above=-1.55pt] {$\cdots$};
		\draw (1.2, 0) node[above=-1.55pt] {$\up$};
		\draw (1.8,0) node[above=-1.55pt] {$\cdots$};
		\end{scope}
		\end{tikzpicture}$$
		and therefore we have $p_{\la,\mu}=q^{j-s}$.
		
		Next we compute the terms occurring on the $i$th position of the resolution of $M(\la)$ with $\la=(j)$ fixed. By Theorem \ref{sumocc} we have
		$$P_i=\bigoplus_{\mu \in \La_N^1} p_{\la,\mu}^{(i)} P(\mu)\langle i \rangle.$$
		  By the above computation these are the projectives $P(s)$ with $j-s=i$. Therefore, the only term occurring is $P(j-i)$.

\subsection*{The differentials in the projective resolution}
For the computation of the differentials of the resolutions in $K_N^{1}-\mod$ we first have to remember those in $K_{N-1}^0-\mod$. Since this category only has one simple module $L=V=P$ which is a cell module and projective at once, the resolution of $V=P$ is $0 \to P \to 0$.

\begin{Theorem}\label{Th:diffproj1}
The chain complex 
$$0 \to P(0) \to P(1) \to \cdots \to P(n)\to 0$$ is a projective resolution of $M(n)$ in $K_N^1-\mod$.
Using the construction from Section \ref{sec:linprojres} one can choose the sign of the differential $\d_{n-k}(n):P(k) \to P(k+1)$ to be $(-1)^{n+k+1}$.
\end{Theorem}
\begin{proof}
In the previous section we  have already seen that we get a resolution like the one in the theorem and since there is only one term in each component, we can choose any sign. Now we just recall the construction to point out, that it is the same projective resolution as it is defined in Section \ref{sec:linprojres}.
Using the construction from Section \ref{sec:linprojres} and the fact that $M(0)$ is projective, inductively we have to take the cone of
$$
\begin{CD}
&0
&\rightarrow &P(0)&\rightarrow\cdots\rightarrow&P(n-1)&\rightarrow&0\\
&&&&&@VVf_0V\\
&&&&0 \rightarrow&G^{t_{n-1}}_{\La\Ga}P&\rightarrow& 0
\end{CD}.
$$
Because $f_0$ is the lift of a non-nullhomotopic morphism, it is itself nonzero.\\
Using equation \eqref{prback} we know that $G^{t_{n-1}}_{\La\Ga}P=P(n)$. Therefore, we get the stated result.
\end{proof}

\subsection*{Dimensions in the $\Ext$-algebra}\label{sec:dimexp1}
Using the formulas \eqref{ca:numcas1}-\eqref{ca:numcas4} from Theorem \ref{Th:dim}, we compute the dimensions of the $\Ext(M(\la),M(\mu))$-spaces for $\la \leq \mu$ in our situation. This is not necessary (later it is not hard to see that we have written down all possible elements) but in bigger situations it is a very helpful tool which we illustrate in this example.
\begin{align}
&&&E^k((j),(j))&=&\begin{cases} 1 &k=0\\ 0 &\text{else}\end{cases}\\
&&&E^k((j),(j-1))&\stackrel{(\ref{ca:numcas3})}{=}&E^{k-1}((j-1),(j-1))\notag\\
&&&&&+E^{k}((j-1),(j-1))\notag\\
&&&&=&\begin{cases} 1 &k \in \{0,1\}\\0 &\text{else}\end{cases}\\
&l<j-1:&& E^k((j),(l))&\stackrel{(\ref{ca:numcas2})}{=}&E^{k-1}((j-1),(l))\notag\\
&&&&\stackrel{Ind}{=}&\begin{cases} 1 &k \in \{j-l-1,j-l\}\\ 0 &\text{else}\end{cases}\end{align}
\begin{Lemma}
The algebra is $(N+1)^2$-dimensional.
\end{Lemma}
\begin{proof}
By the above computations we have $N+1$ idempotents and for each $i$ with $1 \leq i \leq N$ we have
$$\dim \bigoplus_{j < i} Ext^*(M(i),M(j))=2i.$$
Therefore we compute
\begin{align*}
\dim \bigoplus_{i, j} Ext^*(M(i),M(j))\\
&=N+1+\sum_{i=1}^{N}{2i}\\
&=N+1+N(N+1)\\
&=(N+1)^2. \qedhere
\end{align*}
\end{proof}
\subsection*{The elements}
In $Ext^1((j),(j))$ we have the idempotent $Id_{(j)}$.

We have to determine a nonzero element in $\Ext^1(P(j),P(j-1))$. 
\begin{Prop}
There is an element $\Id^{(j)}_{(j-1)} \in \Ext^1(P(j),P(j-1))$ with
\begin{align}
\Id^{(j)}_{(j-1)}: P_\bullet(j) &\to P_\bullet(j-1)[1]\langle 1\rangle\\
P(s) &\to P(s)
\end{align}
\end{Prop}
\begin{proof}
Viewing $\Id^{(j)}_{(j-1)}$ as a map to the complex $P_\bullet(j-1)[1{]_{\hom}}\langle 1\rangle$ with the $[ \quad {]_{\hom}}$-shift defined in Notation \ref{Not:sign}, we have to check that it is an anticommutative chain map, since the degree of the element is one.
We are in the following situation:
\begin{displaymath}
\xymatrix{ 0 \ar[r] &P(0) \ar[r]^{(-1)^{j+1}} \ar[d] &P(1)\ar[d] \ar[r]^{(-1)^{j}} & \cdots \ar[r]^{1} & P(j-2) \ar[d] \ar[r]^{-1} &P(j-1) \ar[d] \ar[r]^{1}&P(j)\ar[d] \ar[r] &0\\
 0 \ar[r] &P(0) \ar[r]^{(-1)^{j}}  &P(1) \ar[r]^{(-1)^{j+1}} & \cdots \ar[r]^{-1} & P(j-2)  \ar[r]^{1} &P(j-1) \ar[r]&0\ar[r]&0}
 \end{displaymath}
Therefore the map is an anticommutative chain map, as it has to be.

The map is not nullhomotopic since otherwise there would exist a map 
$H \in \hom^0(P_\bullet(j),P_\bullet(j-1)\langle 1 \rangle)$. This contradicts Corollary \ref{Cor:dif} since 
$0 \nleq 0-1$.
\end{proof}
\begin{Prop}
There is an element $F^{(j)}_{(j-1)} \in \Ext^0(P(j),P(j-1))$ with
\begin{align}
F^{(j)}_{(j-1)}: P_\bullet(j) &\to P_\bullet(j-1)\langle -1\rangle\\
P(s) &\to P(s-1)
\end{align} 
\end{Prop}
\begin{proof}
Writing down the map, we obtain:
\begin{displaymath}
\xymatrix{ P(0) \ar[r]^{(-1)^{j+1}} \ar[d] &P(1)\ar[d] \ar[r]^{(-1)^{j}} & \cdots \ar[r]^{1} & P(j-2) \ar[d] \ar[r]^{-1} &P(j-1) \ar[d] \ar[r]^{1}&P(j)\ar[d] \ar[r] &0\\
 0\ar[r] &P(0) \ar[r]^{(-1)^{j}}  &P(1) \ar[r]^{(-1)^{j+1}} & \cdots \ar[r]^{-1} & P(j-2)  \ar[r]^{1} &P(j-1) \ar[r]&0}
 \end{displaymath}
All squares in the diagram commute by the relations of $\End(P)$ (cf. Section \ref{sec:quiver1}).
Therefore, it becomes a commutative chain map, as it has to be, since it is of degree $0$.

The map is not nullhomotopic since otherwise there must exist a map $H:P_\bullet(j) \to P_\bullet(j-1)[-1]\langle 1 \rangle$. 
This map cannot exist by Lemma \ref{Le:mapprojres} as $j \nleq (j-1)+1^2-1$.
\end{proof}
The other elements are generated by these two types:
\begin{Prop}\label{Prop:IdF}
For $j \geq l$ we have maps
\begin{align}
\Id^{(j)}_{(l)} &\in \Ext^{j-l}(P(j),P(l)) \notag\\
\Id^{(j)}_{(l)}: P_\bullet(j) &\to P_\bullet(l)[j-l]\langle j-l\rangle\\
P(s) &\to P(s) \ \forall s \leq l
\end{align}
and for $j>l$ maps
\begin{align}
F^{(j)}_{(l)} &\in \Ext^{j-l-1}(P(j),P(l)) \notag\\
F^{(j)}_{(l)}: P_\bullet(j) &\to P_\bullet(l)[j-l-1]\langle j-l-2\rangle\\
P(s) &\to P(s-1) \ \forall s \leq l+1
\end{align}
\end{Prop}
\begin{proof}
By the definition of the maps, we see that
$$\Id^{(j)}_{(l)}=\Id^{(j)}_{(j-1)}  \cdots  \Id^{(l+2)}_{(l+1)} \cdot \Id^{(l+1)}_{(l)}$$
and
$$F^{(j)}_{(l)}=\Id^{(j)}_{(j-1)} \cdots \Id^{(l+2)}_{(l+1)} \cdot F^{(l+1)}_{(l)},$$
so they are elements in the correct dimensions of the $\Ext$-algebra.

We only have to show that they are unequal to zero:

For $\Id^{(j)}_{(l)}$ 
a homotopy would be a map $H \in \hom^{j-l-1}(P_\bullet(j),P_\bullet(l)\langle j-l \rangle)$. This cannot exist by Corollary \ref{Cor:dif} since $0 \nleq (j-l-1)-(j-l)$.
For $F^{(j)}_{(l)}$ a homotopy would be a map $H \in \hom^{j-l-2}(P_\bullet(j),P_\bullet(l)\langle j-l-2 \rangle)$
which cannot exist by Lemma \ref{Le:mapprojres} since $j \nleq l+1^2+(j-l-2)$.
\end{proof}
Now we have found all elements in $\Ext(\bigoplus M(\la),\bigoplus M(\la))$ which can be checked via the dimension formulas.

\subsection*{Multiplications}
The following theorem describes the algebra structure completely:
\begin{Theorem}\label{Th:multi1}
We have the following multiplication rules:
\begin{enumerate}
\item $\Id^{(j)}_{(l)} \cdot \Id^{(l)}_{(m)}= \Id^{(j)}_{(m)}$
\item $\Id^{(j)}_{(l)} \cdot F^{(l)}_{(m)}= F^{(j)}_{(m)}$
\item $F^{(j)}_{(l)} \cdot \Id^{(l)}_{(m)}= F^{(j)}_{(m)}$
\item $F^{(j)}_{(l)} \cdot F^{(l)}_{(m)}= 0$
\end{enumerate}
\end{Theorem}
\begin{proof}
The first three equations are trivial by the definition and the proof of Proposition \ref{Prop:IdF}. As there is no element of degree $2$ by the dimension formulas, the last map must be zero.

For later computations notice that it is already zero as a map in $\hom(P_\bullet,P_\bullet)$ since the map is $P(s) \to P(s-1) \to P(s-2)$ which is zero by the computations above.
\end{proof}
\begin{Remark}\label{Re:nonhom}
Note that all equalities in Theorem \ref{Th:multi1} are already equalities in $\hom(P_\bullet,P_\bullet)$ and not just up to homotopy.
\end{Remark}
\begin{Theorem}\label{mainalg1}
The algebra $\Ext(\bigoplus M(\la),\bigoplus M(\la))$ for $\la \in K_N^1$ is given by the path algebra of the quiver
\begin{displaymath}\xymatrix{
\cdots  \ar@[cyan]@/^/[r] \ar@[black]@/_/[r]& P_\bullet(n+1) \ar@[cyan]@/^/[r] \ar@[black]@/_/[r]& P_\bullet(n)   
 \ar@[cyan]@/^/[r] \ar@[black]@/_/[r]&P_\bullet(n-1)     \ar@/_/[r]  \ar@[cyan]@/^/[r]& \cdots  }
\end{displaymath}
\normalsize
\color{black}
with relations
\begin{enumerate}
\item \begin{displaymath} \xymatrix{ \bullet \ar@[cyan]@/^/[r] &\color{black}\bullet \ar@[cyan]@/^/[r] &\color{black} \bullet}\color{black} =0 \end{displaymath}
\item \begin{displaymath} \xymatrix{ \bullet \ar@[cyan]@/^/[r] &\color{black}\bullet\ar@[black] @{} [d] |{\color{black}=} \ar@[black]@/_/[r] &\color{black} \bullet \color{black}\\
 \bullet \ar@[black]@/_/[r] &\color{black}\bullet\ar@[cyan]@/^/[r] &\color{black} \bullet}\color{black}\end{displaymath}
\end{enumerate}
\color{black}
\end{Theorem}
\color{black}
\newpage
\section{Modules in $K_{N-1}^2-\mod$}\label{sec:Mod2}
In the second case we take $n=2$ and $m=N-1$. Again, we index the weights by the permutation used to obtain them from the maximal weight in the block. A shortest representation is of the form $s_2 \cdot \dots s_k \cdot s_1 \cdot \dots \cdot s_l$ with $0\leq l < k \leq N$. So the weight we get has a $\up$ at the $l$th position and another at the $k$th position (starting to count with position zero). We denote the weight $\lambda=\la_{0}. s_2 \cdot \ldots \cdot s_k \cdot s_1 \cdot \ldots \cdot s_l $ by $(k|l)$. 

\subsection{q-decomposition numbers}
Similar to the first case, we compute the q-decomposition numbers and write down projectives in terms of Vermas (Table \ref{proverm}) and Vermas in terms of simples (Table \ref{vermsimp}). Since there are at most two cups, the decomposition numbers are either $0, 1, q$ or $q^2$. Combining the two tables, we write down projectives in terms of simples (Table \ref{prosimp}).

\begin{table}[ht]

\caption{\label{proverm}$\mu$ such that $d_{\la,\mu}=\_$}
\footnotesize
			\begin{tabular}{|lc|l|l|l|}
		\toprule
		$\la=(m|j)$ &  &$1$&$q$ &$q^2$\\ \hline
		$\begin{array}{l} j\neq m-1\\j\neq 0\end{array}$ &
		\begin{tikzpicture}[scale=1]
		\draw (0,0) node[above=-1.55pt] {$\cdots$};
		\draw (0.4,0) node[above=-1.55pt] {$\down$};
		\draw (0.8, 0) node[above=-1.55pt] {$\up$};
		\draw (1.2,0) node[above=-1.55pt] {$\cdots$};
		\draw (1.6,0) node[above=-1.55pt] {$\down$};
		\draw (2,0) node[above=-1.55pt] {$\up$};
		\draw (2.4,0) node[above=-1.55pt] {$\cdots$};
		\draw (0.4,0) arc (180:360:0.2);
		\draw (1.6,0) arc (180:360:0.2);
		\end{tikzpicture}
														
		 &$(m|j)$&$\begin{array}{l} (m-1|j)\\(m|j-1)\end{array}$&$(m-1|j-1)$\\ \hline
		 
		 $\begin{array}{l} j=m-1\\ j\neq 0\\ j\neq 1 \end{array}$ &
		\begin{tikzpicture}[scale=1]
		\draw (0,0) node[above=-1.55pt] {$\cdots$};
		\draw (0.4,0) node[above=-1.55pt] {$\down$};
		\draw (0.8, 0) node[above=-1.55pt] {$\down$};
		\draw (1.2,0) node[above=-1.55pt] {$\up$};
		\draw (1.6,0) node[above=-1.55pt] {$\up$};
		\draw (2,0) node[above=-1.55pt] {$\cdots$};
		\draw (0.4,0) arc (180:360:0.6);
		\draw (0.8,0) arc (180:360:0.2);
		\end{tikzpicture}
														
		 & $(m|m-1)$ &$\begin{array}{l}(m-1|m-3)\\(m|m-2)\end{array}$	&$(m-2|m-3)$\\ \hline
		 
		 $\begin{array}{l} j=1\\ m=2 \end{array}$ &
		\begin{tikzpicture}[scale=1]
		\draw (0,0) node[above=-1.55pt] {$\down$};
		\draw (0.4,0) node[above=-1.55pt] {$\up$};
		\draw (0.8, 0) node[above=-1.55pt] {$\up$};
		\draw (1.2,0) node[above=-1.55pt] {$\cdots$};
		\draw (0,0) arc (180:360:0.2);
		\draw (0.8, -0.5)--(0.8,0);
		\end{tikzpicture}
														
		 & $(2|1)$ &$(2|0)$ 	&\\ \hline
		 
		$\begin{array}{l} j=0\\ m=1 \end{array}$ &
		\begin{tikzpicture}[scale=1]
		\draw (0,0) node[above=-1.55pt] {$\up$};
		\draw (0.4,0) node[above=-1.55pt] {$\up$};
		\draw (0.8, 0) node[above=-1.55pt] {$\down$};
		\draw (1.2,0) node[above=-1.55pt] {$\cdots$};
		\draw (0, -0.5)--(0,0);
		\draw (0.4, -0.5)--(0.4,0);
		\draw (0.8, -0.5)--(0.8,0);
		\end{tikzpicture}
														
		 & $(1|0)$ &  	&\\ \hline
		 
		$\begin{array}{l} j=0\\ m\neq1 \end{array}$ &
		\begin{tikzpicture}[scale=1]
		\draw (0,0) node[above=-1.55pt] {$\up$};
		\draw (0.4,0) node[above=-1.55pt] {$\cdots$};
		\draw (0.8, 0) node[above=-1.55pt] {$\down$};
		\draw (1.2,0) node[above=-1.55pt] {$\up$};
		\draw (1.6,0) node[above=-1.55pt] {$\cdots$};
		\draw (0.0, -0.5)--(0.0,0);
		\draw (0.8,0) arc (180:360:0.2);
		\end{tikzpicture}
														
		 & $(m|0)$ &$(m-1|0)$  	&\\ \hline
		 
		\end{tabular}
\end{table}

\begin{table}
\caption{\label{vermsimp}$\la$ such that $d_{\la,\mu}=\_$}
\footnotesize
	\begin{tabular}{|m{1.9cm}m{3.9cm}|m{2cm}|m{2.2cm}| m{1.8cm}|}
		\toprule
		$\mu=(m|j)$ &  &$1$&$q$& $q^2$\\ \hline
		$\begin{array}{l} j\neq m-1\\j\neq m-2 \\m \neq N\end{array}$ &
		\begin{tikzpicture}
		\draw (0,0) node[above=-1.55pt] {$\cdots$};
		\draw (0.4,0) node[above=-1.55pt] {$\down$};
		\draw (0.8, 0) node[above=-1.55pt] {$\up$};
		\draw (1.2,0) node[above=-1.55pt] {$\down$};
		\draw (1.6,0) node[above=-1.55pt] {$\cdots$};
		\draw (2,0) node[above=-1.55pt] {$\down$};
		\draw (2.4, 0) node[above=-1.55pt] {$\up$};
		\draw (2.8,0) node[above=-1.55pt] {$\down$};
		\draw (3.2,0) node[above=-1.55pt] {$\cdots$};
				\end{tikzpicture}
														
		 &$(m|j)$&$\begin{array}{l}(m+1|j)\\(m|j+1)\end{array}$ &$(m+1|j+1)$\\ \hline
		 
		 $\begin{array}{l} j=m-1\\ m<N-1 \end{array}$ &
		\begin{tikzpicture}
		\draw (0,0) node[above=-1.55pt] {$\cdots$};
		\draw (0.4,0) node[above=-1.55pt] {$\up$};
		\draw (0.8, 0) node[above=-1.55pt] {$\up$};
		\draw (1.2,0) node[above=-1.55pt] {$\down$};
		\draw (1.6,0) node[above=-1.55pt] {$\down$};
		\draw (2,0) node[above=-1.55pt] {$\cdots$};
		\end{tikzpicture}
														
		 & $(m|m-1)$ &$(m+1|m-1)$ 	&$(m+2|m+1)$\\ \hline
		 
		 $\begin{array}{l} j=m-2\\ m\neq N \end{array}$ &
		\begin{tikzpicture}
		\draw (0,0) node[above=-1.55pt] {$\cdots$};
		\draw (0.4,0) node[above=-1.55pt] {$\up$};
		\draw (0.8, 0) node[above=-1.55pt] {$\down$};
		\draw (1.2,0) node[above=-1.55pt] {$\up$};
		\draw (1.6,0) node[above=-1.55pt] {$\down$};
		\draw (2,0) node[above=-1.55pt] {$\cdots$};			
		\end{tikzpicture}
														
		 & $(m|m-2)$ &$\begin{array}{l}(m+1|m-2)\\(m|m-1)\\(m+1|m)\end{array}$&$(m+1|m-1)$\\ \hline
		 
		$\begin{array}{l} j=m-1\\ m=N-1 \end{array}$ &
		\begin{tikzpicture}
		\draw (0,0) node[above=-1.55pt] {$\cdots$};
		\draw (0.4,0) node[above=-1.55pt] {$\down$};
		\draw (0.8, 0) node[above=-1.55pt] {$\up$};
		\draw (1.2,0) node[above=-1.55pt] {$\up$};
		\draw (1.6,0) node[above=-1.55pt] {$\down$};
		\end{tikzpicture}
														
		 & $(N-1|N-2)$ &$(N|N-2)$ 	&\\ \hline
		 
		 $\begin{array}{l} j=m-1\\ m=N \end{array}$ &
		\begin{tikzpicture}
		\draw (0,0) node[above=-1.55pt] {$\cdots$};
		\draw (0.4,0) node[above=-1.55pt] {$\down$};
		\draw (0.8, 0) node[above=-1.55pt] {$\up$};
		\draw (1.2,0) node[above=-1.55pt] {$\up$};
		\end{tikzpicture}
														
		 & $(N|N-1)$ &  	&\\ \hline
		 
		$\begin{array}{l} j\neq m-1\\j\neq m-2 \\m = N\end{array}$ &
		\begin{tikzpicture}		\draw (0,0) node[above=-1.55pt] {$\cdots$};
		\draw (0.4,0) node[above=-1.55pt] {$\down$};
		\draw (0.8, 0) node[above=-1.55pt] {$\up$};
		\draw (1.2,0) node[above=-1.55pt] {$\down$};
		\draw (1.6,0) node[above=-1.55pt] {$\cdots$};
		\draw (2,0) node[above=-1.55pt] {$\down$};
		\draw (2.4, 0) node[above=-1.55pt] {$\up$};
		\end{tikzpicture}
														
		 &$(N|j)$&$(N|j+1)$ &\\ \hline

		 $\begin{array}{l} j=m-2\\ m= N \end{array}$ &
		\begin{tikzpicture}
		\draw (0,0) node[above=-1.55pt] {$\cdots$};
		\draw (0.4,0) node[above=-1.55pt] {$\up$};
		\draw (0.8, 0) node[above=-1.55pt] {$\down$};
		\draw (1.2,0) node[above=-1.55pt] {$\up$};
		\end{tikzpicture}
														
		 & $(N|N-2)$ &$(N|N-1)$ &\\ \hline
		 
		\end{tabular}
\end{table}
\begin{landscape}
\scriptsize
\begin{longtable}{|l|l|}
\caption{{\captionlabelfont{\label{prosimp}}}\normalsize Filtration of projective module $P(\la)$ by simple modules, same colour belonging to the same Verma module}

\\

\toprule
$\la=(m|j)$ & $P(m|j)$\\
\hline
$\begin{array}{l}j < m-3\\j\neq 0\\m\neq N\end{array}$&
$\begin{array}{l}
\hspace*{1cm}\blu L(m|j)\\
\blu L(m+1|j) L(m|j+1) \hspace{0.2cm} \re L(m-1 |j) \hspace{3.5cm} \ora L(m|j-1)\\
\blu L(m+1|j+1)  \re L(m|j) L(m-1|j+1)  \gre L(m-1|j-1) \ora L(m+1|j-1) L(m|j)\\
\hspace*{3cm} \re L(m|j+1) \hspace{0.2cm}\gre L(m|j-1) L(m-1|j) \hspace{0.3cm}\ora L(m+1|j)\\
\hspace*{6cm}\gre L(m|j)
\end{array}$\\ \hline

$\begin{array}{l}j=m-3\\j\neq 0\\m\neq N\end{array}$&

$\begin{array}{l}
\hspace*{1cm}\blu L(m|m-3)\\
\blu L(m+1|m-3) L(m|m-2) \hspace{0.1cm} \re L(m-1 |m-3) \hspace{4cm} \ora L(m|m-4)\\
\blu L(m+1|m-2)  \re L(m|m-3) L(m-1|m-2) L(m|m-1) \gre L(m-1|m-4) \ora L(m+1|m-4) L(m|m-3)\\
\hspace*{4cm} \re L(m|m-2) \hspace{0.5cm}\gre L(m|m-4) L(m-1|m-3) \hspace{0.5cm}\ora L(m+1|m-3)\\
\hspace*{7cm}\gre L(m|m-3)
\end{array}$\\ \hline

$\begin{array}{l}j=m-2\\j\neq 0\\m\neq N\end{array}$&

$\begin{array}{l}
\hspace*{1cm}\blu L(m|m-2)\\
\blu L(m+1|m-2) L(m|m-1) \hspace{0.1cm} \re L(m-1 |m-2) \hspace{2cm} \ora L(m|m-3)\\
\blu L(m+1|m-1)\hspace{0.5cm}  \re L(m|m-2) \hspace{1cm} \gre L(m-1|m-3) \ora L(m+1|m-3) L(m|m-2)\\
\hspace*{2cm} \re L(m+1|m) \gre L(m|m-3) L(m-1|m-2) L(m|m-1)\ora L(m+1|m-2)\\
\hspace*{5.5cm}\gre L(m|m-2)
\end{array}$\\ \hline

$\begin{array}{l}j=m-1\\j\neq 0\\m< N-1\end{array}$&

$\begin{array}{l}
\blu L(m|m-1)\\
\blu L(m+1|m-1) \hspace{3cm} \re L(m |m-2) \hspace{4cm} \ora L(m-1|m-3)\\
\blu L(m+2|m+1)  \re L(m+1|m-2) L(m|m-1) L(m+1|m) \gre L(m-2|m-3) \ora L(m|m-3) L(m-1|m-2) L(m|m-1)\\
\hspace*{4cm} \re L(m+1|m-1) \hspace{1.1cm}\gre L(m-1|m-3) \ora L(m|m-2)\\
\hspace*{7.4cm}\gre L(m|m-1)
\end{array}$\\ \hline

$\begin{array}{l}j < m-3\\j\neq 0\\m= N\end{array}$&
$\begin{array}{l}
\blu L(N|j)\\
\blu L(N|j+1) \hspace{0.2cm} \re L(N-1 |j) \hspace{2cm} \ora L(N|j-1)\\
\hspace*{0.2cm}  \re L(N|j) L(N-1|j+1)  \gre L(N-1|j-1) \ora L(N|j)\\
\hspace*{1cm} \re L(N|j+1) \hspace{0.5cm}\gre L(N|j-1) L(N-1|j) \\
\hspace*{3cm}\gre L(N|j)
\end{array}$\\ \hline

$\begin{array}{l}j < m-2\\j=0\\m= N\end{array}$&
$\begin{array}{l}
\blu L(N|0)\\
\blu L(N|1) \hspace{0.2cm} \re L(N-1 |0) \\
\hspace*{1cm}  \re L(N|0) L(N-1|1)\\
\hspace*{1.5cm} \re L(N|1)
\end{array}$\\ \hline

$\begin{array}{l}j=m-3\\j\neq 0\\m= N\end{array}$&

$\begin{array}{l}
\blu L(N|N-3)\\
\blu L(N|N-2) \hspace{0.1cm} \re L(N-1 |N-3) \hspace{3.5cm} \ora L(N|N-4)\\
\re L(N|N-3) L(N-1|N-2) L(N|N-1) \gre L(N-1|N-4) \ora L(N|N-3)\\
\hspace*{2cm} \re L(N|N-2) \hspace{0.5cm}\gre L(N|N-4) L(N-1|N-3) \\
\hspace*{5cm}\gre L(N|N-3)
\end{array}$\\ \hline

$\begin{array}{l}j=m-2\\j\neq 0\\m= N\end{array}$&

$\begin{array}{l}
\blu L(N|N-2)\\
\blu L(N|N-1) \re L(N-1 |N-2) \hspace{2cm} \ora L(N|N-3)\\
\hspace*{2cm}  \re L(N|N-2) \gre L(N-1|N-3) \ora L(N|N-2)\\
\hspace*{2cm} \gre L(N|N-3) L(N-1|N-2) L(N|N-1)\\
\hspace*{3.5cm}\gre L(N|N-2)
\end{array}$\\ \hline

$\begin{array}{l}j=m-1\\j\neq 0\\m=N\end{array}$&

$\begin{array}{l}
\blu L(N|N-1)\\
\re L(N |N-2) \hspace{4cm} \ora L(N-1|N-3)\\
\re L(N|N-1) \gre L(N-2|N-3) \ora L(N|N-3) L(N-1|N-2) L(N|N-1)\\
\hspace*{1.5cm} \gre L(N-1|N-3)\hspace{1.5cm} \ora L(N|N-2)\\
\hspace*{2cm}\gre L(N|N-1)
\end{array}$\\ \hline

$\begin{array}{l}j=m-1\\j\neq 0\\m= N-1\end{array}$&

$\begin{array}{l}
\blu L(N-1|N-2)\\
\blu L(N|N-2) \hspace{1cm} \re L(N-1 |N-3) \hspace{5cm} \ora L(N-2|N-4)\\
\re L(N|N-3) L(N-1|N-2) L(N|N-1) \gre L(N-3|N-4) \ora L(N-1|N-4) L(N-2|N-3) L(N-1|N-2)\\
\hspace*{1.5cm} \re L(N|N-2) \hspace{2cm}\gre L(N-2|N-4) \hspace{2cm}\ora L(N-1|N-3)\\
\hspace*{5cm}\gre L(N-1|N-2)
\end{array}$\\ \hline

$\begin{array}{l}j < m-3\\j= 0\\m\neq N\end{array}$&
$\begin{array}{l}
\hspace*{1cm}\blu L(m|0)\\
\blu L(m+1|0) L(m|1) \hspace{0.2cm} \re L(m-1 |0) \\
\hspace*{1cm}\blu L(m+1|1)  \re L(m|0) L(m-1|1) \\
\hspace*{3cm} \re L(m|1) \\
\end{array}$\\ \hline

$\begin{array}{l}j=0\\m=3\\m\neq N\end{array}$&
$\begin{array}{l}
\hspace*{0.5cm}\blu L(3|0)\\
\blu L(4|0) L(3|1) \hspace{0.5cm} \re L(2 |0) \\
\hspace*{0.5cm}\blu L(4|1)  \re L(3|0) L(2|1) L(3|2) \\
\hspace*{2cm} \re L(3|1) \\
\end{array}$\\ \hline

$\begin{array}{l}j=0\\m=2\\m\neq N\end{array}$&
$\begin{array}{l}
\hspace*{1.5cm}\blu L(2|0)\\
\blu L(3|0) L(2|1) L(3|2) \hspace{0.5cm} \re L(1 |0) \\
\hspace*{1cm}\blu L(3|1)  \hspace{1.3cm}\re L(2|0)  \\
\hspace*{3.1cm} \re L(3|2) \\
\end{array}$\\ \hline

$\begin{array}{l}j=0\\m=1\\m<N-1\end{array}$&
$\begin{array}{l}
\blu L(1|0)\\
\blu L(2|0) \\
\blu L(3|2)
\end{array}$\\ \hline
$\begin{array}{l}j=1\\m=2\\m\neq N-1\end{array}$&
$\begin{array}{l}
\blu L(2|1)\\
\blu L(3|1) \hspace{1cm}  \re L(2|0)  \\
\blu L(4|3)\re L(3 |0) L(2|1) L(3|2) \\
\hspace*{2cm} \re L(3|1) \\
\end{array}$\\ \hline
\end{longtable}
\end{landscape}

\normalsize

\subsection{Morphisms between projective modules}\label{sec:mapproj}
Using Table (\ref{prosimp}) we determine some general restrictions on morphisms between two projective modules. Note that this sharpens the conditions obtained in Lemma \ref{Le:condmaps}.
\begin{Lemma}\label{Le:possmaps}
The following conditions are necessary for the existence of degree zero morphisms $P(s|t)\langle i \rangle \to P(k|l)$:
\begin{enumerate}[label={$i=$\arabic*:}]
\setcounter{enumi}{-1}
\item $s+t \leq k+l+1$
\item $s+t \leq k+l+2$ or $(s|t)=(a+1|a)$ and $(k|l)=(a-1|a-2)$
\item $s+t \leq k+l+1$ or $(s|t)=(a+1|a)$ and $(k|l)=(a|a-2)$
\item $(s|t)=(k|l)$
\end{enumerate}
\end{Lemma}
\begin{proof}
Obvious by Table (\ref{prosimp}).
\end{proof}
Similarly to the above computations for $K_N^{1}-\mod$, we determine degree $1$ morphisms between projectives. If we can cut both diagrams simultaneously into two parts (i.e. without cutting a cup or a cap), we can multiply independently. Especially if one part of the diagram is an oriented circle diagram of degree zero and we multiply it with itself, it stays unchanged. 

In some cases one can use Table (\ref{prosimp}) to see that $L(\la)$ does not appear in the degree 2 part of $P(\nu)$, so each morphism $P(\la)\to P(\mu) \to P(\nu)$ must be zero.
As before we read diagrams from bottom to top.
\input{mapprojtables.tex} 
Again, we denote by $P(\la) \to P(\mu)$ the degree one morphism computed in Tables \ref{tab:lamulaupm} - \ref{tab:lamunuspecdnm} (if it is nonzero) and call it the \emph{standard degree one morphism}. Note that since the morphisms are only unique up to scalar, we have made a choice.
\FloatBarrier		 
\subsection{The quiver of $\End(P)$}

For ease of presentation we summarise the relations from the tables and write down the quiver of $\End(P)$.
\begin{Theorem}
The algebra $\End(P)$ is given as the path algebra of the quiver (middle part, for $n>m+3, m\geq 1$)

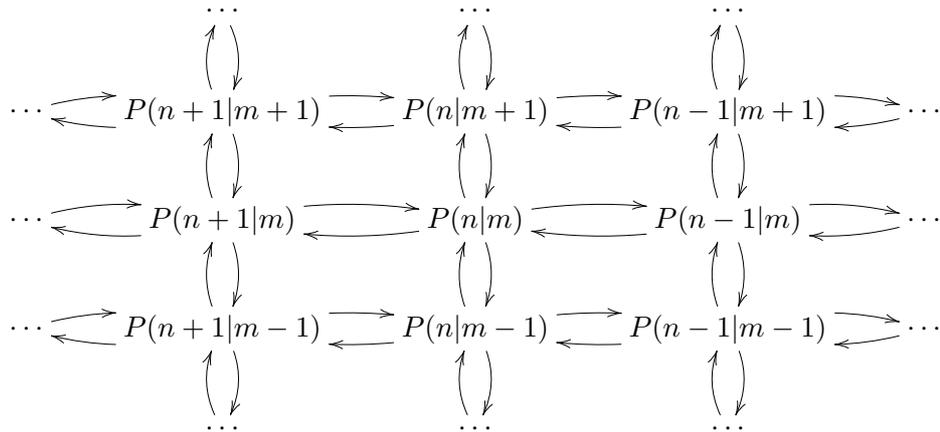
\begin{displaymath}\xymatrix{
& \cdots \ar@/^/[d] & \cdots \ar@/^/[d] & \cdots \ar@/^/[d] &
 \\
\cdots \ar@/^/[r] & P(n+1|m+1)\ar@/^/[d] \ar@/^/[l]  \ar@/^/[u] \ar@/^/[r] & P(n|m+1)\ar@/^/[d] \ar@/^/[l]  \ar@/^/[u] \ar@/^/[r] &P(n-1|m+1)  \ar@/^/[d] \ar@/^/[l]  \ar@/^/[u] \ar@/^/[r]& \cdots \ar@/^/[l] \\
\cdots \ar@/^/[r] & P(n+1|m) \ar@/^/[d] \ar@/^/[l]  \ar@/^/[u] \ar@/^/[r] & P(n|m) \ar@/^/[d] \ar@/^/[l]  \ar@/^/[u] \ar@/^/[r] &P(n-1|m) \ar@/^/[d] \ar@/^/[l]  \ar@/^/[u] \ar@/^/[r] &\cdots \ar@/^/[l] \\
\cdots \ar@/^/[r] &P(n+1|m-1) \ar@/^/[d] \ar@/^/[l]  \ar@/^/[u] \ar@/^/[r] & P(n|m-1) \ar@/^/[d] \ar@/^/[l]  \ar@/^/[u] \ar@/^/[r] &P(n-1|m-1) \ar@/^/[d] \ar@/^/[l]  \ar@/^/[u] \ar@/^/[r] & \cdots \ar@/^/[l]\\
& \cdots \ar@/^/[u]& \cdots \ar@/^/[u]& \cdots \ar@/^/[u] & }\end{displaymath}
At the corners the quiver is given by
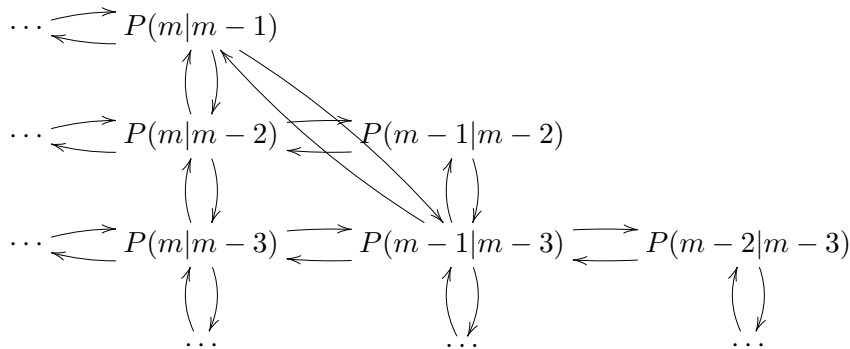
\begin{displaymath}
\xymatrix{
\dots \ar@/^/[r]& P(m|m-1)\ar@/^/[l] \ar@/^/[d] \ar@/^/ [rdd]& &\\
\cdots \ar@/^/[r] &P(m|m-2)\ar@/^/[r] \ar@/^/[l] \ar@/^/[d] \ar@/^/[u] & P(m-1|m-2) \ar@/^/[d] \ar@/^/[l] &\\
\cdots \ar@/^/[r] &P(m|m-3) \ar@/^/[r] \ar@/^/[l] \ar@/^/[d] \ar@/^/[u] & P(m-1|m-3) \ar@/^/[l] \ar@/^/[r] \ar@/^/[d] \ar@/^/[u] \ar@/^/[uul] & P(m-2|m-3)\ar@/^/[d] \ar@/^/[l]\\
& \cdots\ar@/^/[u] & \dots\ar@/^/[u] & \cdots\ar@/^/[u]}\end{displaymath}
with relations (in case that both sides of the relation exist)
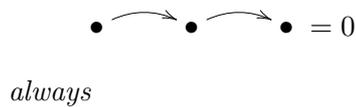
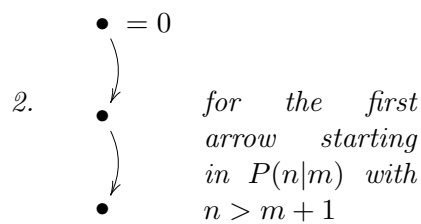
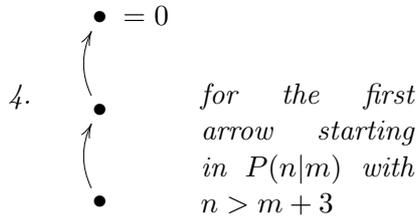
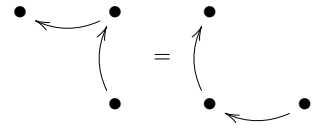
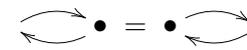
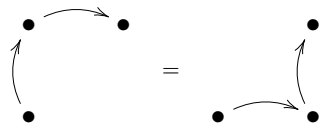
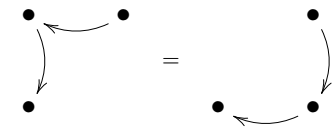
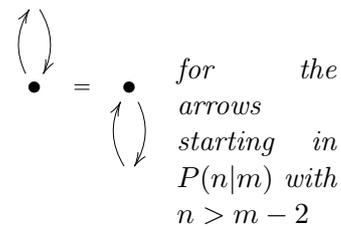
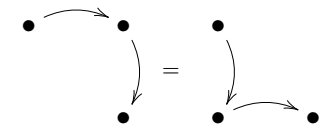
\begin{multicols}{2}
    \raggedcolumns 
\begin{enumerate}
\item \begin{displaymath}
\xymatrix{\bullet \ar@/^/[r]& \bullet \ar@/^/[r]& \bullet}=0\end{displaymath}  always
\item \begin{tabular*}{0.3\textwidth}[]{m{0.1\textwidth}p{0.2\textwidth}}
     \begin{displaymath}\xymatrix{\bullet \ar@/^/[d]\\ \bullet \ar@/^/[d]\\ \bullet}=0\end{displaymath} & for the first arrow starting in $P(n|m)$ with $n>m+1$
    \end{tabular*}
\item \begin{displaymath}\xymatrix{\bullet & \bullet \ar@/^/[l]& \bullet \ar@/^/[l]}=0\end{displaymath} always
\item  \begin{tabular*}{0.3\textwidth}[]{m{0.1\textwidth}p{0.2\textwidth}} \begin{displaymath}\xymatrix{\bullet \\ \bullet \ar@/^/[u]\\ \bullet \ar@/^/[u]}=0\end{displaymath} & for the first arrow starting in $P(n|m)$ with $n>m+3$\end{tabular*}
\item \begin{displaymath}\xymatrix{\ar @{} [drrr] |{=} \bullet \ar@/^/[r] & \bullet & & \bullet\\ \bullet \ar@/^/[u] &&\bullet \ar@/^/[r] &\bullet \ar@/^/[u]}\end{displaymath} always
\item \begin{displaymath}\xymatrix{\ar @{} [drrr] |{=} \bullet \ar@/^/[d] & \bullet \ar@/^/[l] && \bullet\ar@/^/[d]  \\ \bullet&& \bullet &\bullet \ar@/^/[l]}\end{displaymath} always
\item \begin{displaymath}\xymatrix{\ar @{} [drrr] |{=} \bullet \ar@/^/[r] & \bullet \ar@/^/[d]  & \bullet \ar@/^/[d] &  \\ &\bullet &\bullet \ar@/^/[r]& \bullet}\end{displaymath} always
\item \begin{displaymath}\xymatrix{\ar @{} [drrr] |{=} \bullet & \bullet \ar@/^/[l] & \bullet & \\ &\bullet \ar@/^/[u]  & \bullet \ar@/^/[u]  &\bullet \ar@/^/[l]}\end{displaymath} always
\item \begin{displaymath}\xymatrix{\hfill \ar@/^/[r] & \bullet \ar@/^/[l]}=\xymatrix{\bullet \ar@/^/[r] & \hfill \ar@/^/[l]}\end{displaymath} always
\item \begin{tabular*}{0.3\textwidth}[]{m{0.15\textwidth}p{0.15\textwidth}} \begin{displaymath}\xymatrix{\ar @{} [ddr] |{=} \hfill \ar@/^/[d] & \\ \bullet \ar@/^/[u] &\bullet \ar@/^/[d]\\ &\hfill \ar@/^/[u]}\end{displaymath} & for the arrows starting in $P(n|m)$ with $n >m-2$\end{tabular*}
\setcounter{enumglobal}{\value{enumi}}
\end{enumerate}
\end{multicols}
These are all cases occuring in the middle of the quiver, i.e. in the upper diagram. 
We also have to look for those at the corner part.
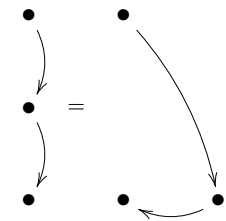
\begin{multicols}{2}
    \raggedcolumns 
\begin{enumerate}
\setcounter{enumi}{\value{enumglobal}}
\item \begin{displaymath}\xymatrix{\ar @{} [ddr] |{=} \bullet \ar@/^/[d]&\bullet \ar@/^/[rdd]&\\\bullet \ar@/^/[d] & &\\\bullet &\bullet &\bullet \ar@/^/[l]}\end{displaymath}  with the arrows starting in $P(m|m-1)$
\item  
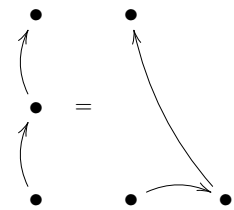
\begin{displaymath}\xymatrix{\ar @{} [ddr] |{=} \bullet &\bullet &\\ \bullet \ar@/^/[u] & &\\\bullet \ar@/^/[u] &\bullet \ar@/^/[r] &\bullet \ar@/^/[uul]}\end{displaymath} with the arrows starting in $P(m|m-3)$
\item 
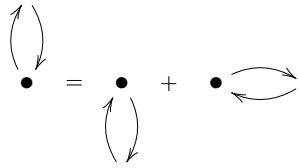
\begin{displaymath}\xymatrix{\ar @{} [ddr] |{=} \hfill \ar@/^/[d] & &&\\ \bullet \ar@/^/[u] &\bullet \ar @{} [r] |{+} \ar@/^/[d]&\bullet  \ar@/^/[r]&\hfill \ar@/^/[l]\\ &\hfill \ar@/^/[u]}\end{displaymath} starting in $P(m|m-2)$
\item 
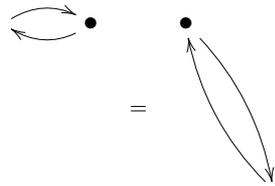
\begin{displaymath}\xymatrix{\hfill \ar@/^/[r]&\bullet \ar@/^/[l]&\bullet \ar@/^/[rdd]&\\&  \hfill \ar @{} [r] |{=} &&\\& & &\hfill \ar@/^/[uul]}\end{displaymath} with the arrows starting in $P(m|m-1)$
\item  
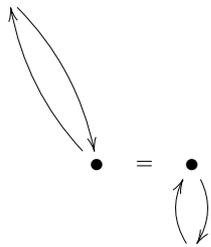
\begin{displaymath}\xymatrix{\hfill \ar@/^/[rdd]&&\\&   &\\&\bullet \ar@/^/[uul]  \ar @{} [r] |{=} &\bullet \ar@/^/[d]\\&&\hfill \ar@/^/[u]}\end{displaymath} with the arrows starting in $P(m|m-2)$
\item \begin{displaymath}\xymatrix{\ar @{} [drrr] |{=} \bullet &  & \bullet & \\ &\bullet \ar@/^/[d]  & \bullet \ar@/^/[u]  &\bullet \ar@/^/[l]\\ &\bullet \ar@/^/[uul] }\end{displaymath} starting in $P(m|m-1)$
\item  \begin{displaymath}\xymatrix{\ar @{} [drrr] |{=} \bullet \ar@/^/[ddr]&  &  & \\ \bullet \ar@/^/[u]  && \bullet \ar@/^/[d]  &\\ &\bullet  &\bullet \ar@/^/[r]&\bullet}\end{displaymath} starting in $P(m|m-2)$
\item   \begin{displaymath}\xymatrix{\ar @{} [drrr] |{=} \bullet  \ar@/_/[d] &  &  & \\ \bullet  && \bullet   &\\ &\bullet\ar@/^/[uul]  &\bullet \ar@/^/[u]&\bullet \ar@/^/[l]}\end{displaymath} starting in $P(m|m-2)$
\item \begin{displaymath}\xymatrix{ \bullet  &\bullet \ar@/^/[l] & \\  && \\ &&\bullet\ar@/^/[uul]}=0\end{displaymath}  
\item \begin{displaymath}\xymatrix{ \bullet  & \\  & \\ &\bullet\ar@/^/[uul]\\ &\bullet\ar@/^/[u] }=0\end{displaymath}
\item \begin{displaymath}\xymatrix{ \bullet \ar@/^/[r]&\bullet  \ar@/^/[ddr]& \\  && \\ &&\bullet}=0\end{displaymath}
\item \begin{displaymath}\xymatrix{ \bullet\ar@/^/[ddr]  & \\  & \\ &\bullet\ar@/^/[d]\\ &\bullet }=0\end{displaymath}
\item   \begin{displaymath}\xymatrix{\ar @{} [drrr] |{=} \bullet \ar@/^/[ddr] &  & \bullet \ar@/^/[d] & \\ &\bullet   & \bullet \ar@/^/[r] &\bullet \\ &\bullet\ar@/_/[u]  }\end{displaymath} starting in $P(m|m-2)$
\setcounter{enumglobal}{\value{enumi}}
\end{enumerate}
\end{multicols}
\begin{enumerate}
\setcounter{enumi}{\value{enumglobal}}
\item  A few extra relations at the lower bound of the quiver (which can easily be seen in the tables).
\end{enumerate}
\end{Theorem}

\subsection{Terms occurring in the projective resolutions of Vermas}
To determine the terms occurring in the projective resolution of Vermas we first compute the combinatorial Kazhdan-Luztig polynomials in Table \ref{table:combkaz}.

\begin{sidewaystable}[ht]
\caption{Kazhdan-Lustig polynomials\label{table:combkaz}}
	\centering		\begin{tabular}{|l|l|l|l|l|}
		\toprule
		$\mu=(s|t)$ &$\la=(n|m)$  & &$|C|$ &$p_{\la,\mu}$\\ \hline
		$\begin{array}{l} t\neq s-1\\t\neq 0\end{array}$ &
		$t\leq m < s$
		&\begin{tikzpicture}[scale=1]
		\draw (-0.6,0) node[above=-1.55pt] {$\mu$};
		\draw (0,0) node[above=-1.55pt] {$\cdots$};
		\draw (0.4,0) node[above=-1.55pt] {$\down$};
		\draw (0.8, 0) node[above=-1.55pt] {$\up$};
		\draw (1.2,0) node[above=-1.55pt] {$\cdots$};
		\draw (1.6,0) node[above=-1.55pt] {$\down$};
		\draw (2,0) node[above=-1.55pt] {$\up$};
		\draw (2.4,0) node[above=-1.55pt] {$\cdots$};
		\draw (0.8,0.3) arc (0:180:0.2);
		\draw (2,0.3) arc (0:180:0.2);
		\begin{scope}[yshift=-0.4cm]
		\draw (-0.6,0) node[above=-1.55pt] {$l_i$};
		\draw (0.4,0) node[above=-1.55pt] {$0$};
		\draw (1.6,0) node[above=-1.55pt] {$0$};
		\end{scope}
		\begin{scope}[yshift=0.15cm]
		\scriptsize	
		\draw (0.6,0) node[above=-1.55pt] {$0$};
		\draw (1.8,0) node[above=-1.55pt] {$0$};	
		\end{scope}	
		\begin{scope}[yshift=-0.8 cm]
		\draw (-0.6,0) node[above=-1.55pt] {$\la$};
		\draw (0,0) node[above=-1.55pt] {$\cdots$};
		\draw (0.4,0) node[above=-1.55pt] {$\down$};
		\draw (0.8,0) node[above=-1.55pt] {$\cdots$};
		\draw (1.2, 0) node[above=-1.55pt] {$\up$};
		\draw (1.8,0) node[above=-1.55pt] {$\cdots$};
		\draw (2.2,0) node[above=-1.55pt] {$\up$};
		\draw (2.6,0) node[above=-1.55pt] {$\cdots$};
		\end{scope}
		\end{tikzpicture}
		& $\left\{ 0 \right\}$
		& $q^{(m+n)-(s+t)}$\\ \hline

		$\begin{array}{l} t\neq s-1\\t\neq 0\end{array}$ &
		$s\leq m$
		&\begin{tikzpicture}[scale=1]
		\draw (-0.6,0) node[above=-1.55pt] {$\mu$};
		\draw (0,0) node[above=-1.55pt] {$\cdots$};
		\draw (0.4,0) node[above=-1.55pt] {$\down$};
		\draw (0.8, 0) node[above=-1.55pt] {$\up$};
		\draw (1.2,0) node[above=-1.55pt] {$\cdots$};
		\draw (1.6,0) node[above=-1.55pt] {$\down$};
		\draw (2,0) node[above=-1.55pt] {$\up$};
		\draw (2.4,0) node[above=-1.55pt] {$\cdots$};
		\draw (0.8,0.3) arc (0:180:0.2);
		\draw (2,0.3) arc (0:180:0.2);
		\begin{scope}[yshift=-0.4cm]
		\draw (-0.6,0) node[above=-1.55pt] {$l_i$};
		\draw (0.4,0) node[above=-1.55pt] {$0$};
		\draw (1.6,0) node[above=-1.55pt] {$1$};
		\end{scope}
		\begin{scope}[yshift=0.15cm]
		\scriptsize	
		\draw (0.6,0) node[above=-1.55pt] {$0$};
		\tiny
		\draw (1.8,0) node[above=-1.55pt] {$0/1$};	
		\end{scope}	
		\begin{scope}[yshift=-0.8 cm]
		\draw (-0.6,0) node[above=-1.55pt] {$\la$};
		\draw (0,0) node[above=-1.55pt] {$\cdots$};
		\draw (0.4,0) node[above=-1.55pt] {$\down$};
		\draw (0.8,0) node[above=-1.55pt] {$\down$};
		\draw (1.2, 0) node[above=-1.55pt] {$\cdots$};
		\draw (1.6,0) node[above=-1.55pt] {$\down$};
		\draw (2,0) node[above=-1.55pt] {$\cdots$};
		\draw (2.4,0) node[above=-1.55pt] {$\up$};
		\draw (2.8,0) node[above=-1.55pt] {$\cdots$};
		\draw (3.2,0) node[above=-1.55pt] {$\up$};
		\end{scope}
		\end{tikzpicture}
		& $\left\{ 0,1 \right\}$
		& $q^{(m+n)-(s+t)}+q^{(m+n)-(s+t)-2}$\\ \hline
		
		$\begin{array}{l} t=0\\s\neq 1\end{array}$ &
		$ m < s$
		&\begin{tikzpicture}[scale=1]
		\draw (-0.6,0) node[above=-1.55pt] {$\mu$};
		\draw (0.0, 0) node[above=-1.55pt] {$\up$};
		\draw (0.4,0) node[above=-1.55pt] {$\cdots$};
		\draw (0.8,0) node[above=-1.55pt] {$\down$};
		\draw (1.2,0) node[above=-1.55pt] {$\up$};
		\draw (1.6,0) node[above=-1.55pt] {$\cdots$};
		\draw (0,0.3) --(0,0.7);
		\draw (1.2,0.3) arc (0:180:0.2);
		\begin{scope}[yshift=-0.4cm]
		\draw (-0.6,0) node[above=-1.55pt] {$l_i$};
		\draw (0.8,0) node[above=-1.55pt] {$0$};
		\end{scope}
		\begin{scope}[yshift=0.15cm]
		\scriptsize	
		\draw (1,0) node[above=-1.55pt] {$0$};
		\end{scope}	
		\begin{scope}[yshift=-0.8 cm]
		\draw (-0.6,0) node[above=-1.55pt] {$\la$};
		\draw (0.2,0) node[above=-1.55pt] {$\cdots$};
		\draw (0.6,0) node[above=-1.55pt] {$\up$};
		\draw (1,0) node[above=-1.55pt] {$\cdots$};
		\draw (1.4, 0) node[above=-1.55pt] {$\up$};
		\draw (1.8,0) node[above=-1.55pt] {$\cdots$};
				\end{scope}
		\end{tikzpicture}
		& $\left\{ 0 \right\}$
		& $q^{(m+n)-(s+t)}$\\ \hline

		$\begin{array}{l} t=0\\s\neq 1\end{array}$ &
		$s \leq m$
		&\begin{tikzpicture}[scale=1]
		\draw (-0.6,0) node[above=-1.55pt] {$\mu$};
		\draw (0.0, 0) node[above=-1.55pt] {$\up$};
		\draw (0.4,0) node[above=-1.55pt] {$\cdots$};
		\draw (0.8,0) node[above=-1.55pt] {$\down$};
		\draw (1.2,0) node[above=-1.55pt] {$\up$};
		\draw (1.6,0) node[above=-1.55pt] {$\cdots$};
		\draw (0,0.3) --(0,0.7);
		\draw (1.2,0.3) arc (0:180:0.2);
		\begin{scope}[yshift=-0.4cm]
		\draw (-0.6,0) node[above=-1.55pt] {$l_i$};
		\draw (0.8,0) node[above=-1.55pt] {$1$};
		\end{scope}		
		\begin{scope}[yshift=0.15cm]
		\tiny	
		\draw (1,0) node[above=-1.55pt] {$0/1$};
		\end{scope}	
		\begin{scope}[yshift=-0.8 cm]
		\draw (-0.6,0) node[above=-1.55pt] {$\la$};
		\draw (0,0) node[above=-1.55pt] {$\down$};
		\draw (0.4,0) node[above=-1.55pt] {$\cdots$};
		\draw (0.8,0) node[above=-1.55pt] {$\down$};
		\draw (1.2, 0) node[above=-1.55pt] {$\cdots$};
		\draw (1.6,0) node[above=-1.55pt] {$\up$};
		\draw (2,0) node[above=-1.55pt] {$\cdots$};
		\draw (2.4,0) node[above=-1.55pt] {$\up$};
		\end{scope}
		\end{tikzpicture}
		& $\left\{ 0,1 \right\}$
		& $q^{(m+n)-(s+t)}+q^{(m+n)-(s+t)-2}$\\ \hline

		$t=s-1$
		&all
		&\begin{tikzpicture}[scale=1]
		\draw (-0.6,0) node[above=-1.55pt] {$\mu$};
		\draw (0.0, 0) node[above=-1.55pt] {$\dots$};
		\draw (0.4,0) node[above=-1.55pt] {$\up$};
		\draw (0.8,0) node[above=-1.55pt] {$\up$};
		\draw (1.2,0) node[above=-1.55pt] {$\down$};
		\draw (1.6,0) node[above=-1.55pt] {$\down$};
		\draw (2,0) node[above=-1.55pt] {$\cdots$};
		\draw (1.2,0.3) arc (0:180:0.2);
		\draw (1.6,0.3) arc (0:180:0.6);
		\begin{scope}[yshift=-0.4cm]
		\draw (-0.6,0) node[above=-1.55pt] {$l_i$};
		\draw (0.8,0) node[above=-1.55pt] {$0$};
		\end{scope}		
		\begin{scope}[yshift=0.2cm]
		\tiny	
		\draw (1,0) node[above=-1.55pt] {$0$};
		\draw (0.6,0) node[above=-1.55pt] {$0$};
		\end{scope}	
		\begin{scope}[yshift=-0.8 cm]
		\draw (-0.6,0) node[above=-1.55pt] {$\la$};
		\draw (0,0) node[above=-1.55pt] {$\cdots$};
		\draw (0.4,0) node[above=-1.55pt] {$\down$};
		\draw (0.8,0) node[above=-1.55pt] {$\down$};
		\draw (1.2, 0) node[above=-1.55pt] {$\cdots$};
		\draw (1.6,0) node[above=-1.55pt] {$\up$};
		\draw (2,0) node[above=-1.55pt] {$\cdots$};
		\draw (2.4,0) node[above=-1.55pt] {$\up$};
		\end{scope}
		\end{tikzpicture}
		& $\left\{ 0\right\}$
		& $q^{(m+n)-(s+t)}$\\	\hline
		\end{tabular}

\end{sidewaystable}

For the terms occurring at the $i$th position of the resolution of $M(\la)$ with $\la=(n|m)$ fixed we get by Theorem \ref{sumocc} $$P_i=\bigoplus_{\mu \in \La_m^n} p_{\la,\mu}^{(i)} P(\mu)\langle i \rangle.$$
Using Table \ref{table:combkaz} we see, that all $\mu=(s|t)$ with $\mu \geq \la$ (i.e. $s\leq n, t\leq m$) and $m+n-(s+t)=i$ occur. These terms we call \emph{A-terms}. So the A-term part of $P_i$ is 
\begin{align*}
\bigoplus_{\substack{s+t+i=m+n\\t\leq m\\s\leq n}}P(s|t).
\end{align*}
There are also other terms occurring, the so-called \emph{B-terms}, for which $s\leq n, l\leq m$, $s \leq m$ and $s\neq t+1$ with $m+n-(s+t)-2=i$.
The B-term part of $P_i$ is
\begin{equation*}
\bigoplus_{\substack{s+t+i+2=m+n\\s \neq t+1\\s\leq m}}P(s|t)
\end{equation*}
From now on, for a fixed $\la=(m|n)$ we denote by $P(s|t)_A$ the projective module $P(s|t)$ occurring in the A-term part and by $P(s|t)_B$ the one occurring in the B-term part.

\FloatBarrier
\subsection{The differentials in the projective resolution}
For the computation of the differentials in the resolutions in $K_{N-1}^2-\mod$ we use the resolutions in $K^1_{N-2}-\mod$ which we know from Theorem \ref{Th:diffproj1}. Inductively we prove:

\begin{Theorem}\label{Th:diff}
All possible degree one maps between $P_{i+1}$ and $P_i$ occur as differentials. Because they are unique up to scalar, the maps can be chosen the following way:
\begin{description}
\item[Maps between the A-terms:]\hfill
\begin{itemize}
	\item[i)] $P(s|t)_A \to P(s+1|t)_A \hspace{2cm} (-1)^{n+m+s+t+1}$
	
	\item [ii)]$ P(s|t)_A \to P(s|t+1)_A \hspace{2cm}  (-1)^{m+t+1}$
\end{itemize}
\item[Maps between the B-terms:]\hfill
\begin{itemize}
	\item[iii)] $P(s|t)_B \to P(s+1|t)_B \hspace{2cm} (-1)^{m+s+1}$
	
	\item [iv)]$ P(s|t)_B \to P(s|t+1)_B \hspace{2cm}  (-1)^{n+m+s+t+1}$
\end{itemize}
\item[Maps from A-terms to B-terms:]\hfill
\begin{itemize}
	\item[v)] $P(s|t)_A \to P(s-1|t)_B \hspace{2cm} (-1)^{(s+t+1)(n+s)+n+m+1}$
	
	\item [vi)]$ P(s|t)_A \to P(s|t-1)_B \hspace{2cm}  (-1)^{(s+t+1)(n+s)+m+s}$
\end{itemize}
\item[Maps from B-terms to A-terms:]\hfill
\begin{itemize}
	\item[vii)] $P(s|s-2)_B \to P(s+1|s)_A \hspace{1.25cm} (-1)^{n+m+1}$
\end{itemize}
\end{description}
\end{Theorem}	
\begin{proof}
Recall from paragraph \ref{sec:linprojres} the way we construct the projective resolutions which is done inductively. In our situation we have $\La=\La_{N-1}^2$ and $\Ga=\La_{N-2}^1$.

For $m=0$ we first construct the resolutions fixing $i=n-1$. 

We know that $0 \to P(1|0)\to M(1|0) \to 0$ is exact (the morphism is an isomorphism), so we may assume $n \geq 2$. $\la''=(0)$, so we get the resolution  $0 \to G^{t_0}_{\La\Ga}P(0) \to G^{t_0}_{\La\Ga}M(0) \to 0$ 
and by  equation \eqref{prback} we know $G^{t_0}_{\La\Ga}P(0)=P(n|0)$. 

By the structure of the projective resolution we also know that each term of the resolution of $M(n|0)$ equals some $P(s|0)_A$, so we only get differentials of type i). Using the cone construction all differentials $P(s|0) \to P(s+1|0)$ with $s+1<n$ equal those in the resolution of $M(n-1)$ multiplied by $(-1)$. Taking the formulas in i) we see that changing $n-1$ to $n$ multiplies the formula by $(-1)$. 

The map $P(n-1|0)_A \to P(n|0)_A$ by construction has sign $1$, which equals $(-1)^{n+0+n-1+0+1}$. So we have shown the theorem for all Vermas of the form $M(n|0)$.
\\
For $m>0$ we fix $i-1=m$ and construct the resolution for $\la'=(n|m-1)$ and $\la''=(n-2)$.\\
 We already know the projective resolution of $M(n-2) \in K^1_{N-2}-\mod$ which gives us the resolution of $G_{\La\Ga}^{t_{m-1}}M(n-2)$
\begin{equation}
0 \rightarrow G^{t_{m-1}}_{\La\Ga}P(0) \rightarrow G^{t_{m-1}}_{\La\Ga}P(1)\rightarrow \cdots \rightarrow G^{t_{m-1}}_{\La\Ga}P(n-2) \rightarrow
G^{t_{m-1}}_{\La\Ga} M(n-2)\rightarrow 0
\end{equation}
By equation (\ref{prback}) we know, that $G^{t_{m-1}}_{\La\Ga}P(j)$ equals the projective module one gets by putting $\down \up$ on the $(m-1)$st and $m$th position. So we obtain
\begin{equation}
G^{t_{m-1}}_{\La\Ga}P(j)=\begin{cases} P(m|j) &\text{for }j<m-1\\ P(j+2|m) &\text{for } j\geq m-1\end{cases}
\end{equation}
Following the cone construction for the differentials we get
\begin{align}
G^{t_{m-1}}_{\La\Ga}P(j) \stackrel{G^{t_{m-1}}_{\La\Ga}d_j(\la')}{\longrightarrow} G^{t_{m-1}}_{\La\Ga}P(j+1)
G^{t_{m-1}}_{\La\Ga}d_j(\la')=(-1)^{n+j+1}d \label{eq:diff}
\end{align}
 where $d$ is the map between the two projective modules defined in Section \ref{sec:mapproj}. Using the cone construction one first observes that all maps which do not end in $G^{t_{m-1}}_{\La\Ga}P_\bullet(n-2)$ are the differential maps from $P_\bullet(n|m-1)$ multiplied with $(-1)$. One easily checks that by changing $(m-1)$ to $m$ all maps in the theorem become multiplied by $(-1)$. Of course, by induction all these maps occur.
 
For the maps between A-terms we just have to look for those ending in A-terms belonging to $G^{t_{m-1}}_{\La\Ga}P_\bullet(n-2)$, i.e. those going to a projective module $P(s|m)_A$.

(i) We have to look at $P(s|m)_A \to P(s+1|m)_A$. This map occurs in the resolution $G^{t_{m-1}}_{\La\Ga}P_\bullet(n-2)$ as the map $G^{t_{m-1}}_{\La\Ga}P(s-2) \to G^{t_{m-1}}_{\La\Ga}P(s-1)$ and by \eqref{eq:diff} the sign is $(-1)^{n+s+1}=(-1)^{n+s+m+m+1}$ so (i) holds.

(ii) In this case we have to determine the map $P(s|m-1)_A \to P(s|m)_A$ which we get from lifting the chain map $M(n|m-1) \stackrel{f}{\rightarrow} G^{t_{m-1}}_{\La\Ga}M(n-2)$. In the construction we take $f$ such that the map $P(n|m-1)_A \to G^{t_{m-1}}_{\La\Ga}P(n-2)=P(n|m)_A$ is the morphism obtained by the multiplication in the algebra (cf. Section \ref{sec:mapproj}). Now we check that all the other maps appear and the signs are equal to $(-1)^{m+m-1+1}=1$.
Given the map $P(s+1|m-1)_A \to P(s+1|m)_A$ with $s\geq m+1$ we get a map 
$$P(s|m-1)_A \stackrel{(-1)^{n+s+1}}{\longrightarrow}P(s+1|m-1)_A \to P(s+1|m)_A.$$There is no other morphism from $P(s|m-1) \to P_{n-s-1} \to P(s+1|m)$ and by the relations obtained in Section \ref{sec:mapproj} we know that the above map equals the map 
$$P(s|m-1)_A\to P(s|m)_A \stackrel{(-1)^{n+s+1}}{\longrightarrow} P(s+1|m)_A.$$
Thus, the map $P(s|m-1)_A \to P(s|m)_A$ must occur and has sign $1$.
\\
For the B-terms all terms without $P(m|t)_B$ already exist in $P_\bullet(n|m-1)$. We only have to look for those maps going to $P(m|t)_B$.

iv) The map $P(m|t)_B \to P(m|t+1)_B$ comes from the map $G^{t_{m-1}}_{\La\Ga}P(t) \to G^{t_{m-1}}_{\La\Ga}P(t+1)$ so it has the sign $(-1)^{n+t+1}=(-1)^{n+t+m+m+1}$.

vii) We just have to look for the map $P(m|m-2)_B \to P(m+1|m)_A$ which comes from the map $G^{t_{m-1}}_{\La\Ga}P(m-2) \to G^{t_{m-1}}_{\La\Ga}P(m-1)$ and therefore has sign $(-1)^{n+m+1}$.
\\
Now we are left to show iii), v) and vi) for those maps ending in $P(m|t)_B$. Those come from the lift of the chain map $f$ (see above). The terms from $P_{n-t-1}(n|m-1)$ possibly mapping to $P(m|t)_B$ are $P(m+1|t)_A$, $P(m|t+1)_A$ and $P(m-1|t)_B$ (this one only exists for $t \neq m-2$).

First check the assertion for $t=m-2$ and for $t=m-3$.

\begin{tikzpicture}
\matrix [matrix of math nodes, column sep=1cm]
{
|(a1)| P(m+1|m-3)_A &|(b1)| P(m+1|m-2)_A &|(d)| P(m+1|m-1)_A \\
|(a2)| P(m|m-2)_A & |(b2)| P(m|m-1)_A \\
|(a3)| P(m-1|m-3)_B &  \\
 \hfill&
 \\
  \hfill&
 \\
 |(c1)| P(m|m-3)_B & |(c2)| P(m|m-2)_B& |(c3)| P(m+1|m)_A\\
};
\begin{scope}[every node/.style={font=\scriptsize }]
\draw [->] (b1) -- node[above]  {$1$} (d);
\draw [->] (b2) -- node[below] {$(-1)^{n+m+1}$} (d);
\draw [->] (c2) -- node[above]  {$(-1)^{n+m+1}$} (c3);
\draw [->] (a1) -- node[above]  {$-1$} (b1);
\draw [->] (a2) -- node [above=-3pt]{$(-1)^{n+m}$} (b1);
\draw [->] (a2) -- node[fill=white,inner sep=2pt] {$1$} (b2);
\draw [->] (a3) -- node[below] {$(-1)^{n+m}$} (b2);
\draw [->] (d) -- node[left] {$1$} (c3);
\draw [->] (a1.west) to [bend right=60 ] node [fill=white,inner sep=2pt]{$d_1$} (c1.west);
\draw [->] (a2.west) to [bend right=40] node[fill=white,inner sep=2pt,near end] {$d_2$} (c1.north west);
\draw [->] (a3) to  node[fill=white,inner sep=2pt] {$d_3$} (c1);
\draw [->] (b2) to  node[fill=white,inner sep=2pt] {$x_2$} (c2);
\draw [->] (b1.-17) to [bend left=20] node [fill=white,inner sep=2pt]{$x_1$} (c2.north east);
\draw [->] (c1) to  node [above]{$(-1)^{n+m}$} (c2);
\end{scope}
\end{tikzpicture}

In the right part of the diagrams, the squares commute by Section \ref{sec:mapproj} and we get the signs
$$x_1=(-1)^{n+m+1}=(-1)^{(m+1+m-2+1)(n+m-2)+n+m+1}$$ and  $$x_2=1=(-1)^{(m+m-1+1)(n+m-1)+m+m}.$$
Now we look at the left square. The part belonging to $P(m+1|m-3)$ commutes and the sign is $$d_1=(-1)^{1+n+m+1+n+m}=1=(-1)^{(m+1+m-3+1)(n+m+1)+n+m+1}.$$ 
By Section \ref{sec:mapproj} the morphism 
$$P(m|m-2) \to P(m|m-1) \to P(m|m-2)$$
 occurring with the sign $1$ equals the sum of the morphisms 
 $$P(m|m-2) \to P(m+1|m-2) \to P(m|m-2)$$ and $$P(m|m-2) \to P(m|m-3) \to P(m|m-2).$$ The first one occurs with sign $(-1)^{n+m+n+m+1}=-1$, so it cancels with the other morphism. The second one yields the existence and the sign of $d_2$ which is $$d_2=(-1)^{n+m}=(-1)^{(m+1+m-3+1)(n+m+1)+m+m+1}.$$
The square including $P(m-1|m-3)_B$ commutes and yields $$d_3=1=(-1)^{m+m-1+1}.$$
For the last step we check the signs and existence for arbitrary $t$, assuming we know them for $t+1$. The situation is the following:
\begin{center}
\begin{tikzpicture}
\matrix [matrix of math nodes,row sep=0.5cm, column sep=2cm]
{
|(a1)| P(m+1|t)_A &|(b1)| P(m+1|t+1)_A \\
|(a2)| P(m|t+1)_A & |(b2)| P(m|t+2)_A \\
|(a3)| P(m-1|t)_B & |(b3)| P(m-1|t+1)_B \\
 \hfill&
 \\
  \hfill&
 \\
 |(c1)| P(m|t)_B & |(c2)| P(m|t+1)_B\\
};
\begin{scope}[every node/.style={font=\scriptsize }]
\draw [->] (a1) -- node[above]  {$(-1)^{m+t}$} (b1);
\draw [->] (a2) -- node [above=-3pt]{$(-1)^{n+t+1}$} (b1);
\draw [->] (a2) -- node[fill=white,inner sep=2pt] {$(-1)^{m+t+1}$} (b2);
\draw [->] (a2) -- node [below=-4pt]{$(-1)^{(m+t)(n+m)+n+m}$} (b3);
\draw [->] (a3) -- node[below] {$(-1)^{n+t+1}$} (b3);
\draw [->] (a1.west) to [bend right=60 ] node [fill=white,inner sep=2pt]{$d_1$} (c1.west);
\draw [->] (a2.west) to [bend right=40] node[fill=white,inner sep=2pt,near end] {$d_2$} (c1.north west);
\draw [->] (a3) to  node[fill=white,inner sep=2pt] {$d_3$} (c1);
\draw [->] (b3) to  node[fill=white,inner sep=2pt] {$x_3$} (c2);
\draw [->] (b2.-17) to [bend left=20] node [fill=white,inner sep=2pt, near end]{$x_1$} (c2.north east);
\draw [->] (b1.-17) to [bend left=60] node [right]{$x_2$} (c2.east);
\draw [->] (c1) to  node [above]{$(-1)^{n+t+1}$} (c2);
\end{scope}
\end{tikzpicture}
\end{center}
By induction we get the signs $x_1=(-1)^{(m+t)(n+m+1)}$, $x_2=(-1)^{(m+t+1)(n+m)}$ and $x_3=1$.
The maps $P(m|t+1) \to P(m|t+2) \to P(m|t+1)$ (occurring with sign $(-1)^{n+t+1+(m+t)(n+m+1)}$) and $P(m|t+1) \to P(m|t) \to P(m|t+1)$ (sign $(-1)^{(m+t)(n+m)+n+m}$) are the same and appear with opposite signs, so their sum is zero.

The square including $P(m+1|t)_A$ commutes and yields \begin{align*} p_1&=(-1)^{m+t+(m+t)(n+m+1)+n+t+1}\\&=(-1)^{(m+1+t)(n+m+1)}\\&=(-1)^{(m+1+t+1)(n+m+1)+n+m+1}\end{align*}
The square including $P(m|t+1)_A$ commutes and yields \begin{align*}p_2&=(-1)^{m+t+1+(m+t+1)(n+m)+n+t+1}\\&=(-1)^{(m+t)(n+m)}\\&=(-1)^{(m+t+1+1)(n+m)+m+m}\end{align*}
The square including $P(m-1|t)_B$ commutes and yields 
\begin{align*}p_3&=(-1)^{n+t+1+n+t+1}\\&=1\\&=(-1)^{m+m-1+1}\end{align*}
So we have proved iii), v), vi).
\end{proof}
\begin{Remark}\label{Re:choices}
Note that the way we have chosen the smaller weight $\la'$ in the construction involves a choice. For $\la=(n|m)$ we have chosen $\la'=(n|m-1)$ if $m >0$. Therefore, by the way of the construction we get a map $P_\bullet(n|m) \to P_\bullet(n|m-1)$. Anyway, we could have chosen $\la'=(n-1|m)$ if $n-1>m$. This possibility delivers us a nonzero map $M(n|m) \to M(n-1|m)$ in the derived category and therefore a nonzero map $P_\bullet(n|m) \to P_\bullet(n-1|m)$.

\textbf{
From now on we fix the projective resolutions constructed in Theorem \ref{Th:diff}.}
\end{Remark}
Combining our knowledge about these projective resolutions with Lemma \ref{Le:possmaps}, similar to Lemma \ref{Le:mapprojres} we obtain the following lemma:
\begin{Lemma}\label{Le:posschmaps}
$\hom^k(P_\bullet(a|b), P_\bullet(c|d)\langle j\rangle)_0={0}$ unless we are in one of the below cases:
\begin{enumerate}
\item \label{poss0} $k-j=0$ and $a+b \leq c+d+k+2$
\item \label{poss1}$k-j=1$ and $a+b \leq c+d+k+3$
\item \label{poss2} $k-j=2$ and $a+b \leq c+d+k+4$
\item \label{poss3} $k-j=3$ and $a+b \leq c+d+k+3$
\item \label{poss4} $k-j=4$ and $a+b \leq c+d+k+2$
\end{enumerate}
\end{Lemma}
\begin{proof}
Assume we have a map $L:P_\bullet(a|b) \to P_\bullet(c|d) \langle j \rangle [k]$, i.e. in each component a morphism
$P_i(c|d) \to P_{i-k}(c|d) \langle j \rangle$. Since our resolutions are linear, we look for morphisms
$P(s|t)\langle i \rangle \to P(s'|t')\langle i-k+j \rangle$ fulfilling the following conditions:
\begin{enumerate}
\item $s+t=a+b-i$ (A-terms) or $s+t=a+b-i-2$ and $s\neq t+1$ (B-terms)
\item $s'+t'=c+d-(i-k)$ (A-terms) or $s'+t'=c+d-(i-k)-2$ and $s\neq t+1$ (B-terms)
\end{enumerate}
Using Lemma \ref{Le:possmaps} one notices that there are only maps for $0 \leq k-j \leq 4$ and that there cannot be any if $a+b$ and $c+d$ differ too much. 

Start with maps from B-terms of $P_\bullet(a|b)$ to A-terms of $P_\bullet(c|d)$, i.e. we look for maps $P(s|t) \to P(s'|t')$ with $s+t=a+b-i-2$ and $s'+t'=c+d-(i-k)$. Note that the two special cases in Lemma \ref{Le:possmaps} cannot occur, since $P(s|s-1)$ cannot occur as a B-term. Therefore from Lemma \ref{Le:possmaps} we get as a condition
\begin{enumerate}
\item $k-j=0$ and $a+b-i-2 \leq c+d-(i-k)$
\item $k-j=1$ and $a+b-i-2 \leq c+d-(i-k)+1$
\item $k-j=2$ and $a+b-i-2 \leq c+d-(i-k)+2$
\item $k-j=3$ and $a+b-i-2 \leq c+d-(i-k)+1$
\item $k-j=4$ and $a+b-i-2 \leq c+d-(i-k)$
\end{enumerate}
Now we look for maps mapping A-terms to A-terms. Doing the same as in the previous case, one observes that the conditions become stricter (since the left side of the inequality becomes $a+b-i$). One only has to verify that for the two special cases occurring in Lemma \ref{Le:posschmaps} with $t=s-1$ and $k-j=2$ or $k-j=3$ the conditions stay unchanged. In these cases one obtains $(s'|t')=(s-2|s-3)$ or $(s'|t')=(s-1|s-3)$, respectively. Putting in $s+t=a+b-i$ and $s'+t'=c+d-(i-k)$ the above conditions in the needed cases stay unchanged.

Analogously to the previous arguments one checks that mapping to a B-term makes the conditions stricter, which proves the lemma.
\end{proof}
\begin{Cor}\label{Cor:posschmaps}
If we denote $a+b-(c+d)=s$, we can rewrite the conditions in Lemma \ref{Le:posschmaps} for a morphism
$$f \in \hom^{s+k}(P_\bullet(a|b), P_\bullet(c|d)\langle s+j\rangle)_0$$
as follows:
\begin{enumerate}
\item \label{Cposs0} $k-j=0$ and $-2 \leq k$
\item \label{Cposs1}$k-j=1$ and $-3 \leq k$
\item \label{Cposs2} $k-j=2$ and $-4 \leq k$
\item \label{Cposs3} $k-j=3$ and $-3 \leq k$
\item \label{Cposs4} $k-j=4$ and $-2 \leq k$
\end{enumerate}
\end{Cor}
\begin{proof}
We only have to check that the inequalities for $k$ encode the same information as those above. Therefore, we write
\begin{align*}
&&a+b &\leq c+d+s+k+i\\
\Leftrightarrow &&0 &\leq -s+s+k+i\\
\Leftrightarrow &&-i &\leq k
\end{align*}
\end{proof}

\subsection{Dimensions of the $\Ext$-algebra}\label{sec:dimexp}
Using the formulas \eqref{ca:numcas1}-\eqref{ca:numcas4} from Theorem \ref{Th:dim} again, we now compute the dimensions of the $\Ext$-spaces in our situation. 
\small{\allowdisplaybreaks{
\begin{align}
&&&E^k((n|m),(n|m))&=&\begin{cases} 1 &k=0\\ 0 &\text{else}\end{cases}\\
&&&E^k((n|m),(n-1|m))&\stackrel{(\ref{ca:numcas3})}{=}&E^{k-1}((n-1|m),(n-1|m))\notag\\
&&&&&+E^{k}((n-1|m),(n-1|m))\notag\\
&&&&=&\begin{cases} 1 &k \in \{0,1\}\\0 &\text{else}\end{cases}\\
&j<n-1:&& E^k((n|m),(j|m))&\stackrel{(\ref{ca:numcas2})}{=}&E^{k-1}((n-1|m),(j|m))\notag\\
&&&&\stackrel{Ind}{=}&\begin{cases} 1 &k \in \{n-j-1,n-j\}\\ 0 &\text{else}\end{cases}\\
&&&E^k((n|m),(n|m-1))&\stackrel{(\ref{ca:numcas3})}{=}&\begin{cases} 1 &k \in \{0,1\}\\0 &\text{else}\end{cases}\\
&j<m-1:&& E^k((n|m),(n|j))&\stackrel{(\ref{ca:numcas2})}{=}&E^{k-1}((n|m-1),(n|j))\notag\\
&&&&\stackrel{Ind}{=}&\begin{cases} 1 &k \in \{m-j-1,m-j\} \\ 0 &\text{else}\end{cases}\\
&m<n-1:&&E^k((n|m),(n-1|m-1))&\stackrel{(\ref{ca:numcas3})}{=}&E^{k-1}((n-1|m),(n-1|m-1))\notag\\
&&&&&+E^{k}((n-1|m),(n-1|m-1))\notag\\
&&&&=&\begin{cases} 1 &k=2\\2 &k=1 \\1 &k=0\\0 &\text{else}\end{cases}\\
&\begin{array}{l}j<l-1\\l<n\\m<l\\j<m\end{array}:&&E^k((n|m),(l|j))&\stackrel{(\ref{ca:numcas2})}{=}&E^{k-(m-j)+1}((n|j+1),(l|j))\notag\\
&&&&\stackrel{(\ref{ca:numcas2})}{=}&E^{k-(m-j)-(n-l)+2}((l+1|j+1),(l|j))\notag\\
&&&&\stackrel{Ind}{=}&\begin{cases} 1 &k=n+m-(l+j)\\2 &k=n+m-(l+j)-1 \\1 &k=n+m-(l+j)-2\\0 &\text{else}\end{cases}\label{ca:klk}\\
&&&E^k((n|m),(l|l-1))&\stackrel{(\ref{ca:numcas3})}{=}&E^{k-(m-l)}((n|l),(l|l-1))\notag\\
&&&&\stackrel{(\ref{ca:numcas3})}{=}&E^{k-(m-l)+1}((n|l-1),(l|l-1))\notag\\
&&&&\stackrel{Ind}{=}&\begin{cases} 1 &k =(n+m)-2l+1\\1 &k=(n+m)-2l \\ 0 &\text{else}\end{cases} \label{ca:-1}\\
&j<m-1:&& E^k((n|m),(n|j))&\stackrel{(\ref{ca:numcas2})}{=}&E^{k-1}((n|m-1),(n|j))\notag\\
&&&&\stackrel{Ind}{=}&\begin{cases} 1 &k \in \{m-j-1,m-j\} \\ 0 &\text{else}\end{cases}\\
&m<n-1:&&E^k((n|m),(n-1|m-1))&\stackrel{(\ref{ca:numcas3})}{=}&E^{k-1}((n-1|m),(n-1|m-1))\notag\\
&&&&&+E^{k}((n-1|m),(n-1|m-1))\notag\\
&&&&=&\begin{cases} 1 &k=2\\2 &k=1 \\1 &k=0\\0 &\text{else}\end{cases}\\
&&&E^k((n|n-1),(n-1|n-2))&\stackrel{(\ref{ca:numcas2})}{=}&E^{k-1}((n|n-2),(n-1|n-2))\notag\\
&&&&=&\begin{cases} 1 &k \in \{1,2\} \\ 0 &\text{else}\end{cases}\\
&m<n-1:&&E^k((n|m),(m|m-1))&\stackrel{(\ref{ca:numcas2})}{=}&E^{k-(n-m)+1}((m+1|m),(m|m-1))\notag\\
&&&&=&\begin{cases} 1 &k \in \{n-m,n-m+1\} \\ 0 &\text{else}\end{cases}\\
&&&E^k((n|m),(m|j))&\stackrel{(\ref{ca:numcas1})}{=}&E^{k}((n|j+1),(j+1|j))\notag\\
&&&&=&\begin{cases} 1 &k \in \{n-j-1,n-j\} \\ 0 &\text{else}\end{cases} \label{ca:eqk}\\
&&&E^k((n|n-1),(n-2|m))&\stackrel{(\ref{ca:numcas4})}{=}&E^{k-1}((n|n-2),(n-2|m))\notag\\
&&&&&-E^{k+1}((n|n-2),(n-2|m))\notag\\
&&&&&+E^{k}((n|n-2),(n-1|m))\notag\\
&&&&=&\begin{cases} 1-0+0=1 &k= n-m+1\\ 1-0+0=1 &k= n-m \\ 0+1-1=0 &k= n-m-1\\0+2-1=1 &k= n-m-2\\0-0+1=1 &k= n-m-3\\0 &\text{else}\end{cases}
\end{align}
\begin{align}
&\begin{array}{l}i\neq j\\l<m\end{array}:&&E^k((n|m),(l|j))&\stackrel{(\ref{ca:numcas2})}{=}&E^{k-(m-(l+1))-(n-(l+2))}((l+2|l+1),(l|j))\notag\\
&&&&=&\begin{cases} 1 &k= (n+m)-(l+j)\\ 1 &k=(n+m)-(l+j)-1  \\ 0 &k= (n+m)-(l+j)-2\\1 &k= (n+m)-(l+j)-3\\1 &k= (n+m)-(l+j)-4\\0 &\text{else}\end{cases}
\end{align}}}
Note that the dimensions are all at most two.

\subsection{The explicit elements in the $\Ext$-algebra}\label{sec:elements2}
\subsubsection{Elements generating the algebra}
The next step is to determine elements in the $\Ext$-spaces explicitly and to show, that they are not homotopic to zero or (if the dimension of the space is greater one) not homotopic to the other element. We are working in two steps. First we determine a few elements which are the generators of the $\Ext$-algebra. To get all elements we compute the multiplication rules. Afterwards we will use the dimension formulas from Section $\ref{sec:dimexp}$  to check that we  generated enough elements in the Ext-algebra.

As already mentioned in Section \ref{sec:Exthom}, for verifying that a map $f:P_\bullet \to Q_\bullet[k]$ is a map between chain complexes (and therefore a cycle in the $\Ext$-algebra), we check if the map $f:P_\bullet \to Q_\bullet[k]_{\hom}$ is commutative or anticommutative for $k$ even or odd, respectively. Here the $[ \quad {]_{\hom}}$-shift denotes the shift without changing the differential defined in Notation \ref{Not:sign}.

The first elements we are looking for are the so-called Identities already mentioned in Remark \ref{Re:choices}.

\begin{Theorem}\label{Th:ElemId}
The following chain maps determine nonzero elements in the $\Ext$-algebra:
\begin{align}
\Id^{(n|m)}_{(n|m-1)}:& \quad P_\bullet(n|m) \to P_\bullet(n|m-1)\langle 1 \rangle[1] \\
&\begin{cases}P(s|t)_A \to P(s|t)_A \\ P(s|t)_B \to P(s|t)_B \end{cases}  \\
\Id^{(n|m)}_{(n-1|m)}:&\quad P_\bullet(n|m) \to P_\bullet(n-1|m)\langle 1 \rangle[1] \\
&\begin{cases} P(s|t)_A \to (-1)^{m+t} P(s|t)_A  \\  P(s|t)_B \to (-1)^{m+s} P(s|t)_B \end{cases}  \end{align}	
\end{Theorem}
\begin{proof}
The first morphism pops out of the construction of $P(n|m)$ for $m \geq 1$. All signs in $P_\bullet(n|m-1)$ are opposite to those in $P_\bullet(n|m)$, so we get an anticommuting map as required since its degree in the $\hom$-algebra is 1 (so it is a cycle). 

For the second morphism we check that it anticommutes with all seven cases in Theorem~\ref{Th:diff}. This is done in the Appendix \ref{Ap:Elem}. By these computations $\Id^{(n|m)}_{(n-1|m)}$ is a cycle in the $\hom$-algebra.

Both elements cannot be nullhomotopic, because a homotopy would be a map
$$H: P_\bullet(n|m) \to P_\bullet(\la')\langle 1 \rangle \text{ with } \la'=(n|m-1) \text { or } \la'=(n-1|m)$$ 
which cannot exist by Lemma \ref{Le:posschmaps}.
\end{proof}

\begin{Theorem}\label{Th:elemF2}
There are two non-nullhomotopic degree zero maps:
\begin{align}
F^{(n|m)}_{(n|m-1)}:&\quad P_\bullet(n|m) \to P_\bullet(n|m-1)\langle -1 \rangle \notag \\
& \begin{cases} P(s|t)_A \to P(s|t-1)_A & t\leq m\\
								P(s+2|s)_A \to (-1)^{n+s}  P(s+1|s)_A  &s \leq m-1\\
								P(s+1|s)_A \to (-1)^{n+s} P(s|s-2)_B  &s \leq m-1\\
								P(s|t)_B \to P(s-1|t)_B &s\leq m\\
								P(s|s-2)_B \to P(s|s-1)_A&s\leq m
\end{cases}\\
\newline
\widetilde {F}^{(n|m)}_{(n-1|m)}:& \quad P_\bullet(n|m) \to P_\bullet(n-1|m) \langle -1 \rangle \notag \\
& \begin{cases} P(s|t)_A \to P(s-1|t)_A & s\leq n\\
								P(s|t)_B \to P(s|t-1)_B &t\leq m\\
								P(s+1|s)_A \to P(s|s-2)_B&s\leq m
\end{cases}
\end{align}
\end{Theorem}
\begin{proof} 
First we prove the assertion for $F^{(n|m)}_{(n|m-1)}$.
We need to check that these maps are commuting chain maps. In other words we need to verify that all diagrams of the form
\begin{equation}\xymatrix{ V \ar[r]\ar[d] &W\ar[d]\\
X \ar[r] &Y}
\end{equation}
with $V$ all possible terms occurring in $P_\bullet(n|m)$ commute.
The cases to consider are the following:
\begin{enumerate}[label={F\arabic*})]
\item \label{FA} $V=P(s|t)_A$ and $s>t+3$
\item \label{FAs3}  $V=P(s+3|s)_A$
\item \label{FAs2}  $V=P(s+2|s)_A$
\item \label{FAs1}  $V=P(s+1|s)_A$
\item \label{FB} $V=P(s|t)_B$ and $s>t+3$
\item \label{FBs2}  $V=P(s|s-2)_B$
\item \label{FBs3}  $V=P(s|s-3)_B$
\end{enumerate}
which are treated in detail in Appendix \ref{Ap:Elem}.

Hence, we have verified that $F^{(n|m)}_{(n|m-1)}$ is a cycle. We are left to show that it is not nullhomotopic. A homotopy would be a morphism 
$$H: P_\bullet(n|m) \to P_\bullet(n|m-1)\langle -1 \rangle [-1]. $$ 
There is no difficulty to prove that this map cannot exist. In Theorem \ref{Th:elf} we will define a family of maps to which the above map belongs and proof that they are not nullhomotopic. Since the proof in this special case is contained in the more general case, we only refer to the proof of Theorem \ref{Th:elf}.

Now consider the map $\widetilde {F}^{(n|m)}_{(n-1|m)}$, where we have less special cases to deal with. In fact we only have to check \begin{equation}\xymatrix{ V \ar[r]\ar[d] &W\ar[d]\\
X \ar[r] &Y}
\end{equation}
with V being one of the following modules:
\begin{enumerate}[label={$\widetilde {\text{F}}$\arabic*})]
\item \label{FtilA} $V=P(s|t)_A$ and $s>t+2$
\item \label{FtilAs2}  $V=P(s+2|s)_A$
\item \label{FtilAs1}  $V=P(s+1|s)_A$
\item \label{FtilB} $V=P(s|t)_B$ and $s>t+2$
\item \label{FtilBs2}  $V=P(s|s-2)_B$
\end{enumerate}
As it is computed in Appendix \ref{Ap:Elem}, the map $\widetilde{F}^{(n|m)}_{(n-1|m)}$ is a cycle. To prove that the element is not nullhomotopic we refer to the proof of Theorem \ref{Th:ftil}.
 \end{proof}
 
\begin{Theorem}\label{Th:elemG}
The following map defines a non-trivial element in \break $\Ext^1(M(m+1|m),M(m-1|m-2))$:
\begin{align}
G^{(m+1|m)}_{(m-1|m-2)}:& \quad P_\bullet(m+1|m) \to P_\bullet(m-1|m-2)[1]\\
& \begin{cases} P(s|t)_A \to 0 & t \neq s-1\\
								P(s|s-1)_A \to (-1) P(s-1|s-3)_A  \\
								P(s|t)_B \to (-1)^{(m+t)(m+s+1)+s+t} P(s-1|t)_A\\ 
								\ \ \ +(-1)^{(m+t)(m+s)+s+t} P(s|t-1)_A &s<m\\
								P(m|t)_B \to  P(m-1|t)_A
								\end{cases}
\end{align}
\end{Theorem}
\begin{proof}
Similar to the above computations, we check that the map anticommutes. Here we have to check the cases:
\begin{enumerate}[label={$G$\arabic*})]
\item \label{GA} $V=P(s|t)_A$ and $s>t+2$
\item \label{GAs2}  $V=P(s+2|s)_A$
\item \label{GAs1}  $V=P(s+1|s)_A$
\item \label{GB} $V=P(s|t)_B$ and $s>t+2$
\item \label{Gs2}  $V=P(s|s-2)_B$
\end{enumerate}
what is done in Appendix \ref{Ap:Elem}.
Since the diagrams in all cases anticommute one only has to show that the map is not nullhomotopic.
A homotopy would be a morphism
$$H: P_\bullet(m+1|m) \to P_\bullet(m-1|m-2) $$ 
which cannot exist by Lemma \ref{Le:posschmaps} case \ref{poss0}, since $m+1+m \nleq m-1+m-2+2$.

\end{proof}
\begin{Theorem}\label{Th:ElemK}
The following assignment defines a non-trivial element in \break $\Ext^0(M(m+1|m),M(m-1|m-2))$:
\begin{align}
K^{(m+1|m)}_{(m-1|m-2)}:& \quad P_\bullet(m+1|m) \to P_\bullet(m-1|m-2)\langle -2 \rangle \\
& \begin{cases} P(s|t)_A \to 0 & t \neq s-1\\
								P(s|s-1)_A \to (-1)^{m+s+1} P(s-2|s-3)_A  \\
								P(s|t)_B \to (-1)^{(m+s)(m+t+1)} P(s-1|t-1)_A\\ 
								\end{cases}
\end{align}
where the degree two maps are chosen to be any nonzero composition of a pair of standard degree one morphisms. 
\end{Theorem}
\begin{proof}
Note that the map is well-defined since for $\la \neq \mu$ we have \\
$\hom(P(\la),P(\mu))_i \leq 1$ and so there is (up to scalar) only one possible choice for the maps above. There are no scalars occurring in the compositions of standard maps in \ref{sec:mapproj}, so one may choose any composition of standard maps.

Again, we have to check that the map is a commutative chain map and for it to check that diagrams of the form \begin{equation}\xymatrix{ V \ar[r]\ar[d] &W\ar[d]\\
X \ar[r] &Y}
\end{equation} commute.

 We simplify the situation by some general computations. Assume we are given a composition  where \begin{displaymath} \xymatrix{ P(s|t)_A \ar[rr]^{differential} &&\text{B-terms} \ar[rr]^{K} &&Y=\text{A-Terms}}
\end{displaymath}
and the terms $P(k|l)_A$ in $Y$ satisfy $k+l=s+t-3$ (by the assignments given in the definition of $K$). The problem reduces to the question whether a degree 3 morphism between projectives exists.
By Lemma \ref{Le:possmaps} this is only possible for $s=t+1$ (treated below). Otherwise the composition must be zero. Since the map $K$ maps $P(s|t)_A$ with $s>t+1$ to zero, the diagram always commutes for $V=P(s|t)_A$ and $s>t+1$.

Another general idea is to determine the degree $2$ maps as composition of degree one maps.
Writing down all other possible cases one obtains the below possibilities for the starting point:
\begin{enumerate}[label={$K$\arabic*})]
\item \label{KAs1}  $V=P(s+1|s)_A$
\item \label{KB} $V=P(s|t)_B$ and $s>t+2$
\item \label{Ks2}  $V=P(s|s-2)_B$
\end{enumerate}
The computations are carried out in  Appendix \ref{Ap:Elem}.

We are left to show that the map is not nullhomotopic. This holds since a homotopy would be a map
$$H: P_\bullet(m+1|m) \to P_\bullet(m-1|m-2)\langle -2 \rangle[-1] $$ 
which cannot exist by Lemma \ref{Le:posschmaps} case \ref{poss1}, since $m+1+m \nleq m-1+m-2+3-1$.

\end{proof}

\subsubsection{Elements obtained as products of generators}
It remains to define a bunch of other elements which one introduces in Theorems \ref{Th:elid}-\ref{Th:elj} as products (up to sign) from the above ones.
\begin{Theorem}\label{Th:elid}
For all $m,n,k,l \in \{0, \ldots N\}$ with $m <n$, $l <k$ and $l \leq m, k \leq n$ consider the product of generators from Theorem \ref{Th:ElemId}
\begin{equation}
\Id^{(n|m)}_{(n|m-1)}  \cdots \Id^{(n|m-1)}_{(n|m-2)} \cdot
\Id^{(n|l+1)}_{(n|l)} \cdot \Id^{(n|l)}_{(n-1|l)}  \cdots \Id^{(k+1|l)}_{(k|l)} \end{equation}
where $\Id^{(a|b)}_{(a|b)} $ denotes the identity morphism on the complex.
Then the following holds:
\begin{enumerate}
\item These maps are exactly given by
\begin{align}\label{al:defId}
\Id^{(n|m)}_{(k|l)}:& \quad P_\bullet(n|m) \to P_\bullet(k|l)\langle (n+m)-(k+l) \rangle[(n+m)-(k+l)]\nonumber \\
&\begin{cases}P(s|t)_A \to P(s|t)_A &(-1)^{(n+k)(l+t)} \\ P(s|t)_B \to P(s|t)_B &(-1)^{(n+k)(l+s)}\end{cases}.  \end{align}	

\item They are not nullhomotopic, hence they define nonzero elements in the $\Ext$-space $\Ext^{(n+m)-(k+l)}(M(n|m),M(k|l))$.
\end{enumerate}
\end{Theorem}
\begin{proof}
By computing the signs in the product $$\Id^{(n|m)}_{(n|m-1)}  \cdots \Id^{(n|m-1)}_{(n|m-2)} \cdot
\Id^{(n|l+1)}_{(n|l)} \cdot \Id^{(n|l)}_{(n-1|l)}  \cdots \Id^{(k+1|l)}_{(k|l)} $$
one gets the appropriate sign in formula \eqref{al:defId}. 

It remains to show that the element is not nullhomotopic. If it was, we would have a map
$$H: P_\bullet(n|m) \to P_\bullet(k|l)\langle (n+m)-(k+l) \rangle[(n+m)-(k+l)-1]$$ which cannot exist by Lemma \ref{Le:posschmaps}.
\end{proof}
\begin{Theorem}\label{Th:elf}
For all $m,n,k,l \in \{0, \ldots N\}$ with $m <n$, $l+1 <k$ and $l < m, k \leq n$ consider the map
\begin{equation}
\Id^{(n|m)}_{(k|l+1)}  \cdot F^{(k|l+1)}_{(k|l)}.
\end{equation}
Again the following holds:
\begin{enumerate}
\item The map is exactly given by
\begin{align}\label{al:comF}\nonumber
F^{(n|m)}_{(k|l)}:& \quad P_\bullet(n|m) \to P_\bullet(k|l)\langle (n+m)-(k+l)-2 \rangle[(n+m)-(k+l)-1]\\
& \begin{cases} P(s|t)_A \to P(s|t-1)_A \:(-1)^{(n+k)(l+t+1)} & t\leq l\\
								P(s+2|s)_A \to P(s+1|s)_A  \:(-1)^{(n+k)(l+s+1)+k+s} &s \leq l\\
								P(s+1|s)_A \to P(s|s-2)_B \:(-1)^{(n+k)(l+s+1)+k+s} &s \leq l\\
								P(s|t)_B \to P(s-1|t)_B \:(-1)^{(n+k)(l+s+1)}&s\leq l+1\\
								P(s|s-2)_B \to P(s|s-1)_A\:(-1)^{(n+k)(l+s+1)}&s\leq l+1
\end{cases}
\end{align}
\item The map defines a nonzero element in $\Ext^{(n+m)-(k+l)-1}(M(n|m),M(k|l))$.
\end{enumerate}
\end{Theorem}
\begin{proof}
It is obvious that the product $\Id^{(n|m)}_{(k|l+1)}  \cdot F^{(k|l+1)}_{(k|l)}$ provides the map in equation \eqref{al:comF}.

To show that $F$ is not nullhomotopic we consider a homotopy 
$$H: P_\bullet(n|m) \to P_\bullet(k|l)\langle (n+m)-(k+l) \rangle[(n+m)-(k+l)].$$ Since both resolutions are linear, $H$ must map the head of a projective to the head of another, so up to scalar it must be the identity on these objects. Therefore, we focus on the first map occurring and show that $H$ cannot exist:
\begin{displaymath}
\xymatrix{\cdots \ar[r] &P(k|l+1)_A \ar[d]^f \ar[r] \ar[ld]^x &Y\ar[ld]^y\\
X \ar[r] & P(k|l)_A \langle -1 \rangle \ar[r] & \cdots}
\end{displaymath}
We have to find maps $x$ and $y$ such that $yd+dx=f$.

\begin{enumerate}
\item We look at $x$ and show that is has to be zero. Since the map has to be the identity on objects  up to a scalar factor it has to map $P(k|l+1)_A$ to $P(k|l+1)$ in $X$. This does not exist since for all $P(s|t)$ occuring in $X$ we have $s+t=k+l-1$ or $s+t=k+l-3$. Therefore, $x$ is the zero map.
\item If the map $y$ exists, it sends $P(k|l)$ to $P(k|l)_A$. This projective occurs as a B-Term in the needed position of the upper resolution for $k\leq m <n$, so we would have $P(k|l)_B \to P(k|l)_A$. Assume this map exists, then we would also get a nonzero map $P(k+1|l)_A \to P(k|l)_B \to P(k|l)_A$. Therefore, we are in the situation
\begin{displaymath}
\xymatrix{\cdots \ar[r] &P(k+1|l)_A  \ar[r] \ar[ld]^{x'} &Y\ar[ld]^{y+y'}\\
X \ar[r] & P(k|l)_A \langle -1 \rangle \ar[r] & \cdots}
\end{displaymath}
and have to find $x'$ and $y'$ such that $dx'+(y+y')d=0$. $y'$ has to be zero, since as mentioned before, the only possible map from $Y$ to $P(k|l)$ is $y$. $x'$ also has to be zero by arguments similar to the above ones. This leads to a contradiction, since $yd \neq 0$.
\end{enumerate}
So we have shown that $H$ cannot exist and the theorem follows.
\end{proof}

\begin{Theorem}\label{Th:ftil}
For all $m,n,k,l \in \{0, \ldots N\}$ with $m <n$, $l <k$ and $l \leq m, k <n$ consider the product of generators
\begin{equation}\Id^{(n|m)}_{(k+1|l)}  \cdot \widetilde{F}^{(k+1|l)}_{(k|l)}. \end{equation}

Then the following holds:
\begin{enumerate}
\item These maps are exactly given by
\begin{align}\nonumber \widetilde{F}^{(n|m)}_{(k|l)}:& \quad P_\bullet(n|m) \to P_\bullet(k|l) \langle (n+m)-(k+l)-2 \rangle [(n+m)-(k+l)-1]\\
& \begin{cases} P(s|t)_A \to P(s-1|t)_A:\ (-1)^{(n+k+1)(l+t)} & s\leq k+1\\
								P(s|t)_B \to P(s|t-1)_B:\ (-1)^{(n+k+1)(l+s)} &t\leq l\\
								P(s+1|s)_A \to P(s|s-2)_B:\ (-1)^{(n+k+1)(l+s)}&s\leq l
\end{cases}.  
\end{align}	
\item Again they are not nullhomotopic, hence they define nonzero elements in the space $\Ext^{(n+m)-(k+l)-1}(M(n|m),M(k|l))$.
\end{enumerate}
\end{Theorem}
\begin{proof}
Again we only have to verify that the map is not nullhomotopic. Since this is similar to the arguments in the previous proof we only sketch the argument:

Assuming we have
$$H: P_\bullet(n|m) \to P_\bullet(k|l)\langle (n+m)-(k+l) \rangle[(n+m)-(k+l)],$$ we are in the situation:
\begin{displaymath}
\xymatrix{\cdots \ar[r] &P(k|l+1)_A \ar[d]^f \ar[r] \ar[ld]^x &Y\ar[ld]^y\\
X \ar[r] & P(k|l)_A \langle -1 \rangle \ar[r] & \cdots}
\end{displaymath}
and have to find maps $x$ and $y$ such that $yd+dx=f$.
\begin{enumerate}
\item $x$ has to be the zero map (same reason as above).
\item If the map $y$ is nonzero, it sends $P(k|l)$ to $P(k|l)_A$. If $P(k|l)_B$ occurs as a B-term then $k \neq l+1$, so we would also get a nonzero map $P(k|l+1)_A \to P(k|l)_B \to P(k|l)_B$. This yields a contradiction.
\end{enumerate}
Hence, the theorem is proved.
\end{proof}

\begin{Theorem}\label{Th:elg}
For all $m,n,k,l \in \{0, \ldots N\}$ with $m <n$, $l <k$ and $k<m$ consider the product
\begin{equation}\Id^{(n|m)}_{(k+2|k+1)} \cdot G^{(k+2|k+1)}_{(k|k-1)}\cdot \Id^{(k|k-1)}_{(k|l)}. \end{equation}

Then the following holds:
\begin{enumerate}
\item These maps are exactly given by
\begin{align}\nonumber
G^{(n|m)}_{(k|l)}:&\quad  P_\bullet(n|m) \to P_\bullet(k|l)\langle (n+m)-(k+l)-4 \rangle [(n+m)-(k+l)-3]\\
& \begin{cases} P(s|t)_A \to 0 & t \neq s-1\\
								P(s|s-1)_A \to (-1)^{(n+k)(k+s)+1} P(s-1|s-3)_A &s-1\leq k \\
								P(s|t)_B \to (-1)^{(k+s+1)(n+t)} P(s-1|t)_A \\ 
								\ \ \ +(-1)^{(k+s+1)(n+t+1)+s+t} P(s|t-1)_A &s<k+1\\
								P(k+1|t)_B \to  P(k|t)_A
								\end{cases}
\end{align}
\item They define nonzero elements in $\Ext^{(n+m)-(k+l)-3}(M(n|m),M(k|l))$.
\end{enumerate}
\end{Theorem}
\begin{proof}
The signs are obtained by writing down the product 
$$\Id^{(n|m)}_{(k+2|k+1)} \cdot G^{(k+2|k+1)}_{(k|k-1)}\cdot \Id^{(k|k-1)}_{(k|l)}$$
explicitly.

Let us assume that $G$ is nullhomotopic, i.e. there is a map
$$H: P_\bullet(n|m) \to P_\bullet(k|l)\langle (n+m)-(k+l)-4 \rangle[(n+m)-(k+l)-4].$$ 
It cannot exist by Corollary \ref{Cor:posschmaps} case \ref{Cposs0}, since $-2 \nleq -4$.
\end{proof}

\begin{Theorem}\label{Th:elk}
For all $m,n,k,l \in \{0, \ldots N\}$ with $m <n$, $l <k$ and $k<m$ consider the product
\begin{equation}\Id^{(n|m)}_{(k+2|k+1)} \cdot K^{(k+2|k+1)}_{(k|k-1)}\cdot \Id^{(k|k-1)}_{(k|l)}. \end{equation}

Then the following holds:
\begin{enumerate}
\item These maps are exactly given by
\begin{align}\nonumber
K^{(n|m)}_{(k|l)}:& \quad P_\bullet(n|m) \to P_\bullet(k|l)\langle (n+m)-(k+l)-6 \rangle [(n+m)-(k+l)-4]\\
& \begin{cases} P(s|t)_A \to 0 & t \neq s-1\\
								P(s|s-1)_A \to (-1)^{(n+k+1)(k+s)} P(s-2|s-3)_A  & s-1 \leq k\\
								P(s|t)_B \to (-1)^{(n+t)(k+s+1)} P(s-1|t-1)_A & s-1 \leq k\\ 
								\end{cases}
\end{align}
\item Again, they define nonzero elements in $\Ext^{(n+m)-(k+l)-4}(M(n|m),M(k|l))$.
\end{enumerate}
\end{Theorem}
\begin{proof}
Writing down the product
$$\Id^{(n|m)}_{(k+2|k+1)} \cdot K^{(k+2|k+1)}_{(k|k-1)}\cdot \Id^{(k|k-1)}_{(k|l)}$$one obtains the signs.

Assume that $K$ is nullhomotopic, i.e. there is a map
$$H: P_\bullet(n|m) \to P_\bullet(k|l)\langle (n+m)-(k+l)-6 \rangle[(n+m)-(k+l)-5].$$ This cannot exist by Corollary \ref{Cor:posschmaps} case \ref{Cposs1}, since $-3 \nleq -5$.
\end{proof}

\begin{Remark}
Note that the maps from Theorems \ref{Th:elid}-\ref{Th:elk} in particular contain the maps constructed already in Theorems \ref{Th:ElemId}-\ref{Th:ElemK}.
\end{Remark}

The last family of elements we define is special in some sense. First, this class only appears as products of at least two non-idempotent elements. Second, we determine maps which seem to belong to this class but later on we will show that they are nullhomotopic. For a benefit of notation, we do not exclude these maps in the definition, but we have to be careful later on.
\begin{Theorem}\label{Th:elj}
For all $m,n,k,l \in \{0, \ldots N\}$ with $m <n$, $l <k$ and $k<n, l<m$ consider the product
\begin{equation}F^{(n|m)}_{(k+1|l)} \cdot \widetilde{F}^{(k+1|l)}_{(k|l)}. \end{equation}
The following holds:
\begin{enumerate}
\item These maps are exactly given by
\begin{align}\nonumber
J^{(n|m)}_{(k|l)}:& \quad P_\bullet(n|m) \to P_\bullet(k|l)\langle (n+m)-(k+l)-4 \rangle [(n+m)-(k+l)-2]\\
& \begin{cases}
								P(s|t)_A \to (-1)^{(n+k+1)(l+t+1)} P(s-1|t-1)_A  & s-1 \leq k, t-1 \leq l\\
								P(s|t)_B \to (-1)^{(n+k+1)(l+s+1)} P(s-1|t-1)_B & s-1 \leq l\\ 
								\end{cases}
\end{align}
\item For \textbf{$m<k$} they define nonzero elements in $\Ext^{(n+m)-(k+l)-4}(M(n|m),M(k|l))$.
\end{enumerate}
\end{Theorem}

\begin{proof}
The maps are well-defined since we know that for $\la \neq \mu$ we have $\hom(P(\la),P(\mu))_i \leq 1$ and by this there is (up to scalar) only one possible choice for the degree two maps above.

For the first assertion we check all compositions occuring in the product of the maps $F$ and $\widetilde{F}$, i.e. analyse all compositions of assignments in the definitions of $F$ and $\widetilde{F}$, respectively.

First, we have
\begin{displaymath} \xymatrix{ P(s|t)_A \ar[rr]^{(-1)^{(n+k+1)(l+t+1)}} &&P(s|t-1)_A  \ar[rr] &&P(s-1|t-1)_A}
\end{displaymath}
where the second assignment is the only possible one, since $P(s|t-1) \neq P(s|s-1)$.

Secondly, we compute
\begin{displaymath} \xymatrix{ P(s+2|s)_A \ar[rr] &&P(s+1|s)_A  \ar[rr] &&P(s|s-2)_A}
\end{displaymath}
which equals zero by Lemma \ref{Le:possmaps}.

Next, we see
\begin{displaymath} \xymatrix{ P(s+1|s)_A \ar[rr] &&P(s-2|s)_B  \ar[rr] &&P(s|s-3)_B}
\end{displaymath}
equals zero, too. 

Turning to the B-terms, we have
\begin{displaymath} \xymatrix{ P(s|t)_B \ar[rr]^{(-1)^{(n+k+1)(l+s+1)}} &&P(s-1|t)_B  \ar[rr] &&P(s-1|t-1)_B}
\end{displaymath}
as asserted.

Last we consider
\begin{displaymath} \xymatrix{ P(s|s-2)_B \ar[rr]^{(-1)^{(n+k+1)(l+s+1)}} &&P(s|s-1)_A  \ar[rr] &&P(s-1|s-3)_B}
\end{displaymath}
which is contained in the second case.

Now we turn to the question whether $J$ is nullhomotopic or not. In the theorem it is only claimed that it is not for $m<k$. Therefore, assume $m<k$ and assume we have a homotopy 
$$H: P_\bullet(n|m) \to P_\bullet(k|l)\langle (n+m)-(k+l)-4 \rangle[(n+m)-(k+l)-3]$$
which leads us to the situation reflected in the diagram:
\begin{displaymath}
\xymatrix{\cdots \ar[r] &P(k+1|l+1)_A \ar[d]^f \ar[r] \ar[ld]^x &Y\ar[ld]^y\\
X \ar[r] & P(k|l)_A \langle -2 \rangle \ar[r] & \cdots}
\end{displaymath}
\begin{enumerate}
\item $x$ has to be the zero map since for all projectives in $X$ (note that their heads sit in degree 1 higher than the ones of $P(k|l)$ i.e. we look for maps from $P(k+1|l+1) \to P(s|t)\langle -1 \rangle$) we have $s+t\leq k+l-1$ and by Lemma \ref{Le:possmaps} a map could just exist if $k+1+l+1 \leq s+t+1$.
\item If the map $y$ is nonzero, it is a map $Y\langle 1 \rangle \to P(k|l) \langle -1 \rangle$. One has to care about the matter that in $Y$ the degree of the heads of the projectives is one higher than the one in $P(k+1|l+1)$, so we look for morphisms $P(s|t) \to P(k|l) \langle -1 \rangle$. By Lemma \ref{Le:possmaps} this may only exist if $s+t \leq k+l+1$. All A-terms in $Y$ have to fulfil $s+t=k+l+3$, therefore, the map must start in a B-term. Since we assume $f=yd$ $y$ must start in a B-term such that the differential from $P(k+1|l+1)$ maps to this term and therefore the only possibilities are $P(k+1|l)_B$ or $P(k|l+1)_B$. But we assumed $k>m$ and these B-terms cannot occur (remember that for B-terms $P(s|t)_B$ we have $s \leq m$). By this $y$ has to be zero.
\end{enumerate}
Summing up the above computations, we have shown that there is no such homotopy $H$ and therefore the theorem is proved.
\end{proof}
\begin{Lemma}
For $m<k$ the maps $F^{(n|m)}_{(k|l)}$ and $\pm \widetilde{F}^{(n|m)}_{(k|l)}$ are not homotopic.
\end{Lemma}
\begin{proof}
Assume we have a homotopy between these two maps,
$$H: P_\bullet(n|m) \to P_\bullet(k|l)\langle (n+m)-(k+l) \rangle[(n+m)-(k+l)].$$ 
Similar to the above computations, we notice that the head of a projective has to be mapped to the head of another projective (so $H$ has to map $P(s|t)$ to $P(s|t)$) and we are in the situation
\begin{displaymath}
\xymatrix{\cdots \ar[r] &{\begin{matrix} P(k+1|l)_A\\P(k|l+1)_A \end{matrix}}\ar[d]^{f-f'} \ar[r] \ar[ld]^x &Y\ar[ld]^y\\
X \ar[r] & P(k|l)_A \langle -1 \rangle \ar[r] & \cdots}
\end{displaymath}
with $f: P(k+1|l) \to P(k|l)$ and $f': P(k|l+1) \to P(k|l)$.
We have to find maps $x$ and $y$ such that $yd+dx=f \pm f'$.\\
\begin{enumerate}
\item $x$ has to be the zero map (same reason as in the proof of Theorem \ref{Th:elf}).
\item If the map $y$ is nonzero, it has to send the $P(k|l)$ to $P(k|l)_A$. But if $P(k|l)$ occured as a summand of $Y$, by Lemma \ref{Le:possmaps} it would be a B-term. However, since $m<k$, $P(k|l)$ cannot occur as a B-term in $P_\bullet(m|n)$.
\end{enumerate}
Hence, there is no such homotopy and the lemma follows.
\end{proof}
\subsubsection{Comparison with the expected dimensions of the $\Ext$-spaces}
Comparing our results to the dimension list we computed in Section \ref{sec:dimexp} one can deduce two important results.
\begin{Cor}
\begin{enumerate}
\item The objects $\Id^{(n|m)}_{(k|l)}$, $F^{(n|m)}_{(k|l)}$, $\widetilde{F}^{(n|m)}_{(k|l)}$, $G^{(n|m)}_{(k|l)}$, $K^{(n|m)}_{(k|l)}$ and $J^{(n|m)}_{(k|l)}$, with $n,m,k,l$ fulfilling the assumptions from the above theorems, form a basis of $\Ext(\bigoplus M(\la),\bigoplus M(\la))$.
\item For $k\leq m$ the element $J^{(n|m)}_{(k|l)}$ is nullhomotopic and $F^{(n|m)}_{(k|l)}$ is homotopic to a multiple of $\widetilde{F}^{(n|m)}_{(k|l)}$.
\end{enumerate}
\end{Cor}
\begin{proof}
The first assertion is checked by comparing the list in Section \ref{sec:dimexp} with the elements defined above (and using the fact that they are not nullhomotopic in the occurring cases).

To see that $J^{(n|m)}_{(k|l)}\simeq 0$ observe that $E^{(n+m)-(k+l)-4}((m|n),(k|l))$ is zero for $k \leq m$ by the list in Section \ref{sec:dimexp}.
Similarly, $E^{(n+m)-(k+l)-1}((m|n),(k|l))$ is onedimensional for $k \leq m$. Since neither $F_{(k|l)}^{(n|m)} \simeq 0$ nor $F^{(n|m)}_{(k|l)}\simeq 0$, a linear combination of them has to be homotopic to zero. By this the assertion of the corollary holds.
\end{proof}
\begin{Remark}
In Section \ref{sec:homotopies} we will determine these homotopies explicitly.
\end{Remark}

\subsection{The algebra structure}\label{sec:multipl}
Now we compute the algebra structure. Therefore, we determine all possible products of the elements from above. For later use, we also compute those which are nullhomotopic. Therefore, we have to define two more families of elements in $\hom(P_\bullet, P_\bullet)$ vanishing in the $\Ext$-algebra.
\begin{Def}
For $b<m-1$, $b+2<a$, $m<n$ define the map
\begin{align}\nonumber
A^{(n|m)}_{(a|b)}:&\quad P_\bullet(n|m) \to P_\bullet(a|b)\langle (n+m)-(a+b)-4 \rangle [(n+m)-(a+b)-2]\\
& \begin{cases}
								P(s+1|s)_A \to (-1)^{(n+a)(b+s)} P(s+1|s-2)_A  \\
								P(s+2|s)_A \to (-1)^{(n+a)(b+s)+a+s} P(s+1|s-1)_A  \\
								P(s+1|s-2)_B \to (-1)^{(n+a)(b+s+1)} P(s|s-1)_A  \\
								P(s|s-2)_B \to (-1)^{(n+a)(b+s)+a+s+1} P(s-1|s-3)_B  \\
								+ (-1)^{(n+a)(b+s)} P(s|s-2)_A \\ 
								\end{cases}								
\end{align}
with the map from $P(s|s-2)_B$ to $P(s|s-2)_A$ being chosen as the composition 
$$P(s|s-2) \to P(s|s-1) \to P(s|s-2)$$
and the map 
\begin{align}\nonumber
B^{(n|m)}_{(a|b)}:&\quad P_\bullet(n|m) \to P_\bullet(a|b)\langle (n+m)-(a+b)-6 \rangle [(n+m)-(a+b)-3]\\
& \begin{cases}
								P(s+1|s)_A \to (-1)^{(n+a+1)(b+s)} P(s|s-2)_A  \\
								P(s+1|s-1)_B \to (-1)^{(n+a+1)(b+s+1)} P(s|s-1)_A  \\
								\end{cases}
\end{align}
\end{Def}
\begin{Lemma}\label{Le:B}
 The maps $A^{(n|m)}_{(a|b)}$ and $B^{(n|m)}_{(a|b)}$ are nullhomotopic.
\end{Lemma}
\begin{proof}
The assertion follows immediately from the fact that the dimension of the corresponding $\Ext$-spaces is zero by the computations in Section \ref{sec:dimexp}.
\end{proof}

\begin{Theorem}\label{Th:Multtable}
In the algebra $\Ext(\bigoplus M(\la),\bigoplus M(\la))$ we get the products listed in Table \ref{tab:Multelem}.
\renewcommand\arraystretch{1.6}
\begin{sidewaystable}
\caption{\label{tab:Multelem} Product $x \cdot y$ of the elements $x$ and $y$}
\begin{tabular}{|c!{\vrule width 1.3pt}c|c|c|c|c|c|}
		\toprule \rule[-0.6cm]{0cm}{1.2cm}
		$x\backslash y$& $\Id^{(k|l)}_{(a|b)}$ & $F^{(k|l)}_{(a|b)}$ & $\widetilde{F}^{(k|l)}_{(a|b)}$ & $G^{(k|l)}_{(a|b)}$ & $K^{(k|l)}_{(a|b)}$ & $J^{(k|l)}_{(a|b)}$ \\ \noalign{\hrule height 1.2pt} \rule[-1cm]{0cm}{2cm}
		$\Id^{(n|m)}_{(k|l)}$& $\begin{matrix}(-1)^{(n+k)(l+b)} \\ \Id^{(n|m)}_{(a|b)}\end{matrix}$ & $\begin{array}{c}(-1)^{(n+k)(l+b+1)}\\ F^{(n|m)}_{(a|b)}\end{array}$ & $\begin{matrix}(-1)^{(n+k)(l+b)}\\ \widetilde{F}^{(n|m)}_{(a|b)}\end{matrix}$ & $\begin{matrix}(-1)^{(n+k)(l+a+1)}\\ G^{(n|m)}_{(a|b)}\end{matrix}$ & $\begin{matrix}(-1)^{(n+k)(l+a+1)}\\ K^{(n|m)}_{(a|b)}\end{matrix}$ & $\begin{matrix}(-1)^{(n+k)(l+b+1)} \\J^{(n|m)}_{(a|b)}\stackrel{a \leq m}{\simeq} 0\end{matrix}$\\ \hline \rule[-1cm]{0cm}{2cm}
  $F^{(n|m)}_{(k|l)}$& $\begin{matrix}(-1)^{(n+k)(l+b)} \\F^{(n|m)}_{(a|b)}\end{matrix}$ & $\begin{matrix} (-1)^{(n+k)(l+b+1)}\\ A^{(n|m)}_{(a|b)} \simeq 0\end{matrix}$ &$\begin{matrix}(-1)^{(n+k)(b+l)} \\J^{(n|m)}_{(a|b)} \stackrel{a \leq m}{\simeq} 0\end{matrix}$ & $\begin{matrix}(-1)^{(n+k)(l+a)+a+k+1}\\ K^{(n|m)}_{(a|b)}\end{matrix}$ &$0$& $\begin{matrix} (-1)^{(n+k)(l+b+1)} \\B^{(n|m)}_{(a|b)}\simeq 0\end{matrix}$\\ \hline \rule[-1cm]{0cm}{2cm}
    $\widetilde{F}^{(n|m)}_{(k|l)}$& $\begin{matrix}(-1)^{(n+k+1)(l+b)}\\ \widetilde{F}^{(n|m)}_{(a|b)}\end{matrix}$ &$\begin{matrix}(-1)^{(n+k+1)(b+l+1)} \\J^{(n|m)}_{(a|b)} \stackrel{a \leq m}{\simeq} 0\end{matrix}$ 
    &$0$ &$\begin{matrix}(-1)^{(n+k+1)(l+a+1)}\\ K^{(n|m)}_{(a|b)}\end{matrix}$ &$0$&$0$\\ \hline \rule[-1cm]{0cm}{2cm}
      $G^{(n|m)}_{(k|l)}$& $\begin{matrix}(-1)^{(a+k)(b+n)} \\G^{(n|m)}_{(a|b)}\end{matrix}$&$\begin{matrix}(-1)^{(a+k)(b+n+1)}\\ K^{(n|m)}_{(a|b)}\end{matrix}$&$\begin{matrix}(-1)^{(a+k+1)(b+n)+a+n}\\ K^{(n|m)}_{(a|b)}\end{matrix}$&$0$&$0$&$0$ \\ \hline \rule[-1cm]{0cm}{2cm}
        $K^{(n|m)}_{(k|l)}$& $\begin{matrix}(-1)^{(a+k)(b+n+1)} \\K^{(n|m)}_{(a|b)}\end{matrix}$ &$0$&$0$&$0$&$0$&$0$\\ \hline \rule[-1cm]{0cm}{2cm}
          $J^{(n|m)}_{(k|l)}$& $\begin{array}{c}(-1)^{(n+k+1)(b+l)} \\J^{(n|m)}_{(a|b)} 
         \stackrel{a \leq m}{\simeq} 0\end{array}$ & $\begin{matrix} (-1)^{(n+k+1)(l+b+1)} \\B^{(n|m)}_{(a|b)}\simeq 0 \end{matrix}$ &$0$&$0$&$0$&$0$\\ \hline
\end{tabular}		

\end{sidewaystable}
\renewcommand\arraystretch{1}

\normalsize
Moreover, $$B^{(n|m)}_{(a|b)}=A^{(n|m)}_{(a+1|b)}\cdot \widetilde{F}^{(a+1|b)}_{(a|b)}.$$
\end{Theorem}
\begin{proof}  The proof is an easy direct calculation based on the previous theorems which can be found in Appendix \ref{Ap:Mult}.
\end{proof}

\subsection{The quiver of $\Ext(\bigoplus M(\la), \bigoplus M(\la))$}
Determining the quiver we obtain
\begin{Theorem}\label{mainalg2}
The algebra $\Ext(\bigoplus M(\la),\bigoplus M(\la))$ is given as the path algebra of the quiver
\\$n>m+2$:
\begin{displaymath}\xymatrix{
&\cdots \ar@[cyan]@/^/[d] \ar@[black]@/_/[d]& {  \cdots} \ar@[cyan]@/^/[d] \ar@[black]@/_/[d]& \cdots \ar@[cyan]@/^/[d] \ar@[black]@/_/[d] &
 \\
\cdots  \ar@[cyan]@/^/[r] \ar@[black]@/_/[r]& P_\bullet(n+1|m+1)\ar@[cyan]@/^/[d] \ar@[black]@/_/[d]    \ar@[cyan]@/^/[r] \ar@[black]@/_/[r]& P_\bullet(n|m+1)\ar@[cyan]@/^/[d] \ar@[black]@/_/[d]   
 \ar@[cyan]@/^/[r] \ar@[black]@/_/[r]&P_\bullet(n-1|m+1)   \ar@[cyan]@/^/[d] \ar@[black]@/_/[d]  \ar@/_/[r]  \ar@[cyan]@/^/[r]& \cdots  \\
\cdots  \ar@[cyan]@/^/[r] \ar@[black]@/_/[r]& P_\bullet(n+1|m)  \ar@[cyan]@/^/[d] \ar@[black]@/_/[d]     \ar@[cyan]@/^/[r] \ar@[black]@/_/[r]& P_\bullet(n|m)  \ar@[cyan]@/^/[d] \ar@[black]@/_/[d]     \ar@[cyan]@/^/[r] \ar@[black]@/_/[r]&P_\bullet(n-1|m)  \ar@[cyan]@/^/[d] \ar@[black]@/_/[d]     \ar@[cyan]@/^/[r] \ar@[black]@/_/[r]&\cdots  \\
\cdots  \ar@[cyan]@/^/[r] \ar@[black]@/_/[r]&P_\bullet(n+1|m-1)  \ar@[cyan]@/^/[d] \ar@[black]@/_/[d]     \ar@[cyan]@/^/[r] \ar@[black]@/_/[r]& P_\bullet(n|m-1)  \ar@[cyan]@/^/[d] \ar@[black]@/_/[d]     \ar@[cyan]@/^/[r] \ar@[black]@/_/[r]&P_\bullet(n-1|m-1)  \ar@[cyan]@/^/[d] \ar@[black]@/_/[d]     \ar@[cyan]@/^/[r] \ar@[black]@/_/[r]& \cdots \\
& \cdots  & \cdots  & \cdots   & }\end{displaymath}
and in the other cases:
\begin{displaymath}
\xymatrix{
\dots  \ar@[cyan]@/^/[r] \ar@[black]@/_/[r]& P_\bullet(m|m-1) \ar@[cyan]@/^/[d] \ar@[black]@/_/[d] \ar@[red]@/^3pc/ [rrdd] \ar@[green]@<1ex>@/^3pc/ [rrdd]& &\\
\cdots   \ar@[cyan]@/^/[r] \ar@[black]@/_/[r]&P_\bullet(m|m-2)  \ar@[cyan]@/^/[r] \ar@[black]@/_/[r] \ar@[cyan]@/^/[d] \ar@[black]@/_/[d]  & P_\bullet(m-1|m-2) \ar@[cyan]@/^/[d] \ar@[black]@/_/[d]  &\\
\cdots   \ar@[cyan]@/^/[r] \ar@[black]@/_/[r]&P_\bullet(m|m-3)   \ar@[cyan]@/^/[r] \ar@[black]@/_/[r] \ar@[cyan]@/^/[d] \ar@[black]@/_/[d]  & P_\bullet(m-1|m-3)    \ar@[cyan]@/^/[r] \ar@[black]@/_/[r]\ar@[cyan]@/^/[d] \ar@[black]@/_/[d]  & P_\bullet(m-2|m-3)\ar@[cyan]@/^/[d] \ar@[black]@/_/[d] \\
& \cdots & \dots & \cdots }\end{displaymath}

\color{black}
with relations given in Table \ref{tab:Multelem}.
\end{Theorem}
\subsection{Homotopies}\label{sec:homotopies}
In this section we explicitly determine the before-mentioned homotopies, i.e. find boundaries to elements becoming zero. Recall that by $\d_\hom$ we denote the differential on the $\hom$-algebra defined in Section \ref{sec:Exthom}.
\begin{Notation}
Given a nullhomotopic chain map $$C_{(a|b)}^{(n|m)} \in \hom^{k}(P_\bullet(n|m),P_\bullet(a|b)\langle j \rangle)$$ we denote by
$H(C)_{(a|b)}^{(n|m)}$ the homotopy, i.e.
$$H(C)_{(a|b)}^{(n|m)} \in \hom^{k-1}(P_\bullet(n|m),P_\bullet(a|b)\langle j \rangle)$$
and
$$C=d _\hom(H(C)_{(a|b)}^{(n|m)}).$$
\end{Notation}

\begin{Lemma}\label{Le:homotopyF}
For $m \geq k$ we have
\begin{align}
F^{(n|m)}_{(k|l)} - (-1)^{n+l}\widetilde{F}^{(n|m)}_{(k|l)} =\d_\hom(H(F-(-1)^{n+l}\widetilde{F})^{(n|m)}_{(k|l)})
\end{align}
with 
\begin{align*}
H(F-(-1)^{n+l}\widetilde{F})^{(n|m)}_{(k|l)}=\begin{cases} P(s|t)_A \to 0 \\P(s|t)_B \to (-1)^{(s+t)(k+s)+(n+k+1)(l+s+1)+n+m+1}P(s|t)_A \end{cases}
\end{align*}
\end{Lemma}
\begin{proof}
We first check the property for $H(F-(-1)^{k+l+1}\widetilde{F})^{(k+1|k)}_{(k|l)}$.
As in the previous proofs, we just have to check diagrams, which is done in Appendix \ref{Ap:Homot}.

Now one observes that
\begin{align*}
&\Id^{(n|m)}_{(k+1|k)} \cdot H(F-(-1)^{k+l+1}\widetilde{F})^{(k+1|k)}_{(k|l)}\\
=&(-1)^{(n+k+1)(k+l+1)+n+m+1} H(F-(-1)^{n+l}\widetilde{F})^{(n|m)}_{(k|l)}
\end{align*}
and
\begin{align*}
&\d_\hom(H(F-(-1)^{n+l}\widetilde{F})^{(n|m)}_{(k|l)}\\
=&(-1)^{(n+k+1)(k+l+1)+n+m+1} \d_\hom(\Id^{(n|m)}_{(k+1|k)} \cdot H(F-(-1)^{k+l+1}\widetilde{F})^{(k+1|k)}_{(k|l)})\\
=&(-1)^{(n+k+1)(k+l+1)+n+m+1} \d_\hom(\Id^{(n|m)}_{(k+1|k)}) \cdot H(F-(-1)^{k+l+1}\widetilde{F})^{(k+1|k)}_{(k|l)}\\
&+(-1)^{(n+k+1)(k+l+1)} \Id^{(n|m)}_{(k+1|k)} \cdot \d_\hom(H(F-(-1)^{k+l+1}\widetilde{F})^{(k+1|k)}_{(k|l)})\\
=&(-1)^{(n+k+1)(k+l+1)} \Id^{(n|m)}_{(k+1|k)} \cdot \d_\hom(H(F-(-1)^{k+l+1}\widetilde{F})^{(k+1|k)}_{(k|l)})\\
=&(-1)^{(n+k+1)(k+l+1)}\Id^{(n|m)}_{(k+1|k)} \cdot (F^{(k+1|k)}_{(k|l)} - (-1)^{k+l+1}\widetilde{F}^{(k+1|k)}_{(k|l)})\\
=&F^{(n|m)}_{(k|l)} - (-1)^{n+l}\widetilde{F}^{(n|m)}_{(k|l)}
\end{align*}
using the fact that $\d_\hom(\Id^{(n|m)}_{(k+1|k)})=0$. 
Therefore, the assertion follows.
\end{proof}
\begin{Lemma}\label{Le:homotopyJ}
For $k \leq m$, $m<n$, $l<k$ we get a homotopy such that\begin{align}
J^{(n|m)}_{(k|l)} =\d_\hom(H(J)^{(n|m)}_{(k|l)})
\end{align}
with 
\begin{align}
H(J)^{(n|m)}_{(k|l)}=\begin{cases} P(s|t)_A \to 0 &s>t+1\\ P(s+1|s)_A \to (-1)^{(n+k)(l+s+1)}P(s|s-2)_A \\P(s|t)_B \to (-1)^{(n+k)(l+t)+(s+t+1)(n+s)}P(s|t-1)_A\end{cases}
\end{align}
\end{Lemma}
\begin{proof}
The computations are done in Appendix \ref{Ap:Homot}.
\end{proof}

\begin{Lemma}\label{Le:homotopyA}
For the map $A^{(n|m)}_{(k|l)}$ defined above we have
\begin{align}
A^{(n|m)}_{(k|l)} =\d_\hom(H(A)^{(n|m)}_{(k|l)})
\end{align}
with 
\begin{align}
H(A)^{(n|m)}_{(k|l)}=\begin{cases} P(s|t)_A \to 0 &s>t+1\\ P(s|t)_B \to 0 &s>t+2\\P(s+1|s)_A \to (-1)^{(n+k)(l+s+1)+n+l+1}P(s|s-2)_A \\P(s+2|s)_B \to (-1)^{(n+k)(l+s)+k+l}P(s+1|s)_A\end{cases}
\end{align}
\end{Lemma}
\begin{proof}
The computations are done in Appendix \ref{Ap:Homot}.
\end{proof}
\begin{Lemma}\label{Le:homotopyB}
For the map $B^{(n|m)}_{(k|l)}$ defined above we have
\begin{align}
B^{(n|m)}_{(k|l)} =\d_\hom(H(B)^{(n|m)}_{(k|l)})
\end{align}
with 
\begin{align}
H(B)^{(n|m)}_{(k|l)}=\begin{cases} P(s|t)_A \to 0 &s>t+1\\ P(s|t)_B \to 0 \\P(s+1|s)_A \to (-1)^{(n+k+1)(l+s+1)+n+l+1}P(s-1|s-2)_A \end{cases}
\end{align}
\end{Lemma}
\begin{proof}
One checks that 
\begin{equation*}H(B)^{(n|m)}_{(k|l)}=H(A)^{(n|m)}_{(k+1|l)}\cdot \widetilde{F}^{(k+1|l)}_{(k|l)} \end{equation*}
This map is a homotopy for $B$ since
\begin{align*}
&\d_\hom(H(B)^{(n|m)}_{(k|l)})\\
=&\d_\hom(H(A)^{(n|m)}_{(k+1|l)}\cdot \widetilde{F}^{(k+1|l)}_{(k|l)} )\\
=&\d_\hom(H(A)^{(n|m)}_{(k+1|l)}) \cdot \widetilde{F}^{(k+1|k)}_{(k|l)}\\
&+(-1)^{n+m+k+l+1} H(A)^{(n|m)}_{(k+1|l)} \cdot \d_\hom(\widetilde{F}^{(k+1|k)}_{(k|l)})\\
=&A^{(n|m)}_{(k+1|l)} \cdot \widetilde{F}^{(k+1|k)}_{(k|l)}\\
=&B^{(n|m)}_{(k|l)}
\end{align*}
since $\d_\hom(\widetilde{F}^{(k+1|k)}_{(k|l)})=0$. Therefore, the lemma is proved.
\end{proof}

\chapter{Excursion on the Koszul-Duality}\label{ch:Koszul}
In the previous section we explicitly determined elements of the $\Ext$-algebra of Verma modules. Some of them have a very natural and simple interpretation when we apply Koszul duality. To explain this we first recall some results of the ordinary category $\O$ and their connection to $O^\p$.
\section{Grading on $\O$}
Recall the notation from Section \ref{sec:Cat}.  
In particular $P^\O(\la)$ denotes the indecomposable module which surjects onto the Verma module $M^\O(\la)$ of highest weight $\la$. Let $\la$ be maximal in its dot-ordering and $\O_\la$ the corresponding block. Then $P_\la^\O=\bigoplus_{\mu \in W\cdot \la} P^\O(\la)$
is a minimal projective generator for $\O_\la$. \\
In \cite{Beil1996} it is shown that $A(\la)=\End(P_\la)$ can be equipped with a positive $\Z$-grading	which turns it into a Koszul-algebra. For basics on Koszul-algebras we refer to (\cite{Beil1996}, \cite{Mazo2009}, \cite{Poli05}). The $\Z$-grading on $A(\la)$ entitles us to consider the corresponding category 
$A(\la)-\gmod$ of graded finite dimensional $A(\la)$-modules. Then we can consider a graded version $\O_\la^\Z$ of $\O_\la$ by considering $A_\la-\gmod$ with the forgetful functor $f:A_\la-\gmod \to A_\la-\mod$ which forgets the grading, cf. \cite{Stro03}. This construction is similar to the one of  the graded version of $\O^\p$ in Remark \ref{grading}.
\section{Koszul duality}
Fix $\p$ a parabolic subalgebra with Levi component $\fr{l}$ and take $\la$ such that the stabilizer under the dot-action of W equals $W_\fr{l}$. Let $\omega_0$ be the longest element in $W$.
Then the following Koszul-duality theorem holds:
\begin{Theorem}[{\cite[Theorem 3.11.1.]{Beil1996}}]\label{Th:koszul}
There exists a contravariant equivalence of triangulated categories over $\C$
$$Kd: \D^{b}(\O_\la^\Z)\longrightarrow \D^b(\O^{\p \Z})$$
such that
\begin{enumerate}
\item the graded modules $P^\O(x\cdot \la),\ M^\O(x\cdot \la)$	and $L^\O(x\cdot \la)$ are sent to the modules $P(x^{-1}\omega_0 \cdot \la), M(x^{-1}\omega_0 \cdot \la)$ and $L(x^{-1}\omega_0 \cdot \la)$, respectively
\item the internal grading shift becomes a diagonal homological-internal grading shift such that
$Kd(M\langle n\rangle)=(Kd(M))\langle n \rangle[n]$ for all $M$.
\end{enumerate}
\end{Theorem}
\begin{Remark}
Since $Kd$ is contravariant and triangulated we have in particular $Kd(M[n])=Kd(M)[-n]$.
\end{Remark}

\section{Proof of a graded version of Verma's Theorem}
As a corollary of the Koszul-duality we deduce a graded version of Verma's Theorem (c.f. \cite{Hump08}) in our special case where $\fr{g}=\fr{gl}_{m+n}$ and working in the principal block of $\O^\p$. (Note that the Koszulity for these blocks was established by elementary arguments in \cite[Section 5]{Brun22008}).
\begin{Cor}
Let $\la \in W \cdot \la_0$ and $W_\la=W_\fr{l}$. Given $0<i<m+n$ such that $\mu=s_i \cdot \la < \la$. Then there exists a graded
embedding $M^\O(\mu)\langle 1 \rangle \subset M^\O(\la)$.
\end{Cor}
\begin{proof}

Write $\la=x \cdot \la^-$ and $\mu=s_i x \cdot \la^-$ with $s_ix > x$ and $\la^-=w_0 \cdot \la_0$.
Using the construction of the projective resolution of a Verma module in section \ref{sec:linprojres} we get that 
$$P_\bullet(\la_0.x^{-1})=P_\bullet(\la_0.x^{-1}s_i)\langle 1 \rangle [1] \oplus P_\bullet(\la_0.(x^{-1})'')$$ the cone of a chain map.
Therefore, we obtain a surjective map 
$$P_\bullet(\la_0.x^{-1}) \to P_\bullet(\la_0.x^{-1}s_i)\langle 1 \rangle [1].$$ Reading this map as a map in the derived category, it is equal to the map 
$$M(\la_0.x^{-1}) \to M(\la_0.x^{-1}s_i)\langle 1 \rangle[1]$$ in $\D^b(\O^{\p \Z})$. Now we take the inverse of the Koszul-duality map $Kd$ and get a morphism 
$$Kd^{-1}(M(\la_0.x^{-1}s_i)\langle 1 \rangle[1]) \to Kd^{-1}(M(\la_0.x^{-1}))$$
which is by Theorem \ref{Th:koszul} a nonzero morphism
$$M^\O(s_ix \cdot \la^- )\langle 1 \rangle \to M^\O( x \cdot \la^-)$$
in $\D^b(\O_\la^\Z)$. Since the functor $Q:\O_\la^\Z \to \D^b(\O_\la^\Z)$ yields an equivalence on $H^0$-complexes (especially on those chains having just one component sitting in zero, cf. \cite[Proposition III.5.2.]{Gelf88}) we get a nonzero morphism 
$$M^\O(s_ix \cdot \la^- )\langle 1 \rangle \to M^\O( x \cdot \la^-)$$ in $\O_\la$.

Since any morphism between Verma modules in $\O$ is injective \cite[Theorem 4.2]{Hump08}, the theorem is proved.
\end{proof}

\begin{Remark}
Using Theorem \ref{Th:equcat} and generalizing Section \ref{sec:linprojres} similar to the proof of \cite[Theorem 5.3.]{Brun22008} for general integral weights, the above proof goes through for all integral weights $\la$.
\end{Remark}

\part{$A_\infty$-structure}
\chapter{$A_\infty$-algebras}\label{ch:Ainf}
$A_\infty$-algebras are a generalization of associative algebras. Some historical and topological motivation and basic material can be found in \cite{Kell2001}. A very detailed exposition with most of the proofs is provided in \cite{lefe2003}.
\section{Definitions}
\begin{Def}
Let $k$ be a field. An \emph{$A_\infty$-algebra over $k$} is a $\Z$-graded vector space
$$A=\bigoplus_{p \in \Z} A^p$$
endowed with a family of graded $k$-linear maps
$$m_n: A^{\otimes n} \to A, \ n \geq 1$$
of degree $2-n$ satisfying the following Stasheff identities:
\begin{align}\label{al:Ainfid}
 \sum (-1)^{r+st} m_{r+t+1}(\Id^{\otimes r} \otimes m_s \otimes \Id^{\otimes t}) =0
\end{align}
where for fixed $n$ the sum runs over all decompositions $n=r+s+t$ with $s\geq 1$, and $r,t \geq 0$.
\end{Def}
\begin{Example}
For $n=1$ the sum in \eqref{al:Ainfid} has only one summand and the identity becomes $m_1m_1=0$. Hence the degree $-1$ map $m_1$ is a differential on $A$.

The map $m_2$ is a degree zero map and plays the role of a not necessarily associative multiplication. Therefore, we write down the next relations:
$$m_1m_2=m_2(m_1 \otimes \Id + \Id \otimes m_1),$$ so $m_1$ is a graded derivation with respect to our multiplication $m_2$.

The third relation expresses the non-associativity of $m_2$ since equation \eqref{al:Ainfid} results in 
$$m_2(\Id \otimes m_2-m_2 \otimes \Id)=m_1m_3+m_3(m_1\otimes \Id \otimes \Id + \Id \otimes m_1 \otimes \Id + \Id \otimes \Id \otimes m_1)$$
for  $n=3$.
If $m_1$ or $m_3$ is zero, then $m_2$ is associative.

In general, one may interpret $m_3$ as a chain homotopy up to which $m_2$ is associative.
\end{Example}
Following Keller, we use the so-called Koszul sign convention for tensor products 
$$(f \otimes g)(x \otimes y)=(-1)^{|g||x|}f(x) \otimes g(y)$$
for $x$ and $y$ homogeneous elements. Note that only the parity of $x$ becomes important.
\begin{Def}
Let $A$ and $B$ be two $A_\infty$-algebras. A \emph{morphism of $A_{\infty}$-algebras} $f:A \to B$ is a family 
$$f_n: A^{\otimes n} \to B$$
of graded $k$-linear maps of degree $1-n$ such that
$$\sum{(-1)^{r+st}f_{r+t+1}(\Id^{\otimes r} \otimes m_s \otimes \Id^{\otimes t})}= \sum{(-1)^w m_q(f_{i_1} \otimes \dots \otimes f_{i_q})}$$
for all $n \geq 1$ with the first sum running over all decompositions $n=r+s+t$ and the second sum running over all $1 \leq 
q \leq n$ and all decompositions $n=i_1+ \dots +i_q$ with all $i_s\geq 1$ and the sign on the right-hand side is given by
$$w=(q-1)(i_1-1)+(q-2)(i_2-1)+ \dots +2(i_{q-2}-1)+(i_{q-1}-1).$$
\end{Def}
A morphism $f$ is a \emph{quasi-isomorphism} if $f_1$ is a quasi-isomorphism. It is \emph{strict} if $f_i=0$ for all $i \neq 1$.
\section{Minimal models}
Our ultimative goal in the whole chapter is to define an $A_\infty$-structure on the $\Ext$-algebra $\Ext (\oplus M(\la), \oplus M(\la))$ from Section \ref{sec:Exthom}. The first step hereby is to introduce an $A_\infty$-structure on the cohomology of an $A_\infty$-algebra (the so-called minimal model). Then we explain how the Ext-algebra $E=\Ext (\oplus M(\la), \oplus M(\la))$ can be viewed as the cohomology of an $A_\infty$-algebra, namely the $\hom$-algebra introduced in Section \ref{sec:Exthom}. Later we shortly review the advantages of these extra information.
\subsection{The existence and construction of a minimal model}\label{sec:minmod}
\begin{Theorem}[\cite{Kade79}]
Let $A$ be an $A_\infty$-algebra and $H^*(A)$ its cohomology. Then there is an $A_\infty$-structure on $H^*(A)$ such that $m_1=0$ and $m_2$ is induced by the multiplication on $A$, and there is a quasi-isomorphism of $A_\infty$-algebras $H^*(A) \to A$ lifting the identity of $H^*(A)$.

Moreover, this structure is unique up to isomorphism of $A_\infty$-algebras.
\end{Theorem}
The proof goes by construction and over the time several different approaches have been worked out. We follow here Merkulov's construction \cite{Merk99}. In our situation we can restrict ourself to the situation of a differential graded algebra. We do not need the full generality of Merkulov's setup.
\begin{Prop}\label{Pr:lambdan}
Take $(A,d)$ a differential graded algebra with grading shift $[\quad]$. Let $B \subset A$ be a subvectorspace of $A$ and $\Pi: A \to B$ a projection commuting with $d$. Assume that we are given a homotopy $Q: A \to A[-1]$ such that 
\begin{equation} \label{condQ} 1-\Pi=dQ+Qd. \end{equation} 

Define $\la_n: A^{\otimes n} \to A$ for $n \geq 2$ by 
$$\la_2(a_1,a_2):=a_1 \cdot a_2$$
and recursively, (setting formally $Q\la_1=-\Id$) for $n \geq 3$
\begin{equation}\label{eq:lambda}\begin{split}
&\la_n(a_1,\ldots,a_n)\\
&= - \sum_{\substack{k+l=n\\k,l \geq 1}}{(-1)^{k+(l-1)(|a_1|+\dots+|a_k|)}Q(\la_k(a_1, \ldots, a_k))  \cdot Q(\la_l(a_{k+1}, \ldots , a_n))}
\end{split}.\end{equation}
Then the maps $m_1=d$ and $m_n=\Pi(\la_n)$ define an $A_\infty$-structure for a minimal model on $B$.
\end{Prop}
\begin{proof}
The proof is precisely Merkulov's, except that he only works with a $\Z_2$-grading. 
As the signs only depend on the parity we do not need to worry about them and can use \cite[chapter 6.4]{Kont2001}.
There, Kontsevich and Soibelman show that (up to the above mentioned signs) Merkulov's approach also works in the graded algebra setting. The Proposition follows. 
\end{proof}
\begin{Remark}
Note that as in the definition, the maps $m_n$ are of degree $2-n$ and therefore the $\la_n$ are of degree $2-n$, too.
\end{Remark}
As a consequence we have the following vanishing results:
\begin{Cor}\label{Cor:disaplambda}
\begin{enumerate}
\item If \begin{equation*}Q(a_1 \cdot a_2)=Q(\la_2(a_1,a_2))=0 \ \forall \ a_1, a_2 \in A\end{equation*}
then $\la_n=0$ and $m_n=0 \ \forall n \geq 3.$
\item If \begin{equation*} Q(\la_3(a_1,a_2, a_3))=0 \text{ and } Q(a_1 \cdot a_2)\cdot Q(a_3 \cdot a_4)=0\ \forall \ a_1, a_2, a_3, a_4 \in A\end{equation*}
then $\la_n=0$ and $m_n=0 \ \forall n \geq 4.$
\end{enumerate}
\end{Cor}
\begin{proof}
\begin{enumerate}
\item Writing down the explicit formula for $\la_3$ we obtain \begin{equation*}
\la_3(a_1,a_2,a_3)=(-1)^{\deg a_1+1}a_1 \cdot Q(a_2\cdot a_3)+Q(a_1 \cdot a_2)\cdot a_3.
\end{equation*} Since both summands are zero by assumption, $\la_3$ equals zero, too. The same holds for the higher $\la_n$, since all summands in \eqref{eq:lambda} vanish.
\item This part follows by exactly the same argument.
\end{enumerate}
\end{proof}
\subsection{The choice of the map $Q$}\label{sec:Q}
Choosing $Q$ in a clever way simplifies computations. Note that since the minimal model only is unique up to quasi-isomorphism, our result will depend on this choice. Our choices and approach resemble the one in \cite[chapter 2]{Lu2009}.

To define $Q$, we first divide the degree $n$ part $A^n$ of $A$ into three subspaces, for this, denote by $Z^n$ the cocycles of $A$ and by $B^n$ the coboundaries. As we work over a field, we can find subspaces $H^n$ and $L^n$ such that
$Z^n= B^n \oplus H^n$ and 
\begin{equation} \label{Identif} A^n=B^n \oplus H^n \oplus L^n.\end{equation}

We identify the $n$th cohomology group $H^n(A)$ via \eqref{Identif} with $H^n$. We want to apply Proposition \ref{Pr:lambdan} with the choice of a subspace $B=H^*(A)$ and the projection $\Pi$ being the projection on the direct summand $H^*$. 

We choose the map $Q$ as follows:
\begin{enumerate}
\item When restricted to $Z^n$ by equation \eqref{condQ} and the condition that $d|_{Z^n}$ equals to zero, the map $Q$ has to satisfy the relation
$$1-\Pi=dQ.$$
In particular, $dQ|_{H}$ has to be zero. We choose $Q|_H=0$.
\item On $B^n$ the map $\Pi$ is zero, and therefore the map $Q|_B$ has to satisfy $1=dQ$, i.e. $Q$ has to be a preimage of $d$. We want to choose this preimage as small as possible i.e. with no non-trivial terms from $Z^n$ (they would anyway be annihilated by $d$). Since $d$ is injective on $L$, we can choose $Q|_B=(d|_L)^{-1}$. 
\item  We briefly outline how to determine $Q$ restricted to $L$ (although it won't play any role in our computations later on). From \eqref{condQ} we get the restriction $$1=Qd+dQ.$$
As $d(a) \in B \text{ for all } a \in A$ we see that $Qd|_{L}=(d|_L)^{-1}d|_L=1$, so we can define $Q|_{L}=0$.
\end{enumerate}
 
In practice, it might be difficult or a lot of work to write down the direct sum decomposition from \eqref{Identif} on elements. One has to define a choice for the large rest $L$. If one wants to show that $Q(a)=0$ for $a \in A^n$ it can be helpful to show that this element cannot contain a boundary. 

\subsection{Advantages of the $A_\infty$-structure}
An $A_\infty$-algebra carries more information than an ordinary algebra. Keller points this out in \cite[Problem 2]{Kell2001}. For instance it can be seen as follows:
 Let $B$ be an associative $k$-algebra and $M_1, \ldots, M_n$ a collection $M$ of $B$-modules.
Let $\filt(M)$ be the full subcategory of $B-\mod$  of modules having subquotients from the same isomorphism classes as the $M_i$'s. In the literature, the objects in $\filt(M)$ are often called {\it modules admitting an $M$-filtration} and the category $\filt(M)$ is {\it the category of M-filtered modules}. The question is:\\

Is $\filt(M)$ determined by $\Ext(\bigoplus M_i, \bigoplus M_i)$?\\
 
There are some well-known classic results answering the question. Given some tilting restrictions on the modules $M_i$ one can use classical tilting theory. If $A$ is a Koszul-algebra and the $M_i$ are all simple modules concentrated in degree zero, the Koszul dual of $A$ is the $\Ext$-algebra of the direct sum of these modules. A proof is given in \cite[Prop. 17]{Mazo2009}. In this case we can reconstruct the algebra $A$ from $\Ext(\bigoplus M_i, \bigoplus M_i)$.

In general, this result does \textbf{not} hold. In \cite[Section 2.1]{Kell2002} Keller gives an example where for different algebras $B$ one gets the same algebra $\Ext(M,M)$ but different subcategories $\filt(M)$.

Therefore, for solving the problem it is not enough to regard $\Ext(\bigoplus M_i, \bigoplus M_i)$ only as an algebra. In \cite[section 5]{Kell2001} Keller points out that we can reconstruct $\filt(M)$ if we also use the $A_\infty$-structure. 

\chapter{The $A_\infty$-structure on $\Ext(\bigoplus M, \bigoplus M)$}\label{ch:expAinf}
In this section we return to the space $\Ext(\bigoplus M(\la), \bigoplus M(\la))$ which we already determined as an algebra in Part I of the thesis. The construction of a minimal model from Section \ref{sec:minmod} applies to our situation if we choose $A=\hom(P_\bullet, P_\bullet)$ and $E=\Ext(\bigoplus M(\la), \bigoplus M(\la))=H^*(A)$. 

Recall that $M(\la)$ are the parabolic Verma modules in the principal block of $\O^\p$ with $\fr{g}=\fr{gl}_{n+m}$ and $\p$ the parabolic subalgebra belonging to the Levi component $\fr{gl}_m \oplus \fr{gl}_n$. Moreover, these modules can be viewed as the cell modules in $K_m^n-\mod$ via the equivalence in Corollary \ref{Cor:equcat}.

\section{General results}
In the following we give an upper bound for the $l$ with $m_l \neq 0$. 

Already in the case $n=2$ we will show that not all $m_l$ for $l>2$ vanish and therefore our specific model provides interesting examples of $A_\infty$-algebras with higher multiplications.

We start by stating the following Lemma generalizing the fact that the multiplication of two morphisms only exists if they lie in appropriate $\hom$-spaces:
\begin{Lemma}\label{Le:Laainf}
Let $a_i$, $1 \leq i \leq l$ be homogeneous elements of degree $k_i$ in \\$\Ext(\oplus M(\la), \oplus M(\la))$ of the form
$$a_i \in \Ext^{k_i}(M(\mu_i),M(\nu_i)) \ 1 \leq i \leq l.$$
Then we have \begin{equation*}
\la_l(a_1,...,a_l)=0 \text{ unless } \nu_i =\mu_{i+1} \ \forall \ 1 \leq i \leq l-1.\end{equation*}
 If $\la_l(a_1,...,a_l) \neq 0$, i.e. the last condition holds, we have
\begin{equation*}\la_l(a_1,...,a_l) \in \hom^{\Sigma k_i+2-l}(P_\bullet(\mu_1),P_\bullet(\nu_l)).\end{equation*}
\end{Lemma}
\begin{proof}
The proof goes by induction on $l$. For all neighboured pairs $i, i+1$ we have 
$$\la_2(a_i, a_{i+1})=a_i \cdot a_{i+1}=0 \text{ unless } \mu_{i+1} = \nu_i.$$
Especially 
\begin{equation} Q(\la_2(a_i, a_{i+1}))=0 \text{ unless }\mu_{i+1} = \nu_i.\label{eq:Q1}\end{equation}

By the construction the map $Q$ yields a preimage for the differential, therefore
\begin{equation}\label{eq:Q2}Q(\la_2(a_i,a_{i+1})) \in \hom^{k_i+k_{i+1}-1}(P_\bullet(\mu_i),P_\bullet(\nu_i)). \end{equation} This yields a second assumption used in the induction step below.

Now for the induction assume that $\forall s <l$ we know:
\begin{enumerate}
\item $Q(\la_s(a_i,...,a_{i+s-1}))=0$ unless $\nu_t = \mu_{t+1} \ \forall \ i \leq t \leq i+s-2$ 
\item $Q(\la_s(a_i,...,a_{i+s-1})) \in \hom^{\Sigma _{r=i}^{i+s-1} k_i+2-s-1}(P_\bullet(\mu_i),P_\bullet(\nu_i))$
\end{enumerate}
For $n=2$ the conditions are fulfilled by \eqref{eq:Q1} and \eqref{eq:Q2}.

We compute
$$\la_l= - \sum_{\substack{k+s=l\\k,l \geq 1}}{(-1)^{k+(s-1)(|a_1|+\dots+|a_k|)}Q(\la_k(a_1, \ldots, a_k))  \cdot Q(\la_s(a_{k+1}, \ldots , a_l))}.$$
Therefore, if $\la_l \neq 0$, at least one summand has to be nonzero. Assume this is the summand belonging to $k$ and $s=l-k$, i.e.
$$Q(\la_k(a_1, \ldots, a_k)) \cdot Q(\la_s(a_{k+1}, \ldots , a_l)) \neq 0.$$
From the assumption we know that 
$$\nu_t=\mu_{t+1} \ \forall \ 1 \leq t \leq k-1 \text{ and } k+1 \leq t \leq l-1$$
and the maps being in $$\hom^{\Sigma _{r=1}^{k} k_i+2-k-1}(P_\bullet(\mu_1),P_\bullet(\nu_k))$$
and
$$\hom^{\Sigma _{r=k+1}^{l} k_i+2-s-1}(P_\bullet(\mu_{k+1}),P_\bullet(\nu_{l})),$$ respectively.
So the composition is zero unless $\nu_{k}=\mu_{k+1}$ and 
$$\la_l(a_1,...,a_l) \in \hom^{\Sigma k_i+2-l}(P_\bullet(\mu_1),P_\bullet(\nu_l)).$$
We also obtain $Q(\la_l(a_1,...,a_{l})) \in \hom^{\Sigma _{r=1}^{l} k_i+2-l-1}(P_\bullet(\mu_i),P_\bullet(\nu_i))$ as we assumed for the induction and so we have proved the Lemma.
\end{proof}
\begin{Theorem}[General Vanishing Theorem]\label{Th:genvanish}
Taking the above construction for a minimal model on $E=\Ext(\oplus M(\la),\oplus M(\la))$ with $M(\la) \in K_m^n-\mod$, the $A_\infty$-structure on $E$ satisfies
\begin{align*} m_l=0 \text{ for all } l >n^2+2.
\end{align*}
\end{Theorem}
\begin{proof}We prove that the degree $2-l$ map $\la_l$ becomes zero for $l >n^2+2$.

Since $\la_l$ is linear, we show the assertion on nonzero homogeneous elements and therefore by Lemma \ref{Le:Laainf} we can take 
$$a_i \in \Ext^{k_i}(M(\mu_i),M(\mu_{i+1})) \ 1 \leq i \leq l.$$ 
By Lemma \ref{Le:Extrest} there are $d_i \geq 0$ such that
\begin{align*} k_i=l(\mu_i)-l(\mu_{i+1})-d_i \end{align*} and therefore 
\begin{align} \label{eq:ki}\sum_{i=1}^l{k_i}=l(\mu_1)-l(\mu_{l+1})-\sum_{i=1}^l d_i.\end{align}
From Lemma \ref{Le:Laainf} we know that
$$\la_l(a_1,...,a_l) \in \hom^{\Sigma k_i+2-l}(P_\bullet(\mu_1),P_\bullet(\nu_l)).$$
Assume $\la_l \neq 0$, so, by Lemma \ref{Le:mapprojres} about the morphisms between our chosen projective resolutions, we know that 
\begin{equation} \label{eq:k2}l(\mu_1) \leq l(\mu_{l+1})+n^2 +\sum k_i+2-l.\end{equation}
Combining \ref{eq:ki} and \ref{eq:k2}, we have
\begin{align*}
&l(\mu_1) \leq l(\mu_{l+1})+l(\mu_1)-l(\mu_{l+1})-\sum_{i=1}^l {d_i}+2-l +n^2,\\
\intertext{which is equivalent to}
 &\sum_{i=1}^l {d_i} \leq n^2+2 -l.
\end{align*}
However, since $\sum_{i=1}^l {d_i} \geq 0$, we obtain  $0\leq n^2+2-l$ or equivalently $l\leq n^2+2$. This provides the asserted upper bound. 
\end{proof}

\section{Explicit computation of the structure for $K_N^{1}$ and $K_{N-1}^2$}
In the previous section we established general vanishing results for the higher multiplications. In this section we study small examples in detail. Again we work in the two cases for $n=1$ and $n=2$ using results already obtained in Chapter \ref{ch:spec}.
\subsection{$\Ext(\bigoplus M(\la), \bigoplus M(\la)) $ for $M(\la) \in K_N^1-\mod$}\label{sec:Exp1}
The first result in this situation is the following:
\begin{Theorem}[1st vanishing Theorem] \label{Th:1stvanish}
For $\p$ the parabolic subalgebra belonging to $\fr{l}=\fr{gl}_1\oplus \fr{gl}_N$ the algebra
$\Ext(\oplus M(\la), \oplus M(\la))$ is formal, i.e. we have a minimal model such that $m_n =0$ for all $n \geq 3$.
\end{Theorem}
\begin{proof}
Recall from Remark \ref{Re:nonhom} that in this specific case all multiplications in the algebra $\Ext(\oplus M(\la), \oplus M(\la))$ are already obtained by the multiplications of the elements in $\hom(P_\bullet, P_\bullet)$. Therefore, for all elements $a_1, a_2 \in \Ext(\oplus M(\la), \oplus M(\la))=H^*(\hom(P_\bullet, P_\bullet))$ identified with the subspace $H^*$ via the decomposition from \eqref{Identif}, the product $a_1 \cdot a_2$ also lies in the subspace $H^*$ and has no boundary component in $B^*$. Since we have chosen $Q|_H=0$, we obtain
$$Q(a_1 \cdot a_2)=0.$$
Therefore, we apply Corollary \ref{Cor:disaplambda} and get
$$m_n = 0 \ \forall \ n \geq 3$$
and the theorem is proved.
\end{proof}
\subsection{$\Ext(\bigoplus M(\la), \bigoplus M(\la)) $ for $M(\la) \in K_{N-1}^2-\mod$}\label{sec:Exp2}
The case of $n=2$ and $m=N-1$, which we deal with in this section, turns out to be more interesting than the case studied in Section \ref{sec:Exp1}, since we have higher multiplications not disappearing. In contrast to the previous example this phenomenon is possible, since from Section \ref{sec:homotopies} we know that some multiplications in $\hom(P_\bullet, P_\bullet)$ are only homotopic to their product in the $\Ext$-algebra. 
Therefore, those elements also consist of a boundary part lying in the direct summand $B$ in the decomposition from \eqref{Identif}. As explained in Section \ref{sec:Q} the construction requires a choice of a {\it minimal} preimage of the differential for the boundary part. It is obvious that the homotopies chosen in Section \ref{sec:homotopies} are preimages. Since neither they, nor their linear combinations
 are contained the kernel of $d_\hom$, they lie in the $L$-part constructed in Section \ref{sec:Q} and therefore they are minimal (in the above sense). Hence we have constructed $Q(\la_2)$ as the homotopies from Section \ref{sec:homotopies} and determine $Q(x \cdot y)$ using Table \ref{tab:Multelem}. This is done in Table \ref{tab:MultelemQ}.
\begin{sidewaystable}
\caption{\label{tab:MultelemQ} $Q(x \cdot y)$ with $x$ and $y \in \Ext(\oplus M(\la), \oplus M(\la))$}
\begin{tabular}{|c!{\vrule width 1.3pt}c|c|c|c|c|c|}
		\toprule \rule[-0.6cm]{0cm}{1.2cm}
		$x\backslash y$& $\Id^{(k|l)}_{(a|b)}$ & $F^{(k|l)}_{(a|b)}$ & $\widetilde{F}^{(k|l)}_{(a|b)}$ & $G^{(k|l)}_{(a|b)}$ & $K^{(k|l)}_{(a|b)}$ & $J^{(k|l)}_{(a|b)}$ \\ \noalign{\hrule height 1.2pt} \rule[-1cm]{0cm}{2cm}
		$\Id^{(n|m)}_{(k|l)}$& $0$ & $0$ & $\begin{matrix}a \leq m\\(-1)^{(n+k)(l+b)+n+b+1}\\ H(F-(-1)^{n+l}\widetilde{F})^{(n|m)}_{(a|b)}\end{matrix}$ & $0$ & $0$ & $\begin{matrix}a\leq m\\(-1)^{(n+k)(l+b+1)} \\H(J)^{(n|m)}_{(a|b)}\end{matrix}$\\ \hline \rule[-1cm]{0cm}{2cm}
  $F^{(n|m)}_{(k|l)}$& $0$ & $\begin{matrix} (-1)^{(n+k)(l+b+1)}\\ H(A)^{(n|m)}_{(a|b)}\end{matrix}$ &$\begin{matrix}a \leq m\\(-1)^{(n+k)(b+l)} \\H(J)^{(n|m)}_{(a|b)} \end{matrix}$ & $0$ &$0$& $\begin{matrix} (-1)^{(n+k)(l+b+1)} \\H(B)^{(n|m)}_{(a|b)}\end{matrix}$\\ \hline \rule[-1cm]{0cm}{2cm}
    $\widetilde{F}^{(n|m)}_{(k|l)}$& $\begin{matrix}a \leq m\\(-1)^{(n+k+1)(l+b)+n+b+1}\\ H(F-(-1)^{n+l}\widetilde{F})^{(n|m)}_{(a|b)}\end{matrix}$ &$\begin{matrix}a \leq m\\(-1)^{(n+k+1)(b+l+1)} \\H(J)^{(n|m)}_{(a|b)} \end{matrix}$ 
    &$0$ &$0$ &$0$&$0$\\ \hline \rule[-1cm]{0cm}{2cm}
      $G^{(n|m)}_{(k|l)}$& $0$&$0$&$0$&$0$&$0$&$0$ \\ \hline \rule[-1cm]{0cm}{2cm}
        $K^{(n|m)}_{(k|l)}$& $0$ &$0$&$0$&$0$&$0$&$0$\\ \hline \rule[-1cm]{0cm}{2cm}
          $J^{(n|m)}_{(k|l)}$& $\begin{array}{c}a\leq m \\(-1)^{(n+k+1)(b+l)} \\H(J)^{(n|m)}_{(a|b)} 
        \end{array}$ & $\begin{matrix} (-1)^{(n+k+1)(l+b+1)} \\H(B)^{(n|m)}_{(a|b)}\end{matrix}$ &$0$&$0$&$0$&$0$\\ \hline
\end{tabular}		

\end{sidewaystable}
\subsubsection{The procedure to determine the higher multiplications}\label{Proc}
We will now work in three steps:
\begin{enumerate}
\item Compute $m_3$.
\item Show that $Q(\la_3)=0$.
\item Show that $Q(\la_2)\cdot Q(\la_2)=0$.
\end{enumerate}
Summing up these steps we conclude that there are no higher multiplications.

\begin{Remark}\label{Re:corexp}
Recall from Corollary \ref{Cor:posschmaps} that the space $\hom^{s+k}(P_\bullet(m|n),\ P_\bullet(a|b)\langle s+j\rangle)_0$ is zero except when certain conditions (depending on $k$ and $j$ only) are satisfied. In particular for 
\begin{align*}
f_1 &\in \hom^{s_1+k_1}(P_\bullet(m|n), P_\bullet(a|b)\langle s_1+j_1\rangle)_0\\
f_2 &\in \hom^{s_2+k_2}(P_\bullet(a|b), P_\bullet(c|d)\langle s_2+j_2\rangle)_0\\
f'_1 &\in \hom^{s'_1+k_1}(P_\bullet(m|n), P_\bullet(a'|b')\langle s'_1+j_1\rangle)_0\\
f'_2 &\in \hom^{s'_2+k_2}(P_\bullet(a'|b'), P_\bullet(c|d)\langle s'_2+j_2\rangle)_0\\
\end{align*}
the products $f_1 \cdot f_2$ and $f'_1 \cdot f'_2$ are elements in 
$$\hom^{s+k}(P_\bullet(m|n), P_\bullet(a|b)\langle s+j\rangle)_0$$
with $s=m+n-(c+d)$, $k=k_1+k_2$, $j=j_1+j_2$. This simplifies the calculations, since if one of the two compositions has to vanish by the restrictions of the corollary, the other has to vanish, too. \end{Remark}
\subsubsection{The multiplication $m_3$}
After this outline of the arguments we turn to the first lemma:
\begin{Theorem}\label{Th:la3}
There are non-vanishing $m_3$ and we have $Q(\la_3)=0$.
\end{Theorem}
\begin{proof}
From Proposition \ref{Pr:lambdan} we know that
\begin{equation}\label{eq:la3}
\la_3(a_1,a_2,a_3)=(-1)^{\deg a_1+1}a_1 \cdot Q(a_2\cdot a_3)+Q(a_1 \cdot a_2)\cdot a_3
\end{equation}
and
\begin{equation}\label{eq:m3}
m_3(a_1,a_2,a_3)=\Pi(\la_3(a_1,a_2,a_3)).
\end{equation}
We first determine the possible $j$ and $k$ introduced in section \ref{Proc} in case
$$a_3 \in \hom^{s_1+k_1}(P_\bullet(m_1|n_1), P_\bullet(a_1|b_1)\langle s_1+j_1\rangle)_0$$
and $$Q(a_1 \cdot a_2)\in \hom^{s_2+k_2}(P_\bullet(m_2|n_2) P_\bullet(a_2|b_2)\langle s_2+j_2\rangle)_0.$$
Then their compositions $a_3 \cdot Q(a_1\cdot a_2)$ and $Q(a_1 \cdot a_2)\cdot a_3$ are in 
$$\hom^{s+(k_1+k_2)}(P_\bullet(m|n), P_\bullet(a|b)\langle s+(j_1+j_2)\rangle)_0$$
with $s=m+n-(a+b)$ and $(m|n)$ and $(a|b)$ appropriate. This is done in Table \ref{tab:kj}. We use Remark \ref{Re:corexp} and therefore only have to write down the composition in one order.
Assuming $$\la_3 \in \hom^{s+k}(P_\bullet(m|n), P_\bullet(a|b)\langle s+j\rangle)_0$$
we would have 
$$Q(\la_3) \in \hom^{s+k-1}(P_\bullet(m|n), P_\bullet(a|b)\langle s+j\rangle)_0.$$
Checking all cases from Table \ref{tab:kj} together with Corollary \ref{Cor:posschmaps} we obtain that $Q(\la_3)$ must be zero in all cases, since for all possible pairs $(j,k-1)$ there exists no nonzero element in $\hom^{s+k-1}(P_\bullet(m|n), P_\bullet(a|b)\langle s+j\rangle)_0$. Hence the second part of the theorem follows.

To compute $m_3$ using \eqref{eq:m3} we apply $\Pi$ to the two summands from \eqref{eq:la3} i.e. we compute
 $\Pi(a_1 \cdot Q(a_2 \cdot a_3))$ and  $\Pi(Q(a_1 \cdot a_2) \cdot a_3)$. The result is presented as a linear combination of the elements 
\begin{align*}
\{ &\Id^{(n,m)}_{(a|b)}, F^{(n,m)}_{(a|b)}, \widetilde{F}^{(n,m)}_{(a|b)}, G^{(n,m)}_{(a|b)}, K^{(n,m)}_{(a|b)}, J^{(n,m)}_{(a|b)},\\ &H(A)^{(n,m)}_{(a|b)}, H(B)^{(n,m)}_{(a|b)}, H(J)^{(n,m)}_{(a|b)}, H(F-\widetilde{F})^{(n,m)}_{(a|b)} \}
\end{align*}
from $H$, $B$ and $L$ introduced in Sections \ref{sec:elements2}-\ref{sec:homotopies} plus the additional element:
\begin{align}
L^{(n|m)}_{(a|b)}:& \quad P_\bullet(n|m) \to P_\bullet(a|b)\langle (n+m)-(a+b)-6 \rangle [(n+m)-(a+b)-3]\nonumber\\
& 
								P(s|s-2)_B \to P(s-1|s-3)_A 
\end{align}
It is easy to see that the resulting set forms a linearly independent set of elements in $\hom(P_\bullet, P_\bullet)$.

Using Table \ref{tab:kj} together with Corollary \ref{Cor:posschmaps} we restrict the computations to those cases where a nonzero element $\la_3 \in \hom^{s+k}(P_\bullet(m|n), P_\bullet(a|b)\langle s+j\rangle)_0$ can exist. The possible combinations for $a_1$ and $Q(a_2 \cdot a_3)$, and $Q(a_1 \cdot a_2)$ and $a_3$ respectively are listed in Table \ref{tab:multhom}. There their product is determined, too. We wrote $\pm$ where the signs were not important for further computations.

Using the results from Table \ref{tab:MultelemQ} and the results of the computations in Table \ref{tab:multhom} we determine all triples $(a_1,a_2,a_3)$ with $a_i \in \Ext(\oplus M(\la),\oplus M(\la))$ such that at least one of the terms $\Pi(a_1 \cdot Q(a_2 \cdot a_3))$ or $\Pi(Q(a_1 \cdot a_2)\cdot a_3)$ is unequal to zero. The result is listed in Table \ref{tab:mult3}. We ignore from now on all signs except for the single special case where both terms are nonzero. There we obtain $(s_1+(-1)^{n+m+k+l+1}s_2)=\pm (1+(-1)^{n+m+k+d+b+a})$. If one wants to know the signs in the other cases, one only has to multiply those in Table \ref{tab:MultelemQ} with the corresponding ones in Table \ref{tab:multhom}. 

Therefore, we have computed all $m_3$ and the non-vanishing of some of them is shown.
\end{proof}
\begin{table}[htb]
\caption{\label{tab:kj} $j$ and $k$ for the compositions $x \cdot y$ and $y \cdot x$}
\begin{tabular}{|c!{\vrule width 1.3pt}c|c|c|c|c|c|}
		\toprule \rule[-0.6cm]{0cm}{1.2cm}
	$x\backslash y$	& $H(F-\widetilde{F})^{(n_2|m_2)}_{(a_2|b_2)}$ & $H(J)^{(n_2|m_2)}_{(a_2|b_2)}$ & $H(A)^{(n_2|m_2)}_{(a_2|b_2)}$ & $H(B)^{(n_2|m_2)}_{(a_2|b_2)}$ \\
		& $j_2=-2$& $j_2=-4$& $j_2=-4$&$j_2=-6$\\
		& $k_2=-2$&$k_2=-3$&$k_2=-3$&$k_2=-4$\\
		 \noalign{\hrule height 1.2pt} 
		$\Id^{(n_1|m_1)}_{(a_1|b_1)}$&&&& \\
		$j_1=0$&$j=-2$&$j=-4$&$j=-4$&$j=-6$\\
		$k_1=0$&$k=-2$&$k=-3$&$k=-3$&$k=-4$\\ \hline 
    $F^{(n_1|m_1)}_{(a_1|b_1)}$&&&& \\
		$j_1=-2$&$j=-4$&$j=-6$&$j=-6$&$j=-8$\\
		$k_1=-1$&$k=-3$&$k=-4$&$k=-4$&$k=-5$\\ \hline 
	  $\widetilde{F}^{(n_1|m_1)}_{(a_1|b_1)}$&&&& \\
		$j_1=-2$&$j=-4$&$j=-6$&$j=-6$&$j=-8$\\
		$k_1=-1$&$k=-3$&$k=-4$&$k=-4$&$k=-5$\\ \hline 
		$G^{(n_1|m_1)}_{(a_1|b_1)}$&&&& \\
		$j_1=-4$&$j=-6$&$j=-8$&$j=-8$&$j=-10$\\
		$k_1=-3$&$k=-5$&$k=-6$&$k=-6$&$k=-7$\\ \hline 
    $K^{(n_1|m_1)}_{(a_1|b_1)}$&&&& \\
		$j_1=-6$&$j=-8$&$j=-10$&$j=-10$&$j=-12$\\
		$k_1=-4$&$k=-6$&$k=-7$&$k=-7$&$k=-8$\\ \hline 
    $J^{(n_1|m_1)}_{(a_1|b_1)}$&&&& \\
		$j_1=-4$&$j=-6$&$j=-8$&$j=-8$&$j=-10$\\
		$k_1=-2$&$k=-4$&$k=-7$&$k=-7$&$k=-8$\\ \hline

	\end{tabular}		

\end{table}
\begin{table} 
\caption{\label{tab:multhom} Possible products for $a_1$ and $Q(a_2 \cdot a_3)$, and $Q(a_1 \cdot a_2)$ and $a_3$, respectively}
\begin{tabular}{|c|c!{\vrule width 1.3pt}l|}
		\toprule \rule[-0.6cm]{0cm}{1.2cm}
		$F^{(n|m)}_{(k|l)} $& $H(F-\widetilde{F})^{(k|l)}_{(a|b)}$ & $\begin{array}{l}(-1)^{(n+a)(a+b)+(l+b)(n+k+1)} G^{(n|m)}_{(a|b)}\\+(-1)^{(l+b)(n+k+1)} H(J)^{(n|m)}_{(a|b)}\\+(-1)^{(l+b)(n+k+1)+n+b} H(A)^{(n|m)}_{(a|b)} \end{array}$\\\hline
		$F^{(n|m)}_{(k|l)} $& $H(A)^{(k|l)}_{(a|b)}$ & $\begin{array}{l}\pm L^{(n|m)}_{(a|b)}\\ + \pm  H(B)^{(n|m)}_{(a|b)} \end{array}$\\\hline
		$F^{(n|m)}_{(k|l)} $& $H(J)^{(k|l)}_{(a|b)}$ & $\begin{array}{l}(-1)^{(l+a)n+(l+b+1)k+ab} K^{(n|m)}_{(a|b)}\\+ \pm  H(B)^{(n|m)}_{(a|b)} \end{array}$\\\hline		
$H(F-\widetilde{F})^{(n|m)}_{(k|l)} $& $F^{(k|l)}_{(a|b)}$ & $\begin{array}{l}(-1)^{(l+b+1)n+m+a} H(J)^{(n|m)}_{(a|b)}\\+(-1)^{(l+b)n+m+(k+1)l+bk+b+a+1} H(A)^{(n|m)}_{(a|b)} \end{array}$\\\hline
		$H(A)^{(n|m)}_{(k|l)} $& $F^{(k|l)}_{(a|b)}$ & $\begin{array}{l}\pm L^{(n|m)}_{(a|b)}\\ + \pm  H(B)^{(n|m)}_{(a|b)} \end{array}$\\\hline
		$H(J)^{(n|m)}_{(k|l)} $& $F^{(k|l)}_{(a|b)}$ & $\pm H(B)^{(n|m)}_{(a|b)}$\\\hline		
		$\widetilde{F}^{(n|m)}_{(k|l)} $& $H(F-\widetilde{F})^{(k|l)}_{(a|b)}$ & $\pm H(J)^{(n|m)}_{(a|b)}$
		\\\hline
		$\widetilde{F}^{(n|m)}_{(k|l)} $& $H(A)^{(k|l)}_{(a|b)}$ & $0$\\\hline
		$\widetilde{F}^{(n|m)}_{(k|l)} $& $H(J)^{(k|l)}_{(a|b)}$ & $0$\\\hline		
$H(F-\widetilde{F})^{(n|m)}_{(k|l)} $& $\widetilde{F}^{(k|l)}_{(a|b)}$ & $\begin{array}{l}(-1)^{(l+a+1)n+m+(k+1)l+(b+1)k+
(a+1)b}( G^{(n|m)}_{(a|b)}\\+(-1)^{(b+a)(n+a)} H(J)^{(n|m)}_{(a|b)}) \end{array}$\\\hline
		$H(A)^{(n|m)}_{(k|l)} $& $\widetilde{F}^{(k|l)}_{(a|b)}$ & $\pm H(B)^{(n|m)}_{(a|b)}$\\\hline
		$H(J)^{(n|m)}_{(k|l)} $& $\widetilde{F}^{(k|l)}_{(a|b)}$ & $(-1)^{(l+a)n+kl+(a+k+1)(b+1)}K^{(n|m)}_{(a|b)}$\\\hline		
		$J^{(n|m)}_{(k|l)} $& $H(F-\widetilde{F})^{(k|l)}_{(a|b)}$ & $\begin{array}{l}(-1)^{(l+a)n+kl+(b+1)(a+k+1)} K^{(n|m)}_{(a|b)}\\+\pm H(B)^{(n|m)}_{(a|b)}\end{array}$
		\\\hline
				$H(F-\widetilde{F})^{(n|m)}_{(k|l)} $& $J^{(k|l)}_{(a|b)}$ & $\begin{array}{l}(-1)^{(l+a+1)n++m+(k+1)l+bk+(b+1)(a+1)}\\ K^{(n|m)}_{(a|b)}\\+\pm H(B)^{(n|m)}_{(a|b)}\end{array}$	\\ \hline	
\end{tabular}
\end{table}
\begin{landscape}
\begin{longtable}{|c|c|c|c|c|c!{\vrule width 1.3pt}c|}
\caption{ The multiplication $m_3(a_1,a_2,a_3)$, all other cases yield zero} \label{tab:mult3} \\
		\toprule \rule[-0.6cm]{0cm}{1.2cm}
		&$a_1$ & $a_2$& $a_3$& $\Pi(a_1 Q(a_2,a_3))$& $\Pi(Q(a_1,a_2)a_3)$ &$m_3$\\ \hline
	$a \leq m$ &$\Id^{(n|m)}_{(k|l)}$ & $\widetilde{F}^{(k|l)}_{(a|b)}$ & $F^{(a|b)}_{(c|d)}$ &$0$&$0$& $0$ \\ \hline
	$a \leq m$ &$\Id^{(n|m)}_{(k|l)}$ & $\widetilde{F}^{(k|l)}_{(a|b)}$ & $F^{(a|b)}_{(c|d)}$ &$0$ & $\pm G^{(n|m)}_{(c|d)}$& $\pm G^{(n|m)}_{(c|d)}$ \\ \hline
	$a \leq m$ &$\Id^{(n|m)}_{(k|l)}$ & $\widetilde{F}^{(k|l)}_{(a|b)}$ & $J^{(a|b)}_{(c|d)}$ &$0$ &$\pm K^{(n|m)}_{(c|d)}$& $\pm K^{(n|m)}_{(c|d)}$ \\ \hline
	$a \leq m$ &$\Id^{(n|m)}_{(k|l)}$ & $J^{(k|l)}_{(a|b)}$ & $F^{(a|b)}_{(c|d)}$ & &$0$ & $0$ $0$ \\ \hline
	$a \leq m$ &$\Id^{(n|m)}_{(k|l)}$ & $J^{(k|l)}_{(a|b)}$ & $\widetilde{F}^{(a|b)}_{(c|d)}$ &$0$ &$\pm K^{(n|m)}_{(c|d)}$ & $\pm K^{(n|m)}_{(c|d)}$ \\ \hline
	&$F^{(n|m)}_{(k|l)}$ & $F^{(k|l)}_{(a|b)}$ & $F^{(a|b)}_{(c|d)}$ &$0$&$0$& $0$ \\ \hline
	$c > l$ &$F^{(n|m)}_{(k|l)}$ & $F^{(k|l)}_{(a|b)}$ & $\widetilde{F}^{(a|b)}_{(c|d)}$ &$0$&$0$& $0$ \\ \hline
	$c  \leq l$ &$F^{(n|m)}_{(k|l)}$ & $F^{(k|l)}_{(a|b)}$ & $\widetilde{F}^{(a|b)}_{(c|d)}$ &$0$&$ \pm K^{(n|m)}_{(c|d)}$ & $\pm K^{(n|m)}_{(c|d)}$ \\ \hline
	$a \leq m$ &$F^{(n|m)}_{(k|l)}$ & $\widetilde{F}^{(k|l)}_{(a|b)}$ & $F^{(a|b)}_{(c|d)}$ &$0$ &$0$&  $0$ \\ \hline
	$\begin{matrix}a \leq m\\ c \leq l\end{matrix}$ &$F^{(n|m)}_{(k|l)}$ & $\widetilde{F}^{(k|l)}_{(a|b)}$ & $F^{(a|b)}_{(c|d)}$ &$0$ & $\pm K^{(n|m)}_{(c|d)}$& $\pm K^{(n|m)}_{(c|d)}$ \\ \hline
	&$F^{(n|m)}_{(k|l)}$ & $\widetilde{F}^{(k|l)}_{(a|b)}$ & $\widetilde{F}^{(a|b)}_{(c|d)}$& $0$ &$\pm K^{(n|m)}_{(c|d)}$ & $\pm K^{(n|m)}_{(c|d)}$ \\ \hline
	$a \leq m$ &$\widetilde{F}^{(n|m)}_{(k|l)}$ & $\Id^{(k|l)}_{(a|b)}$ & $F^{(a|b)}_{(c|d)}$ &$0$&$0$& $0$ \\ \hline
	$\begin{matrix}a \leq m\\ c>l \end{matrix}$ &$\widetilde{F}^{(n|m)}_{(k|l)}$ & $\Id^{(k|l)}_{(a|b)}$ & $F^{(a|b)}_{(c|d)}$ &$0$ & $\pm G^{(n|m)}_{(c|d)}$& $\pm G^{(n|m)}_{(c|d)}$ \\ \hline
		$\begin{matrix}a \leq m\\ c\leq l \end{matrix}$ &$\widetilde{F}^{(n|m)}_{(k|l)}$ & $\Id^{(k|l)}_{(a|b)}$ & $F^{(a|b)}_{(c|d)}$ &$0$ & $\pm G^{(n|m)}_{(c|d)}$& $\pm G^{(n|m)}_{(c|d)}$ \\ \hline
				$\begin{matrix}a > m\\ c> l \end{matrix}$ &$\widetilde{F}^{(n|m)}_{(k|l)}$ & $\Id^{(k|l)}_{(a|b)}$ & $F^{(a|b)}_{(c|d)}$ &$0$ & $\pm 0$& $0$ \\ \hline				
	$a \leq m$ &$\widetilde{F}^{(n|m)}_{(k|l)}$ & $\Id^{(k|l)}_{(a|b)}$ & $J^{(a|b)}_{(c|d)}$ &$0$ &$\pm K^{(n|m)}_{(c|d)}$& $\pm K^{(n|m)}_{(c|d)}$ \\ \hline
	$a \leq m$ &$\widetilde{F}^{(n|m)}_{(k|l)}$ & $F^{(k|l)}_{(a|b)}$ & $F^{(a|b)}_{(c|d)}$& $0$ &$0$ &$0$\\ \hline	
	$a > m$ &$\widetilde{F}^{(n|m)}_{(k|l)}$ & $F^{(k|l)}_{(a|b)}$ & $F^{(a|b)}_{(c|d)}$& $0$ &$0$ &$0$\\ \hline	
	$\begin{matrix} a \leq m \\ c>l\end{matrix}$ &$\widetilde{F}^{(n|m)}_{(k|l)}$ & $F^{(k|l)}_{(a|b)}$ & $\widetilde{F}^{(a|b)}_{(c|d)}$& $0$ &$\pm K^{(n|m)}_{(c|d)}$ & $\pm K^{(n|m)}_{(c|d)}$ \\ \hline
	$\begin{matrix} a \leq m \\ c\leq l\end{matrix}$ &$\widetilde{F}^{(n|m)}_{(k|l)}$ & $F^{(k|l)}_{(a|b)}$ & $\widetilde{F}^{(a|b)}_{(c|d)}$& $0$ &$\pm K^{(n|m)}_{(c|d)}$ & $\pm K^{(n|m)}_{(c|d)}$ \\ \hline
	$\begin{matrix} a>m \\ c\leq l\end{matrix}$ &$\widetilde{F}^{(n|m)}_{(k|l)}$ & $F^{(k|l)}_{(a|b)}$ & $\widetilde{F}^{(a|b)}_{(c|d)}$& $0$ &$0$ & $0$ \\ \hline
$a \leq m$ &$J^{(n|m)}_{(k|l)}$ & $\Id^{(k|l)}_{(a|b)}$ & $F^{(a|b)}_{(c|d)}$& $0$ &$0$ &$0$\\ \hline	
	$\begin{matrix} a \leq m \\ c>l\end{matrix}$ &$J^{(n|m)}_{(k|l)}$ & $\Id^{(k|l)}_{(a|b)}$ & $\widetilde{F}^{(a|b)}_{(c|d)}$& $0$ &$\pm K^{(n|m)}_{(c|d)}$ & $\pm K^{(n|m)}_{(c|d)}$ \\ \hline
	$\begin{matrix} a \leq m \\ c \leq l\end{matrix}$ &$J^{(n|m)}_{(k|l)}$ & $\Id^{(k|l)}_{(a|b)}$ & $\widetilde{F}^{(a|b)}_{(c|d)}$& $s_1 K^{(n|m)}_{(c|d)}$ &$s_2 K^{(n|m)}_{(c|d)}$ & $(s_1 +(-1)^{n+m+k+l+1} s_2) K^{(n|m)}_{(c|d)}$ \\ \hline
		$\begin{matrix} a > m \\ c \leq l\end{matrix}$ &$J^{(n|m)}_{(k|l)}$ & $\Id^{(k|l)}_{(a|b)}$ & $\widetilde{F}^{(a|b)}_{(c|d)}$& &$\pm K^{(n|m)}_{(c|d)}$ $0$ & $\pm K^{(n|m)}_{(c|d)}$ \\ \hline
$c \leq l$ &$F^{(n|m)}_{(k|l)}$ & $\Id^{(k|l)}_{(a|b)}$ & $\widetilde{F}^{(a|b)}_{(c|d)}$& $\pm G^{(n|m)}_{(c|d)}$ &$0$ &$\pm G^{(n|m)}_{(c|d)}$\\ \hline	
$c \leq l$ &$F^{(n|m)}_{(k|l)}$ & $\Id^{(k|l)}_{(a|b)}$ & $J^{(a|b)}_{(c|d)}$& $\pm K^{(n|m)}_{(c|d)}$ &$0$ &$\pm K^{(n|m)}_{(c|d)}$\\ \hline	
$c \leq l$ &$F^{(n|m)}_{(k|l)}$ & $\widetilde{F}^{(k|l)}_{(a|b)}$ & $\Id^{(a|b)}_{(c|d)}$& $\pm G^{(n|m)}_{(c|d)}$ &$0$ &$\pm G^{(n|m)}_{(c|d)}$\\ \hline	
$c \leq l$ &$\widetilde{F}^{(n|m)}_{(k|l)}$ & $\widetilde{F}^{(k|l)}_{(a|b)}$ & $\Id^{(a|b)}_{(c|d)}$& $0$ &$0$ &$0$\\ \hline
$c \leq l$ &$J^{(n|m)}_{(k|l)}$ & $\widetilde{F}^{(k|l)}_{(a|b)}$ & $\Id^{(a|b)}_{(c|d)}$& $\pm K^{(n|m)}_{(c|d)}$ &$0$ &$\pm K^{(n|m)}_{(c|d)}$\\ \hline
$c \leq l$ &$\widetilde{F}^{(n|m)}_{(k|l)}$ & $\widetilde{F}^{(k|l)}_{(a|b)}$ & $F^{(a|b)}_{(c|d)}$& $0$&$0$ &$0$\\ \hline	
$c \leq l$ &$F^{(n|m)}_{(k|l)}$ & $J^{(k|l)}_{(a|b)}$ & $\Id^{(a|b)}_{(c|d)}$& $\pm K^{(n|m)}_{(c|d)}$&$0$ &$\pm K^{(n|m)}_{(c|d)}$\\ \hline	
$c \leq l$ &$\widetilde{F}^{(n|m)}_{(k|l)}$ & $J^{(k|l)}_{(a|b)}$ & $\Id^{(a|b)}_{(c|d)}$& $0$&$0$ &$0$\\ \hline	
	\end{longtable}
	\end{landscape}

\subsubsection{More vanishing results}
Before we are able to prove our main theorem, it is necessary to work out the third step introduced in Section \ref{Proc}:
\begin{Lemma}\label{la2zero}
$$Q(\la_2(a_1, a_2))\cdot Q(\la_2(a_3, a_4))=0 \ \forall \ a_i \ \in \Ext(\bigoplus M(\la),\bigoplus M(\la)).$$
\end{Lemma}
\begin{proof}
Again we argue by the $\hom$-spaces and their dimension. Similar to Table \ref{tab:kj} in Table \ref{tab:kj2} we determine the shifts of the product of $$Q(\la_2(a_1, a_2)) \in \hom^{s+k_1}(P_\bullet(m|n), P_\bullet(a|b)\langle s+j_1\rangle)_0$$
and
$$Q(\la_2(a_3, a_4)) \in \hom^{s+k_2}(P_\bullet(m|n), P_\bullet(a|b)\langle s+j_2\rangle)_0.$$
Applying Corollary \ref{Cor:posschmaps} the lemma follows.
\begin{table}
\caption{\label{tab:kj2} $j$ and $k$ for the compositions of two homotopies}
\begin{tabular}{|c!{\vrule width 1.3pt}c|c|c|c|c|c|}
		\toprule \rule[-0.6cm]{0cm}{1.2cm}
		& $H(F-\widetilde{F})^{(n_2|m_2)}_{(a_2|b_2)}$ & $H(J)^{(n_2|m_2)}_{(a_2|b_2)}$ & $H(A)^{(n_2|m_2)}_{(a_2|b_2)}$ & $H(B)^{(n_2|m_2)}_{(a_2|b_2)}$ \\
		& $j_2=-2$& $j_2=-4$& $j_2=-4$&$j_2=-6$\\
		& $k_2=-2$&$k_2=-3$&$k_2=-3$&$k_2=-4$\\
		 \noalign{\hrule height 1.2pt} 
		$H(F-\widetilde{F})^{(n_1|m_1)}_{(a_1|b_1)}$&&&& \\
		$j_1=-2$&$j=-4$&$j=-6$&$j=-6$&$j=-8$\\
		$k_1=-2$&$k=-4$&$k=-5$&$k=-5$&$k=-6$\\ \hline 
    $H(J)^{(n_1|m_1)}_{(a_1|b_1)}$&&&& \\
		$j_1=-4$&$j=-6$&$j=-8$&$j=-8$&$j=-10$\\
		$k_1=-3$&$k=-5$&$k=-6$&$k=-6$&$k=-7$\\ \hline 
	  $H(A)^{(n_1|m_1)}_{(a_1|b_1)}$&&&& \\
		$j_1=-4$&$j=-6$&$j=-8$&$j=-8$&$j=-10$\\
		$k_1=-3$&$k=-5$&$k=-6$&$k=-6$&$k=-7$\\ \hline 
		$H(B)^{(n_1|m_1)}_{(a_1|b_1)}$&&&& \\
		$j_1=-6$&$j=-8$&$j=-10$&$j=-10$&$j=-12$\\
		$k_1=-4$&$k=-6$&$k=-7$&$k=-7$&$k=-8$\\ \hline 
	
	\end{tabular}		

\end{table}
\end{proof}

\subsubsection{Vanishing of higher multiplications}
The following theorem is the main result in this section:
\begin{Theorem}[2nd Vanishing Theorem]
\label{Th:2ndvanish}
For $\p$ the parabolic subalgebra belonging to $\fr{l}=\fr{gl}_2\oplus \fr{gl}_{N-1}$ and
$\Ext(\oplus M(\la), \oplus M(\la))$ the $A_\infty$-algebra constructed as a minimal model we have 
$$m_n=0 \ \forall n \geq 4.$$
\end{Theorem}
\begin{proof}
By Lemma \ref{Th:la3} we know that $Q(\la_3)=0$ and by Lemma \ref{la2zero} $Q(\la_2)Q(\la_2)=0$. Using Corollary \ref{Cor:disaplambda} we obtain the result.
\end{proof}

\section{Ideas how to proceed}
In the previous section we proved that there is a minimal model with non-vanishing higher multiplications but this does not answer the question if the algebra is formal. To show that the algebra is not formal, we have to prove that no model exists such that $m_n=0$ for all $n \geq 3$. 

As a tool we use Hochschild cohomology. Given a dg-Algebra $A$ one can compute its Hochschild cohomology by using the $A_\infty$-structure on a minimal model of $A$ (cf. \cite[Lemma B.4.1]{lefe2003} and \cite{kade1988}). Especially for a model with $m_n=0$ for $n \geq 3$ one obtains that the Hochschild cohomology is trivial. Consequently, if we can prove that the Hochschild cohomology of $A$ is not trivial, there cannot exist a minimal model with $m_n=0$ for all $n \geq 3$ and therefore $A$ cannot be formal.

Assume that we have found a minimal model on $H^*(A)$ with $m_n=0$ for $3 \leq n \leq p-1$. Then the multiplication $m_p$ defines a cocycle for the Hochschild cohomology of $A$ by the construction in \cite[Lemma B.4.1]{lefe2003}. If we can prove that this class is not trivial, we are done and have shown, that the algebra is not formal. If we cannot, we have to modify our model such that $m_p=0$, too and have to analyse if $m_{p+1}$ vanishes.

A detailed discussion of this topic would go beyond the scope of this thesis. Therefore we can only state the following conjecture:
\begin{Conjecture}
In general the algebra $\Ext(\bigoplus M(\la), \bigoplus M(\la))$ is not formal.
\end{Conjecture}

\FloatBarrier
\clearemptydoublepage
\appendix

\part*{Appendix}
\addcontentsline{toc}{part}{Appendix} 

\chapter{Computations for the proofs in Section \ref{sec:elements2}}\label{Ap:Elem}
\begin{proof}[Computations for the proof of Theorem \ref{Th:ElemId}]\label{Ap:ElemId}
The proof of Theorem \ref{Th:ElemId} requires a verification of the commutativity of the following diagrams:
\begin{multicols}{2}
    \raggedcolumns

\begin{enumerate}
\item[i)] $\xymatrix{P(s|t)_A \ar[rr]^{(-1)^{n+m+s+t+1}} \ar[d]_{m+t} &&P(s+1|t)_A \ar[d]^{m+t} \\P(s|t)_A \ar[rr]^{(-1)^{n+m+s+t}} &&P(s+1|t)_A}$
\item[ii)] $\xymatrix{P(s|t)_A \ar[rr]^{(-1)^{m+t+1}} \ar[d]_{m+t} &&P(s|t+1)_A \ar[d]^{m+t+1} \\P(s|t)_A \ar[rr]^{(-1)^{m+t+1}} &&P(s|t+1)_A}$
\item[iii)] $\xymatrix{P(s|t)_B \ar[rr]^{(-1)^{m+s+1}} \ar[d]_{m+s} &&P(s+1|t)_B \ar[d]^{m+s+1} \\P(s|t)_B \ar[rr]^{(-1)^{m+s+1}} &&P(s+1|t)_B}$
\item[iv)] $\xymatrix{P(s|t)_B \ar[rr]^{(-1)^{n+m+s+t+1}} \ar[d]_{m+s} &&P(s|t+1)_B \ar[d]^{m+s} \\P(s|t)_B \ar[rr]^{(-1)^{n+m+s+t}} &&P(s|t+1)_B}$
\item[v)] $\xymatrix{P(s|t)_A \ar[rr]^{(-1)^{(s+t+1)(n+s)+n+m+1}} \ar[d]_{m+t} &&P(s-1|t)_B \ar[d]^{m+s+1} \\P(s|t)_A \ar[rr]^{(-1)^{(s+t+1)(n+s+1)+n+m}} &&P(s-1|t)_B}$
\item[vi)] $\xymatrix{P(s|t)_A \ar[rr]^{(-1)^{(s+t+1)(n+s)+m+s}} \ar[d]_{m+t} &&P(s|t-1)_B \ar[d]^{m+s} \\P(s|t)_A \ar[rr]^{(-1)^{(s+t+1)(n+s+1)+m+s}} &&P(s|t-1)_B}$
\item[vii)] $\xymatrix{P(s|s-2)_B \ar[rr]^{(-1)^{n+m+1}} \ar[d]_{m+s} &&P(s+1|s)_A \ar[d]^{m+s} \\P(s|s-2)_B \ar[rr]^{(-1)^{n+m}} &&P(s+1|s)_A}$
\end{enumerate}
\end{multicols}
All these diagrams obviously anticommute. 
\end{proof}
\begin{proof}[Computations for the proof of Theorem \ref{Th:elemF2}]\label{Ap:ElemF}
We check that the diagrams $F1$)-$F7$) listed in the proof of Theorem \ref{Th:elemF2} commute.
\begin{itemize}
\item [\ref{FA}] $V=P(s|t)_A$ and $s>t+3$:

We have the following diagram:
\begin{equation}\label{diag:F1}\xymatrix{ P(s|t)_A \ar[r]\ar[d] &{\begin{matrix} P(s+1|t)_A\\P(s|t+1)_A\\P(s-1|t)_B\\P(s|t-1)_B \end{matrix}\ar[d]}\\
P(s|t-1)_A \ar[r] &{\begin{matrix} P(s+1|t-1)_A\\P(s|t)_A\\P(s-1|t-1)_B\\P(s|t-2)_B \\P(s-2|t)_B\end{matrix}}}
\end{equation}
with 
\begin{displaymath} \xymatrix{ P(s|t)_A \ar[rr]^{(-1)^{n+m+s+t+1}} &&P(s+1|t)_A  \ar @{} [d] |{=} \ar[rr] &&P(s+1|t-1)_A\\
P(s|t)_A \ar[rr] &&P(s|t-1)_A \ar[rr]^{(-1)^{n+m+s+t+1}} &&P(s+1|t-1)_A}
\end{displaymath}
for $t>0$ we also have
\begin{displaymath} \xymatrix{ P(s|t)_A \ar[rr]^{(-1)^{m+t+1}} &&P(s|t+1)_A  \ar @{} [d] |{=} \ar[rr] &&P(s|t)_A\\
P(s|t)_A \ar[rr] &&P(s|t-1)_A \ar[rr]^{(-1)^{m+t+1}} &&P(s|t)_A}\end{displaymath}
and for $t=0$ the upper line is zero.\\
For $t+3 \leq s \leq m$ we also have
\begin{displaymath} \xymatrix{ P(s|t)_A \ar[rr]^{(-1)^{(s+t+1)(n+s)+m+s}} &&P(s|t-1)_B  \ar @{} [d] |{=} \ar[rr] &&P(s-1|t-1)_B\\
P(s|t)_A \ar[rr] &&P(s|t-1)_A \ar[rr]^{(-1)^{(s+t)(n+s)+n+m}} &&P(s-1|t-1)_B}\end{displaymath}
and
\begin{displaymath} \xymatrix{ P(s|t)_A \ar[rr] &&P(s|t-1)_B  \ar @{} [d] |{=} \ar[rr] &&P(s|t-2)_B\\
 &&0&&}
\end{displaymath}
 \begin{displaymath} \xymatrix{ P(s|t)_A \ar[rr] &&P(s-1|t)_B  \ar @{} [d] |{=} \ar[rr] &&P(s-2|t)_B\\
 &&0&&}
\end{displaymath}
Therefore diagram \eqref{diag:F1} commutes. \\
\item [\ref{FAs3}]  $V=P(s+3|s)_A$:
\begin{equation}\label{diag:F2}\xymatrix{ P(s+3|s)_A \ar[r]\ar[d] &{\begin{matrix} P(s+4|s)_A\\P(s+3|s+1)_A\\P(s+2|s)_B\\P(s+3|s-1)_B \end{matrix}\ar[d]}\\
P(s+3|s-1)_A \ar[r] &{\begin{matrix} P(s+4|s-1)_A\\P(s+3|s)_A\\P(s+2|s-1)_B\\P(s+3|s-2)_B \end{matrix}}}
\end{equation}
The only part of the diagram which differs from the general case is the part including $P(s+3|s+1)_A$ and $P(s+2|s)_B$. From these two terms, which only exist for $s \leq m-2$, we get apart from those maps we had before, two new maps:
\begin{displaymath} \xymatrix{ P(s+3|s)_A \ar[rr]^{(-1)^{n+m+1}} &&P(s+2|s)_B  \ar @{} [d] |{=} \ar[rr] &&P(s+2|s+1)_A\\
-(P(s+3|s)_A \ar[rr]^{(-1)^{m+s+1}} &&P(s+3|s+1)_A \ar[rr]^{(-1)^{n+s+1}} &&P(s+2|s+1)_A)}\end{displaymath}
So the two maps cancel and the diagram \eqref{diag:F2} commutes.

\item [\ref{FAs2}]  $V=P(s+2|s)_A$:
\begin{enumerate} 
\item $s>0$:

We get the diagram
\begin{equation}\label{diag:F3}\xymatrix{ P(s+2|s)_A \ar[r]\ar[d] &{\begin{matrix} P(s+3|s)_A\\P(s+2|s+1)_A\\P(s+2|s-1)_B \end{matrix}\ar[d]}\\
{\begin{matrix}P(s+2|s-1)_A \\P(s+1|s)_A\end{matrix}} \ar[r] &{\begin{matrix} P(s+3|s-1)_A\\P(s+2|s)_A\\P(s+1|s-1)_B\\P(s+2|s-2)_B \end{matrix}}}
\end{equation}
with
\begin{displaymath} \xymatrix{P(s+2|s)_A \ar[rr]^{(-1)^{n+m+1}} &&P(s+3|s)_A  \ar @{} [d] |{=} \ar[rr] &&P(s+3|s-1)_A\\
P(s+2|s)_A \ar[rr] &&P(s+2|s-1)_A \ar[rr]^{(-1)^{n+m+1}} &&P(s+3|s-1)_A}
\end{displaymath}

\begin{displaymath} \xymatrix{ P(s+2|s)_A \ar[rr]^{(-1)^{m+s+1}} &&P(s+2|s+1)_A  \ar @{} [d] |{=} \ar[rr] &&P(s+2|s)_A\\
P(s+2|s)_A \ar[rr] &&P(s+2|s-1)_A \ar @{} [d] |{+} \ar[rr]^{(-1)^{m+s+1}} &&P(s+2|s)_A\\
P(s+2|s)_A \ar[rr]^{(-1)^{n+s}} &&P(s+1|s)_A  \ar[rr]^{(-1)^{n+m+1}} &&P(s+2|s)_A}\end{displaymath}
and
\begin{displaymath} \xymatrix{ P(s+2|s)_A \ar[rr]^{(-1)^{m+s+1}} &&P(s+2|s+1)_A  \ar @{} [d] |{=} \ar[rr]^{(-1)^{n+s+1}} &&P(s+1|s-1)_B\\
P(s+2|s)_A \ar[rr]^{(-1)^{n+s}} &&P(s+1|s)_A \ar[rr]^{(-1)^{m+s}} &&P(s+1|s-1)_B}\end{displaymath}
and
\begin{displaymath} \xymatrix{ P(s+2|s)_A \ar[rr]^{(-1)^{n+m}} &&P(s+2|s-1)_B  \ar @{} [d] |{=} \ar[rr] &&P(s+1|s-1)_B\\
P(s+2|s)_A \ar[rr] &&P(s+2|s-1)_A \ar[rr]^{(-1)^{n+m}} &&P(s+1|s-1)_B}\end{displaymath}
and
\begin{displaymath} \xymatrix{ P(s+2|s)_A \ar[rr] &&P(s+2|s-1)_A  \ar @{} [d] |{=} \ar[rr] &&P(s+2|s-2)_B\\
 &&0&&}\end{displaymath}
We verified that diagram \eqref{diag:F3} commutes.\\

\item $s=0$:
\begin{equation}\label{diag:F4}\xymatrix{ P(2|0)_A \ar[r]\ar[d] &{\begin{matrix} P(3|0)_A\\P(2|1)_A \end{matrix}\ar[d]}\\
{ P(1|0)_A} \ar[r] &P(2|0)_A}
\end{equation}
with

\begin{displaymath} \xymatrix{ P(2|0)_A \ar[rr]^{(-1)^{m+1}} &&P(2|1)_A  \ar @{} [d] |{=} \ar[rr] &&P(2|0)_A\\
P(2|0)_A \ar[rr]^{(-1)^{n}} &&P(1|0)_A  \ar[rr]^{(-1)^{n+m+1}} &&P(2|0)_A}\end{displaymath}
and \eqref{diag:F4} commutes.\\
\end{enumerate}

\item [\ref{FAs1}]  $V=P(s+1|s)_A$:
\begin{enumerate} 
\item $s>0$:

We get the diagram
\begin{equation}\label{diag:F5}\xymatrix{ P(s+1|s)_A \ar[r]\ar[d] &{\begin{matrix} P(s+2|s)_A\\P(s+1|s-1)_B \end{matrix}\ar[d]}\\
{\begin{matrix}P(s+1|s-1)_A \\P(s|s-2)_B\end{matrix}} \ar[r] &{\begin{matrix} P(s+2|s-1)_A\\P(s+1|s)_A\\P(s+1|s-2)_B \end{matrix}}}
\end{equation}
with
\begin{displaymath} \xymatrix{P(s+1|s)_A \ar[rr]^{(-1)^{n+m}} &&P(s+2|s)_A  \ar @{} [d] |{=} \ar[rr] &&P(s+2|s-1)_A\\
P(s+1|s)_A \ar[rr] &&P(s+1|s-1)_A \ar[rr]^{(-1)^{n+m}} &&P(s+2|s-1)_A}
\end{displaymath}
\begin{displaymath} \xymatrix{P(s+1|s)_A \ar[rr]^{(-1)^{n+m}} &&P(s+2|s)_A  \ar @{} [d] |{=} \ar[rr]^{(-1)^{n+s}} &&P(s+1|s)_A\\
P(s+1|s)_A \ar[rr]^{(-1)^{n+s}} &&P(s|s-2)_B \ar[rr]^{(-1)^{n+m}} &&P(s+1|s)_A}
\end{displaymath}

\begin{displaymath} \xymatrix{ P(s+1|s)_A \ar[rr]^{(-1)^{m+s+1}} &&P(s+1|s-1)_B  \ar @{} [d] |{=} \ar[rr] &&P(s+1|s)_A\\
P(s+1|s)_A \ar[rr] &&P(s+1|s-1)_A \ar[rr]^{(-1)^{m+s+1}} &&P(s+1|s)_A}\end{displaymath}
and
\begin{displaymath} \xymatrix{ P(s+1|s)_A \ar[rr] &&P(s+1|s-1)_A  \ar @{} [d] |{=} \ar[rr]^{(-1)^{n+m+1}} &&P(s+1|s-2)_B\\
-(P(s+1|s)_A \ar[rr]^{(-1)^{n+s}} &&P(s|s-2)_B \ar[rr]^{(-1)^{m+s}} &&P(s+1|s-2)_B)}\end{displaymath}

These computations show that the diagram \eqref{diag:F5} commutes.\\

\item $s=0$:
\begin{equation}\label{diag:F6}\xymatrix{ P(1|0)_A \ar[r]\ar[d] &P(2|0)_A\ar[d]\\
0 \ar[r] &P(1|0)_A}
\end{equation}
This diagram commutes.
\end{enumerate}

\item [\ref{FB}] $V=P(s|t)_B$ and $s>t+3$:

We get the diagram
\begin{equation}\label{diag:F7}\xymatrix{ P(s|t)_B \ar[r]\ar[d] &{\begin{matrix} P(s+1|t)_B\\P(s|t+1)_B \end{matrix}\ar[d]}\\
P(s-1|t)_B \ar[r] &{\begin{matrix} P(s|t)_B\\P(s-1|t+1)_B \end{matrix}}}
\end{equation}
with
\begin{displaymath} \xymatrix{P(s|t)_B \ar[rr]^{(-1)^{n+m+s+t+1}} &&P(s|t+1)_B  \ar @{} [d] |{=} \ar[rr] &&P(s-1|t+1)_B\\
P(s|t)_B \ar[rr] &&P(s-1|t)_B \ar[rr]^{(-1)^{n+m+s+t+1}} &&P(s-1|t+1)_B}
\end{displaymath}
and for $s\leq m-1$
\begin{displaymath} \xymatrix{P(s|t)_B \ar[rr]^{(-1)^{m+s+1}} &&P(s+1|t)_B  \ar @{} [d] |{=} \ar[rr] &&P(s|t)_B\\
P(s|t)_B \ar[rr] &&P(s-1|t)_B \ar[rr]^{(-1)^{m+s+1}} &&P(s|t)_B}
\end{displaymath}
and the diagram \eqref{diag:F7} commutes.

\item [\ref{FBs2}]  $V=P(s|s-2)_B$:

We get the diagram
\begin{equation}\label{diag:F8}\xymatrix{ P(s|s-2)_B \ar[r]\ar[d] &{\begin{matrix} P(s+1|s)_A\\P(s+1|s-2)_B \end{matrix}\ar[d]}\\
P(s|s-1)_A \ar[r] &{\begin{matrix} P(s+1|s-1)_A\\P(s|s-2)_B \end{matrix}}}
\end{equation}
with
\begin{displaymath} \xymatrix{P(s|s-2)_B \ar[rr]^{(-1)^{n+m+1}} &&P(s+1|s)_A  \ar @{} [d] |{=} \ar[rr] &&P(s+1|s-1)_A\\
P(s|s-2)_B \ar[rr] &&P(s|s-1)_A \ar[rr]^{(-1)^{n+m+1}} &&P(s+1|s-1)_A}
\end{displaymath}
and for $s\leq m-1$
\begin{displaymath} \xymatrix{P(s|s-2)_B \ar[rr]^{(-1)^{n+m+1}} &&P(s+1|s)_A  \ar @{} [d] |{+} \ar[rr]^{(-1)^{n+s}} &&P(s|s-2)_B\\
P(s|s-2)_B \ar[rr]^{(-1)^{m+s+1}} &&P(s+1|s-2)_B \ar @{} [d] |{=} \ar[rr] &&P(s|s-2)_B\\
P(s|s-2)_B \ar[rr] &&P(s|s-1)_A \ar[rr]^{(-1)^{m+s+1}} &&P(s|s-2)_B}
\end{displaymath}
The diagram \eqref{diag:F8} commutes.
\item [\ref{FBs3}]  $V=P(s|s-3)_B$:

We get the diagram
\begin{equation}\label{diag:F9}\xymatrix{ P(s|s-3)_B \ar[r]\ar[d] &{\begin{matrix} P(s+1|s-3)_B\\P(s|s-2)_B \end{matrix}\ar[d]}\\
P(s-1|s-3)_B \ar[r] &{\begin{matrix} P(s|s-3)_B\\P(s|s-1)_A \end{matrix}}}
\end{equation}
with
\begin{displaymath} \xymatrix{P(s|s-3)_B \ar[rr]^{(-1)^{n+m}} &&P(s|s-2)_B  \ar @{} [d] |{=} \ar[rr] &&P(s|s-1)_A\\
P(s|s-3)_B \ar[rr] &&P(s-1|s-3)_B \ar[rr]^{(-1)^{n+m}} &&P(s|s-1)_A}
\end{displaymath}
and for $s\leq m-1$
\begin{displaymath} \xymatrix{P(s|s-3)_B \ar[rr]^{(-1)^{m+s+1}} &&P(s+1|s-2)_B \ar @{} [d] |{=} \ar[rr] &&P(s|s-2)_B\\
P(s|s-2)_B \ar[rr] &&P(s|s-1)_A \ar[rr]^{(-1)^{m+s+1}} &&P(s|s-2)_B}
\end{displaymath}
so the diagram \eqref{diag:F9} commutes and all cases for $F$ are checked.

Now we proceed with the computations for the map $\widetilde{F}$:
\end{itemize}
\begin{itemize}
\item [\ref{FtilA}] $V=P(s|t)_A$ and $s>t+2$:

We have the following diagram:
\begin{equation}\label{diag:Ftil1}\xymatrix{ P(s|t)_A \ar[r]\ar[d] &{\begin{matrix} P(s+1|t)_A\\P(s|t+1)_A\\P(s-1|t)_B\\P(s|t-1)_B \end{matrix}\ar[d]}\\
P(s-1|t)_A \ar[r] &{\begin{matrix} P(s-1|t+1)_A\\P(s|t)_A\\P(s-1|t-1)_B\\P(s|t-2)_B \\P(s-2|t)_B\end{matrix}}}
\end{equation}
with 
\begin{displaymath} \xymatrix{ P(s|t)_A \ar[rr]^{(-1)^{m+t+1}} &&P(s|t+1)_A  \ar @{} [d] |{=} \ar[rr] &&P(s-1|t+1)_A\\
P(s|t)_A \ar[rr] &&P(s-1|t)_A \ar[rr]^{(-1)^{m+t+1}} &&P(s-1|t+1)_A}
\end{displaymath}
for $s<n$ we also have
\begin{displaymath} \xymatrix{ P(s|t)_A \ar[rr]^{(-1)^{n+m+s+t+1}} &&P(s+1|t)_A  \ar @{} [d] |{=} \ar[rr] &&P(s|t)_A\\
P(s|t)_A \ar[rr] &&P(s-1|t)_A \ar[rr]^{(-1)^{n+m+s+t+1}} &&P(s|t)_A}\end{displaymath}
\begin{displaymath} \xymatrix{ P(s|t)_A \ar[rr]^{(-1)^{(s+t+1)(n+s)+n+m+1}} &&P(s-1|t)_B  \ar @{} [d] |{=} \ar[rr] &&P(s-1|t-1)_B\\
P(s|t)_A \ar[rr] &&P(s-1|t)_A \ar[rr]^{(-1)^{(s+t)(n+s)+m+s+1}} &&P(s-1|t-1)_B}\end{displaymath}
and
\begin{displaymath} \xymatrix{ P(s|t)_A \ar[rr] &&P(s|t-1)_B  \ar @{} [d] |{=} \ar[rr] &&P(s|t-2)_B\\
 &&0&&}
\end{displaymath}
 \begin{displaymath} \xymatrix{ P(s|t)_A \ar[rr] &&P(s-1|t)_B  \ar @{} [d] |{=} \ar[rr] &&P(s-2|t)_B\\
 &&0&&}
\end{displaymath}
So the diagram \eqref{diag:Ftil1} commutes. \\

\item [\ref{FtilAs2}]  $V=P(s+2|s)_A$:

We get the diagram
\begin{equation}\label{diag:Ftil2}\xymatrix{ P(s+2|s)_A \ar[r]\ar[d] &{\begin{matrix} P(s+3|s)_A\\P(s+2|s+1)_A\\P(s+2|s-1)_B \end{matrix}\ar[d]}\\
P(s+1|s)_A \ar[r] &{\begin{matrix} P(s+2|s)_A\\P(s+1|s-1)_B\\P(s+2|s-2)_B \end{matrix}}}
\end{equation}
with
\begin{displaymath} \xymatrix{P(s+2|s)_A \ar[rr]^{(-1)^{n+m+1}} &&P(s+3|s)_A  \ar @{} [d] |{=} \ar[rr] &&P(s+2|s)_A\\
P(s+2|s)_A \ar[rr] &&P(s+1|s)_A \ar[rr]^{(-1)^{n+m+1}} &&P(s+2|s)_A}
\end{displaymath}

\begin{displaymath} \xymatrix{ P(s+2|s)_A \ar[rr]^{(-1)^{m+s+1}} &&P(s+2|s+1)_A  \ar @{} [d] |{=} \ar[rr] &&P(s+1|s-1)_B\\
P(s+2|s)_A \ar[rr] &&P(s+1|s)_A  \ar[rr]^{(-1)^{m+s+1}} &&P(s+1|s-1)_B}\end{displaymath}
and
\begin{displaymath} \xymatrix{ P(s+2|s)_A \ar[rr] &&P(s+2|s-1)_B  \ar @{} [d] |{=} \ar[rr] &&P(s+2|s-2)_B\\
 &&0&&}
\end{displaymath}
By these computations diagram \eqref{diag:Ftil2} commutes.\\

\item [\ref{FtilAs1}]  $V=P(s+1|s)_A$:

We get the diagram
\begin{equation}\label{diag:Ftil3}\xymatrix{ P(s+1|s)_A \ar[r]\ar[d] &{\begin{matrix} P(s+2|s)_A\\P(s+1|s-1)_B \end{matrix}\ar[d]}\\
P(s|s-2)_B  \ar[r] &{\begin{matrix}P(s+1|s)_A\\P(s+1|s-2)_B \end{matrix}}}
\end{equation}
with
\begin{displaymath} \xymatrix{P(s+1|s)_A \ar[rr]^{(-1)^{n+m}} &&P(s+2|s)_A  \ar @{} [d] |{=} \ar[rr] &&P(s+1|s)_A\\
P(s+1|s)_A \ar[rr] &&P(s|s-2)_B \ar[rr]^{(-1)^{n+m}} &&P(s+1|s)_A}
\end{displaymath}
and
\begin{displaymath} \xymatrix{ P(s+1|s)_A \ar[rr]^{(-1)^{m+s+1}} &&P(s+1|s-1)_B \ar @{} [d] |{=} \ar[rr] &&P(s+1|s-2)_B\\
P(s+1|s)_A \ar[rr] &&P(s|s-2)_B \ar[rr]^{(-1)^{m+s+1}} &&P(s+1|s-2)_B}\end{displaymath}

By these computations diagram \eqref{diag:Ftil3} commutes.\\

\item [\ref{FtilB}] $V=P(s|t)_B$ and $s>t+2$:

\begin{enumerate}
\item $t>0$:

We get the diagram
\begin{equation}\label{diag:Ftil4}\xymatrix{ P(s|t)_B \ar[r]\ar[d] &{\begin{matrix} P(s+1|t)_B\\P(s|t+1)_B \end{matrix}\ar[d]}\\
P(s|t-1)_B \ar[r] &{\begin{matrix} P(s|t)_B\\P(s+1|t-1)_B \end{matrix}}}
\end{equation}
with
\begin{displaymath} \xymatrix{P(s|t)_B \ar[rr]^{(-1)^{n+m+s+t+1}} &&P(s|t+1)_B  \ar @{} [d] |{=} \ar[rr] &&P(s|t)_B\\
P(s|t)_B \ar[rr] &&P(s|t-1)_B \ar[rr]^{(-1)^{n+m+s+t+1}} &&P(s|t)_B}
\end{displaymath}
and for $s\leq m-1$
\begin{displaymath} \xymatrix{P(s|t)_B \ar[rr]^{(-1)^{m+s+1}} &&P(s+1|t)_B  \ar @{} [d] |{=} \ar[rr] &&P(s+1|t-1)_B\\
P(s|t)_B \ar[rr] &&P(s|t-1)_B \ar[rr]^{(-1)^{m+s+1}} &&P(s+1|t-1)_B}
\end{displaymath}
and the diagram \eqref{diag:Ftil4} commutes.
\item $t=0$:

We have the diagram
\begin{equation}\label{diag:Ftil5}\xymatrix{ P(s|0)_B \ar[r]\ar[d] &{\begin{matrix} P(s+1|0)_B\\P(s|1)_B \end{matrix}\ar[d]}\\
0 \ar[r] & P(s|0)_B}
\end{equation}
which commutes since
\begin{displaymath} \xymatrix{ P(s|0)_B \ar[rr] &&P(s|1)_B  \ar @{} [d] |{=} \ar[rr] &&P(s|0)_B\\
 &&0&&}
 \end{displaymath}
\end{enumerate}

\item [\ref{FtilBs2}]  $V=P(s|s-2)_B$:
\begin{enumerate}
\item $s>2$:

We get the diagram
\begin{equation}\label{diag:Ftil6}\xymatrix{ P(s|s-2)_B \ar[r]\ar[d] &{\begin{matrix} P(s+1|s)_A\\P(s+1|s-2)_B \end{matrix}\ar[d]}\\
P(s|s-3)_B \ar[r] &{\begin{matrix} P(s+1|s-1)_A\\P(s|s-2)_B \end{matrix}}}
\end{equation}
with
\begin{displaymath} \xymatrix{P(s|s-2)_B \ar[rr]^{(-1)^{n+m+1}} &&P(s+1|s)_A  \ar @{} [d] |{=} \ar[rr] &&P(s|s-2)_B\\
P(s|s-2)_B \ar[rr] &&P(s|s-3)_B \ar[rr]^{(-1)^{n+m+1}} &&P(s|s-2)_B}
\end{displaymath}

\begin{displaymath} \xymatrix{
P(s|s-2)_B \ar[rr]^{(-1)^{m+s+1}} &&P(s+1|s-2)_B \ar @{} [d] |{=} \ar[rr] &&P(s+1|s-3)_B\\
P(s|s-2)_B \ar[rr] &&P(s|s-3)_B \ar[rr]^{(-1)^{m+s+1}} &&P(s+1|s-3)_B}
\end{displaymath}
The diagram \eqref{diag:Ftil6} commutes.
\item $s=2$:

We get the diagram
\begin{equation}\label{diag:Ftil7}\xymatrix{ P(2|0)_B \ar[r]\ar[d] &{\begin{matrix} P(3|2)_A\\P(3|0)_B \end{matrix}\ar[d]}\\
0 \ar[r] &P(2|0)_B }
\end{equation}
with
\begin{displaymath} \xymatrix{ P(2|0)_B \ar[rr] &&P(3|2)_A  \ar @{} [d] |{=} \ar[rr] &&P(2|0)_B\\
 &&0&&}
\end{displaymath}
\end{enumerate}
\end{itemize}

Therefore, all diagrams commute.
\end{proof}
\begin{proof}[Computations for the proof of Theorem \ref{Th:elemG}]\label{Ap:ElemG}
\begin{itemize}
\item [\ref{GA}] $V=P(s|t)_A$ and $s>t+2$:

We have the following diagram:
\begin{equation}\label{diag:G1}\xymatrix{ P(s|t)_A \ar[r]\ar[d] &{\begin{matrix} P(s-1|t)_B\\P(s|t-1)_B \end{matrix}\ar[d]}\\
0 \ar[r] &{\begin{matrix} P(s|t-2)_A\\P(s-2|t)_A\\P(s-1|t-1)_A \end{matrix}}}
\end{equation}
with
\begin{displaymath} \xymatrix{P(s|t)_A \ar[rr]^{(-1)^{(s+t+1)(m+s+1)}} &&P(s-1|t)_B  \ar @{} [d] |{=} \ar[rrr]^{(-1)^{(m+t)(m+s+1)+s+t+1}} &&&P(s-1|t-1)_A\\
-(P(s|t)_A \ar[rr]^{(-1)^{(s+t+1)(m+s+1)+m+s}} &&P(s|t-1)_A \ar[rrr]^{(-1)^{(m+t+1)(m+s+1)+s+t+1}} &&&P(s-1|t)_A)}
\end{displaymath}
since the upper sign equals $(-1)^{(m+s+1)(m+s+1)+s+t+1}=(-1)^{m+t}$ and the lower $(-1)^{(m+s)(m+s+1)+m+s+1+s+t+1}=(-1)^{m+t}$
\begin{displaymath} \xymatrix{ P(s|t)_A \ar[rr] &&P(s-1|t)_B  \ar @{} [d] |{=} \ar[rr] &&P(s-2|t)_A\\
 &&0&&}
\end{displaymath}
and
\begin{displaymath} \xymatrix{ P(s|t)_A \ar[rr] &&P(s|t-1)_B  \ar @{} [d] |{=} \ar[rr] &&P(s|t-2)_A\\
 &&0&&}
\end{displaymath}
So the diagram \eqref{diag:G1} commutes.
\item [\ref{GAs2}]  $V=P(s|s-2)_A$:

We have the following diagram:
\begin{equation}\label{diag:G2}\xymatrix{ P(s|s-2)_A \ar[r]\ar[d] &{\begin{matrix} P(s|s-1)_A\\P(s|s-3)_B \end{matrix}\ar[d]}\\
0 \ar[r] &{\begin{matrix} P(s-1|s-3)_A\\P(s|s-4)_A \end{matrix}}}
\end{equation}
with
\begin{displaymath} \xymatrix{P(s|s-2)_A \ar[rr]^{(-1)^{m+s+1}} &&P(s|s-1)_A  \ar @{} [d] |{=} \ar[rr]^{-1} &&P(s-1|s-3)_A\\
-(P(s|s-2)_A \ar[rr]^{-1} &&P(s|s-3)_A \ar[rr]^{(-1)^{m+s+1+1}} &&P(s-1|s-3)_A)}
\end{displaymath}
and
\begin{displaymath} \xymatrix{ P(s|s-2)_A \ar[rr] &&P(s|s-3)_B  \ar @{} [d] |{=} \ar[rr] &&P(s|s-4)_A\\
 &&0&&}
\end{displaymath}
So the diagram \eqref{diag:G2} commutes.
\item [\ref{GAs1}]  $V=P(s|s-1)_A$:

\begin{equation}\label{diag:G3}\xymatrix{ P(s|s-1)_A \ar[r]\ar[d] &{\begin{matrix} P(s+1|s-1)_A\\P(s|s-2)_B \end{matrix}\ar[d]}\\
P(s-1|s-3)_A \ar[r] &{\begin{matrix} P(s-1|s-2)_A\\P(s|s-3)_A \end{matrix}}}
\end{equation}
with
\begin{displaymath} \xymatrix{P(s|s-1)_A \ar[rr]^{(-1)^{m+s}} &&P(s|s-2)_B  \ar @{} [d] |{=} \ar[rrr]^{(-1)^{(m+s)(m+s+1)}=-1} &&&P(s-1|s-2)_A\\
-(P(s|s-1)_A \ar[rr]^{-1} &&P(s-1|s-3)_A \ar[rrr]^{(-1)^{m+s}} &&&P(s-1|s-2)_A)}
\end{displaymath}
and
\begin{displaymath} \xymatrix{P(s|s-1)_A \ar[rr]^{(-1)^{m+s}} &&P(s|s-2)_B  \ar @{} [d] |{=} \ar[rr]^{(-1)^{m+s}} &&P(s|s-3)_A\\
P(s|s-1)_A \ar[rr]^{-1} &&P(s-1|s-3)_A \ar[rr]^{1} &&P(s|s-3)_A)}
\end{displaymath}
so the diagram \eqref{diag:G3} anticommutes.
\item [\ref{GB}] $V=P(s|t)_B$ and $s>t+2$:

We are in the situation
\begin{equation}\label{diag:G4}\xymatrix{ P(s|t)_B \ar[r]\ar[d] &{\begin{matrix} P(s|t+1)_B\\P(s+1|t)_B \end{matrix}\ar[d]}\\
{\begin{matrix} P(s-1|t)_A\\P(s|t-1)_A \end{matrix}} \ar[r] &{\begin{matrix} P(s-1|t+1)_A\\P(s|t)_A\\P(s+1|t-1)_A \end{matrix}}}
\end{equation}
with
\begin{displaymath} \xymatrix{P(s|t)_B \ar[rr]^{(-1)^{s+t}} &&P(s|t+1)_B  \ar @{} [d] |{=} \ar[rrr]^{(-1)^{(m+t+1)(m+s+1)+s+t}} &&&P(s-1|t+1)_A\\
P(s|t)_B \ar[rr]^{(-1)^{(m+t)(m+s+1)+s+t+1}} &&P(s-1|t)_A \ar[rrr]^{(-1)^{m+t+1}} &&&P(s-1|t+1)_A)}
\end{displaymath}
which are the onliest terms occurring if $s=m$. For $s<m$ we also have:
\begin{displaymath} \xymatrix{P(s|t)_B \ar[rr]^{(-1)^{m+s+1}} &&P(s+1|t)_B  \ar @{} [d] |{=} \ar[rr]^{(-1)^{(m+t)(m+s)+s+t+1}} &&P(s|t)_A\\
P(s|t)_B \ar[rr]^{(-1)^{(m+t)(m+s+1)+s+t}} &&P(s-1|t)_A \ar[rr]^{(-1)^{s+t+1}} &&P(s|t)_A)}
\end{displaymath}
and for $t>0$ one has
\begin{displaymath} \xymatrix{P(s|t)_B \ar[rr]^{(-1)^{m+s+1}} &&P(s+1|t)_B  \ar @{} [d] |{=} \ar[rr]^{(-1)^{(m+t)(m+s+1)+s+t+1}} &&P(s+1|t-1)_A\\
P(s|t)_B \ar[rr]^{(-1)^{(m+t)(m+s)+s+t}} &&P(s|t-1)_A \ar[rr]^{(-1)^{s+t+1}} &&P(s+1|t-1)_A)}
\end{displaymath}
which does not matter for $t=0$ since the last term does not occur in this case. For $t>0$ we also have
\begin{displaymath} \xymatrix{P(s|t)_B \ar[rr]^{(-1)^{s+t}} &&P(s|t+1)_B  \ar @{} [d] |{=} \ar[rr]^{(-1)^{(m+t+1)(m+s)}} &&P(s|t)_A\\
P(s|t)_B \ar[rr]^{(-1)^{(m+t)(m+s)}} &&P(s|t-1)_A \ar[rr]^{(-1)^{m+t}} &&P(s|t)_A)}
\end{displaymath}
and for $t=0$
\begin{displaymath} \xymatrix{P(s|0)_B \ar[rr] &&P(s|1)_B  \ar @{} [d] |{=} \ar[rr] &&P(s|0)_A\\
&&0&&}
\end{displaymath}
So everything anticommutes.
\item [\ref{Gs2}]  $V=P(s|s-2)_B$:
\begin{enumerate}
\item $s\neq 2$:

We are in the situation
\begin{equation}\label{diag:G5}\xymatrix{ P(s|s-2)_B \ar[r]\ar[d] &{\begin{matrix} P(s+1|s)_A\\P(s+1|s-2)_B \end{matrix}\ar[d]}\\
{\begin{matrix} P(s|s-3)_A\\P(s-1|s-2)_A\end{matrix}}\ar[r] &{\begin{matrix} P(s|s-2)_A\\P(s+1|s-3)_A\end{matrix}}}
\end{equation}
with
\begin{displaymath} \xymatrix{P(s|s-2)_B \ar[rr] &&P(s+1|s)_A  \ar @{} [d] |{=} \ar[rr]^{-1} &&P(s|s-2)_A\\
P(s|s-2)_B \ar[rr]^{(-1)^{m+s}} &&P(s|s-3)_A \ar[rr]^{(-1)^{m+s}} &&P(s|s-2)_A)}
\end{displaymath}
and
\begin{displaymath} \xymatrix{P(s|s-2)_B \ar[rr]^{(-1)^{m+s+1}} &&P(s+1|s-2)_B  \ar @{} [d] |{=} \ar[rr]^{(-1)^{(m+s)(m+s+1)+1}=-1} &&P(s+1|s-3)_A\\
P(s|s-2)_B \ar[rr]^{(-1)^{m+s}} &&P(s|s-3)_A \ar[rr]^{-1} &&P(s+1|s-3)_A)}
\end{displaymath}
and
\begin{displaymath} \xymatrix{P(s|s-2)_B \ar[rr]^{(-1)^{m+s+1}} &&P(s+1|s-2)_B  \ar @{} [d] |{=} \ar[rr]^{(-1)^{m+s+1}} &&P(s|s-2)_A\\
P(s|s-2)_B \ar[rr]^{(-1)^{(m+s)(m+s+1)}=1} &&P(s-1|s-2)_A \ar[rr]^{-1} &&P(s|s-2)_A)}
\end{displaymath}
\item $s=2$:
\begin{equation}\label{diag:G6}\xymatrix{ P(2|0)_B \ar[r]\ar[d] &{\begin{matrix} P(3|2)_A\\P(3|0)_B \end{matrix}\ar[d]}\\
P(1|0)_A\ar[r] &P(2|0)_A}
\end{equation}
with
\begin{displaymath} \xymatrix{P(2|0)_B \ar[rr]^{(-1)^{m+1}} &&P(3|0)_B  \ar @{} [d] |{=} \ar[rr]^{(-1)^{m+1}} &&P(2|0)_A\\
P(2|0)_B \ar[rr]^{1} &&P(1|0)_A \ar[rr]^{-1} &&P(2|0)_A)}
\end{displaymath}
and
\begin{displaymath} \xymatrix{P(2|0)_B \ar[rr] &&P(3|2)_A  \ar @{} [d] |{=} \ar[rr] &&P(2|0)_A\\
&&0&&}
\end{displaymath}
\end{enumerate}
\end{itemize}
\end{proof}

\begin{proof}[Computations for the proof of Theorem \ref{Th:ElemK}]\label{Ap:ElemK}
\begin{itemize}
\item [\ref{KAs1}]  $V=P(s+1|s)_A$:

where we have to compute 
\begin{displaymath} \xymatrix{ P(s+1|s)_A \ar[rr] &&P(s+1|s-1)_B  \ar[rr] &&P(s|s-2)_A
}
\end{displaymath}
and
\begin{displaymath} \xymatrix{ P(s+1|s)_A \ar[rr] &&P(s-1|s-2)_A  \ar[rr] &&P(s|s-2)_A
}
\end{displaymath}
and to check whether any of these maps is zero or not. As in Section \ref{sec:mapproj} we have to multiply the corresponding elements in $K^2_{N-1}$.

In the first case we get

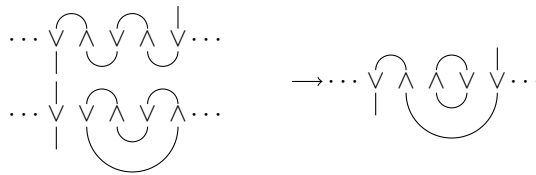
\begin{center}
{\small
		\begin{tikzpicture}
		\begin{scope}
		\draw (0,0) node[above=-1.55pt] {$\cdots$};
		\draw (0.4,0) node[above=-1.55pt] {$\down$};
		\draw (0.8, 0) node[above=-1.55pt] {$\up$};
		\draw (1.2,0) node[above=-1.55pt] {$\down$};
		\draw (1.6,0) node[above=-1.55pt] {$\up$};
		\draw (2,0) node[above=-1.55pt] {$\down$};
		\draw (2.4, 0) node[above=-1.55pt] {$\cdots$};
		\draw (0.4,0.3) arc (180:0:0.2);
		\draw (1.2,0.3) arc (180:0:0.2);
		\draw (2,0.3) -- (2,0.6);
		\draw (0.8,0) arc (180:360:0.2);
		\draw (1.6,0) arc (180:360:0.2);
		\draw (0.4,0) -- (0.4,-0.3);	
		\end{scope}
		
		\begin{scope}[yshift=-1cm]
		\draw (0,0) node[above=-1.55pt] {$\cdots$};
		\draw (0.4,0) node[above=-1.55pt] {$\down$};
		\draw (0.8, 0) node[above=-1.55pt] {$\down$};
		\draw (1.2,0) node[above=-1.55pt] {$\up$};
		\draw (1.6,0) node[above=-1.55pt] {$\down$};
		\draw (2,0) node[above=-1.55pt] {$\up$};
		\draw (2.4, 0) node[above=-1.55pt] {$\cdots$};
		\draw (0.8,0.3) arc (180:0:0.2);
		\draw (0.4,0.3) -- (0.4,0.6);
		\draw (1.6,0.3) arc (180:0:0.2);
		\draw (0.8,0) arc (180:360:0.6);
		\draw (0.4,0) -- (0.4,-0.3);
		\draw (1.2,0) arc (180:360:0.2);
			\end{scope}
		
		\draw[->] (3.5,-0.4)--(3.9,-0.4);
		\begin{scope}[yshift=-0.55cm, xshift=4.2cm]
		\draw (0,0) node[above=-1.55pt] {$\cdots$};
		\draw (0.4,0) node[above=-1.55pt] {$\down$};
		\draw (0.8, 0) node[above=-1.55pt] {$\up$};
		\draw (1.2,0) node[above=-1.55pt] {$\up$};
		\draw (1.6,0) node[above=-1.55pt] {$\down$};
		\draw (2,0) node[above=-1.55pt] {$\down$};
		\draw (2.4, 0) node[above=-1.55pt] {$\cdots$};
		\draw (0.4,0.3) arc (180:0:0.2);
		\draw (1.2,0.3) arc (180:0:0.2);
		\draw (2,0.3) -- (2,0.6);
		\draw (0.8,0) arc (180:360:0.6);
		\draw (0.4,0) -- (0.4,-0.3);
		\draw (1.2,0) arc (180:360:0.2);
											\end{scope}
				\end{tikzpicture}
\normalsize}\\
\end{center}
and for the second morphism\\		
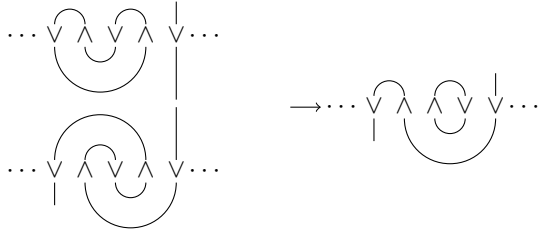
\begin{center}
\small

		\begin{tikzpicture}
		\begin{scope}
		\draw (0,0) node[above=-1.55pt] {$\cdots$};
		\draw (0.4,0) node[above=-1.55pt] {$\down$};
		\draw (0.8, 0) node[above=-1.55pt] {$\up$};
		\draw (1.2,0) node[above=-1.55pt] {$\down$};
		\draw (1.6,0) node[above=-1.55pt] {$\up$};
		\draw (2,0) node[above=-1.55pt] {$\down$};
		\draw (2.4,0) node[above=-1.55pt] {$\cdots$};
		\draw (0.4,0.3) arc (180:0:0.2);
		\draw (1.2,0.3) arc (180:0:0.2);
		\draw (2,0.3) -- (2,0.6);
		\draw (0.4,0) arc (180:360:0.6);
		\draw (0.8,0) arc (180:360:0.2);
		\draw (2,0) -- (2,-0.7);
		\end{scope}
		
				\begin{scope}[yshift=-1.8cm]
				\draw (0,0) node[above=-1.55pt] {$\cdots$};
		\draw (0.4,0) node[above=-1.55pt] {$\down$};
		\draw (0.8, 0) node[above=-1.55pt] {$\up$};
		\draw (1.2,0) node[above=-1.55pt] {$\down$};
		\draw (1.6,0) node[above=-1.55pt] {$\up$};
		\draw (2,0) node[above=-1.55pt] {$\down$};
		\draw (2.4,0) node[above=-1.55pt] {$\cdots$};
		\draw (0.8,0.3) arc (180:0:0.2);
		\draw (0.4,0.3) arc (180:0:0.6);
		\draw (2.0,0.3) -- (2.0,1);
		\draw (1.2,0) arc (180:360:0.2);
		\draw (0.8,0) arc (180:360:0.6);
		\draw (0.4,0) -- (0.4,-0.3);
		
				\end{scope}
		
		\draw[->] (3.5,-0.8)--(3.9,-0.8);
		
		\begin{scope}[yshift=-0.95cm, xshift=4.2cm]
		\draw (0,0) node[above=-1.55pt] {$\cdots$};
		\draw (0.4,0) node[above=-1.55pt] {$\down$};
		\draw (0.8, 0) node[above=-1.55pt] {$\up$};
		\draw (1.2,0) node[above=-1.55pt] {$\up$};
		\draw (1.6,0) node[above=-1.55pt] {$\down$};
		\draw (2,0) node[above=-1.55pt] {$\down$};
		\draw (2.4, 0) node[above=-1.55pt] {$\cdots$};
		\draw (0.4,0.3) arc (180:0:0.2);
		\draw (1.2,0.3) arc (180:0:0.2);
		\draw (2,0.3) -- (2,0.6);
		\draw (0.8,0) arc (180:360:0.6);
		\draw (0.4,0) -- (0.4,-0.3);
		\draw (1.2,0) arc (180:360:0.2);
											\end{scope}
														\end{tikzpicture}
\end{center} \normalsize											
so the two maps commute. We just have to check the signs:\\
\begin{displaymath} \xymatrix{P(s+1|s)_A \ar[rr]^{(-1)^{m+s+1}} &&P(s+1|s-1)_B  \ar @{} [d] |{=} \ar[rr] &&P(s|s-2)_A\\
P(s+1|s)_A \ar[rr]^{(-1)^{m+s}} &&P(s-1|s-2)_A \ar[rr]^{-1} &&P(s|s-2)_A}
\end{displaymath}
\item [\ref{KB}] $V=P(s|t)_B$ and $s>t+2$:

\begin{enumerate}
\item $t>0$:

\begin{equation}\label{diag:K2}\xymatrix{ P(s|t)_B \ar[r]\ar[d] &{\begin{matrix} P(s|t+1)_B\\P(s+1|t)_B \end{matrix}\ar[d]}\\
P(s-1|t-1)_A \ar[r] &{\begin{matrix} P(s-1|t)_A\\P(s|t-1)_A \end{matrix}}}
\end{equation}
We claim
\begin{displaymath} \xymatrix{P(s|t)_B \ar[rr]^{(-1)^{s+t}} &&P(s|t+1)_B  \ar @{} [d] |{=} \ar[rr]^{(-1)^{(m+t)(m+s)}} &&P(s-1|t)_A\\
P(s|t)_B \ar[rr]^{(-1)^{(m+t+1)(m+s)}} &&P(s-1|t-1)_A \ar[rr]^{(-1)^{m+t}} &&P(s-1|t+1)_A)}
\end{displaymath}
To prove this, we have to compute the compositions like we did in the previous case.
The morphism in the first line equals to the product\\
\begin{enumerate}
\item $s>t+3$

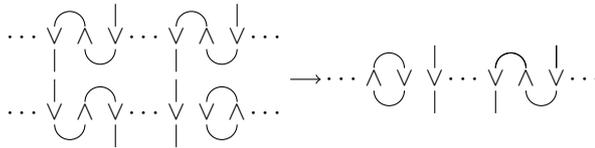
\begin{center} \small
		\begin{tikzpicture}
		\begin{scope}
		\draw (0,0) node[above=-1.55pt] {$\cdots$};
		\draw (0.4,0) node[above=-1.55pt] {$\down$};
		\draw (0.8, 0) node[above=-1.55pt] {$\up$};
		\draw (1.2,0) node[above=-1.55pt] {$\down$};
		\draw (1.6,0) node[above=-1.55pt] {$\cdots$};
		\draw (2,0) node[above=-1.55pt] {$\down$};
		\draw (2.4, 0) node[above=-1.55pt] {$\up$};
		\draw (2.8,0) node[above=-1.55pt] {$\down$};
		\draw (3.2,0) node[above=-1.55pt] {$\cdots$};
		\draw (0.4,0.3) arc (180:0:0.2);
		\draw (1.2,0.3) -- (1.2,0.6);
		\draw (2,0.3) arc (180:0:0.2);
		\draw (2.8,0.3) -- (2.8,0.6);
		\draw (0.8,0) arc (180:360:0.2);
		\draw (0.4,0) -- (0.4,-0.3);
		\draw (2.4,0) arc (180:360:0.2);
		\draw (2,0) -- (2,-0.3);
		\end{scope}
		
				\begin{scope}[yshift=-1cm]
		\draw (0,0) node[above=-1.55pt] {$\cdots$};
		\draw (0.4,0) node[above=-1.55pt] {$\down$};
		\draw (0.8, 0) node[above=-1.55pt] {$\up$};
		\draw (1.2,0) node[above=-1.55pt] {$\down$};
		\draw (1.6,0) node[above=-1.55pt] {$\cdots$};
		\draw (2,0) node[above=-1.55pt] {$\down$};
		\draw (2.4, 0) node[above=-1.55pt] {$\down$};
		\draw (2.8,0) node[above=-1.55pt] {$\up$};
		\draw (3.2,0) node[above=-1.55pt] {$\cdots$};
		\draw (0.8,0.3) arc (180:0:0.2);
		\draw (0.4,0.3) -- (0.4,0.6);
		\draw (2.4,0.3) arc (180:0:0.2);
		\draw (2,0.3) -- (2,0.6);
		\draw (0.4,0) arc (180:360:0.2);
		\draw (1.2,0) -- (1.2,-0.3);
		\draw (2.4,0) arc (180:360:0.2);
		\draw (2,0) -- (2,-0.3);		\end{scope}
		
		\draw[->] (3.5,-0.4)--(3.9,-0.4);
		
		\begin{scope}[yshift=-0.55cm, xshift=4.2cm]
		\draw (0,0) node[above=-1.55pt] {$\cdots$};
		\draw (0.4,0) node[above=-1.55pt] {$\up$};
		\draw (0.8, 0) node[above=-1.55pt] {$\down$};
		\draw (1.2,0) node[above=-1.55pt] {$\down$};
		\draw (1.6,0) node[above=-1.55pt] {$\cdots$};
		\draw (2,0) node[above=-1.55pt] {$\down$};
		\draw (2.4, 0) node[above=-1.55pt] {$\up$};
		\draw (2.8,0) node[above=-1.55pt] {$\down$};
		\draw (3.2,0) node[above=-1.55pt] {$\cdots$};
		\draw (0.4,0.3) arc (180:0:0.2);
		\draw (1.2,0.3) -- (1.2,0.6);
		\draw (2,0.3) arc (180:0:0.2);
		\draw (2.8,0.3) -- (2.8,0.6);
				\draw (0.4,0) arc (180:360:0.2);
				\draw (1.2,0) -- (1.2,-0.3);
		\draw (2,0.3) arc (180:0:0.2);
		\draw (2.8,0.3) -- (2.8,0.6);
		\draw (2.4,0) arc (180:360:0.2);
		\draw (2,0) -- (2,-0.3);
		\end{scope}
				\end{tikzpicture}	
\end{center} \normalsize
\item $s=t+3$

\begin{center} \small
		
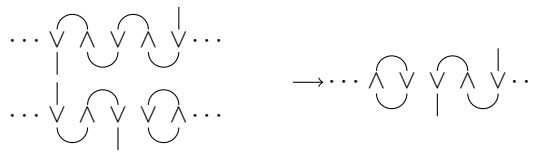
\begin{tikzpicture}
		\begin{scope}
		\draw (0,0) node[above=-1.55pt] {$\cdots$};
		\draw (0.4,0) node[above=-1.55pt] {$\down$};
		\draw (0.8, 0) node[above=-1.55pt] {$\up$};
		\draw (1.2,0) node[above=-1.55pt] {$\down$};
		\draw (1.6,0) node[above=-1.55pt] {$\up$};
		\draw (2,0) node[above=-1.55pt] {$\down$};
		\draw (2.4, 0) node[above=-1.55pt] {$\cdots$};
		\draw (0.4,0.3) arc (180:0:0.2);
		\draw (1.2,0.3) arc (180:0:0.2);
		\draw (2,0.3) -- (2,0.6);
		\draw (0.8,0) arc (180:360:0.2);
		\draw (0.4,0) -- (0.4,-0.3);
		\draw (1.6,0) arc (180:360:0.2);
		\end{scope}
		
				\begin{scope}[yshift=-1cm]
		\draw (0,0) node[above=-1.55pt] {$\cdots$};
		\draw (0.4,0) node[above=-1.55pt] {$\down$};
		\draw (0.8, 0) node[above=-1.55pt] {$\up$};
		\draw (1.2,0) node[above=-1.55pt] {$\down$};
		\draw (1.6,0) node[above=-1.55pt] {$\down$};
		\draw (2,0) node[above=-1.55pt] {$\up$};
		\draw (2.4, 0) node[above=-1.55pt] {$\cdots$};
		\draw (0.8,0.3) arc (180:0:0.2);
		\draw (0.4,0.3) -- (0.4,0.6);
		\draw (1.6,0.3) arc (180:0:0.2);
		\draw (0.4,0) arc (180:360:0.2);
		\draw (1.2,0) -- (1.2,-0.3);
		\draw (1.6,0) arc (180:360:0.2);
		\end{scope}
		
		\draw[->] (3.5,-0.4)--(3.9,-0.4);
		
		\begin{scope}[yshift=-0.55cm, xshift=4.2cm]
		\draw (0,0) node[above=-1.55pt] {$\cdots$};
		\draw (0.4,0) node[above=-1.55pt] {$\up$};
		\draw (0.8, 0) node[above=-1.55pt] {$\down$};
		\draw (1.2,0) node[above=-1.55pt] {$\down$};
		\draw (1.6,0) node[above=-1.55pt] {$\up$};
		\draw (2,0) node[above=-1.55pt] {$\down$};
		\draw (2.4, 0) node[above=-1.55pt] {$\cdots$};
				\draw (0.4,0.3) arc (180:0:0.2);
		\draw (1.2,0.3) arc (180:0:0.2);
		\draw (2,0.3) -- (2,0.6);
		\draw (0.4,0) arc (180:360:0.2);
		\draw (1.2,0) -- (1.2,-0.3);
		\draw (1.6,0) arc (180:360:0.2);
				\end{scope}
				\end{tikzpicture}	
\end{center} \normalsize
\end{enumerate}
and for the morphism in the second line we get\\
\begin{center} \small
		\begin{tikzpicture}
		\begin{scope}
		\draw (0,0) node[above=-1.55pt] {$\cdots$};
		\draw (0.4,0) node[above=-1.55pt] {$\down$};
		\draw (0.8, 0) node[above=-1.55pt] {$\up$};
		\draw (1.2,0) node[above=-1.55pt] {$\down$};
		\draw (1.6,0) node[above=-1.55pt] {$\cdots$};
		\draw (2,0) node[above=-1.55pt] {$\down$};
		\draw (2.4, 0) node[above=-1.55pt] {$\up$};
		\draw (2.8,0) node[above=-1.55pt] {$\down$};
		\draw (3.2,0) node[above=-1.55pt] {$\cdots$};
		\draw (0.8,0.3) arc (180:0:0.2);
		\draw (0.4,0.3) -- (0.4,0.6);
		\draw (2,0.3) arc (180:0:0.2);
		\draw (2.8,0.3) -- (2.8,0.6);
		\draw (0.4,0) arc (180:360:0.2);
		\draw (1.2,0) -- (1.2,-0.3);
		\draw (2,0) arc (180:360:0.2);
		\draw (2.8,0) -- (2.8,-0.3);	
		\end{scope}
		
		\begin{scope}[yshift=-1cm]
		\draw (0,0) node[above=-1.55pt] {$\cdots$};
		\draw (0.4,0) node[above=-1.55pt] {$\down$};
		\draw (0.8, 0) node[above=-1.55pt] {$\up$};
		\draw (1.2,0) node[above=-1.55pt] {$\down$};
		\draw (1.6,0) node[above=-1.55pt] {$\cdots$};
		\draw (2,0) node[above=-1.55pt] {$\down$};
		\draw (2.4, 0) node[above=-1.55pt] {$\up$};
		\draw (2.8,0) node[above=-1.55pt] {$\down$};
		\draw (3.2,0) node[above=-1.55pt] {$\cdots$};
		\draw (0.4,0.3) arc (180:0:0.2);
		\draw (1.2,0.3) -- (1.2,0.6);
		\draw (2,0.3) arc (180:0:0.2);
		\draw (2.8,0.3) -- (2.8,0.6);
		\draw (0.8,0) arc (180:360:0.2);
		\draw (0.4,0) -- (0.4,-0.3);
		\draw (2.4,0) arc (180:360:0.2);
		\draw (2,0) -- (2,-0.3);
			\end{scope}
		
		\draw[->] (3.5,-0.4)--(3.9,-0.4);
		\begin{scope}[yshift=-0.55cm, xshift=4.2cm]
		\draw (0,0) node[above=-1.55pt] {$\cdots$};
		\draw (0.4,0) node[above=-1.55pt] {$\down$};
		\draw (0.8, 0) node[above=-1.55pt] {$\up$};
		\draw (1.2,0) node[above=-1.55pt] {$\down$};
		\draw (1.6,0) node[above=-1.55pt] {$\cdots$};
		\draw (2,0) node[above=-1.55pt] {$\down$};
		\draw (2.4, 0) node[above=-1.55pt] {$\up$};
		\draw (2.8,0) node[above=-1.55pt] {$\down$};
		\draw (3.2,0) node[above=-1.55pt] {$\cdots$};
		\draw (0.8,0.3) arc (180:0:0.2);
		\draw (0.4,0.3) -- (0.4,0.6);
		\draw (2,0.3) arc (180:0:0.2);
		\draw (2.8,0.3) -- (2.8,0.6);
		\draw (0.8,0) arc (180:360:0.2);
		\draw (0.4,0) -- (0.4,-0.3);
		\draw (2.4,0) arc (180:360:0.2);
		\draw (2,0) -- (2,-0.3);
							\end{scope}
				\end{tikzpicture}\\
\end{center} \normalsize
so all these maps are unequal to zero and therefore the diagram \eqref{diag:K2} commutes.
\item $t=0$:

We have the situation
\begin{equation}\label{diag:K3}\xymatrix{ P(s|0)_B \ar[r]\ar[d] &{\begin{matrix} P(s+1|0)_B\\P(s|1)_B \end{matrix}\ar[d]}\\
0\ar[r] &P(s-1|0)_A}
\end{equation}
Changing the images from above to the appropriate case one easily checks that
\begin{displaymath} \xymatrix{P(s|0)_B \ar[rr] &&P(s|1)_B  \ar @{} [d] |{=} \ar[rr] &&P(s-1|0)_A\\
&&0&&}
\end{displaymath} and therefore the diagram \eqref{diag:K3} commutes.
\end{enumerate}
\item [\ref{Ks2}]  $V=P(s|s-2)_B$:

\begin{enumerate}
\item $s\neq 2$:

We are in the situation
\begin{equation}\label{diag:K4}\xymatrix{ P(s|s-2)_B \ar[r]\ar[d] &{\begin{matrix} P(s+1|s)_A\\P(s+1|s-2)_B \end{matrix}\ar[d]}\\
P(s-1|s-3)_A \ar[r] &{\begin{matrix} P(s-1|s-2)_A\\P(s|s-3)_A\end{matrix}}}
\end{equation}
We first show that
\begin{displaymath} \xymatrix{P(s|s-2)_B \ar[rr] &&P(s+1|s)_A  \ar @{} [d] |{=} \ar[rr]^{(-1)^{m+s}} &&P(s-1|s-2)_A\\
P(s|s-2)_B \ar[rr]^{1} &&P(s-1|s-3)_A \ar[rr]^{(-1)^{m+s}} &&P(s-1|s-2)_A)}
\end{displaymath}
We compute the products belonging to the compositions:
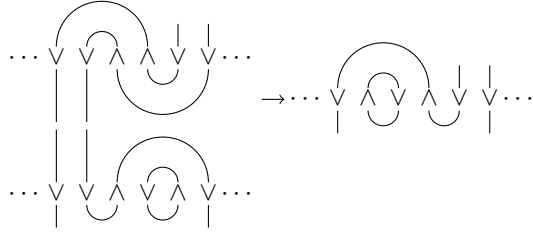
\begin{center}
\small
		\begin{tikzpicture}
		\begin{scope}
		\draw (0,0) node[above=-1.55pt] {$\cdots$};
		\draw (0.4,0) node[above=-1.55pt] {$\down$};
		\draw (0.8, 0) node[above=-1.55pt] {$\down$};
		\draw (1.2,0) node[above=-1.55pt] {$\up$};
		\draw (1.6,0) node[above=-1.55pt] {$\up$};
		\draw (2,0) node[above=-1.55pt] {$\down$};
		\draw (2.4,0) node[above=-1.55pt] {$\down$};
		\draw (2.8,0) node[above=-1.55pt] {$\cdots$};
		\draw (0.4,0.3) arc (180:0:0.6);
		\draw (0.8,0.3) arc (180:0:0.2);
		\draw (2,0.3)--(2,0.6);
		\draw (2.4,0.3)--(2.4,0.6);
		\draw (1.6,0) arc (180:360:0.2);
		\draw (1.2,0) arc (180:360:0.6);
		\draw (0.4,0)--(0.4,-0.7);
		\draw (0.8,0)--(0.8,-0.7);
		
		\end{scope}
		
				\begin{scope}[yshift=-1.8cm]
		\draw (0,0) node[above=-1.55pt] {$\cdots$};
		\draw (0.4,0) node[above=-1.55pt] {$\down$};
		\draw (0.8, 0) node[above=-1.55pt] {$\down$};
		\draw (1.2,0) node[above=-1.55pt] {$\up$};
		\draw (1.6,0) node[above=-1.55pt] {$\down$};
		\draw (2,0) node[above=-1.55pt] {$\up$};
		\draw (2.4,0) node[above=-1.55pt] {$\down$};
		\draw (2.8,0) node[above=-1.55pt] {$\cdots$};
		\draw (1.2,0.3) arc (180:0:0.6);
		\draw (1.6,0.3) arc (180:0:0.2);
		\draw (0.4,0.3) -- (0.4,1);
		\draw (0.8,0.3) -- (0.8,1);
		\draw (0.4,0) -- (0.4,-0.3);
		\draw (0.8,0) arc (180:360:0.2);
		\draw (1.6,0) arc (180:360:0.2);
		\draw (2.4,0) -- (2.4,-0.3);
		\end{scope}
		
		\draw[->] (3.1,-0.4)--(3.4,-0.4);
		
		\begin{scope}[yshift=-0.55cm, xshift=3.7cm]
		\draw (0,0) node[above=-1.55pt] {$\cdots$};
		\draw (0.4,0) node[above=-1.55pt] {$\down$};
		\draw (0.8, 0) node[above=-1.55pt] {$\up$};
		\draw (1.2,0) node[above=-1.55pt] {$\down$};
		\draw (1.6,0) node[above=-1.55pt] {$\up$};
		\draw (2,0) node[above=-1.55pt] {$\down$};
		\draw (2.4,0) node[above=-1.55pt] {$\down$};
		\draw (2.8,0) node[above=-1.55pt] {$\cdots$};
		\draw (0.4,0.3) arc (180:0:0.6);
		\draw (0.8,0.3) arc (180:0:0.2);
		\draw (2,0.3)--(2,0.6);
		\draw (2.4,0.3)--(2.4,0.6);		\draw (0.4,0) -- (0.4,-0.3);
		\draw (0.8,0) arc (180:360:0.2);
		\draw (1.6,0) arc (180:360:0.2);
		\draw (2.4,0) -- (2.4,-0.3);		\end{scope}

			\end{tikzpicture}\\
			\normalsize
			\end{center}
and
\begin{center}
\small
		
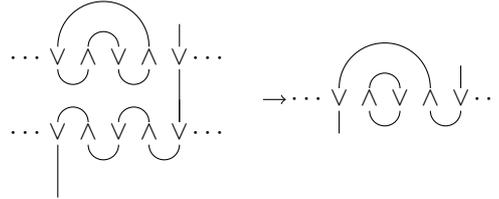
\begin{tikzpicture}
		\begin{scope}
		\draw (0,0) node[above=-1.55pt] {$\cdots$};
		\draw (0.4,0) node[above=-1.55pt] {$\down$};
		\draw (0.8, 0) node[above=-1.55pt] {$\up$};
		\draw (1.2,0) node[above=-1.55pt] {$\down$};
		\draw (1.6,0) node[above=-1.55pt] {$\up$};
		\draw (2,0) node[above=-1.55pt] {$\down$};
		\draw (2.4,0) node[above=-1.55pt] {$\cdots$};
		\draw (0.4,0.3) arc (180:0:0.6);
		\draw (0.8,0.3) arc (180:0:0.2);
		\draw (2,0.3)--(2,0.6);
		\draw (0.4,0) arc (180:360:0.2);
		\draw (1.2,0) arc (180:360:0.2);
		\draw (2,0)--(2,-0.7);
		
		\end{scope}
		\begin{scope}[yshift=-1cm]
		\draw (0,0) node[above=-1.55pt] {$\cdots$};
		\draw (0.4,0) node[above=-1.55pt] {$\down$};
		\draw (0.8, 0) node[above=-1.55pt] {$\up$};
		\draw (1.2,0) node[above=-1.55pt] {$\down$};
		\draw (1.6,0) node[above=-1.55pt] {$\up$};
		\draw (2,0) node[above=-1.55pt] {$\down$};
		\draw (2.4,0) node[above=-1.55pt] {$\cdots$};
		\draw (0.4,0.3) arc (180:0:0.2);
		\draw (1.2,0.3) arc (180:0:0.2);
		\draw (2,0.3)--(2,0.6);
		\draw (0.8,0) arc (180:360:0.2);
		\draw (1.6,0) arc (180:360:0.2);
		\draw (0.4,0)--(0.4,-0.7);	
		\end{scope}
		
		\draw[->] (3.1,-0.4)--(3.4,-0.4);
		
		\begin{scope}[yshift=-0.55cm, xshift=3.7cm]
		\draw (0,0) node[above=-1.55pt] {$\cdots$};
		\draw (0.4,0) node[above=-1.55pt] {$\down$};
		\draw (0.8, 0) node[above=-1.55pt] {$\up$};
		\draw (1.2,0) node[above=-1.55pt] {$\down$};
		\draw (1.6,0) node[above=-1.55pt] {$\up$};
		\draw (2,0) node[above=-1.55pt] {$\down$};
		\draw (2.4,0) node[above=-1.55pt] {$\cdots$};
		\draw (0.4,0.3) arc (180:0:0.6);
		\draw (0.8,0.3) arc (180:0:0.2);
		\draw (2,0.3)--(2,0.6);		
		\draw (0.4,0) -- (0.4,-0.3);
		\draw (0.8,0) arc (180:360:0.2);
		\draw (1.6,0) arc (180:360:0.2);
				\end{scope}
		\end{tikzpicture}
				\normalsize
				\end{center}
Since none of them is zero, the assumption holds.\\
Secondly we have to check that
\begin{displaymath} \xymatrix{P(s|s-2)_B \ar[rr]^{(-1)^{m+s+1}} &&P(s+1|s-2)_B  \ar @{} [d] |{=} \ar[rr]^{(-1)^{m+s+1}} &&P(s|s-3)_A\\
P(s|s-2)_B \ar[rr]^{1} &&P(s-1|s-3)_A \ar[rr]^{1} &&P(s|s-3)_A)}
\end{displaymath}
We compute the products belonging to the compositions:

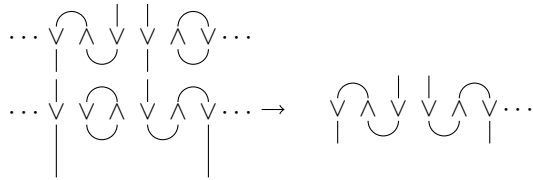
\begin{center}
\small
		\begin{tikzpicture}
		\begin{scope}
		\draw (0,0) node[above=-1.55pt] {$\cdots$};
		\draw (0.4,0) node[above=-1.55pt] {$\down$};
		\draw (0.8, 0) node[above=-1.55pt] {$\up$};
		\draw (1.2,0) node[above=-1.55pt] {$\down$};
		\draw (1.6,0) node[above=-1.55pt] {$\down$};
		\draw (2,0) node[above=-1.55pt] {$\up$};
		\draw (2.4,0) node[above=-1.55pt] {$\down$};
		\draw (2.8,0) node[above=-1.55pt] {$\cdots$};
		\draw (0.4,0.3) arc (180:0:0.2);
		\draw (2,0.3) arc (180:0:0.2);
		\draw (1.2,0.3)--(1.2,0.6);
		\draw (1.6,0.3)--(1.6,0.6);
		\draw (0.8,0) arc (180:360:0.2);
		\draw (2,0) arc (180:360:0.2);
		\draw (0.4,0)--(0.4,-0.3);
		\draw (1.6,0)--(1.6,-0.3);
		
		\end{scope}

		\begin{scope}[yshift=-1cm]
		\draw (0,0) node[above=-1.55pt] {$\cdots$};
		\draw (0.4,0) node[above=-1.55pt] {$\down$};
		\draw (0.8, 0) node[above=-1.55pt] {$\down$};
		\draw (1.2,0) node[above=-1.55pt] {$\up$};
		\draw (1.6,0) node[above=-1.55pt] {$\down$};
		\draw (2,0) node[above=-1.55pt] {$\up$};
		\draw (2.4,0) node[above=-1.55pt] {$\down$};
		\draw (2.8,0) node[above=-1.55pt] {$\cdots$};
		\draw (0.8,0.3) arc (180:0:0.2);
		\draw (2,0.3) arc (180:0:0.2);
		\draw (1.6,0.3)--(1.6,0.6);
				\draw (0.4,0.3)--(0.4,0.6);
		\draw (0.8,0) arc (180:360:0.2);
		\draw (1.6,0) arc (180:360:0.2);
		\draw (0.4,0)--(0.4,-0.7);	
				\draw (2.4,0)--(2.4,-0.7);
		\end{scope}
		
		\draw[->] (3.1,-0.8)--(3.4,-0.8);
		
		\begin{scope}[yshift=-0.95cm, xshift=3.7cm]
		\draw (0.4,0) node[above=-1.55pt] {$\down$};
		\draw (0.8, 0) node[above=-1.55pt] {$\up$};
		\draw (1.2,0) node[above=-1.55pt] {$\down$};
		\draw (1.6,0) node[above=-1.55pt] {$\down$};
		\draw (2,0) node[above=-1.55pt] {$\up$};
		\draw (2.4,0) node[above=-1.55pt] {$\down$};
		\draw (2.8,0) node[above=-1.55pt] {$\cdots$};
		\draw (0.4,0.3) arc (180:0:0.2);
		\draw (2,0.3) arc (180:0:0.2);
		\draw (1.2,0.3)--(1.2,0.6);
		\draw (1.6,0.3)--(1.6,0.6);
		\draw (0.8,0) arc (180:360:0.2);
		\draw (1.6,0) arc (180:360:0.2);
		\draw (0.4,0)--(0.4,-0.3);	
				\draw (2.4,0)--(2.4,-0.3);
								\end{scope}

			\end{tikzpicture}\\
			\normalsize
			\end{center}
and
\begin{center}
\small
		
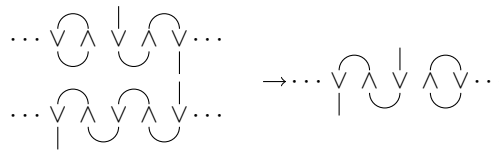
\begin{tikzpicture}
		\begin{scope}
		\draw (0,0) node[above=-1.55pt] {$\cdots$};
		\draw (0.4,0) node[above=-1.55pt] {$\down$};
		\draw (0.8, 0) node[above=-1.55pt] {$\up$};
		\draw (1.2,0) node[above=-1.55pt] {$\down$};
		\draw (1.6,0) node[above=-1.55pt] {$\up$};
		\draw (2,0) node[above=-1.55pt] {$\down$};
		\draw (2.4,0) node[above=-1.55pt] {$\cdots$};
		\draw (0.4,0.3) arc (180:0:0.2);
		\draw (1.6,0.3) arc (180:0:0.2);
		\draw (1.2,0.3)--(1.2,0.6);
		\draw (0.4,0) arc (180:360:0.2);
		\draw (1.2,0) arc (180:360:0.2);
		\draw (2,0)--(2,-0.3);
		
		\end{scope}
		\begin{scope}[yshift=-1cm]
		\draw (0,0) node[above=-1.55pt] {$\cdots$};
		\draw (0.4,0) node[above=-1.55pt] {$\down$};
		\draw (0.8, 0) node[above=-1.55pt] {$\up$};
		\draw (1.2,0) node[above=-1.55pt] {$\down$};
		\draw (1.6,0) node[above=-1.55pt] {$\up$};
		\draw (2,0) node[above=-1.55pt] {$\down$};
		\draw (2.4,0) node[above=-1.55pt] {$\cdots$};
		\draw (0.4,0.3) arc (180:0:0.2);
		\draw (1.2,0.3) arc (180:0:0.2);
		\draw (2,0.3)--(2,0.6);
		\draw (0.8,0) arc (180:360:0.2);
		\draw (1.6,0) arc (180:360:0.2);
		\draw (0.4,0)--(0.4,-0.3);	
		\end{scope}
		
		\draw[->] (3.1,-0.4)--(3.4,-0.4);
		
		\begin{scope}[yshift=-0.55cm, xshift=3.7cm]
		\draw (0,0) node[above=-1.55pt] {$\cdots$};
		\draw (0.4,0) node[above=-1.55pt] {$\down$};
		\draw (0.8, 0) node[above=-1.55pt] {$\up$};
		\draw (1.2,0) node[above=-1.55pt] {$\down$};
		\draw (1.6,0) node[above=-1.55pt] {$\up$};
		\draw (2,0) node[above=-1.55pt] {$\down$};
		\draw (2.4,0) node[above=-1.55pt] {$\cdots$};
		\draw (0.4,0.3) arc (180:0:0.2);
		\draw (1.6,0.3) arc (180:0:0.2);
		\draw (1.2,0.3)--(1.2,0.6);		
		\draw (0.4,0) -- (0.4,-0.3);
		\draw (0.8,0) arc (180:360:0.2);
		\draw (1.6,0) arc (180:360:0.2);
				\end{scope}
		\end{tikzpicture}
				\normalsize
				\end{center}
Since none of them is zero, the assumption holds.\\
Therefore, the diagram \eqref{diag:K4} commutes.
\item $s=2$:

We are in the situation
\begin{equation}\label{diag:K5}\xymatrix{ P(2|0)_B \ar[r]\ar[d] &{\begin{matrix} P(3|2)_A\\P(3|0)_B \end{matrix}\ar[d]}\\
0 \ar[r] &P(1|0)}
\end{equation}
but by \ref{prosimp} there is no degree 3 morphism $P(2|0) \to P(1|0)$, so the above diagram commutes. \qedhere
\end{enumerate}
\end{itemize}
\end{proof}
\small
\chapter{Computations of the multiplications in Section \ref{sec:multipl}\label{Ap:Mult}}
\begin{proof}[Proof of Theorem \ref{Th:Multtable}]
For computing the products, we have to multiply the maps on each projective occurring. Sometimes it might be helpful to use the associativity of the algebra instead.
\begin{enumerate}
\item Multiplication with $\Id^{(n|m)}_{(k|l)}$ from the left:
\begin{itemize}
\item 
\begin{myalign}
&\Id^{(n|m)}_{(k|l)} \cdot \Id^{(k|l)}_{(a|b)}\\
= 
&\begin{cases}P(s|t)_A &\to (-1)^{(n+k)(l+t)} P(s|t)_A \\ &\to (-1)^{(n+k)(l+t)+(k+a)(b+t)}  P(s|t)_A \\ P(s|t)_B &\to(-1)^{(n+k)(l+s)} P(s|t)_B \\&\to(-1)^{(n+k)(l+s)+(k+a)(b+s)} P(s|t)_B\end{cases} 
\end{myalign}	
To compute the sign of the product compared to the one of $\Id^{(n|m)}_{(a|b)}$ we have to compute the difference of both modulo $2$. We get
\begin{align*}
&(n+k)(l+t)+(k+a)(b+t)-((n+a)(b+t)) &\pmod 2 \\
\equiv &(n+k)(l+t) +(n+k)(b+t) &\pmod 2\\
\equiv &(n+k)(b+l) &\pmod 2
\end{align*}
and similarly by substituting $t$ by $s$
\begin{align*}
&(n+k)(l+s)+(k+a)(b+s)-((n+a)(b+s)) &\pmod 2 \\
\equiv &(n+k)(b+l) &\pmod 2
\end{align*}
This yields to 
$$Id^{(n|m)}_{(k|l)} \cdot \Id^{(k|l)}_{(a|b)}= (-1)^{(n+k)(l+b)} Id^{(n|m)}_{(a|b)}.$$
\begin{align*}
\Id^{(n|m)}_{(k|l)} \cdot F^{(k|l)}_{(a|b)}\\
&=\Id^{(n|m)}_{(k|l)} \cdot \Id^{(k|l)}_{(a|b+1)} \cdot F^{(a|b+1)}_{(a|b)} \\
&=(-1)^{(n+k)(b+l+1)} \Id^{(n|m)}_{(a|b+1)} \cdot F^{(a|b+1)}_{(a|b)} &\text{(associativity)}\\
&=(-1)^{(n+k)(b+l+1)} F^{(n|m)}_{(a|b)}
\end{align*}
\item Compute
\begin{align*}
\Id^{(n|m)}_{(k|l)} \cdot \widetilde F^{(k|l)}_{(a|b)}\\
&=\Id^{(n|m)}_{(k|l)} \cdot \Id^{(k|l)}_{(a+1|b)} \cdot \widetilde F^{(a+1|b)}_{(a|b)} \\
&=(-1)^{(n+k)(b+l)} \Id^{(n|m)}_{(a+1|b)} \cdot \widetilde F^{(a+1|b)}_{(a|b)} \\
&=(-1)^{(n+k)(b+l)} \widetilde F^{(n|m)}_{(a|b)}
\end{align*}
\item Similar,

\begin{align*}
\Id^{(n|m)}_{(k|l)} \cdot G^{(k|l)}_{(a|b)}\\
&=\Id^{(n|m)}_{(k|l)} \cdot \Id^{(k|l)}_{(a+2|a+1)} \cdot G^{(a+2|a+1)}_{(a|b)} \\
&=(-1)^{(n+k)(a+l+1)} \Id^{(n|m)}_{(a+2|a+1)} \cdot G^{(a+2|a+1)}_{(a|b)} \\
&=(-1)^{(n+k)(a+l+1)} G^{(n|m)}_{(a|b)}
\end{align*}
\item and
\begin{align*}
\Id^{(n|m)}_{(k|l)} \cdot K^{(k|l)}_{(a|b)}\\
&=\Id^{(n|m)}_{(k|l)} \cdot \Id^{(k|l)}_{(a+2|a+1)} \cdot K^{(a+2|a+1)}_{(a|b)} \\
&=(-1)^{(n+k)(a+l+1)} \Id^{(n|m)}_{(a+2|a+1)} \cdot K^{(a+2|a+1)}_{(a|b)} \\
&=(-1)^{(n+k)(a+l+1)} K^{(n|m)}_{(a|b)}
\end{align*}
\item Using the formulas computed for $F^{(k|l)}_{(a|b)}$, we get
\begin{align*}
\Id^{(n|m)}_{(k|l)} \cdot J^{(k|l)}_{(a|b)}\\
&=\Id^{(n|m)}_{(k|l)} \cdot F^{(k|l)}_{(a+1|b)} \cdot \widetilde F^{(a+1|b)}_{(a|b)} \\
&=(-1)^{(n+k)(b+l+1)} F^{(n|m)}_{(a+1|b)} \cdot \widetilde F^{(a+1|b)}_{(a|b)} \\
&=(-1)^{(n+k)(b+l+1)} J^{(n|m)}_{(a|b)}
\end{align*}
\end{itemize}
\item Multiplication from the left with $ F^{(n|m)}_{(k|l)}$

For this we mostly will have to write down the products on the projectives:

\begin{itemize}
\item Multiplying with $\Id^{(k|l)}_{(a|b)}$ from the right only adds a summand to the exponent of $(-1)$. Writing down the appropriate signs, we obtain:
\begin{align*}
&F^{(n|m)}_{(k|l)} \cdot \Id^{(k|l)}_{(a|b)}\\
=& \begin{cases} P(s|t)_A \to (-1)^{(n+k)(l+t+1)+(a+k)(b+t+1)} P(s|t-1)_A \\
								P(s+2|s)_A \to(-1)^{(n+k)(l+s+1)+k+s+(a+k)(b+s)} P(s+1|s)_A  \\
								P(s+1|s)_A \to (-1)^{(n+k)(l+s+1)+k+s+(a+k)(b+s)} P(s|s-2)_B  \\
								P(s|t)_B \to (-1)^{(n+k)(l+s+1)+(a+k)(b+s+1)} P(s-1|t)_B \\
								P(s|s-2)_B \to (-1)^{(n+k)(l+s+1)+(a+k)(b+s+1)} P(s|s-1)_A
\end{cases}
\end{align*}
Adding the appropriate signs of $F^{(n|m)}_{(a|b)}$ and computing modulo $2$ one gets (we only compute the first one, it is easy to see that the others yield the same result)
\begin{align*}
&(n+k)(l+t+1)+(a+k)(b+t+1)+(n+a)(b+t+1) \\
\equiv& (n+k)(l+t+1)+(n+k)(b+t+1) \\
\equiv& (n+k)(b+l) \pmod 2
\end{align*}
So we have shown that 
$$F^{(n|m)}_{(k|l)} \cdot \Id^{(k|l)}_{(a|b)} = (-1)^{(n+k)(b+l)} F^{(n|m)}_{(a|b)}.$$
\item Now we can compute $F^{(n|m)}_{(k|l)} \cdot F^{(k|l)}_{(a|b)}$, where we have to take care which of the maps will go to zero and which will not.
We have to look for the following cases:
\begin{enumerate}
\item $P(s|t)_A$ and $s>t+2$

There we get
 \begin{displaymath} \xymatrix{ P(s|t)_A \ar[rr] &&P(s|t-1)_A   \ar @{} [d] |{=} \ar[rr] &&P(s|t-2)_A\\
 &&0&&}
\end{displaymath} for $s \neq t+1$, especially in the given case.
\item $P(s|s+2)_A$

First we get the above composition, but in addition we have
 \begin{displaymath} \xymatrix{ P(s+2|s)_A \ar[rr] &&P(s+1|s)_A   \ar @{} [d] |{=} \ar[rr] &&P(s|s-2)_B\\
 &&0&&}
\end{displaymath}
and 
 \begin{displaymath} \xymatrix{ P(s+2|s)_A \ar[rr]^{(-1)^{(n+k)(l+s+1)+k+s}} &&P(s+1|s)_A  \ar[rr]^{(-1)^{(k+a)(b+s)}} &&P(s+1|s-1)_A}
\end{displaymath}
\item $P(s+1|s)_A$

We have
\begin{displaymath}
\xymatrix{ P(s+1|s)_A \ar[rr]^{(-1)^{(n+k)(l+s+1)}} &&P(s+1|s-1)_A  \ar[rr]^{(-1)^{(a+k)(b+s)}} &&P(s+1|s-2)_A}
\end{displaymath}
and 
\begin{displaymath}
\xymatrix{ P(s+1|s)_A \ar[rr]^{(-1)^{(n+k)(l+s+1)}} &&P(s+1|s-1)_A \ar @{} [d] |{=} \ar[rr]^{(-1)^{(a+k)(b+s)+a+s+1}} &&P(s|s-1)_A\\ -(P(s+1|s)_A \ar[rr]^{(-1)^{(n+k)(l+s+1)+k+s}} &&P(s|s-2)_B \ar[rr]^{(-1)^{(a+k)(b+s+1)}} &&P(s|s-1)_A)}
\end{displaymath}
\item $P(s|t)_B$, $s>t+3$

We have
 \begin{displaymath} \xymatrix{ P(s|t)_B \ar[rr] &&P(s-1|t)_B   \ar @{} [d] |{=} \ar[rr] &&P(s-2|t)_B\\
 &&0&&}
\end{displaymath}
\item $P(s+1|s-2)_B$

\begin{displaymath}
\xymatrix{ P(s+1|s-2)_B \ar[rr]^{(-1)^{(n+k)(l+s)}} &&P(s|s-2)_B  \ar[rr]^{(-1)^{(a+k)(b+s+1)}} &&P(s|s-1)_A}
\end{displaymath}
\item $P(s|s-2)_B$

We have two maps
\begin{displaymath}
\xymatrix{ P(s|s-2)_B \ar[rr]^{(-1)^{(n+k)(l+s+1)}} &&P(s|s-1)_A  \ar[rr]^{(-1)^{(a+k)(b+s)}} &&P(s|s-2)_A}
\end{displaymath}
and
\begin{displaymath}
\xymatrix{ P(s|s-2)_B \ar[rr]^{(-1)^{(n+k)(l+s+1)}} &&P(s|s-1)_A  \ar[rr]^{(-1)^{(a+k)(b+s)+a+s+1}} &&P(s-1|s-3)_B}
\end{displaymath}
\end{enumerate}
Taking all nonzero maps together, we get

\begin{align*}
&F^{(n|m)}_{(k|l)} \cdot F^{(k|l)}_{(a|b)}\\
&= \begin{cases}
								P(s+1|s)_A \to (-1)^{(n+k)(l+s+1)+(a+k)(b+s)} P(s+1|s-2)_A  \\
								P(s+2|s)_A \to (-1)^{(n+k)(l+s+1)+k+s+(a+k)(b+s)} P(s+1|s-1)_A  \\
								P(s+1|s-2)_B \to (-1)^{(n+k)(l+s)+(a+k)(b+s+1)} P(s|s-1)_A  \\
								P(s|s-2)_B \to (-1)^{(n+k)(l+s+1)+(a+k)(b+s)+a+s+1} P(s-1|s-3)_B  \\
								+ (-1)^{(n+k)(l+s+1)+(a+k)(b+s)} P(s|s-2)_A \\ 
								\end{cases}	\\
&=(-1)^{(n+k)(l+b+1)}A^{(n|m)}_{(a|b)}															
\end{align*}
For the last equation one just compares the signs similar to above.
\item To multiply with $\widetilde{F}$ we can use the results from above and get:
\begin{align*}
F^{(n|m)}_{(k|l)} \cdot \widetilde{F}^{(k|l)}_{(a|b)}\\
&=F^{(n|m)}_{(k|l)} \cdot \Id^{(k|l)}_{(a+1|b)} \cdot \widetilde F^{(a+1|b)}_{(a|b)} \\
&=(-1)^{(n+k)(b+l)} F^{(n|m)}_{(a+1|b)} \cdot \widetilde F^{(a+1|b)}_{(a|b)} \\
&=(-1)^{(n+k)(b+l)} J^{(n|m)}_{(a|b)}
\end{align*}
\item For the multiplication with $G^{(k|l)}_{(a|b)}$ we have to consider the following cases
\begin{enumerate}
\item $P(s|t)_A$

We compute
 \begin{displaymath} \xymatrix{ P(s|t)_A \ar[rr] &&P(s|t-1)_A   \ar @{} [d] |{=} \ar[rr] &&0\\
 &&0&&}
 \end{displaymath}
\item $P(s+2|s)_A$

Here we also get 
\begin{displaymath} \xymatrix{ P(s+2|s)_A \ar[rr] &&P(s+1|s)_A   \ar @{} [d] |{=} \ar[rr] &&P(s|s-2)_A\\
 &&0&&}
 \end{displaymath}
\item $P(s+1|s)_A$ 

In addition to the first equation we get
\begin{displaymath}
\xymatrix{ P(s+1|s)_A \ar[rr]^{(-1)^{(n+k)(l+s+1)+k+s}} &&P(s|s-2)_B  \ar[rr]^{(-1)^{(a+s+1)(k+s)}} &&P(s-1|s-2)_B}
\end{displaymath}
since for the second summand we get
\begin{displaymath} \xymatrix{ P(s+1|s)_A \ar[rr] &&P(s|s-2)_B   \ar @{} [d] |{=} \ar[rr] &&P(s|s-3)_A\\
 &&0&&}
 \end{displaymath}
\item $P(s|t)_B s>t+2$

Here we have two summands, for the first we get
\begin{displaymath} \xymatrix{ P(s|t)_B \ar[rr] &&P(s-1|t)_B   \ar @{} [d] |{=} \ar[rr] &&P(s-2|t)_B\\
 &&0&&}
 \end{displaymath}
and for the second
\begin{displaymath}
\xymatrix{ P(s|t)_B \ar[rr]^{(-1)^{(n+k)(l+s+1)}} &&P(s-1|t)_B  \ar[rr]^{(-1)^{(a+s)(k+t+1)+s+t+1}} &&P(s-1|t-1)_B}
\end{displaymath}
\item $P(s|s-2)$
\begin{displaymath}
\xymatrix{ P(s|s-2)_B \ar[rr]^{(-1)^{(n+k)(l+s+1)}} &&P(s|s-1)_A  \ar[rr]^{(-1)^{(a+k)(k+s+1)+1}} &&P(s-1|s-3)_B}
\end{displaymath}
\end{enumerate}
Putting all of them together, we obtain
\begin{align*}
&F^{(n|m)}_{(k|l)} \cdot G^{(k|l)}_{(a|b)}\\
&=  \begin{cases} P(s|t)_A \to 0 s>t+1 \\
								P(s|s-1)_A \to \\(-1)^{(n+k)(l+s)+k+s+1+(a+s)(k+s+1)} P(s-2|s-3)_A  \\
								P(s|t)_B \to \\(-1)^{(n+k)(l+s+1)+(a+s)(k+t+1)+s+t+1} P(s-1|t-1)_A \\ 
								\end{cases}\\
&=(-1)^{(n+k)(a+l)+a+k+1} K^{(n|m)}_{(a|b)}															
\end{align*}
\item  For multiplication with $K^{(k|l)}_{(a|b)}$ we can show that no map can exist, since
$$F^{(n|m)}_{(k|l)} \cdot K^{(k|l)}_{(a|b)}\in \hom^{(n+m)-(a+b)-5}(P_\bullet(n|m),P_\bullet(a|b)\langle (n+m)-(a+b)-8 \rangle)$$
and by Lemma \ref{Cor:posschmaps} this cannot exist, since $-3 \nleq -5.$
So by this the composition is
$$F^{(n|m)}_{(k|l)} \cdot K^{(k|l)}_{(a|b)}=0.$$ 

\item Now we compute $F^{(n|m)}_{(k|l)} \cdot J^{(k|l)}_{(a|b)}$.
\begin{enumerate}
\item $P(s|t)_A$ $s>t+1$

 \begin{displaymath} \xymatrix{ P(s|t)_A \ar[rr] &&P(s|t-1)_A   \ar @{} [d] |{=} \ar[rr] &&P(s-1|t-2)_A\\
 &&0&&}
 \end{displaymath}
\item $P(s+2|s)_A$

Here in addition we have 
 \begin{displaymath} \xymatrix{ P(s+2|s)_A \ar[rr] &&P(s+1|s)_A   \ar @{} [d] |{=} \ar[rr] &&P(s|s-1)_A\\
 &&0&&}
 \end{displaymath}
\item $P(s+1|s)_A$ 

First we have
\begin{displaymath}
\xymatrix{ P(s+1|s)_A \ar[rr]^{(-1)^{(n+k)(l+s+1)}} &&P(s+1|s-1)_A  \ar[rr]^{(-1)^{(a+k+1)(b+s)}} &&P(s|s-2)_A}
\end{displaymath}
and 
 \begin{displaymath} \xymatrix{ P(s+1|s)_A \ar[rr] &&P(s|s-2)_B   \ar @{} [d] |{=} \ar[rr] &&P(s-1|s-3)_B\\
 &&0&&}
 \end{displaymath}
\item $P(s|t)_B$, $s>t+2$
 \begin{displaymath} \xymatrix{ P(s|t)_B \ar[rr] &&P(s-1|t)_B   \ar @{} [d] |{=} \ar[rr] &&P(s-2|t-1)_B\\
 &&0&&}
 \end{displaymath}
\item $P(s+1|s-1)_B$
\begin{displaymath}
\xymatrix{ P(s+1|s-1)_B \ar[rr]^{(-1)^{(n+k)(l+s)}} &&P(s+1|s)_A  \ar[rr]^{(-1)^{(a+k+1)(b+s+1)}} &&P(s|s-1)_A}
\end{displaymath} 
\end{enumerate}
All in all, we get
\begin{align*}
&F^{(n|m)}_{(k|l)} \cdot J^{(k|l)}_{(a|b)}\\
&= \begin{cases}
								P(s+1|s)_A \to (-1)^{(n+k)(l+s+1)+(a+k+1)(b+s)} P(s|s-2)_A  \\
								P(s+1|s-1)_B \to (-1)^{(n+k)(l+s)+(a+k+1)(b+s+1)} P(s|s-1)_A  \\
								\end{cases}\\
&=(-1)^{(n+k)(l+b+1)} B^{(n|m)}_{(a|b)}															
\end{align*}
 \end{itemize}
\item We compute multiplications with $ \widetilde{F}^{(n|m)}_{(k|l)}$:
\begin{itemize} 
\item For multiplication with $\Id^{(k|l)}_{(a|b)}$, one only has to change the appropriate signs, so one computes:
\begin{align*}
&\widetilde{F}^{(n|m)}_{(k|l)} \cdot \Id^{(k|l)}_{(a|b)}\\
=& 
\begin{cases} P(s|t)_A \to(-1)^{(n+k+1)(l+t)+(a+k)(b+t)} P(s-1|t)_A\\
								P(s|t)_B \to(-1)^{(n+k+1)(l+s)+(a+k)(b+s)}  P(s|t-1)_B\\
								P(s+1|s)_A \to(-1)^{(n+k+1)(l+s)+(a+k)(b+s)} P(s|s-2)_B
\end{cases}
\end{align*}
Adding the the signs to those of $\widetilde{F}^{(n|m)}_{(a|b)}$ one obtains in the first case (and the others go similar):
\begin{align*}
&(n+k+1)(l+t)+(a+k)(b+t)+(n+a+1)(b+t) \\
\equiv& (n+k+1)(l+t)+(n+k+1)(b+t) \\
\equiv& (n+k+1)(b+l) \pmod 2
\end{align*}
So we have shown that 
$$\widetilde{F}^{(n|m)}_{(k|l)} \cdot \Id^{(k|l)}_{(a|b)} = (-1)^{(n+k+1)(b+l)} \widetilde{F}^{(n|m)}_{(a|b)}.$$
\item Computing $\widetilde{F}^{(n|m)}_{(k|l)} \cdot F^{(k|l)}_{(a|b)}$ we have to check cases:
\begin{enumerate}
\item $P(s|t)$ $s>t+1$

\begin{displaymath}
\xymatrix{ P(s|t)_A \ar[rr]^{(-1)^{(n+k+1)(l+t)}} &&P(s-1|t)_A  \ar[rr]^{(-1)^{(a+k)(b+t+1)}} &&P(s-1|t-1)_A}
\end{displaymath}
\item $P(s+3|s)$

Here we get in addition
 \begin{displaymath} \xymatrix{ P(s+3|s)_A \ar[rr] &&P(s+2|s)_A   \ar @{} [d] |{=} \ar[rr] &&P(s+1|s)_A\\
 &&0&&}
 \end{displaymath}
\item $P(s+2|s)_A$

Here we get the additional equation
 \begin{displaymath} \xymatrix{ P(s+2|s)_A \ar[rr] &&P(s+1|s)_A   \ar @{} [d] |{=} \ar[rr] &&P(s|s-2)_B\\
 &&0&&}
 \end{displaymath}
\item $P(s+1|s)_A$ 

We have
\begin{displaymath}
\xymatrix{ P(s+1|s)_A \ar[rr]^{(-1)^{(n+k+1)(l+s)}} &&P(s|s-2)_B  \ar[rr]^{(-1)^{(a+k)(b+s+1)}} &&P(s|s-1)_A}
\end{displaymath}
\item $P(s|t)_B$

Here we only have to look at one possible case and obtain
\begin{displaymath}
\xymatrix{ P(s|t)_B \ar[rr]^{(-1)^{(n+k+1)(l+s)}} &&P(s|t-1)_B  \ar[rr]^{(-1)^{(a+k)(b+s+1)}} &&P(s-1|t-1)_B}
\end{displaymath}
\end{enumerate}
Taking the above results together we obtain

\begin{align*}
&\widetilde{F}^{(n|m)}_{(k|l)} \cdot F^{(k|l)}_{(a|b)}\\
=
& \begin{cases}
								P(s|t)_A \to (-1)^{(n+k+1)(l+t)+(a+k)(b+t+1)} P(s-1|t-1)_A \\
								P(s|t)_B \to (-1)^{(n+k+1)(l+s)+(a+k)(b+s+1)} P(s-1|t-1)_B \\ 
								\end{cases}\\
&=(-1)^{(n+k+1)(b+l+1)}	J^{(n|m)}_{(a|b)}							
\end{align*}	
\item To show that $\widetilde{F}^{(n|m)}_{(k|l)} \cdot \widetilde{F}^{(k|l)}_{(a|b)}	=0$ we check:
\begin{enumerate}
\item $P(s|t)_A$, $s>t+2$
		 \begin{displaymath} \xymatrix{ P(s|t)_A \ar[rr] &&P(s-1|t)_A   \ar @{} [d] |{=} \ar[rr] &&P(s-2|t)_A\\
 &&0&&}
 \end{displaymath}
\item $P(s+2|s)_A$
		 \begin{displaymath} \xymatrix{ P(s+2|s)_A \ar[rr] &&P(s+1|s)_A   \ar @{} [d] |{=} \ar[rr] &&P(s|s-2)_B\\
 &&0&&} 				  
  \end{displaymath}  

\item $P(s+1|s)_A$
		 \begin{displaymath} \xymatrix{ P(s+1|s)_A \ar[rr] &&P(s|s-2)_B   \ar @{} [d] |{=} \ar[rr] &&P(s|s-3)_B\\
 &&0&&}
   \end{displaymath} 
\item $P(s|t)_B$
		 \begin{displaymath} \xymatrix{ P(s|t)_B \ar[rr] &&P(s|t-1)_B   \ar @{} [d] |{=} \ar[rr] &&P(s|t-2)_B\\
 &&0&&}
   \end{displaymath} 
 since $s \neq t+1$.
\end{enumerate}
Therefore, the composition must be zero.
 
\item For the multiplication with $G^{(k|l)}_{(a|b)}$ we have to consider the following cases
\begin{enumerate}
\item $P(s|t)_A$, $s>t+2$

We compute
 \begin{displaymath} \xymatrix{ P(s|t)_A \ar[rr] &&P(s-1|t)_A   \ar @{} [d] |{=} \ar[rr] &&0\\
 &&0&&}
 \end{displaymath}
\item $P(s+2|s)_A$

Here obtain
\begin{displaymath} \xymatrix{ P(s+2|s)_A \ar[rr] &&P(s+1|s)_A   \ar @{} [d] |{=} \ar[rr] &&P(s|s-2)_A\\
 &&0&&}
 \end{displaymath}
 
\item $P(s+1|s)_A$ 

Here we have
\begin{displaymath}
\xymatrix{ P(s+1|s)_A \ar[rr]^{(-1)^{(n+k+1)(l+s)}} &&P(s|s-2)_B  \ar[rr]^{(-1)^{(a+s+1)(k+s)}} &&P(s-1|s-2)_A}
\end{displaymath}
since for the second summand we get
\begin{displaymath} \xymatrix{ P(s+1|s)_A \ar[rr] &&P(s|s-2)_B   \ar @{} [d] |{=} \ar[rr] &&P(s|s-3)_A\\
 &&0&&}
 \end{displaymath}
\item $P(s|t)_B$ 

Here we have two summands, for the first we get
\begin{displaymath} \xymatrix{ P(s|t)_B \ar[rr] &&P(s|t-1)_B   \ar @{} [d] |{=} \ar[rr] &&P(s|t-2)_A\\
 &&0&&}
 \end{displaymath}
and for the second
\begin{displaymath}
\xymatrix{ P(s|t)_B \ar[rr]^{(-1)^{(n+k+1)(l+s)}} &&P(s|t-1)_B  \ar[rr]^{(-1)^{(a+s+1)(k+t+1)}} &&P(s-1|t-1)_A}
\end{displaymath}
\end{enumerate}
Putting all of them together, we obtain
\begin{align*}
&F^{(n|m)}_{(k|l)} \cdot G^{(k|l)}_{(a|b)}\\
&=  \begin{cases} P(s|t)_A \to 0 s>t+1 \\
								P(s|s-1)_A \to (-1)^{(n+k+1)(l+s+1)+(a+s)(k+s+1)} P(s-2|s-3)_A  \\
								P(s|t)_B \to (-1)^{(n+k+1)(l+s)+(a+s+1)(k+t+1)} P(s-1|t-1)_A \\ 
								\end{cases}\\
&=(-1)^{(n+k+1)(a+l+1)} K^{(n|m)}_{(a|b)}															
\end{align*}

\item  For multiplication with $K^{(k|l)}_{(a|b)}$ as before we can show that no map can exist, since
$$\widetilde{F}^{(n|m)}_{(k|l)} \cdot K^{(k|l)}_{(a|b)}\in \hom^{(n+m)-(a+b)-5}(P_\bullet(n|m),P_\bullet(a|b)\langle (n+m)-(a+b)-8 \rangle)$$
and by Lemma \ref{Cor:posschmaps} this cannot exist, since $-3 \nleq -5.$
%
%
%

 \item Last we have to show $\widetilde{F}^{(n|m)}_{(k|l)} \cdot J^{(k|l)}_{(a|b)}=0$.
\begin{enumerate}
\item $P(s|t)_A$ $s>t+1$
 \begin{displaymath} \xymatrix{ P(s|t)_A \ar[rr] &&P(s-1|t)_A   \ar @{} [d] |{=} \ar[rr] &&P(s-2|t-1)_A\\
 &&0&&}
 \end{displaymath}

\item $P(s+1|s)_A$ 

We get
 \begin{displaymath} \xymatrix{ P(s+1|s)_A \ar[rr] &&P(s|s-2)_B   \ar @{} [d] |{=} \ar[rr] &&P(s-1|s-3)_B\\
 &&0&&}
 \end{displaymath}
 
\item $P(s|t)_B$
 \begin{displaymath} \xymatrix{ P(s|t)_B \ar[rr] &&P(s|t-1)_B   \ar @{} [d] |{=} \ar[rr] &&P(s-1|t-2)_B\\
 &&0&&}
 \end{displaymath}
since $s \neq t+1$.
\end{enumerate}
All in all, we have checked all cases to become zero.
 \end{itemize}
\item Now we are going to compute the multiplication with $G^{(n|m)}_{(k|l)}$.  
\begin{itemize}
\item Again, multiplying with $\Id^{(k|l)}_{(a|b})$ only changes the signs, so we get
\begin{align*}
&G^{(n|m)}_{(k|l)} \cdot \Id^{(k|l)}_{(a|b)}\\
=&  \begin{cases} P(s|t)_A \to 0 & t \neq s-1\\
								P(s|s-1)_A \to (-1)^{(n+k)(k+s)+1+(a+k)(b+s+1)} P(s-1|s-3)_B  \\
								P(s|t)_B \to (-1)^{(k+s+1)(n+t)+(a+k)(b+t)} P(s-1|t)_A \\ 
								\ \ \ +(-1)^{(k+s+1)(n+t+1)+s+t+(a+k)(b+t+1)} P(s|t-1)_A \\
								\end{cases}\\
								&=(-1)^{(a+k)(b+n)} G^{(n|m)}_{(a|b)}															
\end{align*}
\item Multiplication with $F^{(k|l)}_{(a|b)}$ yields to
\begin{enumerate}
\item $P(s|t)_A$ $s>t+1$

We compute
 \begin{displaymath} \xymatrix{ P(s|t)_A \ar[rr] &&0}
 \end{displaymath}
 
\item $P(s+1|s)_A$ 

Here we obtain
\begin{displaymath}
\xymatrix{ P(s+1|s)_A \ar[rr]^{(-1)^{(n+k)(k+s+1)+1}} &&P(s|s-2)_A  \ar[rr]^{(-1)^{(a+k)(b+s+1)+a+s}} &&P(s-1|s-2)_A}
\end{displaymath}
since for the second summand we get
\begin{displaymath} \xymatrix{ P(s+1|s)_A \ar[rr] &&P(s|s-2)_A   \ar @{} [d] |{=} \ar[rr] &&P(s|s-3)_A\\
 &&0&&}
 \end{displaymath}
\item $P(s|t)_B$ 

Here we have two summands, for the first we get
\begin{displaymath} \xymatrix{ P(s|t)_B \ar[rr] &&P(s|t-1)_A   \ar @{} [d] |{=} \ar[rr] &&P(s|t-2)_A\\
 &&0&&}
 \end{displaymath}
and for the second
\begin{displaymath}
\xymatrix{ P(s|t)_B \ar[rr]^{(-1)^{(k+s+1)(n+t)}} &&P(s-1|t)_A  \ar[rr]^{(-1)^{(a+k)(b+t+1)}} &&P(s-1|t-1)_A}
\end{displaymath}
\item $P(s|s-3)_B$

Examining the composition one obtains in addition
\begin{displaymath} \xymatrix{ P(s|s-3)_B \ar[rr] &&P(s-1|s-3)_A   \ar @{} [d] |{=} \ar[rr] &&P(s-2|s-3)_A\\
 &&0&&}
 \end{displaymath}
 \item $P(s|s-2)_B$

And in this case
\begin{displaymath} \xymatrix{ P(s|s-2)_B \ar[rr] &&P(s-1|s-2)_A   \ar @{} [d] |{=} \ar[rr] &&P(s-2|s-4)_B\\
 &&0&&}
 \end{displaymath}
\end{enumerate}
Putting all of them together, we obtain
\begin{align*}
&G^{(n|m)}_{(k|l)} \cdot F^{(k|l)}_{(a|b)}\\
&=  \begin{cases} P(s|t)_A \to 0 s>t+1 \\
								P(s|s-1)_A \to (-1)^{(n+k)(k+s)+(a+k)(b+s)+a+s} P(s-2|s-3)_A  \\
								P(s|t)_B \to (-1)^{(k+s+1)(n+t)+(a+k)(b+t+1)} P(s-1|t-1)_A \\ 
								\end{cases}\\
&=(-1)^{(n+k+1)(a+l+1)} K^{(n|m)}_{(a|b)}															
\end{align*}

\item Now we are going to do the same with  $\widetilde{F}^{(k|l)}_{(a|b})$:
\begin{enumerate}
\item $P(s|t)_A$ $s>t+1$

We compute
 \begin{displaymath} \xymatrix{ P(s|t)_A \ar[rr] &&0}
 \end{displaymath}
 
\item $P(s+1|s)_A$ 

Here we obtain
\begin{displaymath}
\xymatrix{ P(s+1|s)_A \ar[rr]^{(-1)^{(n+k)(k+s+1)+1}} &&P(s|s-2)_A  \ar[rr]^{(-1)^{(a+k+1)(b+s)}} &&P(s-1|s-2)_A}
\end{displaymath}
\item $P(s|t)_B$ 

Here we have two summands, for the first we get for $s>t+2$
\begin{displaymath} \xymatrix{ P(s|t)_B \ar[rr] &&P(s-1|t)_A   \ar @{} [d] |{=} \ar[rr] &&P(s-2|t)_A\\
 &&0&&}
 \end{displaymath}
 
and for the second
\begin{displaymath}
\xymatrix{ P(s|t)_B \ar[rr]^{(-1)^{(k+s+1)(n+t+1)+s+t}} &&P(s|t-1)_A  \ar[rr]^{(-1)^{(a+k+1)(b+t+1)}} &&P(s-1|t-1)_A}
\end{displaymath}
 \item $P(s|s-2)_B$

And in this case
\begin{displaymath} \xymatrix{ P(s|s-2)_B \ar[rr] &&P(s-1|s-2)_A   \ar @{} [d] |{=} \ar[rr] &&P(s-2|s-4)_B\\
 &&0&&}
 \end{displaymath}
\end{enumerate}
Putting all of them together, we obtain
\begin{align*}
&G^{(n|m)}_{(k|l)} \cdot \widetilde{F}^{(k|l)}_{(a|b)}\\
&=  \begin{cases} P(s|t)_A \to 0 s>t+1 \\
								P(s|s-1)_A \to (-1)^{(n+k)(k+s+1)+1+(a+k+1)(b+s)} P(s-2|s-3)_A  \\
								P(s|t)_B \to (-1)^{(k+s+1)(n+t+1)+s+t+(a+k+1)(b+t+1)} P(s-1|t-1)_A \\ 
								\end{cases}\\
&=(-1)^{(a+k+1)(b+n)+a+n} K^{(n|m)}_{(a|b)}															
\end{align*}	
\item  For multiplication with $G^{(k|l)}_{(a|b)}$ we just look at the $\hom$-space and see
\begin{align*}&G^{(n|m)}_{(k|l)} \cdot G^{(k|l)}_{(a|b)}\\&\in \hom^{(n+m)-(a+b)-6}(P_\bullet(n|m),P_\bullet(a|b)\langle (n+m)-(a+b)-8 \rangle)\end{align*}
and by Lemma \ref{Cor:posschmaps} case \ref{Cposs2} this cannot exist, since $-4 \nleq -6$
so the composition is zero.
%
\item  For  $G^{(n|m)}_{(k|l)} \cdot K^{(k|l)}_{(a|b)}$ we have
\begin{align*} &G^{(n|m)}_{(k|l)} \cdot K^{(k|l)}_{(a|b)}\\&\in \hom^{(n+m)-(a+b)-7}(P_\bullet(n|m),P_\bullet(a|b)\langle (n+m)-(a+b)-8 \rangle)\end{align*}
and by Lemma \ref{Cor:posschmaps} case \ref{Cposs1} this cannot exist, since $-3 \nleq -7$.
%
%
\item  The same argument holds for  $G^{(n|m)}_{(k|l)} \cdot J^{(k|l)}_{(a|b)}$ since we have
\begin{align*}&G^{(n|m)}_{(k|l)} \cdot J^{(k|l)}_{(a|b)}\\&\in \hom^{(n+m)-(a+b)-5}(P_\bullet(n|m),P_\bullet(a|b)\langle (n+m)-(a+b)-8 \rangle)\end{align*}
and by Lemma \ref{Cor:posschmaps} case \ref{Cposs3} this cannot exist, since $-3 \nleq -5$

%
%
%
 \end{itemize}
\item Yet we can work out the multiplications with $K^{(n|m)}_{(k|l)}$:
\begin{itemize}
\item Again, multiplying with $\Id^{(k|l)}_{(a|b})$ only changes the signs, so we get
\begin{align*}
&K^{(n|m)}_{(k|l)} \cdot \Id^{(k|l)}_{(a|b)}\\
=& \begin{cases} P(s|t)_A \to 0 \ t \neq s-1\\
								P(s|s-1)_A \to (-1)^{(n+k+1)(k+s)+(a+k)(b+s+1)} P(s-2|s-3)_A  \\
								P(s|t)_B \to (-1)^{(n+t)(k+s+1)+(a+k)(b+t+1)} P(s-1|t-1)_A \\ 
								\end{cases}\\
								&=(-1)^{(a+k)(b+n+1)} G^{(n|m)}_{(a|b)}															
\end{align*}
\item Like above one sees \begin{align*}
&K^{(n|m)}_{(k|l)} \cdot F^{(k|l)}_{(a|b)} \\& \in \hom^{(n+m)-(a+b)-5}(P_\bullet(n|m),P_\bullet(a|b)\langle (n+m)-(a+b)-6 \rangle)\end{align*}
and by Lemma \ref{Cor:posschmaps} case \ref{Cposs1} this cannot exist, since  $-3\nleq -5$.
\item The same holds for multiplication with $\widetilde{F}^{(k|l)}_{(a|b)}$.
\item \begin{myalignst}&K^{(n|m)}_{(k|l)} \cdot G^{(k|l)}_{(a|b)}\\&\in \hom^{(n+m)-(a+b)-7}(P_\bullet(n|m),P_\bullet(a|b)\langle (n+m)-(a+b)-8 \rangle)\end{myalignst}
and by Lemma \ref{Cor:posschmaps} case \ref{Cposs1} this cannot exist since $-3 \nleq -7$.
\item Last we have \begin{align*} &K^{(n|m)}_{(k|l)} \cdot J^{(k|l)}_{(a|b)}\\&\in \hom^{(n+m)-(a+b)-6}(P_\bullet(n|m),P_\bullet(a|b)\langle (n+m)-(a+b)-10 \rangle)\end{align*}
and by Lemma \ref{Cor:posschmaps} case \ref{Cposs4} this cannot exist, since  $-2\nleq -6$.				
\end{itemize}
\item Last we can compute the multiplications with $J^{(n|m)}_{(k|l)}=F^{(n|m)}_{(k+1|l)} \cdot \widetilde{F}^{(k+1|l)}_{(k|l)}$
\begin{itemize} 
\item For multiplication with $\Id$ we can use results already computed:
\begin{align*}
J^{(n|m)}_{(k|l)} \cdot \Id^{(k|l)}_{(a|b)}\\
&=F^{(n|m)}_{(k+1|l)} \cdot \widetilde{F}^{(k+1|l)}_{(k|l)} \cdot \Id^{(k|l)}_{(a|b)} \\
&=F^{(n|m)}_{(k+1|l)} \cdot \Id^{(k+1|l)}_{(a+1|b)} \cdot \widetilde{F}^{(a+1|b)}_{(a|b)} \\
&=(-1)^{(n+k+1)(b+l)} F^{(n|m)}_{(a+1|b)}  \cdot \widetilde{F}^{(a+1|b)}_{(a|b)} \\
&=(-1)^{(n+k+1)(b+l)} J^{(n|m)}_{(a|b)} 
\end{align*}
\item Next we get
\begin{align*}
J^{(n|m)}_{(k|l)} \cdot F^{(k|l)}_{(a|b)}\\
&=F^{(n|m)}_{(k+1|l)} \cdot \widetilde{F}^{(k+1|l)}_{(k|l)} \cdot F^{(k|l)}_{(a|b)} \\
&=F^{(n|m)}_{(k+1|l)} \cdot J^{(k+1|l)}_{(a|b)}\\
&=(-1)^{(n+k+1)(b+l)} B^{(n|m)}_{(a|b)} 
\end{align*}
\item Using the associativity we get
\begin{align*}
J^{(n|m)}_{(k|l)} \cdot \widetilde{F}^{(k|l)}_{(a|b)}\\
&=F^{(n|m)}_{(k+1|l)} \cdot( \widetilde{F}^{(k+1|l)}_{(k|l)} \cdot \widetilde{F}^{(k|l)}_{(a|b)}) \\
&=0
\end{align*}
\item For the multiplication with $G^{(k|l)}_{(a|b)}$ one sees
 \begin{align*}&J^{(n|m)}_{(k|l)} \cdot G^{(k|l)}_{(a|b)}\\&\in \hom^{(n+m)-(a+b)-8}(P_\bullet(n|m),P_\bullet(a|b)\langle (n+m)-(a+b)-5 \rangle)\end{align*}
and by Lemma \ref{Cor:posschmaps} case \ref{Cposs3} this cannot exist, since $-3 \nleq -5$.
\item Multiplying with $K^{(k|l)}_{(a|b)}$ delivers
 \begin{align*}&J^{(n|m)}_{(k|l)} \cdot K^{(k|l)}_{(a|b)}\\&\in \hom^{(n+m)-(a+b)-10}(P_\bullet(n|m),P_\bullet(a|b)\langle (n+m)-(a+b)-6 \rangle)\end{align*}
and by Lemma \ref{Cor:posschmaps} case \ref{Cposs4} this cannot exist, since $-2 \nleq -10$.
\item Last one checks similarly
 \begin{align*}&J^{(n|m)}_{(k|l)} \cdot J^{(k|l)}_{(a|b)}\\&\in \hom^{(n+m)-(a+b)-8}(P_\bullet(n|m),P_\bullet(a|b)\langle (n+m)-(a+b)-4 \rangle)\end{align*}
 which also is empty, since  $$n+m \nleq a+b+(n+m)-(a+b)-4)+2.$$
\end{itemize}
\item Last we check that
$$B^{(n|m)}_{(a|b)}=A^{(n|m)}_{(a+1|b)}\cdot \widetilde{F}^{(a+1|b)}_{(a|b)}.$$
Therefore write
\begin{align*}
&B^{(n|m)}_{(a|b)}\\&=(-1)^{(n+k)(l+b+1)} \cdot F^{(n|m)}_{(k|l)} \cdot F^{(k|l)}_{(a+1|b)} \cdot \widetilde{F}^{(a+1|b)}_{(a|b)}\\
&=A^{(n|m)}_{(a+1|b)}\cdot \widetilde{F}^{(a+1|b)}_{(a|b)}
\end{align*}
\end{enumerate}
\end{proof}

\chapter{Computations of the homotopies in Section \ref{sec:homotopies}}\label{Ap:Homot}
\begin{proof}[Proof of Lemma \ref{Le:homotopyF}]
We will only check the property for $H(F-(-1)^{k+l+1}\widetilde{F})^{(k+1|k)}_{(k|l)}$. 
Since the degree of the map $H$ is $k+l+1$ we have to check that the diagrams of the form
\begin{displaymath}
\xymatrix{\cdots \ar[r] &V \ar[d] \ar[r] \ar[ld] &Y\ar[ld]^{(-1)^{k+l}}\\
X \ar[r] & W  \ar[r] & \cdots}
\end{displaymath}
with $V$ a projective module occurring in the resolution, commute with the diagonal arrows behaving as a homotopy.\\
For $V$ we have to check different cases (shortly write $H$ for $H(F-(-1)^{k+l+1}\widetilde{F})^{(k+1|k)}_{(k|l)}$):
\begin{enumerate}[label={$H$\arabic*})]
\item $V=P(s|t), s>t+2$:

Since $H$ is zero on $P(s|t)_A$ we only have a diagram of the form
\begin{displaymath}
\xymatrix{\cdots \ar[r] &P(s|t)_A \ar[d] \ar[r] \ar[ld] &{\begin{matrix} P(s-1|t)_B \\P(s|t-1)_B \end{matrix}} \ar[ld]^{(-1)^{k+l}}\\
\cdots \ar[r] & {\begin{matrix} P(s-1|t)_A \\P(s|t-1)_A \end{matrix}}  \ar[r] & \cdots}
\end{displaymath}
and we want the triangle to commute. Therefore we check the signs:

\begin{displaymath} \xymatrix{ P(s|t)_A \ar[rr]^{(-1)^{(s+t+1)(k+1+s)}} &&P(s-1|t)_B  \ar @{} [d] |{=} \ar[rr]^{(-1)^{(s+t+1)(k+s+1)+k+l}} &&P(s-1|t)_A\\
& P(s|t)_A \ar[rr]^{(-1)^{k+l}} &&P(s-1|t)_A &}
\end{displaymath} 
and for $s<k+1$ 
\begin{displaymath} \xymatrix{ P(s|t)_A \ar[rr]^{(-1)^{(s+t+1)(k+1+s)+k+s}} &&P(s|t-1)_B  \ar @{} [d] |{=} \ar[rr]^{(-1)^{(s+t+1)(k+s)+k+l}} &&P(s|t-1)_A\\
 P(s|t)_A \ar[rr]^{(-1)^{l+t+1}} &&P(s|t-1)_A &}
\end{displaymath} 
None of the involved terms here will occur for $s=k+1$.
\item $V=P(s+2|s)_A$:

Here we have
\begin{displaymath}
\xymatrix{\cdots \ar[r] &P(s+2|s)_A \ar[d] \ar[r] \ar[ld] &P(s+2|s-1)_B \ar[ld]^{(-1)^{k+l}}\\
\cdots \ar[r] & {\begin{matrix} P(s+1|s)_A \\P(s+2|s-1)_A \end{matrix}}  \ar[r] & \cdots}
\end{displaymath}
If $s+2<k+1$ the diagram
\begin{displaymath} \xymatrix{ P(s+2|s)_A \ar[rr]^{-1} &&P(s+2|s-1)_B  \ar @{} [d] |{=} \ar[rr]^{(-1)^{k+s+k+l}} &&P(s+2|s-1)_A\\
 P(s+2|s)_A \ar[rr]^{(-1)^{l+s+1}} &&P(s+2|s-1)_A &}
\end{displaymath} 
commutes as we have already seen above.
Since we have two maps from $P(s+2|s)_A \to P(s+1|s)_A$, one comming from $F$ and one from $\widetilde{F}$ we obtain
\begin{displaymath} \xymatrix{ P(s+2|s)_A \ar @{} [d] |{=} \ar[rr]^{(-1)^{l+s+1+k+s}} &&P(s+1|s)_A   \\
-(P(s+2|s)_A \ar[rr]^{(-1)^{k+l}} &&P(s+1|s)_A) }
\end{displaymath} 
and by this the two maps cancel.
\item $V=P(s+1|s)_A$:

Here we have
\begin{displaymath}
\xymatrix{\cdots \ar[r] &P(s+1|s)_A \ar[d] \ar[r] \ar[ld] &P(s+1|s-1)_B \ar[ld]^{(-1)^{k+l}}\\
\cdots \ar[r] & {\begin{matrix} P(s|s-2)_B \\P(s+1|s-1)_A \end{matrix}}  \ar[r] & \cdots}
\end{displaymath}
If $s+1<k+1$ the diagram
\begin{displaymath} \xymatrix{ P(s+1|s)_A \ar[rr]^{(-1)^{k+s+1}} &&P(s+1|s-1)_B  \ar @{} [d] |{=} \ar[rr]^{(-1)^{k+l}} &&P(s+2|s-1)_A\\
 P(s+2|s)_A \ar[rr]^{(-1)^{l+s+1}} &&P(s+2|s-1)_A &&}
\end{displaymath} 
commutes as we have already seen above.
Again we have two maps from $P(s+1|s)_A \to P(s|s-2)_B$, one comming from $F$ and one from $\widetilde{F}$, so we have
\begin{displaymath} \xymatrix{ P(s+1|s)_A \ar @{} [d] |{=} \ar[rr]^{(-1)^{l+s+1+k+s}} &&P(s|s-2)_B   \\
-(P(s+1|s)_A \ar[rr]^{(-1)^{k+l}} &&P(s|s-2)_B) }
\end{displaymath} 
and by this the two maps cancel.

\item $V=P(s|t)_B$ and $s>t+2$:

\begin{displaymath}
\xymatrix{\cdots \ar[r] &P(s|t)_B \ar[d] \ar[r] \ar[ld] &{\begin{matrix} P(s+1|t)_B \\P(s|t+1)_B \end{matrix}} \ar[ld]^{(-1)^{k+l}}\\
P(s|t)_A \ar[r] & {\begin{matrix} P(s-1|t)_B \\P(s|t-1)_B\\P(s+1|t)_A \\P(s|t+1)_A \end{matrix}}  \ar[r] & \cdots}
\end{displaymath}
and we have to check that the sum of the diagonal arrows equals the vertical one.
First we check those including two vertical arrows, there we have
\begin{displaymath} \xymatrix{ P(s|t)_B \ar[rr]^{(-1)^{s+t}} &&P(s|t+1)_B  \ar @{} [d] |{=} \ar[rr]^{(-1)^{k+l+(s+t+1)(k+s)}} &&P(s|t+1)_A\\
-(P(s|t)_B \ar[rr]^{(-1)^{(s+t)(k+s)}} &&P(s|t)_A \ar[rr]^{(-1)^{l+t+1}} &&P(s|t+1)_A)}\end{displaymath}
so the sum is zero.\\
Similar we check for $s+1\leq k$
\begin{displaymath} \xymatrix{ P(s|t)_B \ar[rr]^{(-1)^{k+s+1}} &&P(s+1|t)_B  \ar @{} [d] |{=} \ar[rr]^{(-1)^{k+l+(s+t+1)(k+s+1)}} &&P(s+1|t)_A\\
-(P(s|t)_B \ar[rr]^{(-1)^{(s+t)(k+s)}} &&P(s|t)_A \ar[rr]^{(-1)^{k+l+s+t+1}} &&P(s+1|t)_A)}\end{displaymath}
Next we check that the triangles involving one vertical arrow and one diagonal (the left one) commute:
\begin{displaymath} \xymatrix{ P(s|t)_B \ar[rr]^{(-1)^{(s+t)(k+s)}} &&P(s|t)_A  \ar @{} [d] |{=} \ar[rr]^{(-1)^{(s+t+1)(k+s)+k+l+1}} &&P(s-1|t)_B\\
 P(s|t)_B \ar[rr]^{(-1)^{l+s+1}} &&P(s-1|t)_B &}
\end{displaymath}
and
\begin{displaymath} \xymatrix{ P(s|t)_B \ar[rr]^{(-1)^{(s+t)(k+s)}} &&P(s|t)_A  \ar @{} [d] |{=} \ar[rr]^{(-1)^{(s+t+1)(k+s)+l+s}} &&P(s|t-1)_B\\
 P(s|t)_B \ar[rr]^{(-1)^{k+l}} &&P(s|t-1)_B &}
\end{displaymath}
\item   $V=P(s|s-2)_B$:

Last we have to check
\begin{displaymath}
\xymatrix{\cdots \ar[r] &P(s|s-2)_B \ar[d] \ar[r] \ar[ld] &P(s+1|s)_A  \ar[ld]^{(-1)^{k+l}}\\
P(s|s-2)_A \ar[r] & {\begin{matrix} P(s|s-3)_B\\P(s+1|s-2)_A \\P(s|s-1)_A \end{matrix}}  \ar[r] & \cdots}
\end{displaymath}
The one passing over $P(s+1|s-2)_B$ is just commuting by the above computations, so we only have to compute the triangles involving the left one. Here we get
\begin{displaymath} \xymatrix{ P(s|s-2)_B \ar[rr] &&P(s|s-2)_A  \ar @{} [d] |{=} \ar[rr]^{(-1)^{l+s+1}} &&P(s|s-1)_A\\
 P(s|s-2)_B \ar[rr]^{(-1)^{l+s+1}} &&P(s|s-1)_A &}
\end{displaymath}
and
\begin{displaymath} \xymatrix{ P(s|s-2)_B \ar[rr] &&P(s|s-2)_A  \ar @{} [d] |{=} \ar[rr]^{(-1)^{k+l}} &&P(s|s-3)_B\\
 P(s|s-2)_B \ar[rr]^{(-1)^{k+l}} &&P(s|s-3)_B &}
\end{displaymath}\end{enumerate}
So we have checked all possible cases.
\end{proof}

\begin{proof}[Proof of Lemma \ref{Le:homotopyJ}]
Since the degree of the map $J$ is $m+n-(k+l)-2$ and therefore the degree of $H$ is $m+n-(k+l)-3$, we have to check that the diagrams of the form
\begin{displaymath}
\xymatrix{\cdots \ar[r] &V \ar[d] \ar[r] \ar[ld] &Y\ar[ld]^{(-1)^{k+l+m+n}}\\
X \ar[r] & W  \ar[r] & \cdots}
\end{displaymath}
with $V$ a projective module occurring in the resolution, commute with the diagonal arrows behaving as a homotopy.\\
For $V$ we have to check different cases:
\begin{enumerate}[label={$H$\arabic*})]
\item $V=P(s|t)_A, s>t+2$:

$H$ is zero on $P(s|t)_A$ and on all possible terms A-terms occuring in $Y$, so the onliest maps are
\begin{displaymath}
\xymatrix{\cdots \ar[r] &P(s|t)_A \ar[d] \ar[r] \ar[ld] &{\begin{matrix}P(s|t-1)_B\\P(s-1|t)_B\end{matrix}} \ar[ld]^{(-1)^{k+l+m+n}}\\
\cdots \ar[r] & {\begin{matrix}P(s-1|t-1)_A \\P(s|t-2)\end{matrix}} \ar[r] & \cdots}
\end{displaymath}
First check that
\begin{displaymath} \xymatrix{ P(s|t)_A \ar[rr] &&P(s|t-1)_B  \ar @{} [d] |{=} \ar[rr] &&P(s|t-2)_A\\
 &&0&&}
\end{displaymath}
so we are left to look for the triangle
\begin{displaymath} \xymatrix{ P(s|t)_A \ar[rr]^{(-1)^{(s+t+1)(n+s)+n+m+1}} &&P(s-1|t)_B  \ar @{} [d] |{=} \ar[rr]^{(-1)^{(n+k)(l+t)+(s+t)(n+s+1)+n+m+k+l}} &&P(s-1|t-1)_A\\
 P(s|t)_A \ar[rr]^{(-1)^{(n+k+1)(l+t+1)}} &&P(s-1|t-1)_A &}
 \end{displaymath}
 which commutes, too. The B-term always exists if the A-term $P(s-1|t-1)_A$ exists, since $m\geq k$.
 
\item $V=P(s+2|s)_A$:
3
Here we have
\begin{displaymath}
\xymatrix{\cdots \ar[r] &P(s+2|s)_A \ar[d] \ar[r] \ar[ld] &{\begin{matrix}P(s+2|s-1)_B\\P(s+2|s+1)_A\end{matrix}} \ar[ld]^{(-1)^{k+l+m+n}}\\
\cdots \ar[r] & {\begin{matrix}P(s+1|s-1)_A \\P(s+2|s-2)\end{matrix}} \ar[r] & \cdots}
\end{displaymath}
First check that
\begin{displaymath} \xymatrix{ P(s+2|s)_A \ar[rr] &&P(s+2|s-1)_B  \ar @{} [d] |{=} \ar[rr] &&P(s+2|s-2)_A\\
 &&0&&}
\end{displaymath}
so we are left to look for the triangle
\begin{displaymath} \xymatrix{ P(s+2|s)_A \ar[rr]^{(-1)^{m+s+1}} &&P(s+2|s+1)_A  \ar @{} [d] |{=} \ar[rr]^{(-1)^{(n+k)(l+s)+n+m+k+l}} &&P(s+1|s-1)_A\\
 P(s+2|s)_A \ar[rr]^{(-1)^{(n+k+1)(l+s+1)}} &&P(s+1|s-1)_A &}
 \end{displaymath}
 which commutes, too.
 \item $V=P(s+1|s)_A$:
 
Here we have
\begin{displaymath}
\xymatrix{\cdots \ar[r] &P(s+1|s)_A \ar[d] \ar[r] \ar[ld] &P(s+1|s-1)_B \ar[ld]^{(-1)^{m+n+k+l}}\\
P(s|s-2)_A \ar[r] & {\begin{matrix} P(s+1|s-2)_A \\P(s|s-1)_A \\P(s|s-3)_B\end{matrix}}  \ar[r] & \cdots}
\end{displaymath}
First we look at maps both belonging to $H$, one being the right diagonal arrow and one the left
\begin{displaymath} \xymatrix{ P(s+1|s)_A \ar[rr]^{(-1)^{m+s+1}} &&P(s+1|s-1)_B  \ar @{} [d] |{=} \ar[rr]^{(-1)^{(n+k)(l+s+1)+(n+s+1)+m+n+k+l}} &&P(s+1|s-2)_A\\
-(P(s+1|s)_A \ar[rr]^{(-1)^{(n+k)(l+s+1)}} &&P(s|s-2)_A \ar[rr]^{(-1)^{l+k+1}} &&P(s+1|s-2)_A)}\end{displaymath}
and the two maps cancel.\\
Now we check that the triangle including the left diagonal map and $A$ commutes:
\begin{displaymath} \xymatrix{ P(s+1|s)_A \ar[rr]^{(-1)^{(n+k)(l+s+1)}} &&P(s|s-2)_A  \ar @{} [d] |{=} \ar[rr]^{(-1)^{l+s+1}} &&P(s|s-1)_A\\
 P(s+1|s)_A \ar[rr]^{(-1)^{(n+k+1)(l+s+1)}} &&P(s+1|s-2)_A &&}
\end{displaymath} 
Last check that
\begin{displaymath} \xymatrix{ P(s+1|s)_A \ar[rr] &&P(s|s-2)_A  \ar @{} [d] |{=} \ar[rr] &&P(s|s-3)_A\\
 &&0&&}
\end{displaymath}
and therefore we have checked this case completly.

\item $V=P(s|t)_B$ and $s>t+2$:

Writing down only the terms where the maps are nonzero, we get
\begin{displaymath}
\xymatrix{\cdots \ar[r] &P(s|t)_B \ar[d] \ar[r] \ar[ld] &{\begin{matrix}P(s+1|t)_B\\P(s|t+1)_B\end{matrix}} \ar[ld]^{(-1)^{k+l+m+n}}\\
P(s|t-1)_A \ar[r] & {\begin{matrix}P(s+1|t-1)_A \\P(s|t)_A\\P(s-1|t-1)_B\\P(s|t-2)_B \end{matrix}} \ar[r] & \cdots}
\end{displaymath}
First check that
\begin{displaymath} \xymatrix{ P(s|t)_B\ar[rr] &&P(s|t-1)_A  \ar @{} [d] |{=} \ar[rr] &&P(s|t-2)_B\\
 &&0&&}
\end{displaymath}
Now we check that the two maps of $H$ cancel
\begin{displaymath} \xymatrix{ P(s|t)_B \ar[rr]^{(-1)^{n+m+s+t+1}} &&P(s|t+1)_B  \ar @{} [d] |{=} \ar[rr]^{(-1)^{(n+k)(l+t+1)+(s+t)(n+s)+n+m+k+l}} &&P(s|t)_A\\
-(P(s|t)_B \ar[rr]^{(-1)^{(n+k)(l+t)+(s+t+1)(n+s)}} &&P(s|t-1)_A   \ar[rr]^{(-1)^{l+t}} &&P(s|t)_A)}
\end{displaymath}
and
\begin{displaymath} \xymatrix{ P(s|t)_B \ar[rr]^{(-1)^{m+s+1}} &&P(s+1|t)_B  \ar @{} [d] |{=} \ar[rr]^{(-1)^{(n+k)(l+t)+(s+t)(n+s+1)+n+m+k+l}} &&P(s+1|t-1)_A\\
-(P(s|t)_B \ar[rr]^{(-1)^{(n+k)(l+t)+(s+t+1)(n+s)}} &&P(s|t-1)_A   \ar[rr]^{(-1)^{k+l+s+t}} &&P(s+1|t-1)_A)}
\end{displaymath}
 which commutes, too. \\
 We have to check the left triangle, were we get
\begin{displaymath} \xymatrix{ P(s|t)_B \ar[rr]^{(-1)^{(n+k)(l+t)+(s+t+1)(n+s)}} &&P(s|t-1)_A  \ar @{} [d] |{=} \ar[rr]^{(-1)^{(s+t)(k+s)+l+k+1}} &&P(s-1|t-1)_B\\
 P(s|t)_B \ar[rr]^{(-1)^{(n+k+1)(l+s+1)}} &&P(s-1|t-1)_A &&}
\end{displaymath}  
 which obviously commutes.
 
\item   $V=P(s|s-2)_B$:

Last we have to check
\begin{displaymath}
\xymatrix{\cdots \ar[r] &P(s|s-2)_B \ar[d] \ar[r] \ar[ld] &{\begin{matrix}P(s+1|s)_A \\P(s+1|s-2)_B\end{matrix}} \ar[ld]^{(-1)^{n+m+k+l}}\\
P(s|s-3)_A \ar[r] & {\begin{matrix} P(s-1|s-3)_B\\P(s|s-4)_B\\P(s|s-2)_A \\P(s+1|s-3)_A \end{matrix}}  \ar[r] & \cdots}
\end{displaymath} 
All maps involving the right $P(s+1|s-2)_B$ commutate by the same computations as in the previous case. Also the map $A$ and $H$ passing over $P(s|s-3)$ do. We are left to check the case
\begin{displaymath} \xymatrix{ P(s|s-2)_B \ar[rr]^{(-1)^{n+m+1}} &&P(s+1|s)_A  \ar @{} [d] |{=} \ar[rr]^{(-1)^{(n+k)(l+s+1)+n+m+k+l}} &&P(s|t)_A\\
-(P(s|s-2)_B \ar[rr]^{(-1)^{(n+k)(l+s)+(n+s)}} &&P(s|s-3)_A   \ar[rr]^{(-1)^{l+s}} &&P(s|s-2)_A)}
\end{displaymath}
which commutes by the tables in section \ref{sec:mapproj}.
\end{enumerate}
So we have checked all possible cases.
\end{proof}

\begin{proof}[Proof of Lemma \ref{Le:homotopyA}]
Since the degree of the map $A$ is $m+n-(k+l)-2$ and therefore the degree of $H$ is $m+n-(k+l)-3$ we have to check that the diagrams of the form
\begin{displaymath}
\xymatrix{\cdots \ar[r] &V \ar[d] \ar[r] \ar[ld] &Y\ar[ld]^{(-1)^{k+l+m+n}}\\
X \ar[r] & W  \ar[r] & \cdots}
\end{displaymath}
with $V$ a projective module occurring in the resolution, commute with the diagonal arrows behaving as a homotopy.\\
For $V$ we have to check different cases:
\begin{enumerate}[label={$H$\arabic*})]
\item $V=P(s|t)_A, s>t+3$:

$A$ is zero on $P(s|t)_A$ and $H$ is zero on all possible terms occuring in $Y$, so there are no nonzero maps occurring in the diagram and therefore the diagram commutes.
\item $V=P(s+3|s)_A$:

Here we have
\begin{displaymath}
\xymatrix{\cdots \ar[r] &P(s+3|s)_A \ar[d] \ar[r] \ar[ld] &P(s+2|s)_B \ar[ld]^{(-1)^{k+l+m+n}}\\
\cdots \ar[r] & P(s+1|s)_A  \ar[r] & \cdots}
\end{displaymath}
Again $A$ is zero on $P(s+3|s)_A$ and we only check that
\begin{displaymath} \xymatrix{ P(s+3|s)_A \ar[rr] &&P(s+2|s)_B  \ar @{} [d] |{=} \ar[rr] &&P(s+1|s)_A\\
 &&0&&}
\end{displaymath}
what is true.
\item $V=P(s+2|s)_A$:

Here we have
\begin{displaymath}
\xymatrix{\cdots \ar[r] &P(s+2|s)_A \ar[d] \ar[r] \ar[ld] &P(s+1|s)_A \ar[ld]^{(-1)^{m+n+k+l}}\\
\cdots \ar[r] & P(s+1|s-1)_A  \ar[r] & \cdots}
\end{displaymath}
since on all other terms the maps are zero. We get
\begin{displaymath} \xymatrix{ P(s+2|s)_A \ar[rr]^{(-1)^{m+s+1}} &&P(s+2|s+1)_A  \ar @{} [d] |{=} \ar[rr]^{(-1)^{(n+k)(l+s)+n+l+1+m+n+k+l}} &&P(s+1|s-1)_A\\
 P(s+2|s)_A \ar[rr]^{(-1)^{(n+k)(l+s)+k+s}} &&P(s+1|s-1)_A &}
\end{displaymath} 
which commutes.
\item $V=P(s+1|s)_A$:

Here we have
\begin{displaymath}
\xymatrix{\cdots \ar[r] &P(s+1|s)_A \ar[d] \ar[r] \ar[ld] &P(s+1|s-1)_B \ar[ld]^{(-1)^{m+n+k+l}}\\
P(s|s-2)_A \ar[r] & {\begin{matrix} P(s+1|s-2)_A \\P(s|s-1)_A \\P(s|s-3)_B\end{matrix}}  \ar[r] & \cdots}
\end{displaymath}
First we look at maps both belonging to $H$, one being the right diagonal arrow and one the left
\begin{displaymath} \xymatrix{ P(s+1|s)_A \ar[rr]^{(-1)^{m+s+1}} &&P(s+1|s-1)_B  \ar @{} [d] |{=} \ar[rr]^{(-1)^{(n+k)(l+s+1)+k+l+m+n+k+l}} &&P(s|s-1)_A\\
-(P(s+1|s)_A \ar[rr]^{(-1)^{(n+k)(l+s+1)+n+l+1}} &&P(s|s-2)_A \ar[rr]^{(-1)^{l+s+1}} &&P(s|s-1)_A)}\end{displaymath}
and the two maps cancel.\\
Now we check that the triangle including the left diagonal map and $A$ commutes
\begin{displaymath} \xymatrix{ P(s+1|s)_A \ar[rr]^{(-1)^{(n+k)(l+s+1)+n+l+1}} &&P(s|s-2)_A  \ar @{} [d] |{=} \ar[rr]^{(-1)^{k+l+1}} &&P(s+1|s-2)_A\\
 P(s+1|s)_A \ar[rr]^{(-1)^{(n+k)(l+s)}} &&P(s+1|s-2)_A &&}
\end{displaymath} 
Last check that
\begin{displaymath} \xymatrix{ P(s+1|s)_A \ar[rr] &&P(s|s-2)_A  \ar @{} [d] |{=} \ar[rr] &&P(s|s-3)_A\\
 &&0&&}
\end{displaymath}
and therefore we have checked this case completly.

\item $V=P(s|t)_B$ and $s>t+3$:

As above $H$ and $A$ are zero on $P(s|t)_B$ and $H$ is zero on all possible maps occuring in $Y$, so there is nothing to check.
\item $V=P(s|s-3)_B$:

\begin{displaymath}
\xymatrix{\cdots \ar[r] &P(s|s-3)_B \ar[d] \ar[r] \ar[ld] &P(s|s-2)_B\ar[ld]^{(-1)^{m+n+k+l}}\\
\cdots \ar[r] &P(s-1|s-2)_A \ar[r] & \cdots}
\end{displaymath}
Here we have to check that this triangle commutes:
\begin{displaymath} \xymatrix{ P(s|s-3)_B \ar[rr]^{(-1)^{n+m}} &&P(s|s-2)_B  \ar @{} [d] |{=} \ar[rr]^{(-1)^{(n+k)(l+s)+k+l+n+m+k+l}} &&P(s-1|s-2)_A\\
 P(s|s-3)_B \ar[rr]^{(-1)^{(n+k)(l+s)}} &&P(s-1|s-2)_A}
\end{displaymath}

\item   $V=P(s|s-2)_B$:

Last we have to check
\begin{displaymath}
\xymatrix{\cdots \ar[r] &P(s|s-2)_B \ar[d] \ar[r] \ar[ld] &P(s+1|s)_A  \ar[ld]^{(-1)^{n+m+k+l}}\\
P(s-1|s-2)_A \ar[r] & {\begin{matrix} P(s-1|s-3)_B\\P(s|s-2)_A  \end{matrix}}  \ar[r] & \cdots}
\end{displaymath} 
Checking with the tables in section \ref{sec:mapproj}, one gets

\begin{displaymath} \xymatrix{ P(s|s-2)_B \ar[rr]^{(n+k)(l+s)+k+l} &&P(s-1|s-2)_A  \ar @{} [d] |{=} \ar[rr]^{(-1)^{k+l}} &&P(s|s-2)_A\\
-(P(s|s-2)_B \ar[rr]^{(-1)^{n+m+1}} &&P(s+1|s)_A \ar[rr]^{(-1)^{(n+k)(l+s+1)+n+l+1+n+m+k+l}} &&P(s|s-2)_A)\\
 +P(s|s-2)_B \ar[rr]^{(-1)^{(n+k)(l+s)}} &&P(s|s-2)_A &}
\end{displaymath}
since the last morphism is passing over $P(s|s-1)$.
Now we check that the morphism involving only the left triangle commutes:
\begin{displaymath} \xymatrix{ P(s|s-2)_B \ar[rr]^{(n+k)(l+s)+k+l} &&P(s-1|s-2)_A  \ar @{} [d] |{=} \ar[rr]^{(-1)^{l+s+1}} &&P(s-1|s-3)_B\\
 P(s|s-2)_B \ar[rr]^{(-1)^{(n+k)(l+s)+k+s+1}} &&P(s|s-3)_B &}
\end{displaymath}
\end{enumerate}
So we have checked all possible cases.
\end{proof}


\bibliographystyle{alpha}

\begin{thebibliography}{LPWZ09}

\bibitem[AK08]{Asae2008}
M.~Asaeda and M.~Khovanov.
\newblock {Notes on link homology}.
\newblock {\em Arxiv preprint arXiv:0804.1279}, 2008.

\bibitem[ARS97]{Ausl97}
M.~Auslander, I.~Reiten, and S.O. Smal{\o}.
\newblock {Representation theory of Artin algebras, volume 36 of Cambridge
  Studies in Advanced Mathematics}.
\newblock {\em CUP, Cambridge}, 1997.

\bibitem[BC90]{Boe90}
B.D. Boe and D.H. Collingwood.
\newblock {Multiplicity free categories of highest weight representations}.
\newblock {\em Communications in Algebra}, 18(4):947--1032, 1990.

\bibitem[BGS96]{Beil1996}
A.~Beilinson, V.~Ginzburg, and W.~Soergel.
\newblock Koszul duality patterns in representation theory.
\newblock {\em J. Amer. Math. Soc.}, 9(2):473--527, 1996.

\bibitem[Bia04]{Biag2004}
R.~Biagioli.
\newblock Closed product formulas for extensions of generalized {V}erma
  modules.
\newblock {\em Trans. Amer. Math. Soc.}, 356(1):159--184 (electronic), 2004.

\bibitem[Boe85]{Boe85}
B.D. Boe.
\newblock {Homomorphisms between generalized Verma modules}.
\newblock {\em Transactions of the American Mathematical Society}, pages
  791--799, 1985.

\bibitem[BS08a]{Brun12008}
J.~Brundan and C.~Stroppel.
\newblock Highest weight categories arising from {K}hovanov's diagram algebra
  {I}: cellularity.
\newblock {\em arXiv}, 806, 2008.

\bibitem[BS08b]{brun32008}
J.~Brundan and C.~Stroppel.
\newblock Highest weight categories arising from {K}hovanov's diagram algebra
  {III}: category {$\O$}.
\newblock {\em arXiv}, 812, 2008.

\bibitem[BS10]{Brun22008}
J.~Brundan and C.~Stroppel.
\newblock {Highest weight categories arising from {K}hovanov's diagram algebra
  {II}: {K}oszulity}.
\newblock {\em Transformation Groups}, 15(1):1--45, 2010.

\bibitem[CPS88]{Clin1988}
E.~Cline, B.~Parshall, and L.~Scott.
\newblock {Finite dimensional algebras and highest weight categories}.
\newblock {\em J. reine angew. Math}, 391:85--99, 1988.

\bibitem[Del77]{Delo77}
P.~Delorme.
\newblock {Extensions dans la c{\'a}tegorie $\O$ de Bernstein-Gelfand-Gelfand}.
\newblock {\em Applications, preprint, Paris}, 1977.

\bibitem[ES87]{Enri1987}
T.~J. Enright and B.~Shelton.
\newblock Categories of highest weight modules: applications to classical
  {H}ermitian symmetric pairs.
\newblock {\em Mem. Amer. Math. Soc.}, 67(367):iv+94, 1987.

\bibitem[GL04]{Grah2004}
J.~J. Graham and G.~I. Lehrer.
\newblock Cellular algebras and diagram algebras in representation theory.
\newblock In {\em Representation theory of algebraic groups and quantum
  groups}, volume~40 of {\em Adv. Stud. Pure Math.}, pages 141--173. Math. Soc.
  Japan, Tokyo, 2004.

\bibitem[GM96]{Gelf88}
S.~I. Gelfand and Y.~I. Manin.
\newblock {\em Methods of homological algebra}.
\newblock Springer-Verlag, Berlin, 1996.
\newblock Translated from the 1988 Russian original.

\bibitem[HM10]{Hu2010}
J.~Hu and A.~Mathas.
\newblock {Graded cellular bases for the cyclotomic Khovanov-Lauda-Rouquier
  algebras of type A}.
\newblock {\em Advances in Mathematics}, 2010.

\bibitem[Hum08]{Hump08}
J.~E. Humphreys.
\newblock {\em Representations of semisimple {L}ie algebras in the {BGG}
  category {$\O$}}, volume~94 of {\em Graduate Studies in Mathematics}.
\newblock American Mathematical Society, Providence, RI, 2008.

\bibitem[Jan79]{Jan79}
J.~C. Jantzen.
\newblock {\em Moduln mit einem h\"ochsten {G}ewicht}, volume 750 of {\em
  Lecture Notes in Mathematics}.
\newblock Springer, Berlin, 1979.

\bibitem[Kad80]{Kade79}
T.~V. Kadei{\v{s}}vili.
\newblock On the theory of homology of fiber spaces.
\newblock {\em Uspekhi Mat. Nauk}, 35(3(213)):183--188, 1980.
\newblock International Topology Conference (Moscow State Univ., Moscow, 1979).

\bibitem[Kad88]{kade1988}
T.~Kadeishvili.
\newblock {The structure of the $A_\infty$-algebra, and the Hochschild and
  Harrison cohomologies}.
\newblock {\em Trudy Tbiliss. Mat. Inst. Razmadze Akad. Nauk Gruzin. SSR},
  91:19--27, 1988.

\bibitem[Kel01]{Kell2001}
B.~Keller.
\newblock Introduction to {$A$}-infinity algebras and modules.
\newblock {\em Homology Homotopy Appl.}, 3(1):1--35, 2001.

\bibitem[Kel02]{Kell2002}
B.~Keller.
\newblock {$A$}-infinity algebras in representation theory.
\newblock In {\em Representations of algebra. {V}ol. {I}, {II}}, pages 74--86.
  Beijing Norm. Univ. Press, Beijing, 2002.

\bibitem[KS01]{Kont2001}
M.~Kontsevich and Y.~Soibelman.
\newblock Homological mirror symmetry and torus fibrations.
\newblock In {\em Symplectic geometry and mirror symmetry ({S}eoul, 2000)},
  pages 203--263. World Sci. Publ., River Edge, NJ, 2001.

\bibitem[KS02]{khovanov2002quivers}
M.~Khovanov and P.~Seidel.
\newblock {Quivers, Floer cohomology, and braid group actions}.
\newblock {\em Journal of the American Mathematical Society}, 15(1):203--271,
  2002.

\bibitem[LH03]{lefe2003}
K.~Lefevre-Hasegawa.
\newblock {Sur les $A_\infty$-cat{\'e}gories}.
\newblock {\em These de doctorat, Universit{\'e} Denis Diderot--Paris}, 7,
  2003.

\bibitem[LPWZ09]{Lu2009}
D.-M. Lu, J.~H. Palmieri, Q.-S. Wu, and J.~J. Zhang.
\newblock {$A$}-infinity structure on {E}xt-algebras.
\newblock {\em J. Pure Appl. Algebra}, 213(11):2017--2037, 2009.

\bibitem[LS81]{Lasc1981}
A.~Lascoux and M.P. Sch{\"u}tzenberger.
\newblock {Polyn{\^o}mes de Kazhdan-Lusztig pour les grassmanniennes}.
\newblock {\em Ast{\'e}risque}, 87(88):249--266, 1981.

\bibitem[Mer99]{Merk99}
S.~A. Merkulov.
\newblock Strong homotopy algebras of a {K}\"ahler manifold.
\newblock {\em Internat. Math. Res. Notices}, (3):153--164, 1999.

\bibitem[MOS09]{Mazo2009}
V.~Mazorchuk, S.~Ovsienko, and C.~Stroppel.
\newblock Quadratic duals, {K}oszul dual functors, and applications.
\newblock {\em Trans. Amer. Math. Soc.}, 361(3):1129--1172, 2009.

\bibitem[MP95]{Mood95}
R.V. Moody and A.~Pianzola.
\newblock {\em {Lie algebras with triangular decompositions}}.
\newblock Wiley, 1995.

\bibitem[PP05]{Poli05}
A.~Polischuk and L.~Positselski.
\newblock {Quadratic Algebras, a book}.
\newblock {\em AMS University Lecture Series}, 2005.

\bibitem[Sch81]{schm81}
W.~Schmid.
\newblock {Vanishing theorems for Lie algebra cohomology and the cohomology of
  discrete subgroups of semisimple Lie groups}.
\newblock {\em Advances in Mathematics}, 41(1):78--113, 1981.

\bibitem[She88]{Shel1988}
B.~Shelton.
\newblock Extensions between generalized {V}erma modules: the {H}ermitian
  symmetric cases.
\newblock {\em Math. Z.}, 197(3):305--318, 1988.

\bibitem[Str03]{Stro03}
C.~Stroppel.
\newblock {Category $\O$: gradings and translation functors}.
\newblock {\em Journal of Algebra}, 268(1):301--326, 2003.

\bibitem[Str05]{Stro2005}
C.~Stroppel.
\newblock Categorification of the {T}emperley-{L}ieb category, tangles, and
  cobordisms via projective functors.
\newblock {\em Duke Math. J.}, 126(3):547--596, 2005.

\end{thebibliography}

\end{document}